\documentclass{amsart}
\usepackage[dvipsnames]{xcolor}
\usepackage{amsmath,amssymb,amsthm}
\usepackage{amsfonts,amssymb}
\usepackage{mathrsfs}
\usepackage{enumitem}
\usepackage{mathtools}

\usepackage{ stmaryrd }

\usepackage{natbib}

\usepackage{tikz}
\usepackage{tikz-cd}
\tikzset{overcross/.style={double, line width=1.5, white, double=#1, double distance=\knotlinewidth},
    overcross/.default={black},
    knot/.style={line width=\knotlinewidth, baseline=-.5ex}}
\newcommand{\knotlinewidth}{.7pt}
\usetikzlibrary { decorations.pathmorphing, decorations.pathreplacing, decorations.shapes, }
\usetikzlibrary{matrix}

\usepackage{float}
\usepackage{graphicx, accents}
\usepackage{caption}
\usepackage[margin=10pt]{subcaption}
\usepackage{tabularx}
\usepackage{booktabs}
\usepackage{array}
\usepackage{diagbox}
\usepackage{nicematrix}
\usepackage{hhline}
\usepackage{multirow}
\usepackage{tabularray}
\usepackage[normalem]{ulem}
\usepackage{wrapfig}

\usepackage{xfrac}

\usepackage{xcolor}
  \usepackage[table]{xcolor}

\newcommand{\highlight}[1]{%
  \colorbox{yellow!30}{$#1$}}

\makeatletter
\newcommand{\subalign}[1]{%
  \vcenter{%
    \Let@ \restore@math@cr \default@tag
    \baselineskip\fontdimen10 \scriptfont\tw@
    \advance\baselineskip\fontdimen12 \scriptfont\tw@
    \lineskip\thr@@\fontdimen8 \scriptfont\thr@@
    \lineskiplimit\lineskip
    \ialign{\hfil$\m@th\scriptstyle##$&$\m@th\scriptstyle{}##$\hfil\crcr
      #1\crcr
    }%
  }%
}
\makeatother

\makeatletter
\newcommand{\raisemath}[1]{\mathpalette{\raisem@th{#1}}}
\newcommand{\raisem@th}[3]{\raisebox{#1}{$#2#3$}}
\makeatother

\makeatletter
\DeclareRobustCommand{\cev}[1]{%
  \mathpalette\do@cev{#1}%
}
\newcommand{\do@cev}[2]{%
  \fix@cev{#1}{+}%
  \reflectbox{$\m@th#1\vec{\reflectbox{$\fix@cev{#1}{-}\m@th#1#2\fix@cev{#1}{+}$}}$}%
  \fix@cev{#1}{-}%
}
\newcommand{\fix@cev}[2]{%
  \ifx#1\displaystyle
    \mkern#23mu
  \else
    \ifx#1\textstyle
      \mkern#23mu
    \else
      \ifx#1\scriptstyle
        \mkern#22mu
      \else
        \mkern#22mu
      \fi
    \fi
  \fi
}
  
\makeatother

\newcommand{\nocontentsline}[3]{}
\newcommand\stoptoc{%
   \let\origcontentsline\addcontentsline
   \let\addcontentsline\nocontentsline
}
\newcommand\resumetoc{%
   \let\addcontentsline\origcontentsline
}

\usepackage{hyperref}
\newtheorem{theorem}{Theorem}[section]
\newtheorem{conjecture}{Conjecture}[section]
\newtheorem{corollary}[theorem]{Corollary}
\newtheorem{proposition}[theorem]{Proposition}
\newtheorem{lemma}[theorem]{Lemma}

\theoremstyle{definition}
\newtheorem{definition}[theorem]{Definition}
\newtheorem{example}[theorem]{Example}
\newtheorem{question}[theorem]{Question}

\theoremstyle{remark}
\newtheorem*{remark}{Remark}

\newcommand{\Z}{\mathbb{Z}}
\newcommand{\N}{\mathbb{N}}

\newcommand{\F}{\mathbb{F}}
\newcommand{\RP}{\mathbb{RP}}
\newcommand{\Kh}{\mathrm{Kh}}

\newcommand{\lk}{\mathrm{lk}}

\newcommand{\abs}[1]{\left| #1 \right|}

\newcommand{\cut}{\mathsf{Cut}}
\newcommand{\cyc}{\mathsf{Cyc}}
\newcommand{\G}{\mathbb{G}}

\newcommand{\blue}[1]{\textcolor{blue}{#1}}
\newcommand{\red}[1]{\textcolor{red}{#1}}
\newcommand{\green}[1]{\textcolor{ForestGreen}{#1}}
\newcommand{\orange}[1]{\textcolor{orange}{#1}}

\usepackage{geometry}

\title{Myopic Tutte polynomials and Khovanov homology in $\RP^3$}

\author{Keegan Boyle}
\address{Department of Mathematical Sciences, New Mexico State University, Las Cruces, NM 88003}
\email{kboyle@nmsu.edu}
\author{Dean Spyropoulos}
\address{Department of Mathematical Sciences, New Mexico State University, Las Cruces, NM 88003}
\email{dspyro@nmsu.edu}
\date{}

\begin{document}

\begin{abstract}
We present a ``myopic'' Tutte polynomial for graphs on $\RP^2$ which takes only nullhomologous spanning subgraphs as input. It recovers the generalized Krushkal polynomial and Drobotukhina's analogue of the Jones polynomial for alternating, nullhomologous links in $\RP^3$. We use this myopic Tutte polynomial to prove an analogue of the Kauffman-Murasugi-Thistlethwaite Theorem, relating the Jones polynomial of an alternating link to certain refinements of the crossing number. Finally, we construct a spanning tree model for the Khovanov homology of nullhomologous links, mirroring work by Champanerkar-Kofman and Wehrli for links in $S^3$. For alternating links, we use our model to prove that the Khovanov homology in $\mathbb{Z}/2\mathbb{Z}$ coefficients is determined entirely by the Jones polynomial and signatures of the link.
\end{abstract}

\maketitle

\vspace{-0.5cm}

\tableofcontents

\vspace{-1cm}

\section{Introduction}

Recently, there has been a surge of interest in the Khovanov homology \cite{MR1740682} of links in $\RP^3$ \cite{MR2113902,MR2320160,MR3189291}; for instance, see \cite{MR4904031, MR4862260, MR5083272, MR4919586, ren2025intrinsic, MR4983443, rushworth2026some}. In light of the growing curiosity, we revisit the fundamental problem of characterizing Khovanov thinness \cite{MR1917056, MR2034399}, now in $\RP^3$. In $S^3$, a link is called \emph{Khovanov-thin} if its Khovanov homology is supported on precisely two diagonals. Such links are considered ``boring'' in a sense, as their Khovanov homology in $\F := \Z / 2\Z$-coefficients is no stronger than their (unnormalized) Jones polynomial $\hat{J}_L(t)$ (obtained as the graded Euler characteristic of the homology) together with the signature $\sigma(L)$ (obtained from intercepts of the diagonals).

To begin our investigations in real projective space, view $\RP^3 - \infty$ as a twisted interval bundle over $\RP^2$, so that any link $L\subset \RP^3$ may be represented by a projection $D\subset \RP^2$, called a \emph{diagram} of $L$. For example, below is a diagram for a three-crossing knot in $\RP^3$. To its right is its Khovanov homology in $\F$-coefficients.\footnote{We work over $\F := \Z/2\Z$-coefficients, corresponding to the theory of Asaeda, Przytycki, and Sikora \cite{MR2113902}. Gabrov\v sek \cite{MR3189291} fixed a sign convention extending the construction to the integers (thereafter completed by Manolescu-Willis \cite{MR4904031}).}

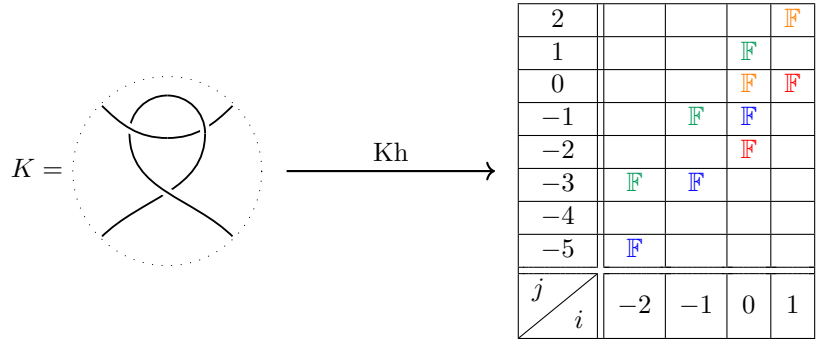
\[
\tikz{
\node(K) at (0,0) {$K = \tikz[baseline={([yshift=-.5ex]current bounding box.center)}, scale=.25]{
	\draw[knot, overcross] (0,1.75) to[out=0, in=-135] (5*1.41/2, 5*1.41/2);
	\draw[knot, overcross] (2,2) to[out=-90, in=45] (-5*1.41/2, -5*1.41/2);
	\draw[knot, overcross] (5*1.41/2, -5*1.41/2) to[out=135, in=-90] (-2,2);
	\draw[knot, overcross] (-2,2) to[out=90, in=180] (0,4);
	\draw[knot, overcross] (-5*1.41/2, 5*1.41/2) to[out=-45, in=180] (0, 1.75);
	\draw[knot, overcross] (0,4) to[out=0, in=90] (2,2);
	\draw[white, ultra thick] (0,0) circle (5cm);
	\draw[line cap=round, dash pattern=on 0pt off 3.5pt] (0,0) circle (5cm);
}$};
\node(T) at (7,0) {\begin{tabular}{|c||c|c|c|c|} 
\hline
$2$  &              &              &              & $\orange{\mathbb{F}}$  \\ 
\hline
$1$  &              &              & $\green{\mathbb{F}}$ &               \\ 
\hline
$0$  &              &              & $\orange{\mathbb{F}}$ & $\red{\mathbb{F}}$  \\ 
\hline
$-1$ &              & $\green{\mathbb{F}}$ & $\blue{\mathbb{F}}$ &               \\ 
\hline
$-2$ &              &              & $\red{\mathbb{F}}$ &               \\ 
\hline
$-3$ & $\green{\mathbb{F}}$ & $\blue{\mathbb{F}}$ &              &               \\ 
\hline
$-4$ &              &              &              &               \\ 
\hline
$-5$ & $\blue{\mathbb{F}}$ &              &              &               \\ 
\hhline{|=::====|}
   \slashbox{$j$}{$i$}   & $-2$         & $-1$         & $0$          & $1$           \\
\hline
\end{tabular}};
\draw[->, shorten <=5pt, shorten >=5pt, thick] (K) -- node[above, midway]{$\Kh$} (T)
}
\]

Now, this knot does not seem Khovanov-thin in the traditional sense: the homology is supported on four diagonals, represented above by four colors. However, this reflects a feature: the Jones polynomial of some links in $\RP^3$ (as defined by Drobotukhina\footnote{We frequently cite Yulia Drobotukhina, who is now known as Julia Viro.} in \cite{MR1073213}) have exponents which are both integers and half-integers. For example,
\[
J_K(t) = t^{-2} - t^{-1} + 1 - t^{\frac{1}{2}} + t^{\frac{1}{2}}
\]
which can be obtained as the graded Euler characteristic of the Khovanov homology above by taking 
\begin{align*}
\hat{J}_K(q) & = (q+q^{-1}) \cdot J_K(q = -t^{\frac{1}{2}}) \\
& = \blue{q^{-5}} + \green{q} + \red{q^{-2}} - \orange{q^{2}}.
\end{align*}
Moreover, some links in $\RP^3$ have two signatures---the placement of the cool-colored diagonals in the Khovanov homology above reflects the fact that one signature of $K$ is 0, while the placement of the warm-colored diagonals matches the fact that the other signature of $K$ is 1---see Example \ref{ex:maingoeritz} and Proposition \ref{prop:goeritzsig}.

The links in $\RP^3$ which have two Jones polynomials, two signatures, and two Khovanov homologies are precisely the non-local, nullhomologous links. We call $L$ \emph{local} if it is contained within a 3-ball in $\RP^3$; otherwise, it is called \emph{non-local}. We call $L$ \emph{nullhomologous} (resp., \emph{homologically essential}) if $[L] = 0$ (resp. $[L] = 1$) in $H_1(\RP^3) \cong \F$. A diagram $D\subset \RP^2$ is called \emph{alternating} if, along it, underpasses and overpasses alternate if and only if the arc between successive crossings either fails to intersect the boundary of the disk of the diagram, or intersects it in $4k$-many points. A link $L\subset \RP^3$ is called alternating if it has an alternating diagram. For example, the knot $K$ above is non-local, nullhomologous, and alternating. Note that alternating links must be nullhomologous by definition. At the outset, our goal is to prove the following.

\begin{theorem}
\label{thm:main}
If $L$ is a non-local, alternating link in $\RP^3$, then its Khovanov homology in $\Z/2\Z$-coefficients is supported on exactly four diagonals given by $j - 2i = - \sigma_{\pm}(L) \pm 1$. Moreover, it is equivalent to the data $\{J_L(t) = J_L^+(L) + J_L^-(t), \sigma_+(L), \sigma_-(L)\}$.
\end{theorem}

\begin{wrapfigure}[10]{R}{0.25\textwidth}
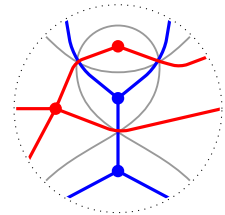

\vspace{-15pt}
\centering
\begin{minipage}{0.25\textwidth}
\[\tikz[scale=.275]{
	\begin{scope}[white!60!black]
	\draw[knot] (0,1.75) to[out=0, in=-135] (5*1.41/2, 5*1.41/2);
	\draw[knot] (2,2) to[out=-90, in=45] (-5*1.41/2, -5*1.41/2);
	\draw[knot] (5*1.41/2, -5*1.41/2) to[out=135, in=-90] (-2,2);
	\draw[knot] (-2,2) to[out=90, in=180] (0,4);
	\draw[knot] (-5*1.41/2, 5*1.41/2) to[out=-45, in=180] (0, 1.75);
	\draw[knot] (0,4) to[out=0, in=90] (2,2);
	\end{scope}
	\begin{scope}[blue, very thick]
		\fill (0, 0.5) circle (.3cm);
		\fill (0, -3) circle (.3cm);
		\draw (0, 0.5) -- (0, -3);
		\draw[rounded corners = 3mm] (0, 0.5) -- (-2.15,2.25) -- (-5*0.5,5*0.866);
		\draw[rounded corners = 3mm] (0, 0.5) -- (2.15,2.25) -- (5*0.5,5*0.866);
		\draw (-5*0.5,-5*0.866) -- (0,-3);
		\draw (5*0.5,-5*0.866) -- (0,-3);
	\end{scope}
	\begin{scope}[red, very thick]
		\fill (0, 3) circle (.3cm);
		\fill (-3, 0) circle (.3cm);
		\draw[rounded corners = 1mm] (0,3) -- (-2.05,2.25) -- (-3, 0) -- (-5,0);
		\draw (-3, 0) -- (-5*0.866, -5*0.5);
		\draw[rounded corners = 1mm] (0,3) -- (2.05,2.25) -- (3, 2) -- (5*0.866, 5*0.5);
		\draw[rounded corners] (-3, 0) -- (0,-1.15) -- (5,0);
	\end{scope}
	\draw[white, ultra thick] (0,0) circle (5cm);
	\draw[dotted] (0,0) circle (5cm);
}\]
\end{minipage}
\caption{Tait graphs.}
\label{fig:pdtaits}
\end{wrapfigure}
While such a result is likely attainable via the extension of Lee homology provided by Manolescu-Willis \cite{MR4904031}, we select to take a detour based on the following observation. First, the nullhomologous links in $\RP^3$ are precisely the ones which admit checkerboard-colorable diagrams. Hence, they admit Tait graphs which are embedded graphs in $\RP^2$. For instance, in Figure \ref{fig:pdtaits} we picture the two Tait graphs associated to the diagram for $K$ above. Mirroring Thistlethwaite \cite{MR899051}, one can ask if the two Jones polynomials of an alternating link can be obtained by some ``Tutte-like'' polynomial for the two Poincar\'e dual Tait graphs of $D$. To the affirmative, we present a polynomial invariant $\psi_G \in \mathbb{Z}[x,y]$ of graphs $G \hookrightarrow \RP^2$ and prove the following.

\begin{theorem}
\label{thm:jonespolynomialalternatingtutte}
If $D$ is a connected, reduced, alternating diagram for a non-split, non-local alternating link in $\RP^3$, and $G_+$ and $G_-$ are its two checkerboard graphs, then there are $k_+, k_- \in \mathbb{Z}[\frac{1}{2}]$ such that
\[
J_L(t) = (-1)^{w(D)} \left( t^{k_+} \psi_{G_+}(-t, -t^{-1}) + t^{k_-} \psi_{G_-}(-t^{-1}, -t)\right)
\]
and $k_+ - k_- = \frac{1}{2}$.
\end{theorem}

As with the ordinary Tutte polynomial, ours admits three definitions, the most notable being a spanning tree expansion. The relevance to Khovanov homology and Theorem \ref{thm:main} is provided by Champanerkar-Kofman \cite{MR2480298} and Wehrli \cite{MR2477595}, who showed that for any link $L$ in $S^3$, there is a cochain complex generated by spanning trees of the Tait graph of $L$ whose homology is the reduced Khovanov homology of $L$. So too for non-local nullhomologous links in $\RP^3$; we use their approach to prove Theorem \ref{thm:main}.

We call the polynomial $\psi_G$ the \emph{myopic Tutte polynomial} since it takes only the homologically trivial spanning subgraphs of $G$ as input. The literature on extensions of the Tutte polynomial to graphs embedded into surfaces is vast---we refer the interested reader to the surveys \cite{MR4972578} and \cite{MR4952626}. The most general of these extensions is the generalized Krushkal polynomial \cite{MR2769192, MR3739494}, a four-variable polynomial $K_G \in \mathbb{Z}[x, y, A, B]$ defined for graphs embedded into any surface $\Sigma$. An interesting feature of our presumably weak polynomial is that it is equivalent to $K_G$ when $\Sigma = \RP^2$. In particular, this implies that in $\RP^2$, $K_G$ admits a spanning tree expansion (in arbitrary surfaces, only a quasi-tree expansion is guaranteed; see \cite[Theorem 4.4]{MR3739494}). Let $G^*$ denote the Poincar\'e dual of $G$ in $\RP^2$. Using the property that $K_{G^*}(x,y, A, B) = K_{G}(y, x, B, A)$, we prove the following.

\begin{theorem}
\label{thm:equalkrushkal}
For any connected, non-local graph $G \subset \RP^2$, the polynomial $K_G(x, y, A, B)$ is equivalent to $\{\psi_G(x,y), \psi_{G^*}(x,y)\}$. Namely,
\[
K_G(x,y,0,1) = \psi_G(x,y)
\]
and
\[
K_G(x,y, A, B) = \sqrt{B} \, \psi_G(x,y) + \sqrt{A} \, \psi_{G^*}(y,x).
\]
\end{theorem}

The myopic Tutte polynomial description of the Jones polynomial of alternating links also allows us to make some overlooked observations. Using the notation in Theorem \ref{thm:jonespolynomialalternatingtutte}, we write
\[
J_L^+(t) := (-1)^{w(D)} t^{k_+} \psi_{G_+}(-t, -t^{-1}) 
\qquad \text{and} \qquad
J_L^-(t) := (-1)^{w(D)} t^{k_-} \psi_{G_-}(-t^{-1}, -t)
\]
so that $J_L(t) = J_L^+(t) + J_L^-(t)$ for non-local alternating links. Assume that $D$ is the connected, reduced, alternating diagram for $L$ inducing the Tait graphs $G_+$ and $G_-$. If $G$ is a graph, let $E(G)$ denote the set of edges of $G$. Define the \emph{projective rank} of a graph $G\subset \RP^2$ to be
\[
\alpha(G) := \min_{E \subset E(G)}\{\abs{E} : G - E~\text{is local}\}.
\]
See Example \ref{ex:alphacomp} for a sample computation of $\alpha(G)$. Denote the projective ranks of $G_+$ and $G_-$ by $\alpha^+(D)$ and $\alpha^-(D)$ respectively. For a Laurent polynomial $P$, let $d_{\max} P$ and $d_{\min} P$ denote the maximum and minimum exponent of $P$ respectively. Let $c(L)$ denote the crossing number of $L$.

\begin{theorem}
\label{thm:altbreadth}
Assume that $L$ is a non-split, non-local, alternating link in $\RP^3$ with connected, reduced, alternating diagram $D \subset \RP^2$. Then
\[
d_{\max} J_L^\pm - d_{\min} J_L^\pm = c(L) - \alpha^\pm(D).
\]
In particular, $\alpha^\pm(D)$ are invariants of the non-local alternating link $L \subset \RP^3$.
\end{theorem}

\begin{wrapfigure}[11]{R}{0.3\textwidth}
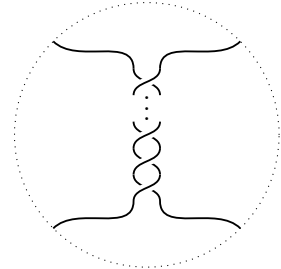

\vspace{-15pt}
\begin{minipage}{0.3\textwidth}
\[
\tikz[baseline={([yshift=-.5ex]current bounding box.center)}, scale=0.35]
{
\draw[knot] (1,0) to[out=90, in=-90] (0,1);
\draw[knot] (1,1) to[out=90,in=-90] (0,2);
\draw[knot] (1,2) to[out=90,in=-90] (0,3);
\draw[knot] (1,4) to[out=90, in=-90] (0,5);
\draw[knot, overcross] (0,0) to[out=90, in=-90] (1,1);
\draw[knot, overcross] (0,1) to[out=90, in=-90] (1,2);
\draw[knot, overcross] (0,2) to[out=90, in=-90] (1,3);
\node at (0.5,3.8) {$\vdots$};
\draw[knot, overcross] (0,4) to[out=90, in=-90] (1,5);
\draw[knot] (0.5+5*0.707, 2.5+5*0.707) to[out=-135, in=90] (1,5);
\draw[knot] (0.5-5*0.707, 2.5+5*0.707) to[out=-45, in=90] (0,5);
\draw[knot] (0.5+5*0.707, 2.5-5*0.707) to[out=135, in=-90] (1,0);
\draw[knot] (0.5-5*0.707, 2.5-5*0.707) to[out=45, in=-90] (0,0);
\draw[dotted] (0.5,2.5) circle(5cm);
}
\]
\end{minipage}
\caption{The link $\Upsilon_n$.}
\label{fig:upsilon}
\end{wrapfigure}
Kauffman \cite{MR899057}, Murasugi \cite{MR895570}, and Thistlethwaite \cite{MR899051} proved that the breadth of the Jones polynomial of an alternating link equals its crossing number, thus the Jones polynomial detects the unknot among alternating links. Using Theorem \ref{thm:altbreadth}, we are able to show that Drobotukhina's analogue of the Jones polynomial detects a set $\mathscr{L}$ of a few infinite families of (unoriented) links among alternating links in $\RP^3$; see (\ref{eq:scriptell}) and Figure \ref{fig:betafams}. The simplest family of links in $\mathscr{L}$ is depicted in Figure \ref{fig:upsilon}, and are denoted $\Upsilon_n$ for $n\in \Z$, where $\abs{n}$ is the crossing number of the link. If $n < 0$, $\Upsilon_n$ means the mirror of the link in Figure \ref{fig:upsilon}.

\begin{theorem} \label{thm:detection}
If $L\in \mathscr{L}$ and $L'$ is an alternating link, then $L$ is equivalent to $L'$ as an unoriented link whenever $J_L(t) = J_{L'}(t)$.
\end{theorem}

For example, to prove Theorem \ref{thm:detection} for the family $\{\Upsilon_n\}$, we note that $\min \{\text{breadth}\, J^+_L, \text{breadth}\, J^-_L\} = 0$ if and only if $L$ is $\Upsilon_n$ for some $n$. This is because $\abs{E(G)} = \alpha(G)$ if and only if $G$ is the single vertex graph whose edges are all non-local loops. Then, the result follows if there are no duplicates of the Jones polynomial within this family. The other families belonging to $\mathscr{L}$ are described in \S \ref{ss:tabulation}; these are some other alternating links which have $\min \{\text{breadth} \, J^+_L, \text{breadth} \, J^-_L\} \le 2$.

\subsection{Relation to the Jones unknotting conjecture}

A conjectural generalization of Theorem \ref{thm:altbreadth} to non-alternating links would be strong enough to imply to Jones unknotting conjecture. In order to state this carefully we need the following definition. 

For an arbitrary nullhomologous link $L\subset \RP^3$, we write $J_L(t) = J^\Re_L(t) + J^\Im_L(t)$ to separate the portions of the Jones polynomial with integral and half-integral exponents respectively (when $L$ is local, one summand is zero). For any diagram $D$ of $L$, Proposition \ref{prop:KBgraphpolynomial} says that, when $L$ is non-local, these two parts come from the spanning trees of a Tait graph for $D$ and its dual. Therefore, it makes sense to write $\alpha_{\Re}(D)$ and $\alpha_{\Im}(D)$ to mean the projective rank of the graph contributing to $J^\Re_L$ and $J^\Im_L$ respectively.

\begin{definition} \label{def:localcrossingno}
Let $\varepsilon \in \{\Re, \Im\}$ and $L \subset \RP^3$ be a nullhomologous link with diagram $D$. Let $c(D)$ be the crossing number of $D$. The \emph{$\varepsilon$ projective crossing number} $\alpha_{\varepsilon}(L)$ is defined as 
\[
\alpha_{\varepsilon}(L)  := \min_{D} \alpha_{\varepsilon}(D),
\]	
and the \emph{$\varepsilon$ local crossing number} $\beta_{\varepsilon}$ is defined as
\[
\beta_{\varepsilon}(L) : = \min_{D} [c(D) - \alpha_{\varepsilon}(D)],
\]	
where both minimums are taken over all diagrams $D$ in $\RP^2$ which represent $L$. 
\end{definition}

\begin{remark}
Note that $\alpha_\varepsilon(D)$ counts a subset of the crossings of $D$. Resolving $D$ at these crossings produces a link for which the checkerboard surface is now local. In particular, $\alpha_\varepsilon(L)$ can be thought of as the minimum number of crossings, over all diagrams, which need to be resolved (in the direction consistent with a checkerboard surface of the indicated parity) to obtain a local link. 
\end{remark}

Although there are obvious inequalities $0 \leq \alpha_{\varepsilon}(L) \leq c(L)$ and $0 \leq \beta_{\varepsilon}(L) \leq c(L)$, we don't know of any relationship between $\alpha$ and $\beta$. We can start by considering the possible generalizations of Theorem  \ref{thm:altbreadth} to non-alternating links. We first observe that two possible generalizations are false.

\begin{example} Perhaps the most straightforward generalization of Theorem \ref{thm:altbreadth} would be that 
	\[ 
	d_{\max}J_L^{\varepsilon} - d_{\min}J_L^{\varepsilon} \leq \beta_{\varepsilon}(L).
	\] 
However, we do not have to look far to see that this is false. Consider the first non-alternating link $4_3^2$ of Drobotukhina's table \cite{MR1296890}; a diagram of $4_3^2$ with its associated (signed) Tait graphs is provided in Figure \ref{fig:guess1iswrong}.
\begin{figure}[ht]
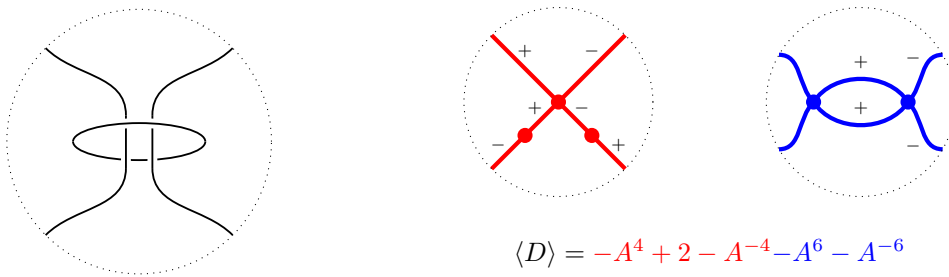

\begin{minipage}{.3\linewidth}
\[
\tikz[baseline={([yshift=-.5ex]current bounding box.center)}, scale=0.35]
{
\draw[knot] (-2,2.5) arc (-180:0:2.5cm and .7cm);
\draw[knot, overcross] (1,1.5) -- (1,3.5);
\draw[knot, overcross] (0,1.5) -- (0,3.5);
\draw[knot] (0.5+5*0.707, 2.5+5*0.707) to[out=-135, in=90] (1,3.5);
\draw[knot, overcross] (0.5-5*0.707, 2.5+5*0.707) to[out=-45, in=90] (0,3.5);
\draw[knot] (0.5+5*0.707, 2.5-5*0.707) to[out=135, in=-90] (1,1.5);
\draw[knot] (0.5-5*0.707, 2.5-5*0.707) to[out=45, in=-90] (0,1.5);
\draw[knot,overcross] (-2,2.5) arc (-180:0:2.5cm and -.7cm);
\draw[dotted] (0.5,2.5) circle(5cm);
}
\]
\end{minipage}
\begin{minipage}{.6\linewidth}
\[
\tikz[baseline={([yshift=-.5ex]current bounding box.center)}, scale=0.5]
	{
\begin{scope}[red]
\fill (1.5,1.5) circle (.2cm);
\fill (1.5+1.25*0.7071,1.5-1.25*0.7071) circle (.2cm);
\fill (1.5-1.25*0.7071,1.5-1.25*0.7071) circle (.2cm);
\draw[ultra thick] (1.5-2.5*0.7071,1.5-2.5*0.7071) -- node[black, above, pos=0.75]{\tiny$-$}node[black, above, pos=0.325]{\tiny$+$}node[black, above, pos=0.05]{\tiny$-$} (1.5+2.5*0.7071,1.5+2.5*0.7071);
\draw[ultra thick] (1.5+2.5*0.7071,1.5-2.5*0.7071) -- node[black, above, pos=0.05]{\tiny$+$}node[black, above, pos=0.325]{\tiny$-$}node[black, above, pos=0.75]{\tiny$+$} (1.5-2.5*0.7071,1.5+2.5*0.7071);
\end{scope}
\draw[dotted] (1.5,1.5) circle(2.5);
\begin{scope}[xshift=8cm]
\begin{scope}[blue]
\fill (2.75,1.5) circle (.2cm);
\fill (0.25,1.5) circle (.2cm);
\draw[ultra thick] (2.75,1.5) to[bend left=2cm] node[black, midway, above]{\tiny$+$} (0.25,1.5);
\draw[ultra thick] (2.75,1.5) to[bend right=2cm] node[black, midway, above]{\tiny$+$} (0.25,1.5);
\draw[ultra thick] (2.75,1.5) to[out=45,in=180] node[black, pos=0.8, left]{\tiny$-$} (1.5+0.866*2.5, 1.5+0.5*2.5);
\draw[ultra thick] (2.75,1.5) to[out=-45, in=180] node[black, pos=0.8, left]{\tiny$-$} (1.5+0.866*2.5, 1.5-0.5*2.5);
\draw[ultra thick] (0.25,1.5) to[out=135, in=0] (1.5-0.866*2.5, 1.5+0.5*2.5);
\draw[ultra thick] (0.25,1.5) to[out=-135, in=0] (1.5-0.866*2.5, 1.5-0.5*2.5);
\end{scope}
\draw[dotted] (1.5,1.5) circle(2.5);
\end{scope}
\node at (5.5, -2.5) {$\langle D \rangle =  \red{-A^4 + 2 - A^{-4}} \blue{-A^6 - A^{-6}}$}
    }
\]
\end{minipage}
\caption{The non-local nullhomologous link $4_3^2$ (left) and a computation of $J_K(t)$ associated to each Tait graph (right). Note that $4_3^2$ is non-alternating (and amphichiral).}
\label{fig:guess1iswrong}
\end{figure}
Proposition \ref{prop:KBgraphpolynomial} says that the Kauffman bracket of this diagram $D$ for $4_3^2$ can be broken up into two pieces corresponding to the two parts of the Jones polynomial: one for each Tait graph associated to $D$. A quick check verifies that $w(D) = \pm 4$ for any choice of orientation on $D$. So, for $L = 4_3^2$,
\[
d_{\max}J_L^{\Re} - d_{\min}J_L^{\Re} = 2
\qquad\text{and}\qquad
d_{\max}J_L^{\Im} - d_{\min}J_L^{\Im} = 3.
\]
However, Figure \ref{fig:guess1iswrong} establishes that $\beta_\varepsilon(4_3^2) \le 2$ for both of $\varepsilon \in \{\Re, \Im\}$. (Indeed, Theorem \ref{thm:betatabulation} implies that $\beta_\varepsilon(L) = 2$ for each $\varepsilon$.) Thus
\[
d_{\max}J_L^{\Im} - d_{\min}J_L^{\Im} = 3 > 2 = \beta_\Im(L).
\]

\end{example}
\begin{example}
Alternatively, using Theorem \ref{thm:dromain}, Theorem \ref{thm:altbreadth} can be rewritten as 
\[
(d_{\max}J_L - d_{\min}J_L) - (d_{\max}J_L^{\varepsilon} - d_{\min}J_L^{\varepsilon}) + \dfrac{1}{2} = \alpha_{\varepsilon}(L),
\]
which might suggest the inequality 
\[
(d_{\max}J_L - d_{\min}J_L) - (d_{\max}J_L^{\varepsilon} - d_{\min}J_L^{\varepsilon}) + \dfrac{1}{2} \geq \alpha_{\varepsilon}(L). 
\]
This is also false, with counter example provided by $L = 4_3^2$. Indeed, since $L$ is amphichiral, we have that 
\[
d_{\min} J_L = d_{\min} J_L^\Im
\qquad\text{and}\qquad
d_{\max} J_L = d_{\max} J_L^\Im
\]
so that the left-hand side is equal to zero when $\varepsilon = \Im$. Since $4_3^2$ is non-local, $\alpha_\varepsilon(L) >0$.
\end{example}

Interestingly, reversing each of the inequalities in the examples above gives plausible statements. We list them here as conjectures.

\begin{conjecture} \label{conj:alphabetainequalities}
Let $L \subset \RP^3$ be a nullhomologous link and let $\varepsilon \in \{\Re, \Im\}$. Then 
\begin{equation}
(d_{\max}J_L - d_{\min}J_L) - (d_{\max}J_L^{\varepsilon} - d_{\min}J_L^{\varepsilon}) + \dfrac{1}{2} \leq \alpha_{\varepsilon}(L).
\end{equation}
Furthermore, if $L$ is local, suppose that $L$ is a knot and $\varepsilon = \Re$. Then
\begin{equation}
\label{eq:impliesjones}
d_{\max}J_L^{\varepsilon} - d_{\min}J_L^{\varepsilon} \geq \beta_{\varepsilon}(L), 
\end{equation}
\end{conjecture}

Part (\ref{eq:impliesjones}) of Conjecture \ref{conj:alphabetainequalities} implies the Jones unknotting conjecture. Indeed, suppose that $K \subset S^3$ is a knot with $J_K(t) = 1$. Then viewing $K$ as a local, and hence nullhomologous, knot in $\RP^3$, part (\ref{eq:impliesjones}) of Conjecture \ref{conj:alphabetainequalities} implies that $\beta_{\Re}(K) = 0$. However, there is precisely one infinite family of diagrams in $\RP^3$ with $c(D) = \alpha_{\varepsilon}(D)$, so that $\beta = 0$; see Theorem \ref{thm:betatabulation}. Within this family, it is easy to verify that two diagrams represent the same link if and only if they have the same Jones polynomial. In fact, it follows that Conjecture \ref{conj:alphabetainequalities} implies that the Jones polynomial detects each $\Upsilon_n$, not only the unknot. 

\begin{remark}
In part (\ref{eq:impliesjones}) of Conjecture \ref{conj:alphabetainequalities}, some restriction on local links is necessary. To see this, we appeal to work of Thistlethwaite \cite{MR1831681} and Eliahou-Kauffman-Thistlethwaite \cite{MR1928648} which provides infinite families of prime $k$-component links with Jones polynomial equal to the Jones polynomial of the $k$-component unlink for each $k > 1$. So, let $L$ be a nontrivial 2-component link with Jones polynomial $-t^{-\frac{1}{2}} - t^{\frac{1}{2}}$; hence $d_{\max}J_L^\Im - d_{\min}J_L^\Im = 1$. For (\ref{eq:impliesjones}) to be satisfied, that must mean that $\beta_\Im(L)$ must be 0 or 1. However, in Theorem \ref{thm:betatabulation}, we provide all possible links with $\min\{\beta_\Re(L), \beta_\Im(L)\}$ equal to 0 or 1 (or 2, for that matter). Besides the unlinks, no local links belong to either of these families.
\end{remark}

\subsection{Amendments to the literature: amphichirality of links in $\RP^3$}

In her formative paper investigating the Jones polynomial for links in real projective space, Drobotukhina proved the following theorem, serving as an analogue to the work of Kauffman \cite{MR899057}, Murasugi \cite{MR895570}, and Thistlethwaite \cite{MR899051}.

\begin{theorem}[\cite{MR1073213}]
\label{thm:dromain}
If $L$ is a non-split, non-local link in $\RP^3$, then
\[
d_{\max}J_L(t) - d_{\min}J_L(t) \le c(L) - \frac{1}{2}
\]
with equality when $L$ is an alternating link.
\end{theorem}

Drobotukhina's theorem serves as motivation for our Theorem \ref{thm:altbreadth} and Conjecture \ref{conj:alphabetainequalities}, as we are interested in the possibility and use for similar statements about $J^\Re_L$ and $J^\Im_L$. Before proceeding, we highlight the following immediate corollary of her work, which does not seem to appear in the literature.

\begin{corollary}
\label{cor:amphichiralknots}
Non-local alternating knots in $\RP^3$ are not amphichiral.
\end{corollary}

The proof is provided in \S \ref{ss:drobjp}, following the definition of $J_L(t)$. This corollary is clearly false for links: the single-crossing 2-component link $\Upsilon_1$ is alternating and amphichiral. We also remind the reader that alternating links are assumed to be nullhomologous by definition (for example, both unknots are amphichiral). Figure \ref{fig:non-localnull-homologousamphichiral} depicts the knot $6_{17}$ of Drobotukhina's table \cite{MR1296890}, which is the first non-local nullhomologous amphichiral knot; we conclude that it is non-alternating. Notably, it has zero writhe.

\begin{figure}[ht]
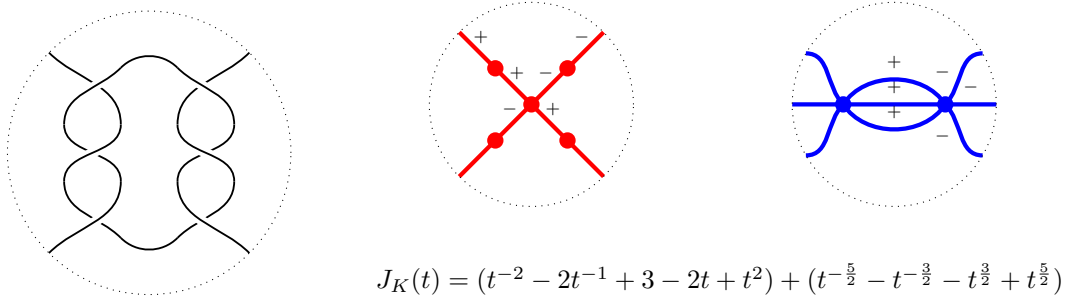

\begin{minipage}{.3\linewidth}
\[
\tikz[baseline={([yshift=-.5ex]current bounding box.center)}, scale=0.75]
	{
\draw[knot] (1,0) to[out=135, in=-90] (0,1);
\draw[knot] (1,1) to[out=90, in=-90] (0,2);
\draw[knot] (1,2) to[out=90, in=-45] (1.5-2.5*0.7071, 1.5+2.5*0.7071);
\draw[knot] (2,0) to[out=45, in=-90] (3,1);
\draw[knot] (2,1) to[out=90, in=-90] (3,2);
\draw[knot] (2,2) to[out=90, in=-135] (1.5+2.5*0.7071, 1.5+2.5*0.7071);
\draw[knot, overcross] (1.5-2.5*0.7071, 1.5-2.5*0.7071) to[out=45, in=-90] (1,1);
\draw[knot, overcross] (0,1) to[out=90, in=-90] (1,2);
\draw[knot, overcross] (0,2) to[out=90, in=-135] (1,3);
\draw[knot, overcross] (1.5+2.5*0.7071, 1.5-2.5*0.7071) to[out=135, in=-90] (2,1);
\draw[knot, overcross] (3,1) to[out=90, in=-90] (2,2);
\draw[knot, overcross] (3,2) to[out=90, in=-45] (2,3);
\draw[knot] (1,0) to[out=-45, in=-135] (2,0);
\draw[knot] (1,3) to[out=45, in=135] (2,3);
\draw[dotted] (1.5,1.5) circle(2.5);
}
\]
\end{minipage}
\begin{minipage}{.6\linewidth}
\[
\tikz[baseline={([yshift=-.5ex]current bounding box.center)}, scale=0.6]
	{
\begin{scope}[red, scale=0.9]
\fill (1.5,1.5) circle (.2cm);
\fill (1.5+1.25*0.7071,1.5+1.25*0.7071) circle (.2cm);
\fill (1.5-1.25*0.7071,1.5+1.25*0.7071) circle (.2cm);
\fill (1.5+1.25*0.7071,1.5-1.25*0.7071) circle (.2cm);
\fill (1.5-1.25*0.7071,1.5-1.25*0.7071) circle (.2cm);
\draw[ultra thick] (1.5-2.5*0.7071,1.5-2.5*0.7071) -- node[black, above, pos=0.85]{\tiny$-$}node[black, above, pos=0.6]{\tiny$-$}node[black, above, pos=0.35]{\tiny$-$} (1.5+2.5*0.7071,1.5+2.5*0.7071);
\draw[ultra thick] (1.5-2.5*0.7071,1.5+2.5*0.7071) -- node[black, above, pos=0.15]{\tiny$+$}node[black, above, pos=0.4]{\tiny$+$}node[black, above, pos=0.65]{\tiny$+$} (1.5+2.5*0.7071,1.5-2.5*0.7071);
\end{scope}
\draw[dotted, scale=0.9] (1.5,1.5) circle(2.5);
\begin{scope}[xshift=8cm]
\begin{scope}[blue, scale=0.9]
\fill (2.75,1.5) circle (.2cm);
\fill (0.25,1.5) circle (.2cm);
\draw[ultra thick] (2.75,1.5) -- node[black, midway, above]{\tiny$+$} (0.25,1.5);
\draw[ultra thick] (2.75,1.5) to[bend left=2cm] node[black, midway, above]{\tiny$+$} (0.25,1.5);
\draw[ultra thick] (2.75,1.5) to[bend right=2cm] node[black, midway, above]{\tiny$+$} (0.25,1.5);
\draw[ultra thick] (2.75,1.5) to[out=45,in=180] node[black, midway, left]{\tiny$-$} (1.5+0.866*2.5, 1.5+0.5*2.5);
\draw[ultra thick] (2.75,1.5) -- node[black, midway, above]{\tiny$-$} (4, 1.5);
\draw[ultra thick] (2.75,1.5) to[out=-45, in=180] node[black, midway, left]{\tiny$-$} (1.5+0.866*2.5, 1.5-0.5*2.5);
\draw[ultra thick] (0.25,1.5) to[out=135, in=0] (1.5-0.866*2.5, 1.5+0.5*2.5);
\draw[ultra thick] (0.25,1.5) -- (-1, 1.5);
\draw[ultra thick] (0.25,1.5) to[out=-135, in=0] (1.5-0.866*2.5, 1.5-0.5*2.5);
\end{scope}
\draw[dotted, scale=0.9] (1.5,1.5) circle(2.5);
\end{scope}
\node at (5.5, -2.5) {$J_K(t) = (t^{-2} - 2t^{-1} + 3 - 2t + t^2) + (t^{-\frac{5}{2}} - t^{-\frac{3}{2}} - t^{\frac{3}{2}} + t^{\frac{5}{2}})$}
    }
\]
\end{minipage}
\caption{A non-local nullhomologous amphichiral knot $K$ (left) and a computation of $J_K(t)$ (right).}
\label{fig:non-localnull-homologousamphichiral}
\end{figure}

In \cite{MR1296890}, Drobotukhina introduced the polynomial
\[
f_L(t) = \left. (-A^3)^{-w'(D)} \langle D \rangle \right|_{A=t^{-1/4}}
\]
where $w'(D)$ is the difference between the number of positive and negative crossings, where only self-crossings of components of $D$ are considered. Then $f_L$ is an invariant of non-oriented links in $\RP^3$. It agrees with $J_L$ when $L$ is a knot, and it provides an obstruction to amphichirality. Now, following the proof of Theorem \ref{thm:jonespolynomialalternatingtutte}, it is just as easy to obtain the following expansion of $f_L$ via myopic Tutte polynomials when $L$ is alternating:
\begin{equation} \label{eq:flsplitting}
f_L(t) = (-1)^{w'(D)} \left(t^{\frac{3w'(D) - k_+'}{4}}\psi_{G_+}(-t, -t^{-1}) + t^{\frac{3w'(D) + k_-'}{4}}\psi_{G_-}(-t^{-1}, -t)\right).
\end{equation}
(Compare (\ref{eq:flsplitting}) with (\ref{eq:jlsplitting}) in the proof of Theorem \ref{thm:jonespolynomialalternatingtutte}.) The reason why Corollary \ref{cor:amphichiralknots} fails for links is that, while $V_L(t) \in \mathbb{Z}[t^{\pm\frac{1}{2}}]$, $f_L(t)$ lives in $\mathbb{Z}[t^{\pm\frac{1}{4}}]$. Indeed, it is still the case that $k_+ - k_- = \frac{1}{2}$ (interpreted appropriately). For example, Figure \ref{fig:5_5^3} shows a calculation of $f_L(t)$ for the second alternating amphichiral link in Drobotukhina's table, labeled $5_5^3$.

\begin{figure}[ht]
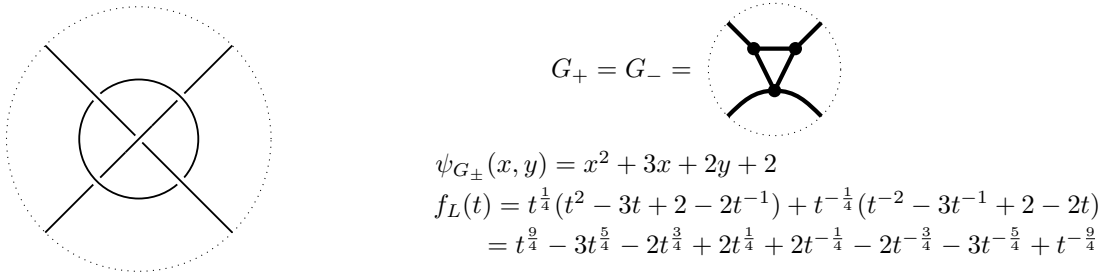

\begin{minipage}{.25\linewidth}
\[
\tikz[baseline={([yshift=-.5ex]current bounding box.center)}, scale=0.7]
	{
\draw[knot] (1.5-2.5*0.7071, 1.5-2.5*0.7071) -- (1.25,1.25);
\draw[knot] (1.75,1.75) -- (1.5+2.5*0.7071, 1.5+2.5*0.7071);
\draw[knot, overcross] (1.5,1.5) circle(1.125);
\draw[knot, overcross] (1.5-2.5*0.7071, 1.5+2.5*0.7071) -- (1.5+2.5*0.7071, 1.5-2.5*0.7071);
\draw[knot, overcross] (1.25,1.25) -- (1.75,1.75);
\draw[dotted] (1.5,1.5) circle(2.5);
    }
\]
\end{minipage}
\begin{minipage}{.7\linewidth}
\[
\qquad
\tikz[baseline={([yshift=-.5ex]current bounding box.center)}, scale=0.5]
	{
\node at (1.5,1) {$G_+ = G_- =$};
\begin{scope}[xshift=4.5cm, scale=0.7]
\fill (1.5+1.125*0.7, 1.5+1.125*0.7) circle (.25cm);
\fill (1.5-1.125*0.7, 1.5+1.125*0.7) circle (.25cm);
\fill (1.5, 1.5-0.8) circle (.25cm);
\begin{scope}[ultra thick]
\draw (1.5+2.5*0.7071, 1.5+2.5*0.7071) -- (1.5+1.125*0.7, 1.5+1.125*0.7);
\draw (1.5-2.5*0.7071, 1.5+2.5*0.7071) -- (1.5-1.125*0.7, 1.5+1.125*0.7);
\draw (1.5, 1.5-0.8) to[out=180,in=45] (1.5-2.5*0.7071, 1.5-2.5*0.7071);
\draw (1.5, 1.5-0.8) to[out=0,in=135]  (1.5+2.5*0.7071, 1.5-2.5*0.7071);
\draw (1.5, 1.5-0.8)
-- (1.5+1.125*0.7, 1.5+1.125*0.7);
\draw (1.5, 1.5-0.8) -- (1.5-1.125*0.7, 1.5+1.125*0.7);
\draw (1.5-1.125*0.7, 1.5+1.125*0.7) -- (1.5+1.125*0.7, 1.5+1.125*0.7);
\end{scope}
\draw[dotted] (1.5,1.5) circle(2.5);
\end{scope}
\node at (1.1, -1.5) {$\psi_{G_\pm}(x,y) = x^2 + 3x + 2y + 2$};
\node at (5.35, -2.5) {$f_L(t) = t^{\frac{1}{4}} (t^2 - 3t + 2 - 2t^{-1}) + t^{-\frac{1}{4}} (t^{-2}-3t^{-1} + 2 - 2t)$};
\node at (6.05, -3.5) {$ = t^{\frac{9}{4}} - 3t^{\frac{5}{4}} - 2t^{\frac{3}{4}} + 2 t^{\frac{1}{4}} + 2 t^{-\frac{1}{4}} - 2t^{-\frac{3}{4}} - 3t^{-\frac{5}{4}} + t^{-\frac{9}{4}}$};
    }
\]
\end{minipage}
\caption{A computation of $f_L(t)$ (right) for the alternating amphichiral link $L = 5_5^3$ (left).}
\label{fig:5_5^3}
\end{figure}

For both of the alternating amphichiral links in Drobotukhina's table, the Tait graphs are self-dual: $G_+ = G_-$. Such duality for alternating links implies amphichirality, so our work implies that only links can have self-dual Tait graphs.

\begin{corollary}
\label{cor:amphichiral}
Let $D$ be a reduced alternating diagram of a non-local alternating link $L \subset \RP^2$. If $G_+(D) = G_-(D)$, then $L$ is amphichiral and, moreover, has at least two components.
\end{corollary}

The converse of Corollary \ref{cor:amphichiral} is a projective version of the Kauffman conjecture \cite{MR1078665, MR1078636}. The original conjecture was disproven by Dasbach-Hougardy \cite{MR1414091} and subsequently reformulated by Kauffman-Jablan in \cite{MR3024020}. From Figure \ref{fig:5_5^3}, note the apparent infinite family of alternating amphichiral links, consisting of one $n$-component link for each $n \ge 2$. Each has $G_+(D) = G_-(D)$.

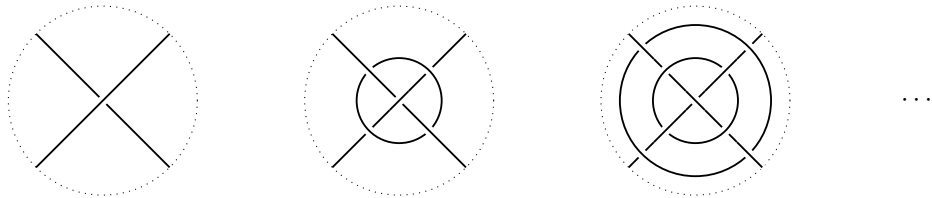
\[
\tikz[baseline={([yshift=-.5ex]current bounding box.center)}, scale=0.5]
	{
\draw[knot] (1.5-2.5*0.7071, 1.5+2.5*0.7071) -- (1.5+2.5*0.7071, 1.5-2.5*0.7071);
\draw[knot, overcross] (1.5-2.5*0.7071, 1.5-2.5*0.7071) -- (1.5+2.5*0.7071, 1.5+2.5*0.7071);
\draw[dotted] (1.5,1.5) circle(2.5);
    }
\qquad\qquad
\tikz[baseline={([yshift=-.5ex]current bounding box.center)}, scale=0.5]
	{
\draw[knot] (1.5-2.5*0.7071, 1.5-2.5*0.7071) -- (1.25,1.25);
\draw[knot] (1.75,1.75) -- (1.5+2.5*0.7071, 1.5+2.5*0.7071);
\draw[knot, overcross] (1.5,1.5) circle(1.125);
\draw[knot, overcross] (1.5-2.5*0.7071, 1.5+2.5*0.7071) -- (1.5+2.5*0.7071, 1.5-2.5*0.7071);
\draw[knot, overcross] (1.25,1.25) -- (1.75,1.75);
\draw[dotted] (1.5,1.5) circle(2.5);
    }
\qquad\qquad
\tikz[baseline={([yshift=-.5ex]current bounding box.center)}, scale=0.5]
	{
\draw[knot] (1.5-2.5*0.7071, 1.5-2.5*0.7071) -- (0.4,0.4);
\draw[knot] (1.25,1.25) -- (1.75,1.75);
\draw[knot] (2.6,2.6) -- (1.5+2.5*0.7071, 1.5+2.5*0.7071);
\draw[knot] (0.4,2.6) -- (1.25, 1.75);
\draw[knot] (1.75, 1.25) -- (2.6, 0.4);
\draw[knot, overcross] (1.5,1.5) circle(1.125);
\draw[knot, overcross]
(1.5,1.5) circle(2);
\draw[knot, overcross] (0.4,0.4) -- (1.25,1.25);
\draw[dotted] (1.5,1.5) circle(2.5);
\draw[knot, overcross] (1.75,1.75) -- (2.6,2.6);
\draw[knot, overcross] (1.5-2.5*0.7071, 1.5+2.5*0.7071) -- (0.4,2.6);
\draw[knot, overcross] (1.25, 1.75) -- (1.75, 1.25);
\draw[knot, overcross] (2.6, 0.4) -- (1.5+2.5*0.7071, 1.5-2.5*0.7071);
    }
\qquad\qquad
\cdots
\]

\stoptoc

\subsection*{Speculations}

The work presented in this article motivates a few projects which warrant investigation. We list a couple of them below.

\subsubsection*{Quasialternating links in $\RP^3$}

\begin{wrapfigure}[16]{R}{0.3\textwidth}
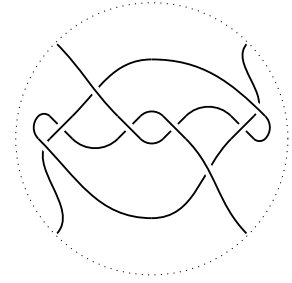

\vspace{-10pt}
\centering
\begin{minipage}{0.3\textwidth}
\[
\tikz[baseline={([yshift=-.5ex]current bounding box.center)}, scale=.3]{
    \draw[knot] (0.5,0) -- (1.5,1);
    \draw[knot, overcross] (0.5,1) -- (1.5,0);
    \draw[knot] (-1.5,0) -- (-0.5,1);
    \draw[knot, overcross] (-1.5,1) -- (-0.5,0);
    \draw[knot] (-0.5,1) to[out=45, in=135] (0.5,1);
    \draw[knot] (-0.5,0) to[out=-45, in=-135] (0.5,0);
    \draw[knot] (3.5,1) -- (4.5,0);
    \draw[knot, overcross] (3.5,0) -- (4.5,1);
    \draw[knot] (1.5,1) to[out=45, in=135] (3.5,1);
    \draw[knot] (3.5,0) to[out=-135, in=0] (0,-3.5);
    \draw[knot, overcross] (1.5,0) to[out=-45, in=135] (6*1.41/2, -6*1.41/2);
    \draw[knot] (-3.5,0) -- (-4.5,1);
    \draw[knot, overcross] (-3.5,1) -- (-4.5,0);
    \draw[knot] (-1.5,0) to[out=-135, in=-45] (-3.5,0);
    \draw[knot] (-3.5,1) to[out=45, in=180] (0,3.5);
    \draw[knot, overcross] (-1.5,1) to[out=135, in=-45] (-6*1.41/2, 6*1.41/2);
    \draw[knot] (4.5,1) to[out=45, in=-135] (6*1.41/2, 6*1.41/2);
    \draw[knot] (4.5,0) to[out=-45, in=-135] (5, 0) to[out=45, in=-45] (5,1);
    \draw[knot, overcross] (5,1) to[out=135, in=0] (0,3.5);
    \draw[knot] (-4.5,0) to[out=-135, in=45] (-6*1.41/2, -6*1.41/2);
    \draw[knot] (-4.5,1) to[out=135, in=45] (-5, 1) to[out=-135, in=135] (-5,0);
    \draw[knot, overcross] (-5,0) to[out=-45, in=180] (0,-3.5);
	\draw[white, ultra thick] (0,0) circle (6cm);
	\draw[dotted] (0,0) circle (6cm);
    }
\]
\end{minipage}
\caption{An 8-crossing, non-alternating nullhomologous knot in $\RP^3$.}
\label{fig:weaklyqa}
\end{wrapfigure}

Following Theorem \ref{thm:main}, we hope to search for larger classes of Khovanov-thin links in $\RP^3$. At the time of writing, there is no known geometric classification of homologically-thin links in $S^3$--- to the authors' knowledge, there are no examples of homologically-thin links which are not \emph{two-fold quasi-alternating} \cite{MR3760881}. These contain the class of \emph{quasi-alternating links}, introduced in \cite{MR2141852}. The known homologically-thin non-quasi-alternating links (\textit{e.g.}, those introduced in \cite{MR2592726}) are two-fold quasi-alternating.

What does it mean for a link in $\RP^3$ to be quasi-alternating? If a link $L$ is Khovanov-thin, then the link surgeries spectral sequence of \cite{MR2141852} implies that the branched double cover of $S^3$ branched along $mL$ is an $L$-space. More recently, Chen \cite{MR5083272} showed that there are similar spectral sequences for nullhomologous links in $\RP^3$. Observe that a non-local, nullhomologous link in $\RP^3$ has two non-trivial branched double covers. It also has two chances to be Khovanov-thin. For example, we claim that one of the Khovanov homologies of the knot pictured in Figure \ref{fig:weaklyqa} is thick, and the other is thin.

\subsubsection*{Does Khovanov homology detect infinitely many knots in $\RP^3$?}

\begin{wrapfigure}[14]{L}{0.3\textwidth}
\vspace{-10pt}
\centering
\begin{minipage}{0.3\textwidth}
\centering
\begin{tabular}{|c||c|c|c|c|} 
\hline
$3n$   &            &             &          & $\blue{\F}$  \\ 
\hline
$3n-1$   &            &             &          &              \\ 
\hline
$3n-2$     &            &             & & $\blue{\F}$  \\ 
\hline
$\vdots$ &            &             &     $\iddots$     &              \\ 
\hline
$n+2$    &            & $\blue{\F}$ &  &              \\ 
\hline
$n+1$    & $\red{\F}$ &             &          &              \\ 
\hline
$n$   &            & $\blue{\F}$ &          &              \\ 
\hline
$n-1$    & $\red{\F}$ &             &          &              \\ 
\hhline{|=::====|}
\slashbox{$j$}{$i$}          & $0$        & $1$         & $\cdots$ & $n$          \\
\hline
\end{tabular}\end{minipage}
\caption{$\Kh(\Upsilon_n; \F).$}
\end{wrapfigure}

In a seminal article \cite{MR2805599}, Kronheimer and Mrowka proved that Khovanov homology detects the unknot. Since then, Khovanov homology has been shown to detect several other small-crossing knots \cite{MR3190305, MR4049809, MR4275096, MR4516040, MR4393789, MR4889247}, but no infinite family is known to be detected.

Let $K$ be a knot in $\RP^3$. In work-in-progress, the second author and Chen Zhang claim that if $\widetilde{\Kh}(K) \cong \widetilde{\Kh}(\Upsilon_{2k})$ for any $k\in \Z$, then $K = \Upsilon_{2k}$. Like the unknot, the reduced Khovanov homology of $\Upsilon_n$ is very simple: one part of the invariant is always dimension one, and the other is dimension $n$. The literature provides a few avenues to pursue this result. We are especially interested in studying the knot Floer homology of nullhomologous links in $\RP^3$, and the development of a Dowlin spectral sequence \cite{MR4777638}.

\subsection*{Article outline}

Our article is organized as follows.

\begin{itemize}

\item In \S \ref{s:background}, we give background on the Jones polynomial, Goeritz matrices, and Khovanov homology for nullhomologous links in $\RP^3$. We can prove Corollary \ref{cor:amphichiralknots} right away in \S \ref{ss:drobjp}. Proposition \ref{prop:goeritzsig}, proven in \S \ref{ss:Goeritz}, is essential only to the proof of Theorem \ref{thm:main}.

\item In \S \ref{s:myopictutte}, we describe the myopic Tutte polynomial for graphs embedded in $\RP^2$. We provide three definitions: one in terms of edge deletion/contraction, another via a state-sum formula, and finally, in \S \ref{ss:spanningtreemyopic}, one in terms of spanning trees. We show that one can recover the generalized Krushkal polynomial from the myopic Tutte polynomial and prove Theorem \ref{thm:equalkrushkal} in \S \ref{ss:krushkal}. 

\item In \S \ref{s:spanningtreejones}, we give a description of the Kauffman bracket of nullhomologous links in terms of spanning trees of Tait graphs; see Proposition \ref{prop:KBgraphpolynomial}. This is followed by a proof of Theorem \ref{thm:jonespolynomialalternatingtutte}.

\item In \S \ref{s:applications}, we provide a few applications. These include proofs of Theorems \ref{thm:altbreadth} and \ref{thm:detection}. We also exhibit the tabulation of nullhomologous links in $\RP^3$ by local crossing number in \S \ref{ss:tabulation}.

\item Finally, in \S \ref{s:spanningtreekh}, we develop the theory of spanning tree complexes for nullhomologous links in $\RP^3$. In order to make full use of our myopic Tutte polynomials, we follow Champanerkar-Kofman \cite{MR2480298} and use their spectral sequence to conclude with a proof of Theorem \ref{thm:main}.

\end{itemize}

\subsection*{AI disclosure}

No artificial intelligence was used in the development of this article.

\subsection*{Acknowledgements}

The authors wish to thank Adam Lowrance for his thoughts on several drafts of this article. We also thank Daren Chen, Ilya Kofman, and Chen Zhang for helpful and inspiring conversations.

\resumetoc

\section{Background}
\label{s:background}

\subsection{Drobotukhina's analogue of the Jones polynomial}
\label{ss:drobjp}

The Jones polynomial for links in $\RP^3$ was defined by Drobotukhina \cite{MR1073213} via the Kauffman bracket.

\begin{definition}
\label{def:kauffmanbracket}
Given $L\subset \RP^3$, pick a diagram $D \subset \RP^2$ for $L$. Define the \emph{Kauffman bracket} $\langle D \rangle \in \mathbb{Z}[A, A^{-1}, \theta]$ by the following rules.
\begin{enumerate}[label=(\arabic*)]
\item Near any crossing in $D$, $\left\langle~
\tikz[baseline=.6ex, scale = .4]{
\draw (0,0) -- (1,1);
\draw (1,0) -- (.7,.3);
\draw (.3,.7) -- (0,1);
}
~\right\rangle = A \left\langle ~ \tikz[baseline=.6ex, scale = .4]{
\draw[rounded corners = 1mm] (0,0) -- (.45,.5) -- (0,1);
\draw[rounded corners = 1mm] (1,0) -- (.55,.5) -- (1,1);
}~\right\rangle + A^{-1} \left\langle~ \tikz[baseline=.6ex, scale = .4]{
\draw[rounded corners = 1mm] (0,0) -- (.5,.45) -- (1,0);
\draw[rounded corners = 1mm] (0,1) -- (.5,.55) -- (1,1);
}~\right\rangle$. The first is called the $A$-resolution, and the second is the $B$-resolution.
\item $\left\langle D \sqcup \bigcirc \right\rangle = (-A^{-2} - A^2) \langle D \rangle$; \textit{i.e.}, we can remove a crossingless class-0 unknot at the expense of multiplying by the provided Laurent binomial.
\item $\langle \circledcirc \rangle = 1$ and $\langle \ominus \rangle = \theta$; \textit{i.e.}, the crossingless class-0 unknot evaluates to 1, and the crossingless class-1 unknot evaluates to $\theta$.
\end{enumerate}
\end{definition}

For an $n$-crossing diagram $D$, order the crossings arbitrarily and pick $\mathbf{v} = (v_1,\ldots, v_n) \in \{A, B\}^n$. Let $D_\mathbf{v}$ denote the resolution of $D$ obtained by performing the $v_i$-resolution at the $i$th crossing of $D$. 

\begin{lemma}[\cite{MR4904031}, Lemma 2.3]
\label{lem:ManWillis}
A crossingless class-1 unknot exists in a resolution $D_\mathbf{v}$ if and only if $D$ is a diagram for a class-1 link $L$.
\end{lemma}

In addition, there can be at most one class-1 unknot in $D_\mathbf{v}$. Therefore,
\begin{enumerate}
\item if $D$ is a diagram for a class-0 link, then $\langle D \rangle \in \mathbb{Z}[A, A^{-1}]$, and
\item if $D$ is a diagram for a class-1 link, then $\langle D \rangle = \theta \cdot f$ for $f\in \mathbb{Z}[A, A^{-1}]$.
\end{enumerate}
Drobotukhina's version of the bracket polynomial is obtained by setting $\theta=1$.

In \cite{MR1073213}, Drobotukhina proved that two links $L, L' \subset \RP^3$ are isotopic if and only if there is a sequence of moves $\Omega_1$--$\Omega_5$, taking a diagram of $L$ to a diagram of $L'$. The moves $\Omega_1$--$\Omega_3$ are the ordinary Reidemeister moves, and $\Omega_4$ and $\Omega_5$ are the moves pictured in Figure \ref{fig:rmoves}.

\begin{figure}[ht]
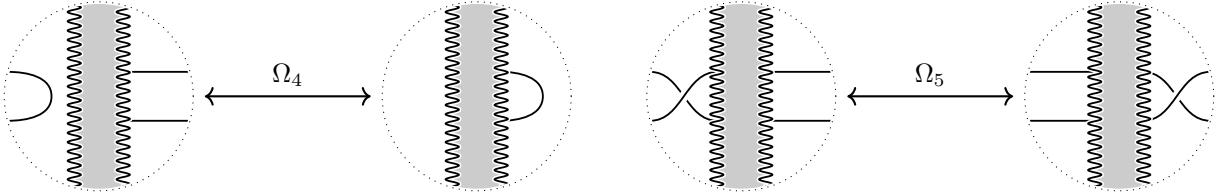

\tikz[]{
\node(1) at (0,0) {$
\tikz[baseline={([yshift=-.5ex]current bounding box.center)}, scale=.25]{
	\fill[black!20!white] (-5*0.259,5*0.966) -- (-5*0.259,-5*0.966) to[out=-15,in=-165] (5*0.259,-5*0.966) -- (5*0.259,5*0.966) to[out=165,in=15] (-5*0.259,5*0.966);
	\draw[knot] (1.25,5*0.259) -- (5*0.966,5*0.259);
	\draw[knot] (1.25,-5*0.259) -- (5*0.966,-5*0.259);
	\draw[knot] (-5*0.966,5*0.259) to[out=0, in=90] (-2.5,0) to[out=-90, in=0] (-5*0.966,-5*0.259);
	\draw[white, line width=1.85pt, decorate,decoration={coil,segment length=3.5pt,aspect=0}] (5*0.259,5*0.966) -- (5*0.259,-5*0.966);
	\draw[white, line width=1.85pt, decorate,decoration={coil,segment length=3.5pt,aspect=0}] (-5*0.259,5*0.966) -- (-5*0.259,-5*0.966);
	\draw[knot, decorate,decoration={coil,segment length=3.5pt,aspect=0}] (5*0.259,5*0.966) -- (5*0.259,-5*0.966);
	\draw[knot, decorate,decoration={coil,segment length=3.5pt,aspect=0}] (-5*0.259,5*0.966) -- (-5*0.259,-5*0.966);
	\draw[white, ultra thick] (0,0) circle (5cm);
	\draw[dotted] (0,0) circle (5cm);
}$};
\node(2) at (5,0) {$
\tikz[baseline={([yshift=-.5ex]current bounding box.center)}, scale=.25]{
	\fill[black!20!white] (-5*0.259,5*0.966) -- (-5*0.259,-5*0.966) to[out=-15,in=-165] (5*0.259,-5*0.966) -- (5*0.259,5*0.966) to[out=165,in=15] (-5*0.259,5*0.966);
	\draw[knot] (1.25,5*0.259) to[out=0, in=90] (3.5,0) to[out=-90, in=0] (1.25,-5*0.259);
	\draw[white, line width=1.85pt, decorate,decoration={coil,segment length=3.5pt,aspect=0}] (5*0.259,5*0.966) -- (5*0.259,-5*0.966);
	\draw[white, line width=1.85pt, decorate,decoration={coil,segment length=3.5pt,aspect=0}] (-5*0.259,5*0.966) -- (-5*0.259,-5*0.966);
	\draw[knot, decorate,decoration={coil,segment length=3.5pt,aspect=0}] (5*0.259,5*0.966) -- (5*0.259,-5*0.966);
	\draw[knot, decorate,decoration={coil,segment length=3.5pt,aspect=0}] (-5*0.259,5*0.966) -- (-5*0.259,-5*0.966);
	\draw[white, ultra thick] (0,0) circle (5cm);
	\draw[dotted] (0,0) circle (5cm);
}$};
\draw[<->, thick] (1) -- node[above, midway]{$\Omega_4$} (2);
\begin{scope}[xshift=8.5cm]
\node(1') at (0,0) {$
\tikz[baseline={([yshift=-.5ex]current bounding box.center)}, scale=.25]{
	\fill[black!20!white] (-5*0.259,5*0.966) -- (-5*0.259,-5*0.966) to[out=-15,in=-165] (5*0.259,-5*0.966) -- (5*0.259,5*0.966) to[out=165,in=15] (-5*0.259,5*0.966);
	\draw[knot] (1.25,5*0.259) -- (5*0.966,5*0.259);
	\draw[knot] (1.25,-5*0.259) -- (5*0.966,-5*0.259);
	\draw[knot] (-1.25,-5*0.259) to[out=180, in=0] (-5*0.966,5*0.259);
	\draw[knot, overcross] (-1.25,5*0.259) to[out=180, in=0] (-5*0.966,-5*0.259);
	\draw[white, line width=1.85pt, decorate,decoration={coil,segment length=3.5pt,aspect=0}] (5*0.259,5*0.966) -- (5*0.259,-5*0.966);
	\draw[white, line width=1.85pt, decorate,decoration={coil,segment length=3.5pt,aspect=0}] (-5*0.259,5*0.966) -- (-5*0.259,-5*0.966);
	\draw[knot, decorate,decoration={coil,segment length=3.5pt,aspect=0}] (5*0.259,5*0.966) -- (5*0.259,-5*0.966);
	\draw[knot, decorate,decoration={coil,segment length=3.5pt,aspect=0}] (-5*0.259,5*0.966) -- (-5*0.259,-5*0.966);
	\draw[white, ultra thick] (0,0) circle (5cm);
	\draw[dotted] (0,0) circle (5cm);
}$};
\node(2') at (5,0) {$
\tikz[baseline={([yshift=-.5ex]current bounding box.center)}, scale=.25]{
	\fill[black!20!white] (-5*0.259,5*0.966) -- (-5*0.259,-5*0.966) to[out=-15,in=-165] (5*0.259,-5*0.966) -- (5*0.259,5*0.966) to[out=165,in=15] (-5*0.259,5*0.966);
	\draw[knot] (1.25,5*0.259) to[out=0, in=180] (5*0.966,-5*0.259);
	\draw[knot, overcross] (1.25,-5*0.259) to[out=0, in=180] (5*0.966,5*0.259);
	\draw[knot] (-1.25,5*0.259) -- (-5*0.966,5*0.259);
	\draw[knot] (-1.25,-5*0.259) -- (-5*0.966,-5*0.259);
	\draw[white, line width=1.85pt, decorate,decoration={coil,segment length=3.5pt,aspect=0}] (5*0.259,5*0.966) -- (5*0.259,-5*0.966);
	\draw[white, line width=1.85pt, decorate,decoration={coil,segment length=3.5pt,aspect=0}] (-5*0.259,5*0.966) -- (-5*0.259,-5*0.966);
	\draw[knot, decorate,decoration={coil,segment length=3.5pt,aspect=0}] (5*0.259,5*0.966) -- (5*0.259,-5*0.966);
	\draw[knot, decorate,decoration={coil,segment length=3.5pt,aspect=0}] (-5*0.259,5*0.966) -- (-5*0.259,-5*0.966);
	\draw[white, ultra thick] (0,0) circle (5cm);
	\draw[dotted] (0,0) circle (5cm);
}$};
\draw[<->, thick] (1') -- node[above, midway]{$\Omega_5$} (2');
\end{scope}
}
\caption{Drobotukhina's moves $\Omega_4$ and $\Omega_5$. Here, the shaded region is understood to contain an arbitrary link diagram.}
\label{fig:rmoves}
\end{figure}

It is routinely verified that the Kauffman bracket obeys $\Omega_2$--$\Omega_5$. Per usual, to obtain an invariant, one passes to oriented links. Given an oriented link $L \subset \RP^3$ and a diagram $D$ for $L$, define the \emph{writhe} $w(D)$ as the difference between the number of positive crossings $\left(~\tikz[baseline=.6ex, scale = .4]{
\draw[->] (0,0) -- (1,1);
\draw (1,0) -- (.7,.3);
\draw[->] (.3,.7) -- (0,1);
}~\right)$ and the number of negative crossings $\left(~\tikz[baseline=.6ex, scale = .4]{
\draw[->] (.7,.7) -- (1,1);
\draw[->] (1,0) -- (0,1);
\draw (0,0) -- (.3,.3);
}~\right)$ in $D$. Then, Drobotukhina's analogue of the Jones polynomial is defined as 
\begin{equation}\label{eq:jpdefinition}
J_L(t, \theta) = \left.(-A^3)^{-w(D)}\langle D \rangle\right|_{A=t^{-1/4}} \in \mathbb{Z}[t^{\frac{1}{2}}, t^{-\frac{1}{2}}, \theta],
\end{equation}
and it is invariant under $\Omega_1$--$\Omega_5$.

We are mainly interested in nullhomologous links. Since $\theta$ is irrelevant for such links, we drop it from the notation. For nullhomologous links, the Jones polynomial can be supported in both integer and half-integer exponents. Therefore, there are two summands in the Jones polynomial corresponding to the real and imaginary part of $\left.J_L(t)\right|_{t=-1}$. We write
\[
J_L(-1) = \pm \mathrm{det}_{\Re}(L) \pm \mathrm{det}_{\Im}(L)\, i
\]
for $\det_{\Re}, \det_{\Im} \ge 0$, and call $\det_{\Re}(L)$ and $\det_{\Im}(L)$ the \emph{real} and \emph{imaginary determinants} of $L$, respectively.

\begin{example}
\label{ex:PT_KB}
Let us return to our first example of a knot in $\RP^3$. It is the knot $3_1$ from Drobotukhina's tabulation \cite{MR1296890}. One can compute that
\[
\left\langle
\tikz[baseline={([yshift=-.5ex]current bounding box.center)}, scale=.125]{
	\draw[knot, overcross] (0,1.75) to[out=0, in=-135] (5*1.41/2, 5*1.41/2);
	\draw[knot, overcross] (2,2) to[out=-90, in=45] (-5*1.41/2, -5*1.41/2);
	\draw[knot, overcross] (5*1.41/2, -5*1.41/2) to[out=135, in=-90] (-2,2);
	\draw[knot, overcross] (-2,2) to[out=90, in=180] (0,4);
	\draw[knot, overcross] (-5*1.41/2, 5*1.41/2) to[out=-45, in=180] (0, 1.75);
	\draw[knot, overcross] (0,4) to[out=0, in=90] (2,2);
	\draw[white, ultra thick] (0,0) circle (5cm);
	\draw[line cap=round, dash pattern=on 0pt off 3.5pt] (0,0) circle (5cm);
}
\right\rangle = -A^5 + A + A^{-1} - A^{-3} -A^{-5}.
\]
This link has a single component, so its Jones polynomial does not depend on the choice of orientation: $w(D) = -1$ and we calculate
\[
J_L(t) = 1 - t^{-1} + t^{-2} + t^{\frac{1}{2}} - t^{-\frac{1}{2}},
\]
so
\[
J_L(-1) = 3 + 2i,
\]
meaning that $\det_{\Re}(L) = 3$ and $\det_{\Im}(L) = 2$.
\end{example}

Let $mD$ denote the mirror image of $D$. As is the case in $S^3$, we have that $w(D) = -w(mD)$ and $\langle mD \rangle = \left.\langle D \rangle\right|_{A \to A^{-1}}$. Thus, if $mL$ is the link with diagram $mD$, we have that $J_{mL}(t, \theta)$ is obtained from $J_{L}(t, \theta)$ by interchanging $t^{\frac{1}{2}}$ and $t^{-\frac{1}{2}}$. For $L$ in Example \ref{ex:PT_KB},
\[
J_{mL}(t) = t^2 - t + 1 - t^{\frac{1}{2}} + t^{-\frac{1}{2}} \qquad \text{and} \qquad \det(mL) = 3-2i.
\]

\begin{proof}[Proof of Corollary \ref{cor:amphichiralknots}]
In addition to being an invariant of oriented links, the Jones polynomial is an invariant of non-oriented knots. Thus if $K \subset \RP^3$ is amphichiral, it must be the case that $J_K(t) = J_K(t^{-1})$. In particular, if $M \in \Z[\pm \frac{1}{2}]$ is the maximum degree of $J_K$ for an amphichiral knot $K$, them $-M$ is the minimum degree of $J_K$. Thus the Jones polynomial of $K$ takes the form
\[
J_K(t) = c_{-M} t^{-M} + \cdots + c_M t^M
\]
so that the breadth of $J_K(t)$ is $2M \in \Z$. However, by Theorem \ref{thm:dromain}, if $K$ is alternating, then the breadth of $J_K(t)$ is a half-integer, from which we deduce that $K$ is not alternating. 
\end{proof}

\begin{example}
\label{ex: twists_KB}
Notice that
\[
\left\langle
\tikz[baseline={([yshift=-.5ex]current bounding box.center)}, scale=.125]{
	\draw[knot] (5*1.41/2, -5*1.41/2) -- (-5*1.41/2, 5*1.41/2);
	\draw[knot, overcross] (-5*1.41/2, -5*1.41/2) -- (5*1.41/2, 5*1.41/2); 
	\draw[white, ultra thick] (0,0) circle (5cm);
	\draw[dotted] (0,0) circle (5cm);
}
\right\rangle = A + A^{-1}.
\]
This is a diagram for the two-component link $\Upsilon_1$ in $\RP^3$, so we have that $w(D) = \pm1$, and
\[
J_{\tikz[baseline={([yshift=-.5ex]current bounding box.center)}, scale=0.05]{
	\draw[dotted] (0,0) circle (5cm);
	\draw[knot, ->] (5*1.41/2, -5*1.41/2) -- (-5*1.41/2, 5*1.41/2);
	\draw[knot, overcross] (-5*1.41/2, -5*1.41/2) -- (5*1.41/2, 5*1.41/2); 
	\draw[knot, ->] (-5*1.41/2, -5*1.41/2) -- (5*1.41/2, 5*1.41/2); 
}}(t) = - t - t^{\frac{1}{2}} 
\qquad \text{and} \qquad
J_{\tikz[baseline={([yshift=-.5ex]current bounding box.center)}, scale=0.05]{
	\draw[dotted] (0,0) circle (5cm);
	\draw[knot, ->] (5*1.41/2, -5*1.41/2) -- (-5*1.41/2, 5*1.41/2);
	\draw[knot, overcross] (-5*1.41/2, -5*1.41/2) -- (5*1.41/2, 5*1.41/2);
	\draw[knot, <-] (-5*1.41/2, -5*1.41/2) -- (5*1.41/2, 5*1.41/2); 
}}(t) = -t^{-1} - t^{-\frac{1}{2}}
\]
so $J_{{\tikz[baseline={([yshift=-.5ex]current bounding box.center)}, scale=0.045]{
	\draw[dotted] (0,0) circle (5cm);
	\draw[knot, ->] (5*1.41/2, -5*1.41/2) -- (-5*1.41/2, 5*1.41/2);
	\draw[knot, overcross] (-5*1.41/2, -5*1.41/2) -- (5*1.41/2, 5*1.41/2); 
	\draw[knot, ->] (-5*1.41/2, -5*1.41/2) -- (5*1.41/2, 5*1.41/2); 
}} }(-1) = 1-i$ and $J_{{\tikz[baseline={([yshift=-.5ex]current bounding box.center)}, scale=0.045]{
	\draw[dotted] (0,0) circle (5cm);
	\draw[knot, ->] (5*1.41/2, -5*1.41/2) -- (-5*1.41/2, 5*1.41/2);
	\draw[knot, overcross] (-5*1.41/2, -5*1.41/2) -- (5*1.41/2, 5*1.41/2); 
	\draw[knot, <-] (-5*1.41/2, -5*1.41/2) -- (5*1.41/2, 5*1.41/2); 
}}}(-1) = 1 + i$. Now, from the observation that the Kauffman bracket of a negative Reidemeister I move evaluates to $-A^{-3}$, we have that
\begin{align*}
\left\langle
\tikz[baseline={([yshift=-.5ex]current bounding box.center)}, scale=.125]{
	\draw[knot] (5*0.5, -5*0.866) to[out=90, in=-90] (1,0) to[out=90, in=-90] (5*0.5, 5*0.866);
	\draw[knot] (-5*0.5, -5*0.866) to[out=90, in=-90] (-1,0) to[out=90, in=-90] (-5*0.5, 5*0.866);
	\node[scale=0.8, draw, fill=white] at (0,0) {$~n~$};
	\draw[white, ultra thick] (0,0) circle (5cm);
	\draw[dotted] (0,0) circle (5cm);
}
\right\rangle
& =
A \left\langle
\tikz[baseline={([yshift=-.5ex]current bounding box.center)}, scale=.125]{
	\draw[knot] (5*0.5, -5*0.866) to[out=90, in=-90] (1,0) to[out=90, in=-90] (5*0.5, 5*0.866);
	\draw[knot] (-5*0.5, -5*0.866) to[out=90, in=-90] (-1,0) to[out=90, in=-90] (-5*0.5, 5*0.866);
	\node[scale=0.8, draw, fill=white] at (0,0) {$n-1$};
	\draw[white, ultra thick] (0,0) circle (5cm);
	\draw[dotted] (0,0) circle (5cm);
}
\right\rangle
+
A^{-1}
\left\langle
\tikz[baseline={([yshift=-.5ex]current bounding box.center)}, scale=.125]{
	\draw[knot] (5*0.5, -5*0.866) to[out=90, in=-90] (1,0);
	\draw[knot] (-5*0.5, -5*0.866) to[out=90, in=-90] (-1,0);
	\draw[knot] (-5*0.5, 5*0.866) to[out=-30, in=-150] (5*0.5, 5*0.866);
	\draw[knot] (-2,1.8) to[out=90, in=90] (2,1.8);
	\node[scale=0.8, draw, fill=white] at (0,0) {$n-1$};
	\draw[white, ultra thick] (0,0) circle (5cm);
	\draw[dotted] (0,0) circle (5cm);
}
\right\rangle
\\
& =
A \left\langle
\tikz[baseline={([yshift=-.5ex]current bounding box.center)}, scale=.125]{
	\draw[knot] (5*0.5, -5*0.866) to[out=90, in=-90] (1,0) to[out=90, in=-90] (5*0.5, 5*0.866);
	\draw[knot] (-5*0.5, -5*0.866) to[out=90, in=-90] (-1,0) to[out=90, in=-90] (-5*0.5, 5*0.866);
	\node[scale=0.8, draw, fill=white] at (0,0) {$n-1$};
	\draw[white, ultra thick] (0,0) circle (5cm);
	\draw[dotted] (0,0) circle (5cm);
}
\right\rangle
+
A^{-1}
(-A^{-3})^{n-1}.
\end{align*}
Iterating this, we conclude that
\begin{equation}
\label{eq:M2KB}
\left\langle
\tikz[baseline={([yshift=-.5ex]current bounding box.center)}, scale=.125]{
	\draw[knot] (5*0.5, -5*0.866) to[out=90, in=-90] (1,0) to[out=90, in=-90] (5*0.5, 5*0.866);
	\draw[knot] (-5*0.5, -5*0.866) to[out=90, in=-90] (-1,0) to[out=90, in=-90] (-5*0.5, 5*0.866);
	\node[scale=0.8, draw, fill=white] at (0,0) {$~n~$};
	\draw[white, ultra thick] (0,0) circle (5cm);
	\draw[dotted] (0,0) circle (5cm);
}
\right\rangle
= A^n + \sum_{i=0}^{n-1} (-1)^i A^{n - 2 - 4i}.
\end{equation}
Let $\Upsilon_n^\uparrow$ be the link $\Upsilon_n$ oriented so that $w(D) = n$ (which is always the case if $n$ is even). Then we conclude that
\[
J_{\Upsilon_n^\uparrow}(t) = (-1)^n \left( t^{\frac{n}{2}} + \sum_{i=0}^{n-1}(-1)^i t^{\frac{n+1}{2} + i}\right).
\]
We record that $\det_\varepsilon(\Upsilon_n) \in \{1, n\}$. For completeness, note that if $n$ is odd and $\Upsilon_n^\downarrow$ is $\Upsilon_n$ oriented so that $w(D) = -n$, we have
\[
J_{\Upsilon_n^\downarrow}(t) = -t^{-n} - \sum_{i=0}^{n-1} (-1)^i t^{-n+\frac{1}{2}+i}.
\]
\end{example}

In Example \ref{ex: twists_KB}, it is apparent that the real and imaginary determinants are interchanged if $n$ is odd and the orientation on a single strand is reversed. In general, if $L^*$ differs from $L$ by reversing the orientation of one strand $K$ of $L$, then
\[
J_{L^*}(t) = t^{-3\lk(K, L-K)}J_L(t)
\]
as is the case classically (c.f. \cite[Chapter 3]{MR1472978}). Therefore, reversing the orientation of a component $K$ of $L$ interchanges the real and imaginary determinants if and only if $\lk(K, L-K)$ is a half-integer.

\subsubsection{The determinant and spanning trees}
\label{ss:dettrees}

For alternating links $L \subset S^3$, it is a classical result (for example, see \cite{MR99665}) that the determinant of $L$ is equal to the number of spanning trees of the Tait graph coming from either checkerboard surface induced by a reduced, alternating diagram of $L$. For example, the determinant of the trefoil is 3, which is reflected by the fact that there are three spanning trees for either of the Tait graphs corresponding to the diagram below.
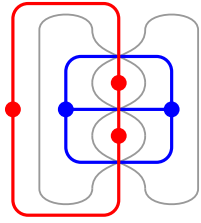
\[
\tikz[baseline={([yshift=-.5ex]current bounding box.center)}, scale=.7]{
	\begin{scope}[white!60!black]
		\draw[knot] (1,2) to[out=90, in=-90] (0,3);
		\draw[knot] (0,2) to[out=90, in=-90] (1,3);
		\draw[knot] (1,1) to[out=90, in=-90] (0,2);
		\draw[knot] (0,1) to[out=90, in=-90] (1,2);
		\draw[knot] (1,0) to[out=90, in=-90] (0,1);
		\draw[knot] (0,0) to[out=90, in=-90] (1,1);
		\draw[knot] (0,0) to[out=-90, in=-90] (-1,0) -- (-1,3) to[out=90, in=90] (0,3);
		\draw[knot] (1,0) to[out=-90, in=-90] (2,0) -- (2,3) to[out=90, in=90] (1,3);
	\end{scope}
	\begin{scope}[blue, very thick]
		\fill (-0.5, 1.5) circle (.15cm);
		\fill (1.5, 1.5) circle (.15cm);
		\draw (-0.5, 1.5) -- (1.5,1.5);
		\draw[rounded corners = 2mm] (-0.5, 1.5) -- (-0.5, 2.5) -- (1.5, 2.5) -- (1.5, 1.5);
		\draw[rounded corners = 2mm] (-0.5, 1.5) -- (-0.5,0.5) -- (1.5,0.5) -- (1.5, 1.5);
	\end{scope}
	\begin{scope}[red, very thick]
		\fill (0.5, 1) circle (.15cm);
		\fill (0.5, 2) circle (.15cm);
		\fill (-1.5, 1.5) circle (.15cm);
		\draw (0.5, 1) -- (0.5, 2);
		\draw[rounded corners = 2mm] (0.5, 2) -- (0.5,3.5) -- (-1.5, 3.5) -- (-1.5, 1.5);
		\draw[rounded corners = 2mm] (0.5, 1) -- (0.5, -0.5) -- (-1.5, -0.5) -- (-1.5, 1.5);
	\end{scope}
}
\]
Interestingly, the two Tait graphs associated to a diagram of a non-local nullhomologous link in $\RP^3$, while dual as graphs, serve distinct purposes. For example, we can perform isotopies on the Tait graphs associated to the diagram in Example \ref{ex:PT_KB}, pictured in Figure \ref{fig:pdtaits}, to obtain the graphs
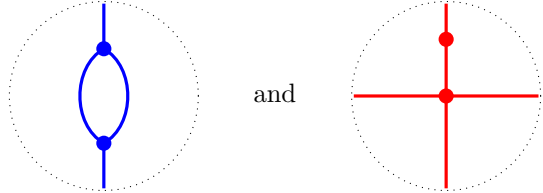
\[
\tikz[baseline={([yshift=-.5ex]current bounding box.center)}, scale=.25]{
	\begin{scope}[blue, very thick]
		\fill (0,2.5) circle (.4cm);
		\fill (0,-2.5) circle (.4cm);
		\draw (0,-2.5) -- (0,-5);
		\draw (0,2.5) -- (0,5);
		\draw (0,-2.5) to[out=30, in=-30] (0,2.5);
		\draw (0,-2.5) to[out=150, in=-150] (0,2.5);
	\end{scope}
	\draw[white, ultra thick] (0,0) circle (5cm);
	\draw[dotted] (0,0) circle (5cm);
}
\qquad \text{and} \qquad
\tikz[baseline={([yshift=-.5ex]current bounding box.center)}, scale=.25]{
	\begin{scope}[red, very thick]
		\fill (0, 3) circle (.4cm);
		\fill (0,0) circle (.4cm);
		\draw (0,-5) -- (0,5);
		\draw (-5,0) -- (5,0);
	\end{scope}
	\draw[white, ultra thick] (0,0) circle (5cm);
	\draw[dotted] (0,0) circle (5cm);
}
\]
in $\RP^2$. Observe that the blue Tait graph has 3 spanning trees and the red Tait graph has 2 spanning trees. This corresponds to the fact that $\det_{\Re}(L) = 3$ and $\det_{\Im}(L) = 2$ in this case. Corollary \ref{cor:detsandtrees} asserts that this is always the case for nullhomologous alternating links.

\subsection{Goeritz matrices for alternating links}
\label{ss:Goeritz}

Associated to any nullhomologous link in $\RP^3$ are two Goeritz matrices. We provide a brief description here; see \cite{JeffMarshMilne} for more details. We postpone the relation of these matrices to the two 2-fold branched covers until future work; the main content of this section is a proof that the determinants of these Goeritz matrices are the real and imaginary determinants of $L$. We will also provide combinatorial formulas for the signatures of nullhomologous links in $\RP^3$.

Our definition follows \cite{MR500905} exactly. Let $L\subset \RP^3$ be a nullhomologous link and $D \subset \RP^2$ be a diagram for it. Pick a checkerboard coloring $\kappa$ for $D$; as before, the faces colored black give rise to the associated Tait graph. Let $R_0, R_1, \ldots, R_n$ denote the white regions in $\mathbb{D}^2 - D$, so that no white region is incident to itself. To each crossing $c$, assign the number $\mu(c)$ according to Figure \ref{fig:incidence}.

\begin{figure}[ht]
\begin{tikzpicture}[thick, scale=1.75]
\node at (0.125,0.5) {$R_i$};
\node at (0.875,0.5) {$R_j$};
\fill[black!20!white] (0,0) -- (.5,.5) -- (1,0);
\fill[black!20!white] (1,1) -- (.5,.5) -- (0,1);
\draw (0,0) -- (1,1);
\draw (0,1) -- (.3,.7);
\draw (1,0) -- (.7,.3);
\node at (0.5,-.35) {$+1$};
\begin{scope}[xshift=3cm]
\node at (0.5,0.125) {$R_i$};
\node at (0.5,0.875) {$R_j$};
\fill[black!20!white] (0,0) -- (.5,.5) -- (0,1);
\fill[black!20!white] (1,1) -- (.5,.5) -- (1,0);
\draw (0,0) -- (1,1);
\draw (0,1) -- (.3,.7);
\draw (1,0) -- (.7,.3);
\node at (0.5,-.35) {$-1$};
\end{scope}
\end{tikzpicture}
\caption{Incidence number associated to a crossing.}
\label{fig:incidence}
\end{figure}
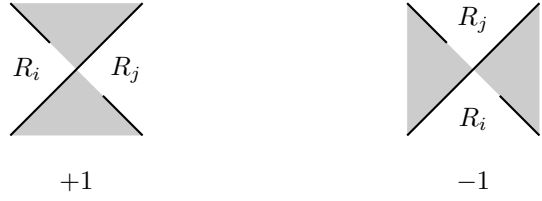

The \emph{Goeritz matrix} $\mathcal{G} = \mathcal{G}(D, \kappa)$ associated to a fixed checkerboard coloring $\kappa$ of $D$ is defined as follows. For each $i, j \in \{0, 1, \ldots, n\}$, let
\[
g_{ij} = - \sum_{c \in R_i \cap R_j} \mu(c) ~\text{if}~ i\not=j, 
\qquad \text{and} \qquad
g_{ii} = -\sum_{i \not=j} g_{ij} ~ \text{otherwise.}
\]
Then $\mathcal{G}$ is the matrix with entries $g_{i,j}$ for $i, j\in \{1, \ldots, n\}$. Notice that $\mathcal{G}$ is an invariant of $(D, \kappa)$ up to matrix similarity. If $D$ is a reduced, alternating diagram of an alternating link $L \subset \RP^3$, then we denote by $\mathcal{G}_+$ and $\mathcal{G}_-$ the Goeritz matrices coming from the checkerboard coloring yielding $G_+(D)$ and $G_-(D)$, respectively.

For $L \subset \RP^2$ oriented, define
\[
\mu(D, \kappa) = \sum_{c~\text{of type II}} \mu(c)
\]
according to Figure \ref{fig:type}. We define
\[
\sigma_{(D, \kappa)}(L) = \mathrm{sig}(\mathcal{G}(D, \kappa)) - \mu(D, \kappa).
\]
When $D$ is a reduced alternating diagram, we write $\sigma_\pm = \mathrm{sig}(\mathcal{G}_\pm) - \mu(D_\pm)$, where $D_\pm$ is the diagram $D$ with checkerboard coloring yielding $G_\pm$.

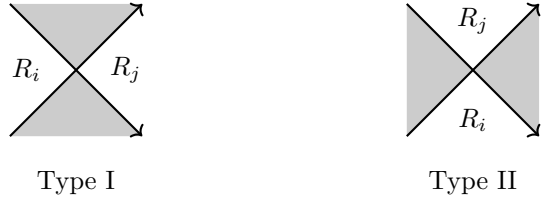
\begin{figure}[ht]
\begin{tikzpicture}[thick, scale=1.75]
\node at (0.125,0.5) {$R_i$};
\node at (0.875,0.5) {$R_j$};
\fill[black!20!white] (0,0) -- (.5,.5) -- (1,0);
\fill[black!20!white] (1,1) -- (.5,.5) -- (0,1);
\draw[->] (0,0) -- (1,1);
\draw[->] (0,1) -- (1,0);
\node at (0.5,-.35) {Type I};
\begin{scope}[xshift=3cm]
\node at (0.5,0.125) {$R_i$};
\node at (0.5,0.875) {$R_j$};
\fill[black!20!white] (0,0) -- (.5,.5) -- (0,1);
\fill[black!20!white] (1,1) -- (.5,.5) -- (1,0);
\draw[->] (0,0) -- (1,1);
\draw[->] (0,1) -- (1,0);
\node at (0.5,-.35) {Type II};
\end{scope}
\end{tikzpicture}
\caption{Types of crossings in an oriented diagram with checkerboard coloring.}
\label{fig:type}
\end{figure}

\begin{example}
\label{ex:maingoeritz}
Returning to our main example,
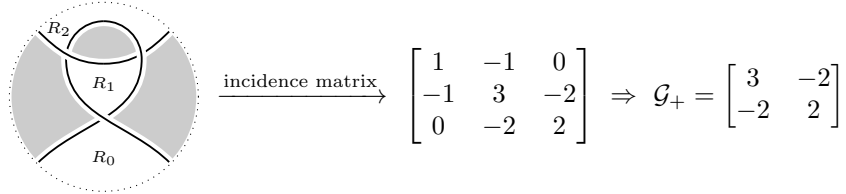
\[
\tikz[baseline={([yshift=-.5ex]current bounding box.center)}, scale=.25]{
	\fill[black!20!white] (-5*1.41/2, 5*1.41/2) to[out=-45, in=135] (-2, 2.25) to[out=-100, in=135] (0,-1.17) to[out=-135, in=45] (-5*1.41/2, -5*1.41/2) to[out=135, in=-90, looseness=0.8] (-5,0) to[out=90,in=-135] (-5*1.41/2, 5*1.41/2);
	\fill[black!20!white] (5*1.41/2, 5*1.41/2) to[out=-135, in=45] (2, 2.25) to[out=-80, in=45] (0,-1.17) to[out=-45, in=135] (5*1.41/2, -5*1.41/2) to[out=45, in=-90, looseness=0.8] (5,0) to[out=90,in=-45] (5*1.41/2, 5*1.41/2);
	\fill[black!20!white] (-2, 2.25) to[out=100, in=180] (0,4) to[out=0, in=80] (2, 2.25) to[out=210, in=0] (0,1.75) to[out=180, in=-30] (-2, 2.25);
	\draw[knot, overcross] (0,1.75) to[out=0, in=-135] (5*1.41/2, 5*1.41/2);
	\draw[knot, overcross] (2,2) to[out=-90, in=45] (-5*1.41/2, -5*1.41/2);
	\draw[knot, overcross] (5*1.41/2, -5*1.41/2) to[out=135, in=-90] (-2,2);
	\draw[knot, overcross] (-2,2) to[out=90, in=180] (0,4);
	\draw[knot, overcross] (-5*1.41/2, 5*1.41/2) to[out=-45, in=180] (0, 1.75);
	\draw[knot, overcross] (0,4) to[out=0, in=90] (2,2);
	\draw[white, ultra thick] (0,0) circle (5cm);
	\draw[dotted] (0,0) circle (5cm);
	\node at (0,-3.25) {\tiny $R_0$};
	\node at (0,0.65) {\tiny $R_1$};
	\node at (-2.4,3.6) {\tiny $R_2$};
}
~\xrightarrow{\text{incidence matrix}}~
\begin{bmatrix}
1 & -1 & 0 \\
-1 & 3 & -2 \\
0 & -2 & 2
\end{bmatrix}
~\Rightarrow ~
\mathcal{G}_+ = \begin{bmatrix} 3 & -2 \\ -2 & 2 \end{bmatrix}
\]
(recall that this checkerboard coloring yields the positive graph; see Equation (\ref{eq:PT_taits})) and
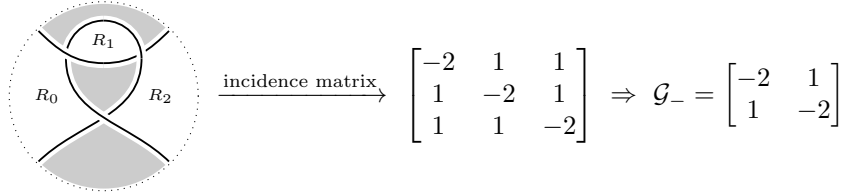
\[
\tikz[baseline={([yshift=-.5ex]current bounding box.center)}, scale=.25]{
	\fill[black!20!white] (-2, 2.25) to[out=-45, in=180] (0,1.75) to[out=0,in=-135] (2,2.25) to[out=-80, in=45] (0,-1.17) to[out=135, in=-100] (-2, 2.25);
	\fill[black!20!white] (-5*1.41/2, 5*1.41/2) to[out=45, in=180] (0,5) to[out=0, in=135] (5*1.41/2, 5*1.41/2) to[out=225, in=30] (2, 2.25) to[out=100, in=0] (0,4) to[out=180, in=80] (-2, 2.25) to[out=150, in=-45] (-5*1.41/2, 5*1.41/2);
	\fill[black!20!white] (-5*1.41/2, -5*1.41/2) to[out=45, in=-135] (0,-1.17) to[out=-45, in=135] (5*1.41/2, -5*1.41/2) to[out=-135, in=0] (0,-5) to[out=180, in=-45] (-5*1.41/2, -5*1.41/2);
	\draw[knot, overcross] (0,1.75) to[out=0, in=-135] (5*1.41/2, 5*1.41/2);
	\draw[knot, overcross] (2,2) to[out=-90, in=45] (-5*1.41/2, -5*1.41/2);
	\draw[knot, overcross] (5*1.41/2, -5*1.41/2) to[out=135, in=-90] (-2,2);
	\draw[knot, overcross] (-2,2) to[out=90, in=180] (0,4);
	\draw[knot, overcross] (-5*1.41/2, 5*1.41/2) to[out=-45, in=180] (0, 1.75);
	\draw[knot, overcross] (0,4) to[out=0, in=90] (2,2);
	\draw[white, ultra thick] (0,0) circle (5cm);
	\draw[dotted] (0,0) circle (5cm);
	\node at (-3,0) {\tiny $R_0$};
	\node at (0,2.9) {\tiny $R_1$};
	\node at (3,0) {\tiny $R_2$};
}
~\xrightarrow{\text{incidence matrix}}~
\begin{bmatrix}
-2 & 1 & 1 \\
1 & -2 & 1 \\
1 & 1 & -2
\end{bmatrix}
~\Rightarrow ~
\mathcal{G}_- = \begin{bmatrix} -2 & 1 \\ 1 & -2 \end{bmatrix}
\]
Notice that $\det(\mathcal{G}_+) = \det_\Im(L) = 2$ and $\det(\mathcal{G}_-) = \det_\Re(L) = 3$. We can also compute that
\[
\mu \left(
\tikz[baseline={([yshift=-.5ex]current bounding box.center)}, scale=.25]{
	\fill[black!20!white] (-5*1.41/2, 5*1.41/2) to[out=-45, in=135] (-2, 2.25) to[out=-100, in=135] (0,-1.17) to[out=-135, in=45] (-5*1.41/2, -5*1.41/2) to[out=135, in=-90, looseness=0.8] (-5,0) to[out=90,in=-135] (-5*1.41/2, 5*1.41/2);
	\fill[black!20!white] (5*1.41/2, 5*1.41/2) to[out=-135, in=45] (2, 2.25) to[out=-80, in=45] (0,-1.17) to[out=-45, in=135] (5*1.41/2, -5*1.41/2) to[out=45, in=-90, looseness=0.8] (5,0) to[out=90,in=-45] (5*1.41/2, 5*1.41/2);
	\fill[black!20!white] (-2, 2.25) to[out=100, in=180] (0,4) to[out=0, in=80] (2, 2.25) to[out=210, in=0] (0,1.75) to[out=180, in=-30] (-2, 2.25);
	\draw[knot, overcross] (0,1.75) to[out=0, in=-135] (5*1.41/2, 5*1.41/2);
	\draw[knot, overcross] (2,2) to[out=-90, in=45] (-5*1.41/2, -5*1.41/2);
	\draw[knot, overcross] (5*1.41/2, -5*1.41/2) to[out=135, in=-90] (-2,2);
	\draw[knot, overcross] (-2,2) to[out=90, in=180] (0,4);
	\draw[knot, overcross] (-5*1.41/2, 5*1.41/2) to[out=-45, in=180] (0, 1.75);
	\draw[knot, overcross,] (0,4) to[out=0, in=90] (2,2);
	\draw[knot, <-] (0,4) to[out=0, in=90] (2,2);
	\draw[white, ultra thick] (0,0) circle (5cm);
	\draw[dotted] (0,0) circle (5cm);
} \right) = 1
\qquad \text{and} \qquad
\mu \left(
\tikz[baseline={([yshift=-.5ex]current bounding box.center)}, scale=.25]{
	\fill[black!20!white] (-2, 2.25) to[out=-45, in=180] (0,1.75) to[out=0,in=-135] (2,2.25) to[out=-80, in=45] (0,-1.17) to[out=135, in=-100] (-2, 2.25);
	\fill[black!20!white] (-5*1.41/2, 5*1.41/2) to[out=45, in=180] (0,5) to[out=0, in=135] (5*1.41/2, 5*1.41/2) to[out=225, in=30] (2, 2.25) to[out=100, in=0] (0,4) to[out=180, in=80] (-2, 2.25) to[out=150, in=-45] (-5*1.41/2, 5*1.41/2);
	\fill[black!20!white] (-5*1.41/2, -5*1.41/2) to[out=45, in=-135] (0,-1.17) to[out=-45, in=135] (5*1.41/2, -5*1.41/2) to[out=-135, in=0] (0,-5) to[out=180, in=-45] (-5*1.41/2, -5*1.41/2);
	\draw[knot, overcross] (0,1.75) to[out=0, in=-135] (5*1.41/2, 5*1.41/2);
	\draw[knot, overcross] (2,2) to[out=-90, in=45] (-5*1.41/2, -5*1.41/2);
	\draw[knot, overcross] (5*1.41/2, -5*1.41/2) to[out=135, in=-90] (-2,2);
	\draw[knot, overcross] (-2,2) to[out=90, in=180] (0,4);
	\draw[knot, overcross] (-5*1.41/2, 5*1.41/2) to[out=-45, in=180] (0, 1.75);
	\draw[knot, overcross] (0,4) to[out=0, in=90] (2,2);
	\draw[knot, <-] (0,4) to[out=0, in=90] (2,2);
	\draw[white, ultra thick] (0,0) circle (5cm);
	\draw[dotted] (0,0) circle (5cm);
} \right) = -2.
\]
Therefore $\sigma_+(L) = 1$ and $\sigma_-(L) = 0$.
\end{example}

It is clear that the two Goeritz matrices are nothing but the regular Goeritz matrices obtained by two distinct ``plat-like'' closures of $D$. Namely, given a diagram $D$ with coloring $\kappa$, let $\overline{D(\kappa)}$ be the diagram in the plane obtained by taking a crossingless closure of $D \cap \mathbb{D}^2$ in which each endpoint of $D$ on $\partial \mathbb{D}^2$ is connected to an adjacent endpoint so that each white region $R_i$ does not encroach on a distinct region. For example, the Goeritz matrices of the knot in Example \ref{ex:maingoeritz} are obtained as classical Goeritz matrices of the Hopf link and right-handed trefoil, since the two relevant closures are
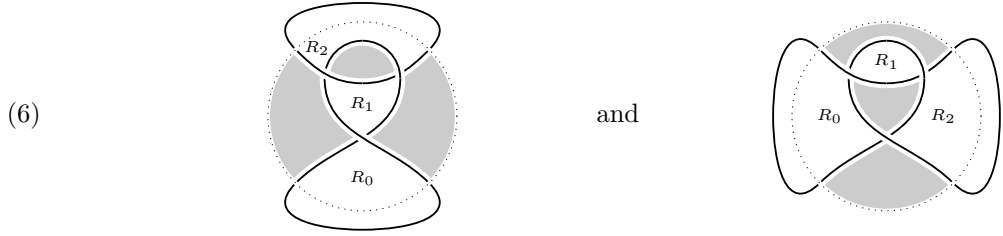
\begin{equation}
\label{eq:platish}
\tikz[baseline={([yshift=-.5ex]current bounding box.center)}, scale=.25]{
	\fill[black!20!white] (-5*1.41/2, 5*1.41/2) to[out=-45, in=135] (-2, 2.25) to[out=-100, in=135] (0,-1.17) to[out=-135, in=45] (-5*1.41/2, -5*1.41/2) to[out=135, in=-90, looseness=0.8] (-5,0) to[out=90,in=-135] (-5*1.41/2, 5*1.41/2);
	\fill[black!20!white] (5*1.41/2, 5*1.41/2) to[out=-135, in=45] (2, 2.25) to[out=-80, in=45] (0,-1.17) to[out=-45, in=135] (5*1.41/2, -5*1.41/2) to[out=45, in=-90, looseness=0.8] (5,0) to[out=90,in=-45] (5*1.41/2, 5*1.41/2);
	\fill[black!20!white] (-2, 2.25) to[out=100, in=180] (0,4) to[out=0, in=80] (2, 2.25) to[out=210, in=0] (0,1.75) to[out=180, in=-30] (-2, 2.25);
	\draw[knot, overcross] (0,1.75) to[out=0, in=-135] (5*1.41/2, 5*1.41/2);
	\draw[knot, overcross] (2,2) to[out=-90, in=45] (-5*1.41/2, -5*1.41/2);
	\draw[knot, overcross] (5*1.41/2, -5*1.41/2) to[out=135, in=-90] (-2,2);
	\draw[knot, overcross] (-2,2) to[out=90, in=180] (0,4);
	\draw[knot, overcross] (-5*1.41/2, 5*1.41/2) to[out=-45, in=180] (0, 1.75);
	\draw[knot, overcross] (0,4) to[out=0, in=90] (2,2);
	\draw[knot, looseness=1.5] (-5*1.41/2, -5*1.41/2) to[out=-135, in=180] (0,-6) to[out=0, in=-45] (5*1.41/2, -5*1.41/2);
	\draw[knot, looseness=1.5] (-5*1.41/2, 5*1.41/2) to[out=135, in=180] (0,6) to[out=0, in=45] (5*1.41/2, 5*1.41/2);
	\draw[white, ultra thick] (0,0) circle (5cm);
	\draw[dotted] (0,0) circle (5cm);
	\node at (0,-3.25) {\tiny $R_0$};
	\node at (0,0.65) {\tiny $R_1$};
	\node at (-2.4,3.6) {\tiny $R_2$};
}
\qquad \quad \qquad \text{and} \qquad \quad \qquad
\tikz[baseline={([yshift=-.5ex]current bounding box.center)}, scale=.25]{
	\fill[black!20!white] (-2, 2.25) to[out=-45, in=180] (0,1.75) to[out=0,in=-135] (2,2.25) to[out=-80, in=45] (0,-1.17) to[out=135, in=-100] (-2, 2.25);
	\fill[black!20!white] (-5*1.41/2, 5*1.41/2) to[out=45, in=180] (0,5) to[out=0, in=135] (5*1.41/2, 5*1.41/2) to[out=225, in=30] (2, 2.25) to[out=100, in=0] (0,4) to[out=180, in=80] (-2, 2.25) to[out=150, in=-45] (-5*1.41/2, 5*1.41/2);
	\fill[black!20!white] (-5*1.41/2, -5*1.41/2) to[out=45, in=-135] (0,-1.17) to[out=-45, in=135] (5*1.41/2, -5*1.41/2) to[out=-135, in=0] (0,-5) to[out=180, in=-45] (-5*1.41/2, -5*1.41/2);
	\draw[knot, overcross] (0,1.75) to[out=0, in=-135] (5*1.41/2, 5*1.41/2);
	\draw[knot, overcross] (2,2) to[out=-90, in=45] (-5*1.41/2, -5*1.41/2);
	\draw[knot, overcross] (5*1.41/2, -5*1.41/2) to[out=135, in=-90] (-2,2);
	\draw[knot, overcross] (-2,2) to[out=90, in=180] (0,4);
	\draw[knot, overcross] (-5*1.41/2, 5*1.41/2) to[out=-45, in=180] (0, 1.75);
	\draw[knot, overcross] (0,4) to[out=0, in=90] (2,2);
	\draw[knot, looseness=1.5] (-5*1.41/2, -5*1.41/2) to[out=-135, in=-90] (-6,0) to[out=90, in=135] (-5*1.41/2, 5*1.41/2);
	\draw[knot, looseness=1.5] (5*1.41/2, -5*1.41/2) to[out=-45, in=-90] (6,0) to[out=90, in=45] (5*1.41/2, 5*1.41/2);
	\draw[white, ultra thick] (0,0) circle (5cm);
	\draw[dotted] (0,0) circle (5cm);
	\node at (-3,0) {\tiny $R_0$};
	\node at (0,2.9) {\tiny $R_1$};
	\node at (3,0) {\tiny $R_2$};
}
\end{equation}
respectively. If $D$ is a reduced, alternating diagram in $\RP^2$, denote by $D^\pm$ the diagram $\overline{D(\kappa_\pm)}$ in $\mathbb{R}^2$ corresponding to the positive/negative checkerboard colorings. If $D$ is a diagram in $\mathbb{R}^2$, let $s_\mathbf{A}(D)$ and $s_{\mathbf{B}}(D)$ denote the number of circles in the all $A$- and all $B$-smoothings of $D$ respectively.

\begin{proposition}
\label{prop:goeritzsig}
If $D$ is an alternating diagram of a non-split, non-local link $L \subset \RP^3$, then
\newline
\begin{minipage}{.45\linewidth}
\begin{align*}
\sigma_+(L) &= s_\mathbf{A}(D^+) - n_+(D) - 1 \\
& = 1 + n_-(D) - s_\mathbf{B}(D^+)
\end{align*}
\end{minipage}
\begin{minipage}{.1\linewidth}
\[and\]
\end{minipage}
\begin{minipage}{.45\linewidth}
\begin{align*}
\sigma_-(L) & = s_\mathbf{A}(D^-) - n_+(D) - 1 \\
& = 1 + n_-(D) - s_\mathbf{B}(D^-).
\end{align*}
\end{minipage}
\end{proposition}

\begin{proof}
It is clear by definition that $\mathcal{G}_\pm = \mathcal{G}(D^\pm)$. In \cite{MR898151}, Murasugi shows that the signature of the latter is equal to $s_\mathbf{A}(D^\pm) - 1$ (see also \cite{MR2128055}). Since $G_\pm$ has only $\pm$-signed edges, $\mu(\kappa_\pm) = n_+(D)$. The other equalities follow similarly.
\end{proof}

\begin{proposition}
\label{prop:goeritzdet}
If $D$ is a reduced, alternating diagram for $L \subset \RP^3$, then 
\[
\det(\mathcal{G}_\pm) = \pm \abs{\{\text{spanning trees of}~G_\pm\}}.
\]
\end{proposition}

\begin{proof}
Notice that the Tait graphs obtained from the diagrams produced by the procedure in (\ref{eq:platish}) are equal as abstract graphs to the Tait graphs obtained from the checkerboard-colored diagram in $\RP^2$. Thus the result follows from the analogous statement in $S^3$ (which is proven using, say, Kirchhoff's Matrix-Tree Theorem).
\end{proof}

\subsection{Khovanov homology in $\RP^3$}
\label{ss:khbackground}

We briefly describe the $\F:= \Z/2\Z$-valued Khovanov complex for links in $\RP^3$. Given a diagram $D \subset \RP^2$ of $L\subset \RP^3$, take its hypercube of resolutions in the usual sense. That is, at each crossing, perform a smoothing as follows.
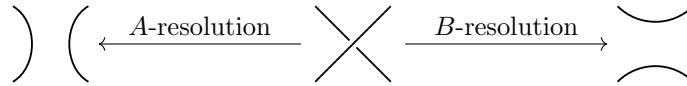
\[
\tikz[]{
\node(c) at (0,0) {$\tikz{
\draw[knot] (0,1) -- (1,0);
\draw[knot, overcross] (0,0) -- (1,1);
	}$};
\node(A) at (-4,0) {$\tikz{
\draw[knot] (0,0) to[out=30, in=-30] (0,1);
\draw[knot] (1,0) to[out=150, in=-150] (1,1);
	}$};
\node(B) at (4,0) {$\tikz{
\draw[knot] (0,0) to[out=45, in=135] (1,0);
\draw[knot] (0,1) to[out=-45, in=-135] (1,1);
	}$};
\draw[->] (c) -- node[above, midway]{$A$-resolution} (A);
\draw[->] (c) -- node[above, midway]{$B$-resolution} (B);
}
\]
Recall that for any $\mathbf{v} \in \{A, B\}^n$, $D_{\mathbf{v}}$ denotes the resolution obtained by applying the $v_i$-resolution at the $i$th crossing of $D$ for each $i=1,\ldots, n$.

Then, for $R$ a commutative ring, define two free $R$-modules
\[
V = R \langle 1, X \rangle, \qquad \text{and}, \qquad \bar{V} = R \langle \bar{1}, \bar{X} \rangle.
\]
To each resolution $D_\mathbf{v}$, assign a triply-graded chain group $\hat{C}(D_\mathbf{v})$ as follows. First, each component of $D_\mathbf{v}$ is either nullhomologous or homologically essential. As an $R$-module, $\hat{C}(D_\mathbf{v})$ is the tensor product of factors of $V$ for each trivial circle and $\bar{V}$ for an essential circle, if present (there can be at most one; see Lemma \ref{lem:ManWillis}). Then the tri-grading on any generator of $\hat{C}(D_\mathbf{v})$ is given by
\begin{align*}
& i = \abs{\mathbf{v}} \\
& j = i + \#(\text{circles labeled with}~1~\text{or}~\bar{1}) - \#(\text{circles labeled with}~X~\text{or}~\bar{X}) \\
& k = \#(\text{circles labeled with}~\bar{1}) - \#(\text{circles labeled with}~\bar{X})
\end{align*}
where $\abs{\mathbf{v}} = \#(B\text{-resolutions in}~\mathbf{v})$.

We say that the edge $e: \mathbf{v} \to \mathbf{u}$ of this hypercube corresponds to:
\begin{itemize}
\item a \emph{merge} if $D_\mathbf{u}$ is obtained from $D_\mathbf{v}$ by merging two circles into one;
\item a \emph{split} if $D_\mathbf{u}$ is obtained from $D_\mathbf{v}$ by splitting one circle into two;
\item a \emph{1-1 bifurcation} if $D_\mathbf{u}$ and $D_\mathbf{v}$ have the same number of circles.
\end{itemize}
Note that 1-1 bifurcations do not appear as possible cobordisms between resolutions of links in $S^3$.

\begin{example}
In Figure \ref{fig:maincube}, we picture a cube of resolutions associated to our usual diagram of $3_1$. The 1-1 bifurcations are represented by dotted arrows.

\begin{figure}[ht]
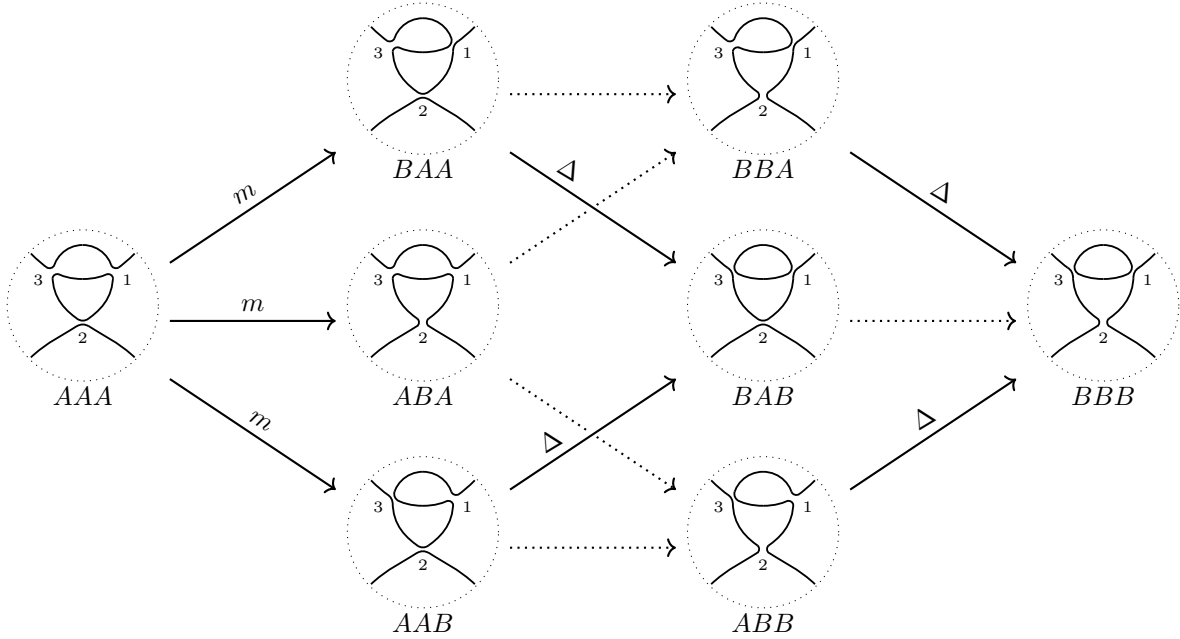

\tikz[yscale=1, xscale=0.9]{
	\node(AAA) at (0,0) {$\tikz[baseline={([yshift=-.5ex]current bounding box.center)}, scale=.2]{
	\node at (0,-6) {$AAA$};
	\draw[knot, overcross] (0,1.75) to[out=0, in=-135] (5*1.41/2, 5*1.41/2);
	\draw[knot, overcross] (2,2) to[out=-90, in=45] (-5*1.41/2, -5*1.41/2);
	\draw[knot, overcross] (5*1.41/2, -5*1.41/2) to[out=135, in=-90] (-2,2);
	\draw[knot, overcross] (-2,2) to[out=90, in=180] (0,4);
	\draw[knot, overcross] (-5*1.41/2, 5*1.41/2) to[out=-45, in=180] (0, 1.75);
	\draw[knot, overcross] (0,4) to[out=0, in=90] (2,2);
	\draw[white, ultra thick] (0,0) circle (5cm);
	\draw[dotted] (0,0) circle (5cm);
	\node[anchor=north] at (0,-1.18) {\tiny$2$};
	\node[anchor=west] at (2, 1.7) {\tiny$1$};
	\node[anchor=east] at (-2, 1.7) {\tiny$3$};
	\fill[white, very thick] (2, 2.25) circle (.45cm);
        \draw[knot] (2.375, 2.51) to[out=210,in=-75] (1.88, 2.685);
        \draw[knot] (1.995, 1.8) to[out=90, in=25] (1.58, 2.055);
	\fill[white, very thick] (0,-1.15) circle (.45cm);
        \draw[knot] (-0.4,-1.38) to[out=32,in=148] (0.4,-1.38);
        \draw[knot] (-0.375,-0.875) to[out=-36,in=-144] (0.375,-0.875);
	\fill[white, very thick] (-2, 2.25) circle (.45cm);
        \draw[knot] (-2.375, 2.51) to[out=-30,in=-105] (-1.88, 2.685);
        \draw[knot] (-1.995, 1.8) to[out=90, in=155] (-1.58, 2.055);
}$};
	\node(BAA) at (5,3) {$\tikz[baseline={([yshift=-.5ex]current bounding box.center)}, scale=.2]{
	\node at (0,-6) {$BAA$};
	\draw[knot, overcross] (0,1.75) to[out=0, in=-135] (5*1.41/2, 5*1.41/2);
	\draw[knot, overcross] (2,2) to[out=-90, in=45] (-5*1.41/2, -5*1.41/2);
	\draw[knot, overcross] (5*1.41/2, -5*1.41/2) to[out=135, in=-90] (-2,2);
	\draw[knot, overcross] (-2,2) to[out=90, in=180] (0,4);
	\draw[knot, overcross] (-5*1.41/2, 5*1.41/2) to[out=-45, in=180] (0, 1.75);
	\draw[knot, overcross] (0,4) to[out=0, in=90] (2,2);
	\draw[white, ultra thick] (0,0) circle (5cm);
	\draw[dotted] (0,0) circle (5cm);
	\node[anchor=north] at (0,-1.18) {\tiny$2$};
	\node[anchor=west] at (2, 1.7) {\tiny$1$};
	\node[anchor=east] at (-2, 1.7) {\tiny$3$};
	\fill[white, very thick] (2, 2.25) circle (.45cm);
       \draw[knot] (2.375, 2.51) to[out=210,in=90] (1.995, 1.8);
        \draw[knot] (1.88, 2.685) to[out=-75,in=25] (1.58, 2.055);
	\fill[white, very thick] (0,-1.15) circle (.45cm);
        \draw[knot] (-0.4,-1.38) to[out=32,in=148] (0.4,-1.38);
        \draw[knot] (-0.375,-0.875) to[out=-36,in=-144] (0.375,-0.875);
	\fill[white, very thick] (-2, 2.25) circle (.45cm);
        \draw[knot] (-2.375, 2.51) to[out=-30,in=-105] (-1.88, 2.685);
        \draw[knot] (-1.995, 1.8) to[out=90, in=155] (-1.58, 2.055);
}$};
	\node(ABA) at (5,0) {$\tikz[baseline={([yshift=-.5ex]current bounding box.center)}, scale=.2]{
	\node at (0,-6) {$ABA$};
	\draw[knot, overcross] (0,1.75) to[out=0, in=-135] (5*1.41/2, 5*1.41/2);
	\draw[knot, overcross] (2,2) to[out=-90, in=45] (-5*1.41/2, -5*1.41/2);
	\draw[knot, overcross] (5*1.41/2, -5*1.41/2) to[out=135, in=-90] (-2,2);
	\draw[knot, overcross] (-2,2) to[out=90, in=180] (0,4);
	\draw[knot, overcross] (-5*1.41/2, 5*1.41/2) to[out=-45, in=180] (0, 1.75);
	\draw[knot, overcross] (0,4) to[out=0, in=90] (2,2);
	\draw[white, ultra thick] (0,0) circle (5cm);
	\draw[dotted] (0,0) circle (5cm);
	\node[anchor=north] at (0,-1.18) {\tiny$2$};
	\node[anchor=west] at (2, 1.7) {\tiny$1$};
	\node[anchor=east] at (-2, 1.7) {\tiny$3$};
	\fill[white, very thick] (2, 2.25) circle (.45cm);
        \draw[knot] (2.375, 2.51) to[out=210,in=-75] (1.88, 2.685);
        \draw[knot] (1.995, 1.8) to[out=90, in=25] (1.58, 2.055);
	\fill[white, very thick] (0,-1.15) circle (.45cm);
        \draw[knot] (-0.4,-1.38) to[out=30,in=-33] (-0.375,-0.875);
        \draw[knot] (0.4,-1.38) to[out=150,in=213] (0.375,-0.875);
	\fill[white, very thick] (-2, 2.25) circle (.45cm);
        \draw[knot] (-2.375, 2.51) to[out=-30,in=-105] (-1.88, 2.685);
        \draw[knot] (-1.995, 1.8) to[out=90, in=155] (-1.58, 2.055);
}$};
	\node(AAB) at (5,-3) {$\tikz[baseline={([yshift=-.5ex]current bounding box.center)}, scale=.2]{
	\node at (0,-6) {$AAB$};
	\draw[knot, overcross] (0,1.75) to[out=0, in=-135] (5*1.41/2, 5*1.41/2);
	\draw[knot, overcross] (2,2) to[out=-90, in=45] (-5*1.41/2, -5*1.41/2);
	\draw[knot, overcross] (5*1.41/2, -5*1.41/2) to[out=135, in=-90] (-2,2);
	\draw[knot, overcross] (-2,2) to[out=90, in=180] (0,4);
	\draw[knot, overcross] (-5*1.41/2, 5*1.41/2) to[out=-45, in=180] (0, 1.75);
	\draw[knot, overcross] (0,4) to[out=0, in=90] (2,2);
	\draw[white, ultra thick] (0,0) circle (5cm);
	\draw[dotted] (0,0) circle (5cm);
	\node[anchor=north] at (0,-1.18) {\tiny$2$};
	\node[anchor=west] at (2, 1.7) {\tiny$1$};
	\node[anchor=east] at (-2, 1.7) {\tiny$3$};
	\fill[white, very thick] (2, 2.25) circle (.45cm);
        \draw[knot] (2.375, 2.51) to[out=210,in=-75] (1.88, 2.685);
        \draw[knot] (1.995, 1.8) to[out=90, in=25] (1.58, 2.055);
	\fill[white, very thick] (0,-1.15) circle (.45cm);
        \draw[knot] (-0.4,-1.38) to[out=32,in=148] (0.4,-1.38);
        \draw[knot] (-0.375,-0.875) to[out=-36,in=-144] (0.375,-0.875);
	\fill[white, very thick] (-2, 2.25) circle (.45cm);
        \draw[knot] (-2.375, 2.51) to[out=-30,in=90] (-1.995, 1.8);
        \draw[knot] (-1.88, 2.685) to[out=-105,in=155] (-1.58, 2.055);
}$};
	\node(BBA) at (10,3) {$\tikz[baseline={([yshift=-.5ex]current bounding box.center)}, scale=.2]{
	\node at (0,-6) {$BBA$};
	\draw[knot, overcross] (0,1.75) to[out=0, in=-135] (5*1.41/2, 5*1.41/2);
	\draw[knot, overcross] (2,2) to[out=-90, in=45] (-5*1.41/2, -5*1.41/2);
	\draw[knot, overcross] (5*1.41/2, -5*1.41/2) to[out=135, in=-90] (-2,2);
	\draw[knot, overcross] (-2,2) to[out=90, in=180] (0,4);
	\draw[knot, overcross] (-5*1.41/2, 5*1.41/2) to[out=-45, in=180] (0, 1.75);
	\draw[knot, overcross] (0,4) to[out=0, in=90] (2,2);
	\draw[white, ultra thick] (0,0) circle (5cm);
	\draw[dotted] (0,0) circle (5cm);
	\node[anchor=north] at (0,-1.18) {\tiny$2$};
	\node[anchor=west] at (2, 1.7) {\tiny$1$};
	\node[anchor=east] at (-2, 1.7) {\tiny$3$};
	\fill[white, very thick] (2, 2.25) circle (.45cm);
       \draw[knot] (2.375, 2.51) to[out=210,in=90] (1.995, 1.8);
        \draw[knot] (1.88, 2.685) to[out=-75,in=25] (1.58, 2.055);
	\fill[white, very thick] (0,-1.15) circle (.45cm);
        \draw[knot] (-0.4,-1.38) to[out=30,in=-33] (-0.375,-0.875);
        \draw[knot] (0.4,-1.38) to[out=150,in=213] (0.375,-0.875);
	\fill[white, very thick] (-2, 2.25) circle (.45cm);
        \draw[knot] (-2.375, 2.51) to[out=-30,in=-105] (-1.88, 2.685);
        \draw[knot] (-1.995, 1.8) to[out=90, in=155] (-1.58, 2.055);
}$};
	\node(BAB) at (10,0) {$\tikz[baseline={([yshift=-.5ex]current bounding box.center)}, scale=.2]{
	\node at (0,-6) {$BAB$};
	\draw[knot, overcross] (0,1.75) to[out=0, in=-135] (5*1.41/2, 5*1.41/2);
	\draw[knot, overcross] (2,2) to[out=-90, in=45] (-5*1.41/2, -5*1.41/2);
	\draw[knot, overcross] (5*1.41/2, -5*1.41/2) to[out=135, in=-90] (-2,2);
	\draw[knot, overcross] (-2,2) to[out=90, in=180] (0,4);
	\draw[knot, overcross] (-5*1.41/2, 5*1.41/2) to[out=-45, in=180] (0, 1.75);
	\draw[knot, overcross] (0,4) to[out=0, in=90] (2,2);
	\draw[white, ultra thick] (0,0) circle (5cm);
	\draw[dotted] (0,0) circle (5cm);
	\node[anchor=north] at (0,-1.18) {\tiny$2$};
	\node[anchor=west] at (2, 1.7) {\tiny$1$};
	\node[anchor=east] at (-2, 1.7) {\tiny$3$};
	\fill[white, very thick] (2, 2.25) circle (.45cm);
       \draw[knot] (2.375, 2.51) to[out=210,in=90] (1.995, 1.8);
        \draw[knot] (1.88, 2.685) to[out=-75,in=25] (1.58, 2.055);
	\fill[white, very thick] (0,-1.15) circle (.45cm);
        \draw[knot] (-0.4,-1.38) to[out=32,in=148] (0.4,-1.38);
        \draw[knot] (-0.375,-0.875) to[out=-36,in=-144] (0.375,-0.875);
	\fill[white, very thick] (-2, 2.25) circle (.45cm);
        \draw[knot] (-2.375, 2.51) to[out=-30,in=90] (-1.995, 1.8);
        \draw[knot] (-1.88, 2.685) to[out=-105,in=155] (-1.58, 2.055);
}$};
	\node(ABB) at (10,-3) {$\tikz[baseline={([yshift=-.5ex]current bounding box.center)}, scale=.2]{
	\node at (0,-6) {$ABB$};
	\draw[knot, overcross] (0,1.75) to[out=0, in=-135] (5*1.41/2, 5*1.41/2);
	\draw[knot, overcross] (2,2) to[out=-90, in=45] (-5*1.41/2, -5*1.41/2);
	\draw[knot, overcross] (5*1.41/2, -5*1.41/2) to[out=135, in=-90] (-2,2);
	\draw[knot, overcross] (-2,2) to[out=90, in=180] (0,4);
	\draw[knot, overcross] (-5*1.41/2, 5*1.41/2) to[out=-45, in=180] (0, 1.75);
	\draw[knot, overcross] (0,4) to[out=0, in=90] (2,2);
	\draw[white, ultra thick] (0,0) circle (5cm);
	\draw[dotted] (0,0) circle (5cm);
	\node[anchor=north] at (0,-1.18) {\tiny$2$};
	\node[anchor=west] at (2, 1.7) {\tiny$1$};
	\node[anchor=east] at (-2, 1.7) {\tiny$3$};
	\fill[white, very thick] (2, 2.25) circle (.45cm);
        \draw[knot] (2.375, 2.51) to[out=210,in=-75] (1.88, 2.685);
        \draw[knot] (1.995, 1.8) to[out=90, in=25] (1.58, 2.055);
	\fill[white, very thick] (0,-1.15) circle (.45cm);
        \draw[knot] (-0.4,-1.38) to[out=30,in=-33] (-0.375,-0.875);
        \draw[knot] (0.4,-1.38) to[out=150,in=213] (0.375,-0.875);
	\fill[white, very thick] (-2, 2.25) circle (.45cm);
        \draw[knot] (-2.375, 2.51) to[out=-30,in=90] (-1.995, 1.8);
        \draw[knot] (-1.88, 2.685) to[out=-105,in=155] (-1.58, 2.055);
}$};
	\node(BBB) at (15,0) {$\tikz[baseline={([yshift=-.5ex]current bounding box.center)}, scale=.2]{
	\node at (0,-6) {$BBB$};
	\draw[knot, overcross] (0,1.75) to[out=0, in=-135] (5*1.41/2, 5*1.41/2);
	\draw[knot, overcross] (2,2) to[out=-90, in=45] (-5*1.41/2, -5*1.41/2);
	\draw[knot, overcross] (5*1.41/2, -5*1.41/2) to[out=135, in=-90] (-2,2);
	\draw[knot, overcross] (-2,2) to[out=90, in=180] (0,4);
	\draw[knot, overcross] (-5*1.41/2, 5*1.41/2) to[out=-45, in=180] (0, 1.75);
	\draw[knot, overcross] (0,4) to[out=0, in=90] (2,2);
	\draw[white, ultra thick] (0,0) circle (5cm);
	\draw[dotted] (0,0) circle (5cm);
	\node[anchor=north] at (0,-1.18) {\tiny$2$};
	\node[anchor=west] at (2, 1.7) {\tiny$1$};
	\node[anchor=east] at (-2, 1.7) {\tiny$3$};
	\fill[white, very thick] (2, 2.25) circle (.45cm);
       \draw[knot] (2.375, 2.51) to[out=210,in=90] (1.995, 1.8);
        \draw[knot] (1.88, 2.685) to[out=-75,in=25] (1.58, 2.055);
	\fill[white, very thick] (0,-1.15) circle (.45cm);
        \draw[knot] (-0.4,-1.38) to[out=30,in=-33] (-0.375,-0.875);
        \draw[knot] (0.4,-1.38) to[out=150,in=213] (0.375,-0.875);
	\fill[white, very thick] (-2, 2.25) circle (.45cm);
        \draw[knot] (-2.375, 2.51) to[out=-30,in=90] (-1.995, 1.8);
        \draw[knot] (-1.88, 2.685) to[out=-105,in=155] (-1.58, 2.055);
}$};
\draw[->, thick] (AAA) -- node[above, midway, sloped]{$m$} (BAA);
\draw[->, thick] (AAA) -- node[above, midway, sloped]{$m$} (ABA);
\draw[->, thick] (AAA) -- node[above, midway, sloped]{$m$} (AAB);
\draw[->, thick, dotted] (ABA) -- (BBA);
\draw[->, thick, dotted] (ABA) -- (ABB);
\draw[->, thick, dotted] (BAA) -- (BBA);
\draw[->, line width=3pt, white] (BAA) -- (BAB);
\draw[->, thick] (BAA) -- node[above, pos=0.3, sloped]{$\Delta$} (BAB);
\draw[->, thick, dotted] (AAB) -- (ABB);
\draw[->, line width=3pt, white] (AAB) -- (BAB);
\draw[->, thick] (AAB) -- node[above, pos=0.3, sloped]{$\Delta$}(BAB);
\draw[->, thick] (BBA) -- node[above, midway, sloped]{$\Delta$} (BBB);
\draw[->, thick, dotted] (BAB) -- (BBB);
\draw[->, thick] (ABB) -- node[above, midway, sloped]{$\Delta$} (BBB);
}
\caption{The hypercube associated to the diagram of the 3-crossing nullhomologous knot $3_1$ from Drobotukhina's table.}
\label{fig:maincube}
\end{figure}
\end{example}

To any edge $e$ of the hypercube, we assign a map $d_e: \hat{C}(D_\mathbf{v}) \to \hat{C}(D_\mathbf{u})$. In this paper, we are only interested in the case where $L$ is nullhomologous, so we limit the definition of this map to the case where all components are class-0 and the third grading is trivial.
\begin{itemize}
\item If $e$ is a merge, then $d_e = \mathrm{Id}^{\otimes k} \otimes m \otimes \mathrm{Id}^{\otimes n-k-2}$ where
\[
m : V \otimes V \to V \qquad \text{by} \qquad \begin{cases} 1 \otimes 1 & \mapsto 1 \\ 1 \otimes X & \mapsto X \\ X \otimes 1 & \mapsto X \\ X \otimes X & \mapsto 0 \end{cases}
\]
\item If $e$ is a split, then $d_e = \mathrm{Id}^{\otimes k} \otimes \Delta \otimes \mathrm{Id}^{\otimes n - k -1}$ where
\[
\Delta : V \to V \otimes V \qquad \text{by} \qquad \begin{cases} 1 \mapsto 1 \otimes X + X \otimes 1 \\ X \mapsto X \otimes X \end{cases}
\]
\item If $e$ is a 1-1 bifurcation, then $d_e = 0$.
\end{itemize}

To simplify our discussion, we assume that $R = \Z/2\Z := \mathbb{F}$, so that $d^i = \sum_{\substack{e: \mathbf{v} \to \mathbf{u} \\ \abs{\mathbf{v}} = i}} d_e$ is a differential without mention of auxiliary choices (see \cite[\S 2]{MR4904031}). In conclusion, given an orientation on $L$, we define
\[
C(D) := \hat{C}(D) [-n_-, n_+ - 2n_-]
\]
where $[a, b]$ indicates a global shift of the bigrading and $n_{\pm}$ is the number of positive/negative crossings in $D$. The \emph{Khovanov homology} $\Kh(L; \F)$ of the oriented link $L$ is the cohomology of $C(D)$.

We can define a reduced version of Khovanov homology for links in $\RP^3$ exactly as in \cite[\S 3.1]{MR2034399}. We denote the \emph{reduced Khovanov complex} by $\widetilde{C}(D)$. The homology is an invariant of based links, and is denoted $\widetilde{\Kh}(L)$. In $\F$-coefficients, the Khovanov homology is two copies of the reduced homology:
\[
\Kh^{i,j}(L; \F) \cong \widetilde{\Kh}^{i, j-1}(L; \F) \oplus \widetilde{\Kh}^{i, j+1}(L; \F).
\]
To ensure that our conventions align with \cite{MR2480298}, we renormalize so that the reduced Khovanov homology of the unknot is supported in the bigrading $(0,-1)$.

\begin{example}
\label{ex:khmain}
We invite the reader to verify the computation asserted in the introduction:
\[
\tikz{
\node(K) at (0,0) {$\tikz[baseline={([yshift=-.5ex]current bounding box.center)}, scale=.25]{
	\draw[knot, overcross] (0,1.75) to[out=0, in=-135] (5*1.41/2, 5*1.41/2);
	\draw[knot, overcross] (2,2) to[out=-90, in=45] (-5*1.41/2, -5*1.41/2);
	\draw[knot, overcross] (5*1.41/2, -5*1.41/2) to[out=135, in=-90] (-2,2);
	\draw[knot, overcross] (-2,2) to[out=90, in=180] (0,4);
	\draw[knot, overcross] (-5*1.41/2, 5*1.41/2) to[out=-45, in=180] (0, 1.75);
	\draw[knot, overcross] (0,4) to[out=0, in=90] (2,2);
	\draw[white, ultra thick] (0,0) circle (5cm);
	\draw[dotted] (0,0) circle (5cm);
}$};
\node(T) at (7,0) {\begin{tabular}{|c||c|c|c|c|} 
\hline
$2$  &              &              &              & $\orange{\mathbb{F}}$  \\ 
\hline
$1$  &              &              & $\green{\mathbb{F}}$ &               \\ 
\hline
$0$  &              &              & $\orange{\mathbb{F}}$ & $\red{\mathbb{F}}$  \\ 
\hline
$-1$ &              & $\green{\mathbb{F}}$ & $\blue{\mathbb{F}}$ &               \\ 
\hline
$-2$ &              &              & $\red{\mathbb{F}}$ &               \\ 
\hline
$-3$ & $\green{\mathbb{F}}$ & $\blue{\mathbb{F}}$ &              &               \\ 
\hline
$-4$ &              &              &              &               \\ 
\hline
$-5$ & $\blue{\mathbb{F}}$ &              &              &               \\ 
\hhline{|=::====|}
   \slashbox{$j$}{$i$}   & $-2$         & $-1$         & $0$          & $1$           \\
\hline
\end{tabular}};
\draw[->, shorten <=5pt, shorten >=5pt, thick] (K) -- node[above, midway]{$\Kh$} (T)
}
\]
The two copies of the reduced homology are given in blue/red and green/orange respectively. Notice that this diagram for $3_1$ has $n_+ = 1$ and $n_- = 2$. We again remark upon a motivating irregularity: the $j$-gradings are not supported in a single parity. Summands supported in odd (resp. even) $j$-grading are pictured in cool (resp. warm) colors. Recalling our computation in Example \ref{ex:maingoeritz}, it is also satisfying that the average of the $j$-intercepts of the cool-colored diagonals is $0 = \sigma_-(L)$, and $-1 = -\sigma_+(L)$ for the warm-colored diagonals.
\end{example}

\section{The myopic Tutte polynomial for graphs in $\RP^2$}
\label{s:myopictutte}

The Tutte polynomial \cite{MR61366, MR1813436} of a graph $G$ is a polynomial $\chi_G \in \mathbb{Z}[x, y]$. If $G$ is the Tait graph of a reduced, alternating diagram, then $\chi_G(x=-t, y=-t^{-1})$ agrees with the Jones polynomial up to a multiple of $\pm t^k$ for some $k$ \cite{MR899051}. Classically, the Tutte polynomial is an invariant of $G$ as an abstract graph---it does not distinguish between different embeddings of $G$.

For our purposes, we need a version of the Tutte polynomial for graphs embedded in surfaces. There is a significant quantity of literature in this direction: extensions include the relative Tutte polynomial \cite{MR2607372}, Bollob\'as-Riordan polynomials \cite{MR1851080, MR1906909}, the Las Vergnas polynomial \cite{MR597150}, and, ultimately, the generalized Krushkal polynomial \cite{MR2769192, MR3739494}. We introduce a Tutte-like polynomial for graphs embedded in $\RP^2$ and show that it produces a model for Drobotukhina's polynomial for alternating links. This was already accomplished by the invariant introduced in \cite{MR3438450}. The novelty of our polynomial is that it is sensitive to the splitting $J_L(t) = J_L^+(t) + J_L^-(t)$: each summand corresponds to the distinct Tait graphs induced by a reduced, non-local, alternating diagram. In addition, while defined similarly to the abstract Tutte polynomial, it turns out that our polynomial recovers the Krushkal polynomial of graphs in $\RP^2$. We omit a comparison of our polynomial with the relative Tutte polynomial, but note that the Jones polynomial of a virtual link can be computed from the relative Tutte polynomial of an associated graph.

We start by fixing some terminology and notation. We call an edge $e$ of $G$ a \emph{bridge} if the removal of $e$ increases the number of components of $G$. We call an edge a \emph{loop} if its endpoints lie on the same vertex. If an edge is neither a bridge nor a loop, we call it \emph{regular}. If a graph has no edges apart from loops and bridges, we call it \emph{terminal}. We write $e \in G$ to mean ``an edge $e$ of $G$.'' The graph resulting from deleting a regular edge $e\in G$ is denoted by $G - e$. The graph resulting from contracting a regular edge $e\in G$ is denoted by $G/ e$.

We will write $G\subset \RP^2$ to mean an embedding of a graph $G$ into $\RP^2$. (We are most interested in this case because the Tait graphs of diagrams of nullhomologous links are naturally embeddings of graphs into $\RP^2$.) We conflate $G$ with the image of the embedding of $G$. We say that $G_1, G_2\in \RP^2$ are equal if there is an ambient isotopy taking one to the other. We call $G \subset \RP^2$ \emph{local} if there is a disk $\mathbb{D}^2$ such that $G \subset \mathbb{D}^2 \subset \RP^2$. A connected, non-local graph $G \subset \RP^2$ is called \emph{cellular}. 

\begin{definition}
\label{def:protutte}
The \emph{myopic Tutte polynomial} of a graph $G \subset \RP^2$ is a polynomial $\psi_G \in \mathbb{Z}[x, y]$ defined by the following rules.
\begin{enumerate}
\item If $e\in G$ is a regular edge, then $\psi_G = \psi_{G-e} + \psi_{G/e}$.
\item If $e\in G$ is not regular, then we have
\[
\psi_G = \begin{cases}
x \, \psi_{G/e} & \text{if $e$ is a bridge}, \\
y \, \psi_{G-e} & \text{if $e$ is a loop with $[e]=0\in H_1(\RP^2)$}, \\
\psi_{G-e} & \text{if $e$ is a loop with $[e]=1\in H_1(\RP^2)$}.
\end{cases}
\]
\item If $G$ is a graph with no edges, then $\psi_G = 1$.
\end{enumerate}
\end{definition}

We can use (1) to reduce a graph into terminal summands, and then apply (2) and (3). It is worth comparing this definition with Definition \ref{def:kauffmanbracket}. Note that $\psi_G$ would be ill-defined if we introduced a new variable for which $\psi_G = z \, \psi_{G - e}$ whenever $[e] = 1$; for example (conflating a graph $G$ with $\psi_G$), we would obtain
\begin{align*}
\tikz[baseline={([yshift=-.5ex]current bounding box.center)}, scale=.125]{
	\begin{scope}[blue, very thick]
		\fill (0,2.5) circle (.55cm);
		\fill (0,-2.5) circle (.55cm);
		\draw (0,-2.5) -- (0,-5);
		\draw (0,2.5) -- (0,5);
		\draw (0,-2.5) to[out=30, in=-30] (0,2.5);
		\draw (0,-2.5) to[out=150, in=-150] (0,2.5);
	\end{scope}
	\draw[white, ultra thick] (0,0) circle (5cm);
	\draw[dotted] (0,0) circle (5cm);
}
& = \tikz[baseline={([yshift=-.5ex]current bounding box.center)}, scale=.125]{
	\begin{scope}[blue, very thick]
		\fill (0,2.5) circle (.55cm);
		\fill (0,-2.5) circle (.55cm);
		\draw (0,-2.5) -- (0,-5);
		\draw (0,2.5) -- (0,5);
		\draw (0,-2.5) -- (0,2.5);
	\end{scope}
	\draw[white, ultra thick] (0,0) circle (5cm);
	\draw[dotted] (0,0) circle (5cm);
} + 
\tikz[baseline={([yshift=-.5ex]current bounding box.center)}, scale=.125]{
	\begin{scope}[blue, very thick]
		\fill (0,0) circle (.55cm);
		\draw (0,0) -- (0,-5);
		\draw (0,0) -- (0,5);
		\draw (0,0) edge [loop, out=45, in=-45, distance=5cm] (0,0);
	\end{scope}
	\draw[white, ultra thick] (0,0) circle (5cm);
	\draw[dotted] (0,0) circle (5cm);
}
\\
&=
\tikz[baseline={([yshift=-.5ex]current bounding box.center)}, scale=.125]{
	\begin{scope}[blue, very thick]
		\fill (0,2.5) circle (.55cm);
		\fill (0,-2.5) circle (.55cm);
		\draw (0,-2.5) -- (0,-5);
		\draw (0,2.5) -- (0,5);
	\end{scope}
	\draw[white, ultra thick] (0,0) circle (5cm);
	\draw[dotted] (0,0) circle (5cm);
} +
\tikz[baseline={([yshift=-.5ex]current bounding box.center)}, scale=.125]{
	\begin{scope}[blue, very thick]
		\fill (0,0) circle (.55cm);
		\draw (0,0) -- (0,-5);
		\draw (0,0) -- (0,5);
	\end{scope}
	\draw[white, ultra thick] (0,0) circle (5cm);
	\draw[dotted] (0,0) circle (5cm);
} +
\tikz[baseline={([yshift=-.5ex]current bounding box.center)}, scale=.125]{
	\begin{scope}[blue, very thick]
		\fill (0,0) circle (.55cm);
		\draw (0,0) -- (0,-5);
		\draw (0,0) -- (0,5);
		\draw (0,0) edge [loop, out=45, in=-45, distance=5cm] (0,0);
	\end{scope}
	\draw[white, ultra thick] (0,0) circle (5cm);
	\draw[dotted] (0,0) circle (5cm);
}
\\ &= x + z + yz
\end{align*}
if the initial application of (1) was applied along an internal edge, but
\begin{align*}
\tikz[baseline={([yshift=-.5ex]current bounding box.center)}, scale=.125]{
	\begin{scope}[blue, very thick]
		\fill (0,2.5) circle (.55cm);
		\fill (0,-2.5) circle (.55cm);
		\draw (0,-2.5) -- (0,-5);
		\draw (0,2.5) -- (0,5);
		\draw (0,-2.5) to[out=30, in=-30] (0,2.5);
		\draw (0,-2.5) to[out=150, in=-150] (0,2.5);
	\end{scope}
	\draw[white, ultra thick] (0,0) circle (5cm);
	\draw[dotted] (0,0) circle (5cm);
}
& = \tikz[baseline={([yshift=-.5ex]current bounding box.center)}, scale=.125]{
	\begin{scope}[blue, very thick]
		\fill (0,2.5) circle (.55cm);
		\fill (0,-2.5) circle (.55cm);
		\draw (0,-2.5) to[out=30, in=-30] (0,2.5);
		\draw (0,-2.5) to[out=150, in=-150] (0,2.5);
	\end{scope}
	\draw[white, ultra thick] (0,0) circle (5cm);
	\draw[dotted] (0,0) circle (5cm);
} + 
\tikz[baseline={([yshift=-.5ex]current bounding box.center)}, scale=.125]{
	\begin{scope}[blue, very thick]
		\fill (0,0) circle (.55cm);
		\draw (-5*0.70711,-5*0.70711) -- (5*0.70711, 5*0.70711);
		\draw (-5*0.70711, 5*0.70711) -- (5*0.70711,-5*0.70711);
	\end{scope}
	\draw[white, ultra thick] (0,0) circle (5cm);
	\draw[dotted] (0,0) circle (5cm);
}
\\
&=
\tikz[baseline={([yshift=-.5ex]current bounding box.center)}, scale=.125]{
	\begin{scope}[blue, very thick]
		\fill (0,2.5) circle (.55cm);
		\fill (0,-2.5) circle (.55cm);
		\draw (0,-2.5) -- (0,2.5);
	\end{scope}
	\draw[white, ultra thick] (0,0) circle (5cm);
	\draw[dotted] (0,0) circle (5cm);
} +
\tikz[baseline={([yshift=-.5ex]current bounding box.center)}, scale=.125]{
	\begin{scope}[blue, very thick]
		\fill (0,0) circle (.55cm);
		\draw (0,0) edge [loop, out=135, in=-135, distance=5cm] (0,0);
	\end{scope}
	\draw[white, ultra thick] (0,0) circle (5cm);
	\draw[dotted] (0,0) circle (5cm);
} +
\tikz[baseline={([yshift=-.5ex]current bounding box.center)}, scale=.125]{
	\begin{scope}[blue, very thick]
		\fill (0,0) circle (.55cm);
		\draw (-5*0.70711,-5*0.70711) -- (5*0.70711, 5*0.70711);
		\draw (-5*0.70711, 5*0.70711) -- (5*0.70711,-5*0.70711);
	\end{scope}
	\draw[white, ultra thick] (0,0) circle (5cm);
	\draw[dotted] (0,0) circle (5cm);
}
\\ &= x + y + z^2
\end{align*}
if we perform our initial application of (1) along the edge passing through the boundary of the disk-model. These expressions are equivalent if and only if $z = 1$ or $z = y$. The latter choice produces the ordinary Tutte polynomial; we will show that the former choice produces a well-defined invariant of graphs in $\RP^2$ by giving a state-sum model for $\psi_G$ in Proposition \ref{prop:tutte2}.

\begin{example}
\label{ex:PT_Tutte}
We have just computed the myopic Tutte polynomial for one of the cellular graphs obtained from the diagram in Example \ref{ex:PT_KB}:
\begin{align*}
\tikz[baseline={([yshift=-.5ex]current bounding box.center)}, scale=.125]{
	\begin{scope}[blue, very thick]
		\fill (0,2.5) circle (.55cm);
		\fill (0,-2.5) circle (.55cm);
		\draw (0,-2.5) -- (0,-5);
		\draw (0,2.5) -- (0,5);
		\draw (0,-2.5) to[out=30, in=-30] (0,2.5);
		\draw (0,-2.5) to[out=150, in=-150] (0,2.5);
	\end{scope}
	\draw[white, ultra thick] (0,0) circle (5cm);
	\draw[dotted] (0,0) circle (5cm);
}
= x + 1 + y.
\end{align*}
For the other, we have
\begin{align*}
\tikz[baseline={([yshift=-.5ex]current bounding box.center)}, scale=.125]{
	\begin{scope}[red, very thick]
		\fill (0, 3) circle (.55cm);
		\fill (0,0) circle (.55cm);
		\draw (0,-5) -- (0,5);
		\draw (-5,0) -- (5,0);
	\end{scope}
	\draw[white, ultra thick] (0,0) circle (5cm);
	\draw[dotted] (0,0) circle (5cm);
} &= \tikz[baseline={([yshift=-.5ex]current bounding box.center)}, scale=.125]{
	\begin{scope}[red, very thick]
		\fill (0, 3) circle (.55cm);
		\fill (0,0) circle (.55cm);
		\draw (-5,0) -- (5,0);
		\draw (0,3) -- (0,5);
		\draw (0,0) -- (0,-5);
	\end{scope}
	\draw[white, ultra thick] (0,0) circle (5cm);
	\draw[dotted] (0,0) circle (5cm);
} +
\tikz[baseline={([yshift=-.5ex]current bounding box.center)}, scale=.125]{
	\begin{scope}[red, very thick]
		\fill (0,0) circle (.55cm);
		\draw (0,-5) -- (0,5);
		\draw (-5,0) -- (5,0);
	\end{scope}
	\draw[white, ultra thick] (0,0) circle (5cm);
	\draw[dotted] (0,0) circle (5cm);
} 
 =
x + 1.
\end{align*}
Note that the classical Tutte polynomial $\chi_G$ satisfies $\chi_{G^*}(x,y) = \chi_G(y,x)$ for $G^*$ the graph dual to $G$, but this computation illustrates that no such duality holds for $\psi_G$. Now, denote the blue graph by $G_1$ and the dual red graph by $G_2$. Remarkably, setting $x = -t$ and $y=-t^{-1}$, we find that
\begin{align*}
(-t^{-1}) \psi_{G_1}(-t, -t^{-1}) + (-t^{-\frac{1}{2}}) \psi_{G_2}(-t, -t^{-1}) &= (-t^{-1})(-t + 1 -t^{-1}) + (-t^{-\frac{1}{2}})(-t+1)
\\ &= (1 -t^{-1} + t^{-2}) + (t^{\frac{1}{2}} - t^{-\frac{1}{2}})
\\ &= J_L(t).
\end{align*}
Theorem \ref{thm:jonespolynomialalternatingtutte} asserts that this computation is indicative of a general fact: to obtain $J_L(t)$ for any non-local, alternating link $L \subset \RP^3$, we can find $k_1$ and $k_2$ such that
\[
J_L(t) = \pm (t^{k_1} \psi_{G_1}(-t, -t^{-1}) + t^{k_2} \psi_{G_2}(-t^{-1}, -t)).
\]
\end{example}

By a \emph{spanning subgraph} of $G$, we mean a subgraph $H$ whose vertex set coincides with that of $G$. Let $i: G \hookrightarrow \RP^2$ be the embedding of $G$ into $\RP^2$ and let $i^H := i|_{H}$ for any subgraph $H\subseteq G$. Denote the induced map on homology by $i_*^H: H_1(H; \F)\cong \F^{e(H) - v(H) + b_0(H)} \to H_1(\RP^2) \cong \F$.

\begin{proposition}
\label{prop:tutte2}
For any embedded graph $G \subset \RP^2$,
\begin{equation}
\label{eq:tutte2}
\psi_G(x, y) = \sum_{H \subseteq G} \delta(H) (x-1)^{b_0(H) - b_0(G)} (y-1)^{b_1(H)}
\end{equation}
where the sum is taken over all spanning subgraphs and
\[
\delta(H) = \begin{cases} 0 & \text{if} ~ b_1(H) > \dim \ker(i_*^H),~\text{and} \\ 1 & \text{if} ~ b_1(H) = \dim \ker(i_*^H). \end{cases}
\]
In particular, $\psi_G$ is well-defined.
\end{proposition}

We call (\ref{eq:tutte2}) the \emph{state-sum model} for $\psi_G$. Notice that $b_1(H) = \dim \ker(i_*^H)$ if and only if $H$ is contained within a disk in $\RP^2$, \textit{i.e.}, $H$ is a local spanning subgraph of $G$.

\begin{proof}
The result follows from a routine verification that the state-sum model obeys the rules provided in Definition \ref{def:protutte}. We write $\widetilde{\psi}_G = \sum_{H \subseteq G} \widetilde{\psi}_G^H$ where $\widetilde{\psi}_G^H = (x-1)^{b_0(H) - b_0(G)} (y-1)^{b_1(H)}$ under the assumption that $H$ is a local spanning subgraph of $G$. Note that $H \subseteq G$ naturally induces spanning subgraphs $H'$ on $G - e$ and $G / e$.

For \ref{def:protutte}(1), assume that $e \in G$ is a regular edge. We write
\[
\widetilde{\psi}_G = \sum_{\substack{H \subseteq G \\ e \in H}} \widetilde{\psi}_G^H + \sum_{\substack{H \subseteq G \\ e \not \in H}} \widetilde{\psi}_G^H
\]
and claim that the first summation is $\widetilde{\psi}_{G/e}$ and the second is $\widetilde{\psi}_{G - e}$. Both claims are straightforward. First, assuming $e\in H$, we have $\delta(H/e) = \delta(H)$. Edge contraction preserves the number of components and, since $e$ is not a loop, $b_1(H) = b_1(H/e)$. Next, if $e \not\in H$, then $H$ is a spanning subgraph of $G - e$. Since $e$ is not a bridge, $b_0(G - e) = b_0(G)$, and the result follows.

For \ref{def:protutte}(2), by the reasoning above, we can write
\[
\widetilde{\psi}_G = \widetilde{\psi}_{G / e} + \sum_{\substack{H \subseteq G \\ e \not \in H}} \widetilde{\psi}_G^H 
\]
if $e\in G$ is a bridge. If $e\not\in H$, let $H'$ be the spanning subgraph of $G / e$ with exactly the same edge set as $H$ but with vertex set $V(G) / e$. Notice that $b_0(G/e) = b_0(G)$ and $b_1(H') = b_1(H)$ but that $b_0(H') = b_0(H) - 1$. Thus
\[
\widetilde{\psi}_G = \widetilde{\psi}_{G/e} + (x-1) \widetilde{\psi}_{G/e} = x\, \widetilde{\psi}_{G/e}.
\]
Similarly, if $e$ is a loop, then
\[
\widetilde{\psi}_G = \sum_{\substack{H \subseteq G \\ e \in H}} \widetilde{\psi}_G^H + \widetilde{\psi}_{G-e}.
\]
If $e \in H$, let $H'$ be the spanning subgraph $H - e$ of $G - e$. Then $b_0(H) = b_0(H')$ and $b_0(G) = b_0(G-e)$, but $b_1(H') = b_1(H) - 1$. Then, on one hand, if $[e]=0$, then
\[
\widetilde{\psi}_G = (y-1) \widetilde{\psi}_{G-e} + \widetilde{\psi}_{G-e} = y \, \widetilde{\psi}_{G-e}.
\]
On the other hand, when $[e]=1$, then $\delta(H) = 0$ if $e \in H$, so we have $\widetilde{\psi}_G = \widetilde{\psi}_{G-e}$.

Finally, for \ref{def:protutte}(3), if $G$ has no edges, then it is its only spanning subgraph, so $\widetilde{\psi}_G = 1$. We conclude that $\widetilde{\psi}_G(x,y) = \psi_G(x,y)$.
\end{proof}

\begin{remark}
We note that there are generalizations of the myopic Tutte polynomial. First, we could define the myopic Tutte polynomial for an embedded graph $\G = (G, \iota: G \hookrightarrow \Sigma)$ on any surface $\Sigma$, with the only difference being that $H_1(\RP^2)$ is replaced by $H_1(\Sigma)$ in the definition of $i_*^H$. We could also consider the polynomial for $\G = (G, \iota: G \hookrightarrow \RP^3)$, an embedded graph whose image is subsequently projected onto $\RP^2$. We do not investigate graphs embedded into $\RP^3$ or their myopic Tutte polynomials in this paper.
\end{remark}

\begin{example}
At this point, we can observe some features of myopic Tutte polynomials. Let's start by comparing the state-sum model of $\psi_G$ with those of the abstract Tutte polynomial $\chi_G$ and the generalized Krushkal polynomial \cite{MR2769192, MR3739494}, denoted by $K_G \in \mathbb{Z}[x, y, \sqrt{A}, \sqrt{B}]$ (see Definition \ref{def:krushkal}). A relevant computation is provided in Figure \ref{fig:statesumcomps}, where we consider the dual graphs from Example \ref{ex:PT_Tutte}.

\begin{figure}[ht]
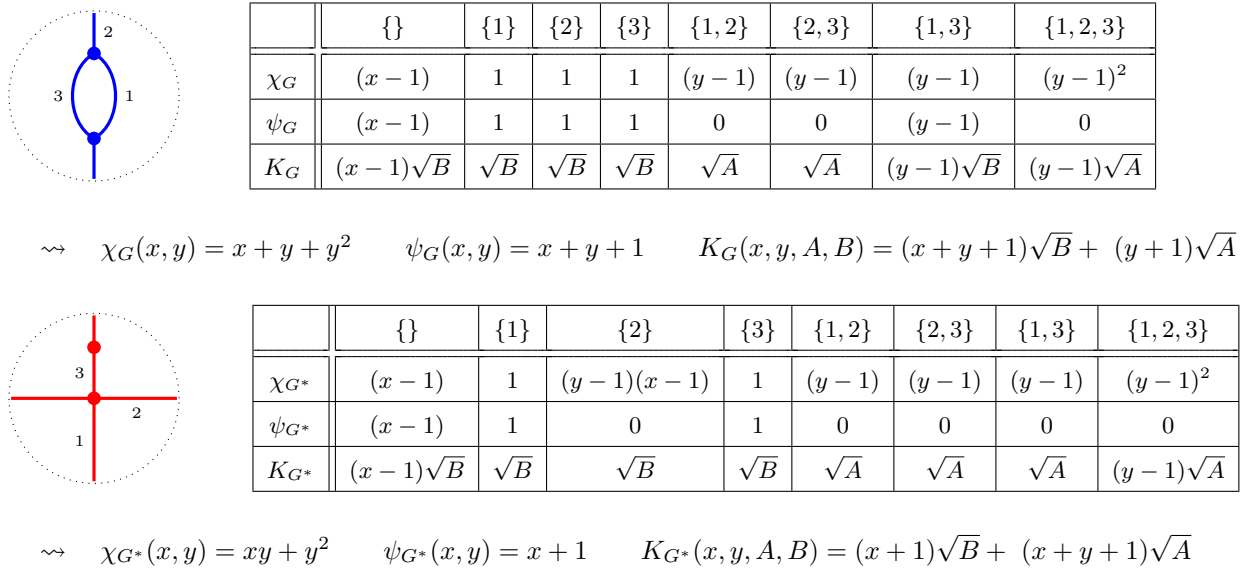

\tikz[baseline={([yshift=-.5ex]current bounding box.center)}, scale=.2]{
\node at (0,0) {\tikz[baseline={([yshift=-.5ex]current bounding box.center)}, scale=.225]{
\begin{scope}[blue, very thick]
	\fill (0,2.5) circle (.4cm);
	\fill (0,-2.5) circle (.4cm);
	\draw (0,-2.5) -- (0,-5);
	\draw (0,2.5) -- (0,5);
	\draw (0,-2.5) to[out=30, in=-30] (0,2.5);
	\draw (0,-2.5) to[out=150, in=-150] (0,2.5);
\end{scope}
\node[anchor=west] at (0,3.75) {\tiny$2$};
\node[anchor=east] at (-1.25,0) {\tiny$3$};
\node[anchor=west] at (1.25,0) {\tiny$1$};
\draw[white, ultra thick] (0,0) circle (5cm);
\draw[dotted] (0,0) circle (5cm);
}};
\node at (40,0) {\begingroup
\small
\renewcommand{\arraystretch}{1.5}
\begin{tabular}{|c||c|c|c|c|c|c|c|c|} 
\hline
            & $\{\}$          & $\{1\}$    & $\{2\}$    & $\{3\}$    & $\{1,2\}$  & $\{2,3\}$  & $\{1,3\}$       & $\{1,2,3\}$      \\ 
\hhline{|=::========|}
$\chi_G$             & $(x-1)$         & $1$        & $1$        & $1$        & $(y-1)$    & $(y-1)$    & $(y-1)$         & $(y-1)^2$        \\ 
\hline
$\psi_G$             & $(x-1)$         & $1$        & $1$        & $1$        & $0$        & $0$        & $(y-1)$         & $0$              \\ 
\hline
$K_G$                & $(x-1)\sqrt{B}$ & $\sqrt{B}$ & $\sqrt{B}$ & $\sqrt{B}$ & $\sqrt{A}$ & $\sqrt{A}$ & $(y-1)\sqrt{B}$ & $(y-1)\sqrt{A}$  \\ 
\hline                                
\end{tabular}
\endgroup};
\node at (36,-10) {$\leadsto \quad \chi_G(x,y) = x + y +y^2 \qquad \psi_G(x,y) = x+y+1 \qquad K_G(x, y, A, B) = (x+y+1)\sqrt{B} +  (y+1)\sqrt{A}$};
\node at (0,-20) {\tikz[baseline={([yshift=-.5ex]current bounding box.center)}, scale=.225]{
\begin{scope}[red, very thick]
	\fill (0, 3) circle (.4cm);
	\fill (0,0) circle (.4cm);
	\draw (0,-5) -- (0,5);
	\draw (-5,0) -- (5,0);
\end{scope}
\node[anchor=north] at (2.5, 0) {\tiny$2$};
\node[anchor=east] at (0,1.5) {\tiny$3$};
\node[anchor=east] at (0,-2.5) {\tiny$1$};
\draw[white, ultra thick] (0,0) circle (5cm);
\draw[dotted] (0,0) circle (5cm);
}};
\node at (42.85,-20) {\begingroup
\small
\renewcommand{\arraystretch}{1.5}
\begin{tabular}{|c||c|c|c|c|c|c|c|c|} 
\hline
                     & $\{\}$          & $\{1\}$    & $\{2\}$      & $\{3\}$    & $\{1,2\}$  & $\{2,3\}$  & $\{1,3\}$  & $\{1,2,3\}$      \\ 
\hhline{|=::========|}
$\chi_{G^*}$         & $(x-1)$         & $1$        & $(y-1)(x-1)$ & $1$        & $(y-1)$    & $(y-1)$    & $(y-1)$    & $(y-1)^2$        \\ 
\hline
$\psi_{G^*}$         & $(x-1)$         & $1$        & $0$          & $1$        & $0$        & $0$        & $0$        & $0$              \\ 
\hline
$K_{G^*}$            & $(x-1)\sqrt{B}$ & $\sqrt{B}$ & $\sqrt{B}$   & $\sqrt{B}$ & $\sqrt{A}$ & $\sqrt{A}$ & $\sqrt{A}$ & $(y-1)\sqrt{A}$  \\ 
\hline                        
\end{tabular}
\endgroup};
\node at (34.5,-30) {$\leadsto \quad \chi_{G^*}(x,y) = xy+y^2 \qquad \psi_{G^*}(x,y) = x+1 \qquad K_{G^*}(x, y, A, B) = (x+1)\sqrt{B} +  (x+y+1)\sqrt{A}$};
}
\caption{Computing the Tutte, myopic Tutte, and Krushkal polynomials via their state-sum definitions for the graphs obtained from a diagram of $3_1$ from Drobotukhina's table.}
\label{fig:statesumcomps}
\end{figure}

It is well known that the Krushkal polynomial generalizes the Tutte polynomial for embedded graphs and that $K_G$ satisfies a duality statement analogous to the Tutte polynomial's; we review these properties in \S \ref{ss:krushkal}. Proposition \ref{prop:tutte2} establishes that the myopic Tutte polynomial $\psi_G$ of embedded graphs is related to the Tutte polynomial for abstract graphs by dismissing summands of the state-sum formula which are homologically essential. It is also evident in this example that
\[
K_G(x, y, 0, 1) = \psi_G(x, y).
\]
Proposition \ref{prop:fromkrushkal} states that this holds in general. Notably, we will also prove in Theorem \ref{thm:equalkrushkal} that, for cellular graphs $G \subset \RP^2$, the generalized Krushkal polynomial $K_G$ is equivalent to the data of $\{\psi_G, \psi_{G^*}\}$. Namely,
\[
K_G(x, y, A, B) = \sqrt{B} \, \psi_G(x,y) + \sqrt{A} \, \psi_{G^*}(y, x),
\]
as is the case in this example. Finally, these results imply that the (abstract) Tutte polynomial can be recovered: in general,
\[
\chi_G(x,y) = \psi_G(x,y) + (y-1) \, \psi_{G^*}(y, x),
\]
see Corollary \ref{cor:tuttefrommyopic}.
\end{example}

\subsection{Spanning tree expansion for the myopic Tutte polynomial}
\label{ss:spanningtreemyopic}

We will need one more description of the myopic Tutte polynomial in terms of spanning trees, mirroring Tutte's original definition \cite{MR61366}. It requires a few relatively burdensome definitions which will be used throughout the rest of the paper.

Assume that $G$ is a graph in $\RP^2$ with $n$ vertices. An acyclic subgraph $H$ of $G$ is a spanning tree of $G$ if and only if $H$ has $n-1$ edges. Assume that $T$ is a spanning tree of $G$ and that $e \in G$.
\begin{itemize}
\item If $e \in T$, then $T - e$ has two components, $P_1$ and $P_2$. Let $\mathsf{Cut}(T, e)$ denote the subgraph of $G$ consisting of edges of $G$ with one endpoint in $P_1$ and the other endpoint in $P_2$.
\item If $e \not\in T$, then $T \cup e$ has exactly one cycle. Let $\mathsf{Cyc}(T,e)$ denote the subgraph of $T \cup e$ consisting only of this cycle.
\end{itemize}
Note that in either situation, $e \in \cut(T,e)$ and $e\in \cyc(T,e)$. In addition, $e\in \cut(T,e')$ if and only if $e' \in \cyc(T, e)$.

Now, assume that $G$ is a connected graph and place an arbitrary ordering on its edges. We say the edge $e_i$ \emph{precedes} the edge $e_j$ if $i < j$. Let $e_i \in G$ and assume that $T$ is a spanning tree of $G$.
\begin{itemize}
\item If $e_i \in T$, it is called \emph{internally active with respect to $T$} if $e_i$ precedes all other edges in $\cut(T, e_i)$. Otherwise, it is called \emph{internally inactive with respect to $T$}.
\item If $e_i \not\in T$, it is called \emph{externally active with respect to $T$} if $e_i$ precedes all other edges in $\cyc(T,e_i)$. Moreover, in this case, $e_i$ is called
\begin{itemize}
\item \emph{locally externally active with respect to $T$} if $[\cyc(T,e_i)] = 0 \in H_1(\RP^2)$, and
\item \emph{projectively externally active with respect to $T$} if $[\cyc(T, e_i)] = 1 \in H_1(\RP^2)$.
\end{itemize}
Otherwise, it is called \emph{externally inactive with respect to $T$}.
\end{itemize}

Finally, for a given spanning tree $T$ of a connected graph $G$ with ordered edges, we define:
\begin{itemize}
\item the \emph{internal activity of $T$} to be the number of edges of $G$ which are internally active with respect to $T$;
\item the \emph{local external activity of $T$} to be the number of edges of $G$ which are locally externally active with respect to $T$; and
\item the \emph{projective external activity of $T$} to be the number of edges of $G$ which are projectively externally active with respect to $T$.
\end{itemize}

\begin{proposition}
\label{prop:tutte_spanningtrees}
If $G\subset \RP^2$ is a connected graph with ordered edges, then
\[
\psi_G(x, y) = \sum_{T \subset G} x^i y^j
\]
where the sum is taken over all spanning trees of $G$, $i$ is the internal activity of $T$, and $j$ is the local external activity of $T$. In particular, the myopic Tutte polynomial is independent of edge ordering.
\end{proposition}

We will introduce some notation which will be helpful for our proof and the succeeding discussion. Each edge $e$ of any spanning tree $T$ of $G$ falls into one of five states with respect to $T$ depending on whether $e$ is internal or external, active or inactive and, if externally active, local or projective. We associate to each of these states a symbol\footnote{prescribed originally by John Conway; see \cite{MR899051}.} according to the decision tree of Figure \ref{fig:decisiontree1}.

\begin{figure}[h]
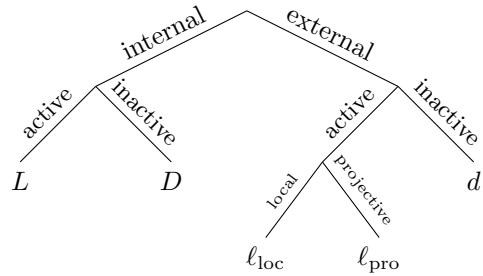

\tikz[]{
	\draw (0,0) -- (2,-1) node [midway, above, sloped] {external};
	\draw (0,0) -- (-2,-1) node [midway, above, sloped] {internal};
	\draw (2,-1) -- (1,-2) node [midway, above, sloped] {\small active};
	\draw (2,-1) -- (3,-2) node [midway, above, sloped] {\small inactive};
	\draw (-2,-1) -- (-1,-2) node [midway, above, sloped] {\small inactive};
	\draw (-2,-1) -- (-3,-2) node [midway, above, sloped] {\small active};
	\draw (1,-2) -- (0.25, -3) node [midway, above, sloped] {\tiny local};
	\draw (1,-2) -- (1.75, -3) node [midway, above, sloped] {\tiny projective};
	\node[anchor=north] at (-3,-2) {$L$};
	\node[anchor=north] at (-1,-2) {$D$};
	\node[anchor=north] at (0.25,-3) {$\ell_{\text{loc}}$};
	\node[anchor=north] at (1.75,-3) {$\ell_{\text{pro}}$};
	\node[anchor=north] at (3,-2) {$d$};
}
\caption{Determining the state of any edge in a graph relative to a spanning tree.}
\label{fig:decisiontree1}
\end{figure}

If $e$ is an edge of $G$, we denote the state of $e$ relative $T$ by $s_T(e)$. If we want to suppress whether an externally active edge is local or projective, we use the label $\ell$, and write $s^0_T(e) \in \{L, D, \ell, d\}$ to denote the suppressed state. The state of the spanning tree $T$ relative an edge ordering of $G$ is the (unordered) set $s(T) = \{s_T(e_1), \ldots, s_T(e_m)\}$, usually presented as a word $s_T(e_1)\cdots s_T(e_m)$

\begin{proof}[Proof of Proposition \ref{prop:tutte_spanningtrees}]
Let $\displaystyle \psi_G'(x, y) = \sum_{T \subset G} x^i y^j$. Again, we will show that $\psi_G'$ satisfies the defining relations of $\psi_G$. First, assume that $e$ is a regular edge, and (assuming the second statement of the proposition) that $e$ is the highest-ordered regular edge. Thus, $e$ is inactive with respect to each spanning tree of $G$, so it never contributes to the internal or external activities of any spanning tree. Hence
\begin{align*}
\psi_G'(x, y) = \sum_{T \subset G} x^i y^j = \sum_{\substack{T \subset G \\ e \in T}} x^i y^j + \sum_{\substack{T \subset G \\ e\not\in T}} x^i y^j = \psi_{G/e}' + \psi_{G-e}' 
\end{align*}
where the edge orderings on $G-e$ and $G/e$ are the natural ones induced by the ordering on $G$. Now, assume that $e$ is irregular. If $e$ is a bridge, then it belongs to each spanning tree of $G$, and $\cut(T, e)$ is the subgraph of $G$ consisting only of $e$. Thus $e$ is internally active with respect to each spanning tree $T$, so
\[
\psi_G' = x \, \psi_{G/e}'.
\]
If $e$ is a loop, then $\cyc(T, e)$ is exactly $e$. Depending on whether $e$ is a local loop or a projective loop, it follows that
\[
\psi_G' = y\, \psi_{G- e}' \qquad \text{or} \qquad \psi_G' = \psi_{G-e}'
\]
respectively. Finally, since $G$ is connected by assumption, any graph with no edges must be the graph with a single vertex. The graph itself is the only spanning tree, hence $\psi'_{\odot} = x^0 y^0 = 1$.

Now we prove that $\psi_G$ is independent under edge ordering. It suffices to examine the effect of transposing labels adjacent in a given edge ordering. Let $G$ and $G'$ be identical graphs in $\RP^2$ with edge orderings which satisfy
\[
e_i' = e_{i+1}, \qquad e_{i+1}' = e_i, \qquad \text{and}\qquad e_j' = e_j ~ \text{for each}~ j \not= i, i+1.
\]
First, it is apparent that $s_T(e_j) = s_T(e_j')$ for any spanning tree $T$ if $j \not= i, i+1$; transposing adjacent labels cannot affect activity or inactivity, and $[T \cup e_j] = [T \cup e_j'] \in H_1(\RP^2)$ since $e_j = e_j'$. Next, in \cite[pp. 85--88]{MR61366} (or \cite[Chapter IX]{MR1813436}), Tutte proves that a change in activity of $e_i$ or $e_{i+1}$ is only possible if
\begin{enumerate}[label=(\roman*)]
\item one of them, say $e_i$, is in $T$, but $e_{i+1} \not\in T$,
\item $e_i \in \cyc(T, e_{i+1})$, and
\item each edge $e_j$ for $j \not=i, i+1$ satisfies $s_T^0(e_j) = s_{\sigma(T)}^0(e_j)$ where $\sigma(T)$ is the tree obtained by replacing $e_i\in T$ with $e_{i+1} \not\in T$.
\end{enumerate}
Tutte's work holds in our setting, but there is a nontriviality: it can now be the case that $s_T(e_j) \not= s_{\sigma(T)}(e_j)$ when $s_T(e_j) = \ell_{\text{pro}}$ and $s_{\sigma(T)}(e_j) = \ell_{\text{loc}}$, or vice versa.

Let us enumerate the possible changes in activity which follow from Tutte's description. Start by assuming that $e_i\in T$ is internally active. Since $e_i \in \cyc(T, e_{i+1})$, $e_{i+1}$ is forced to be externally inactive. Recall that $e_i \in \cyc(T, e_{i+1})$ implies that $e_{i+1} \in \cyc(\sigma(T), e_i)$. Thus, $s_{\sigma(T)}(e_{i+1}) = D$ independent of whether $s_{\sigma(T)} (e_i) = d$ or $s_{\sigma(T)}^0(e_i) = \ell$. This means that the states associated $e_i, e_{i+1}$ and $e_i', e_{i+1}'$ are interchanged by the reordering. Otherwise, $e_i$ is internally inactive with respect to $T$, so $s_T(e_{i+1}) = d$. If $s_{\sigma(T)} (e_i) = d$, then $s_{\sigma(T)}(e_{i+1}) = D$, and there is no change in activity after reordering. If $s^0_{\sigma(T)}(e_i) = \ell$, then we still have $s_{\sigma(T)} (e_{i+1}) = D$, but there is a change in activity interchanging the states associated to $e_i, e_{i+1}$ and $e_i', e_{i+1}'$. In \cite[\S 3]{MR899051}, Thistlethwaite summarizes these findings, which we replicate in Table \ref{tab:thisthlethwaite}. Here, we use $\ell_\bullet$ for $\bullet \in \{\mathrm{loc}, \mathrm{pro}\}$ fixed. Notice that the type of external activity cannot change between $e_i$ and $e_i'$.

\begin{table}[h]
\centering
\begin{tabular}{cc||cc|cc}
                       &                        & \multicolumn{2}{c|}{Old edge ordering} & \multicolumn{2}{c}{New edge ordering}  \\
\multicolumn{1}{c}{}   & \multicolumn{1}{c||}{} & $e_i$  & $e_{i+1}$                     & $e_i'$ & $e_{i+1}'$                    \\ 
\hhline{==::====}
\multirow{2}{*}{Case 1}   & $T$                    & $L$    & $d$                           & $d$    & $D$                           \\
                       & $\sigma(T)$            & $d$    & $D$                           & $L$    & $d$                           \\ 
\hline
\multirow{2}{*}{Case 2}  & $T$                    & $D$    & $d$                           & $\ell_{\bullet}$ & $D$                           \\
                       & $\sigma(T)$            & $\ell_{\bullet}$ & $D$                           & $D$    & $d$                           \\ 
\hline
\multirow{2}{*}{Case 3} & $T$                    & $L$    & $d$                           & $\ell_{\bullet}$ & $D$                           \\
                       & $\sigma(T)$            & $\ell_{\bullet}$ & $D$                           & $L$    & $d$                          
\end{tabular}
\caption{}
\label{tab:thisthlethwaite}
\end{table}

In what follows, we use $s(T)$ to denote the state of $T$ relative the old ordering of edges, and $s'(T)$ to denote the state of $T$ relative the new ordering. For Tutte, the proof is finished, since $s(T) = s'(\sigma(T))$ and $s(\sigma(T)) = s'(T)$ if there is no distinction between $\ell_\mathrm{loc}$ and $\ell_\mathrm{pro}$. We have to consider the possibility that $s_T(e_j) \not= s_{\sigma(T)}(e_j)$, $j\not=i, i+1$, for $e_j$ an edge which is externally active with respect to $T$ and $\sigma(T)$ subject to the conditions (i)--(iii).

Assume $e_i \in T$. Following \cite{MR61366}, $T - e_i$ consists of two components, denoted by $C$ and $D$. Since $e_i \in \cyc(T, e_{i+1})$, we know that $e_{i+1}$ has ends in $C$ and $D$; \emph{i.e.}, $e_{i+1} \in \cut(T, e_i)$. If $e_j$ starts and ends at the same component $C$ or $D$, then $\cyc(T, e_j) = \cyc(\sigma(T), e_j)$, so $s_T(e_j) = s_{\sigma(T)}(e_j)$. Otherwise, $e_j \in \cut(T, e_i)$. Therefore, if $e_i$ is internally active with respect to $T$, then $s_T(e_j) = s_{\sigma(T)}(e_j) = d$. We must conclude that if $s_T(e_j) \not= s_{\sigma(T)}(e_j)$, then $e_i$ must be internally inactive with respect to $T$, putting us squarely within Case 2 of Table \ref{tab:thisthlethwaite}.

Let $\gamma \in H_1(\RP^2)$ be the homology class represented by $\cyc(T, e_{i+1}) = \cyc(\sigma(T), e_i)$ under the conditions (i)--(iii). We claim that 
\[
\gamma = [\cyc(T, e_j)] + [\cyc(\sigma(T), e_j)]
\]
for $e_j \in \cut(T, e_i)$. Let $a, \alpha, c$ be the vertices of $e_i$, $e_j$, and $e_{i+1}$ respectively which reside in $C\subset T$, and $b, \beta, d$ be the vertices of each which reside in $D \subset T$. See Figure \ref{fig:cycleaddition} for a schematic. Since $T$ and $\sigma(T)$ are trees, there are unique paths $P_1$ from $\alpha$ to $\beta$ in $T$ and $P_2$ from $\alpha$ to $\beta$ in $\sigma(T)$. Define $\alpha'$ to be the last vertex in $P_1$ preceding $e_i$ which is a term of $P_2$, and $\beta'$ to be the first vertex in $P_1$ succeeding $e_i$ which is a term of $P_2$. Finally, define $R_1$ to be the subpath of $P_1$ going from $\alpha'$ to $\beta'$ and $R_2$ to be the subpath of $P_2$ going from $\alpha'$ to $\beta'$. Then
\[
\gamma = [R_1 \cup R_2] = [R_1 \cup e_j] + [R_2 \cup e_j] = [\cyc(T, e_j)] + [\cyc(\sigma(T), e_j)]
\]
as desired.

\begin{figure}[h]
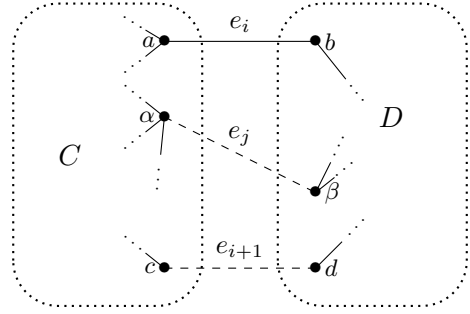

\tikz[]{
	\draw[rounded corners=6mm, dotted, thick] (0,0) rectangle (2.5,4);
	\node(c) at (2, 0.5) {$\bullet$};
		\node[anchor=east] at (c) {\small $c$};
		\draw (2,0.5) -- (1.75, 0.7) node[pos=1.6, sloped] {\tiny$\cdots$};
	\node(alpha) at (2, 2.5) {$\bullet$};
		\node[anchor=east] at (alpha) {\small $\alpha$};
		\draw (2,2.5) -- (1.95, 2)node[pos=1.6, sloped] {\tiny$\cdots$};
		\draw (2,2.5) -- (1.75, 2.3)node[pos=1.6, sloped] {\tiny$\cdots$};
		\draw (2,2.5) -- (1.75, 2.7)node[pos=1.6, sloped] {\tiny$\cdots$};
	\node(a) at (2, 3.5) {$\bullet$};
		\node[anchor=east] at (a) {\small $a$};
		\draw (2,3.5) -- (1.75, 3.65)node[pos=1.6, sloped] {\tiny$\cdots$};
		\draw (2,3.5) -- (1.75, 3.3)node[pos=1.6, sloped] {\tiny$\cdots$};
	\node at (0.75,2) {\large $C$};
	\draw[rounded corners=6mm, dotted, thick] (3.5,0) rectangle (6,4);
	\node(d) at (4, 0.5) {$\bullet$};
		\node[anchor=west] at (d) {\small $d$};
		\draw (4,0.5) -- (4.35, 0.85)node[pos=1.6, sloped] {\tiny$\cdots$};
	\node(beta) at (4, 1.5) {$\bullet$};
		\node[anchor=west] at (beta) {\small $\beta$};
		\draw (4,1.5) -- (4.25, 1.7)node[pos=1.6, sloped] {\tiny$\cdots$};
		\draw (4,1.5) -- (4.2, 1.9)node[pos=1.6, sloped] {\tiny$\cdots$};
	\node(b) at (4, 3.5) {$\bullet$};
		\node[anchor=west] at (b) {\small $b$};
		\draw (4,3.5) -- (4.4, 3)node[pos=1.4, sloped] {\tiny$\cdots$};
	\node at (5,2.5) {\large $D$};
	\draw (a.center) -- (b.center) node [midway, above] {$e_i$};
	\draw[dashed] (alpha.center) -- (beta.center) node [midway, above] {$e_j$};
	\draw[dashed] (c.center) -- (d.center) node [midway, above] {$e_{i+1}$};
}
\caption{A tree $T$ with $e_i\in T$ and $e_j, e_{i+1}\in \cut(T, e_i)$. The tree $\sigma(T)$ is obtained by replacing $e_i$ with $e_{i+1}$.}
\label{fig:cycleaddition}
\end{figure}

Returning to Case 2 of Table \ref{tab:thisthlethwaite}, there are two possibilities.
\begin{itemize}
\item If $\bullet = \mathrm{loc}$, then we have that $0 = [\cyc(T, e_j)] + [\cyc(\sigma(T), e_j)]$, so $s_T(e_j) = s_{\sigma(T)}(e_j)$. Thus $s(T) = s'(\sigma(T))$ and $s(\sigma(T)) = s'(T)$.
\item If $\bullet = \mathrm{pro}$, then $1 = [\cyc(T, e_j)] + [\cyc(\sigma(T), e_j)]$. This means that either $s_T(e_j) = s_{\sigma(T)}(e_j) = d$, $s_T(e_j) = \ell_{\mathrm{pro}}$ and $s_{\sigma(T)}(e_j) = \ell_{\mathrm{loc}}$, or $s_T(e_j) = \ell_{\mathrm{loc}}$ and $s_{\sigma(T)}(e_j) = \ell_{\mathrm{pro}}$. But neither $d$ nor $\ell_{\mathrm{pro}}$ contribute to the local external activity of spanning tree, thus $s(T) = s'(T)$ and $s(\sigma(T)) = s' (\sigma(T))$.
\end{itemize}
\end{proof}

\begin{example}
Below, we compute $\psi_G$ for one of the Tait graphs from Example \ref{ex:PT_Tutte} using the spanning tree definition and two distinct edge orderings. The change in activity after reordering described here is the one guaranteed by Case 2 in Table \ref{tab:thisthlethwaite}. Notice that $e_2$ and $e_3$ are interchanged, and that $s_T(e_1) \not= s_{\sigma(T)}(e_1)$ for $T = e_2$ in $G$.
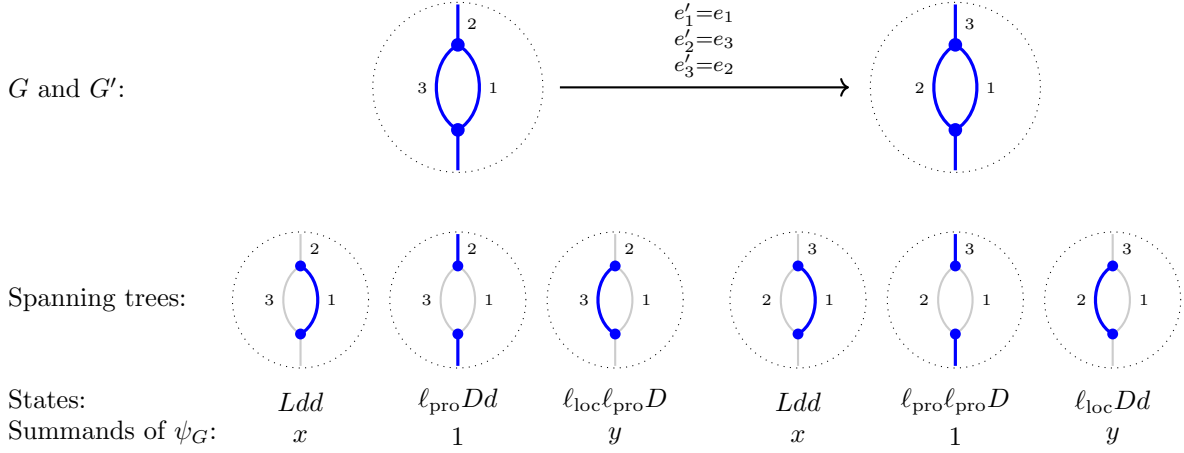
\[
\tikz[baseline={([yshift=-.5ex]current bounding box.center)}, scale=.225]{
\node[anchor=west] at (-25,0) {$G$ and $G'$:};
\node[anchor=west]  at (-25,-12.5) {Spanning trees:};
\node[anchor=west]  at (-25,-18.3) {States:};
\node[anchor=west]  at (-25,-20.3) {Summands of $\psi_G$:};
	\begin{scope}[xshift=2cm]
		\begin{scope}[blue, very thick]
			\fill (0,2.5) circle (.4cm);
			\fill (0,-2.5) circle (.4cm);
			\draw (0,-2.5) -- (0,-5);
			\draw (0,2.5) -- (0,5);
			\draw (0,-2.5) to[out=30, in=-30] (0,2.5);
			\draw (0,-2.5) to[out=150, in=-150] (0,2.5);
		\end{scope}
		\node[anchor=west] at (0,3.75) {\tiny$2$};
		\node[anchor=east] at (-1.25,0) {\tiny$3$};
		\node[anchor=west] at (1.25,0) {\tiny$1$};
		\draw[white, ultra thick] (0,0) circle (5cm);
		\draw[dotted] (0,0) circle (5cm);
	\end{scope}
	\begin{scope}[yshift=-12.5cm, xshift=-7.25cm, scale=0.8]
	\draw[black!20!white, thick] (0,-2.5) to[out=150, in=-150] (0,2.5);
	\draw[black!20!white, thick] (0,-2.5) -- (0,-5);
	\draw[black!20!white, thick] (0,2.5) -- (0,5);
		\begin{scope}[blue, very thick]
			\fill (0,2.5) circle (.4cm);
			\fill (0,-2.5) circle (.4cm);
			\draw (0,-2.5) to[out=30, in=-30] (0,2.5);
		\end{scope}
		\node[anchor=west] at (0,3.75) {\tiny$2$};
		\node[anchor=east] at (-1.25,0) {\tiny$3$};
		\node[anchor=west] at (1.25,0) {\tiny$1$};
		\draw[white, ultra thick] (0,0) circle (5cm);
		\draw[dotted] (0,0) circle (5cm);
		\node at (0,-7.5) {$Ldd$};
		\node at (0,-10) {$x$};
	\end{scope}
	\begin{scope}[yshift=-12.5cm, xshift=2cm, scale=0.8]
	\draw[black!20!white, thick] (0,-2.5) to[out=150, in=-150] (0,2.5);
	\draw[black!20!white, thick]  (0,-2.5) to[out=30, in=-30] (0,2.5);
		\begin{scope}[blue, very thick]
			\fill (0,2.5) circle (.4cm);
			\fill (0,-2.5) circle (.4cm);
			\draw (0,-2.5) -- (0,-5);
			\draw (0,2.5) -- (0,5);
		\end{scope}
		\node[anchor=west] at (0,3.75) {\tiny$2$};
		\node[anchor=east] at (-1.25,0) {\tiny$3$};
		\node[anchor=west] at (1.25,0) {\tiny$1$};
		\draw[white, ultra thick] (0,0) circle (5cm);
		\draw[dotted] (0,0) circle (5cm);
		\node at (0,-7.5) {$\ell_{\text{pro}}Dd$};
		\node at (0,-10) {$1$};
	\end{scope}
	\begin{scope}[yshift=-12.5cm, xshift=11.25cm, scale=0.8]
	\draw[black!20!white, thick] (0,-2.5) to[out=30, in=-30] (0,2.5);
	\draw[black!20!white, thick] (0,-2.5) -- (0,-5);
	\draw[black!20!white, thick] (0,2.5) -- (0,5);
		\begin{scope}[blue, very thick]
			\fill (0,2.5) circle (.4cm);
			\fill (0,-2.5) circle (.4cm);
			\draw (0,-2.5) to[out=150, in=-150] (0,2.5);
		\end{scope}
		\node[anchor=west] at (0,3.75) {\tiny$2$};
		\node[anchor=east] at (-1.25,0) {\tiny$3$};
		\node[anchor=west] at (1.25,0) {\tiny$1$};
		\draw[white, ultra thick] (0,0) circle (5cm);
		\draw[dotted] (0,0) circle (5cm);
		\node at (0,-7.5) {$\ell_{\text{loc}}\ell_{\text{pro}}D$};
		\node at (0,-10) {$y$};
	\end{scope}
	\begin{scope}[xshift=31.25cm]
		\begin{scope}[blue, very thick]
			\fill (0,2.5) circle (.4cm);
			\fill (0,-2.5) circle (.4cm);
			\draw (0,-2.5) -- (0,-5);
			\draw (0,2.5) -- (0,5);
			\draw (0,-2.5) to[out=30, in=-30] (0,2.5);
			\draw (0,-2.5) to[out=150, in=-150] (0,2.5);
		\end{scope}
		\node[anchor=west] at (0,3.75) {\tiny$3$};
		\node[anchor=east] at (-1.25,0) {\tiny$2$};
		\node[anchor=west] at (1.25,0) {\tiny$1$};
		\draw[white, ultra thick] (0,0) circle (5cm);
		\draw[dotted] (0,0) circle (5cm);
	\end{scope}
	\begin{scope}[xshift=22cm, yshift=-12.5cm, scale=0.8]
	\draw[black!20!white, thick] (0,-2.5) to[out=150, in=-150] (0,2.5);
	\draw[black!20!white, thick] (0,-2.5) -- (0,-5);
	\draw[black!20!white, thick] (0,2.5) -- (0,5);
		\begin{scope}[blue, very thick]
			\fill (0,2.5) circle (.4cm);
			\fill (0,-2.5) circle (.4cm);
			\draw (0,-2.5) to[out=30, in=-30] (0,2.5);
		\end{scope}
		\node[anchor=west] at (0,3.75) {\tiny$3$};
		\node[anchor=east] at (-1.25,0) {\tiny$2$};
		\node[anchor=west] at (1.25,0) {\tiny$1$};
		\draw[white, ultra thick] (0,0) circle (5cm);
		\draw[dotted] (0,0) circle (5cm);
		\node at (0,-7.5) {$Ldd$};
		\node at (0,-10) {$x$};
	\end{scope}
	\begin{scope}[xshift=31.25cm, yshift=-12.5cm, scale=0.8]
	\draw[black!20!white, thick] (0,-2.5) to[out=150, in=-150] (0,2.5);
	\draw[black!20!white, thick]  (0,-2.5) to[out=30, in=-30] (0,2.5);
		\begin{scope}[blue, very thick]
			\fill (0,2.5) circle (.4cm);
			\fill (0,-2.5) circle (.4cm);
			\draw (0,-2.5) -- (0,-5);
			\draw (0,2.5) -- (0,5);
		\end{scope}
		\node[anchor=west] at (0,3.75) {\tiny$3$};
		\node[anchor=east] at (-1.25,0) {\tiny$2$};
		\node[anchor=west] at (1.25,0) {\tiny$1$};
		\draw[white, ultra thick] (0,0) circle (5cm);
		\draw[dotted] (0,0) circle (5cm);
		\node at (0,-7.5) {$\ell_{\text{pro}}\ell_{\text{pro}}D$};
		\node at (0,-10) {$1$};
	\end{scope}
	\begin{scope}[xshift=40.5cm, yshift=-12.5cm, scale=0.8]
	\draw[black!20!white, thick] (0,-2.5) to[out=30, in=-30] (0,2.5);
	\draw[black!20!white, thick] (0,-2.5) -- (0,-5);
	\draw[black!20!white, thick] (0,2.5) -- (0,5);
		\begin{scope}[blue, very thick]
			\fill (0,2.5) circle (.4cm);
			\fill (0,-2.5) circle (.4cm);
			\draw (0,-2.5) to[out=150, in=-150] (0,2.5);
		\end{scope}
		\node[anchor=west] at (0,3.75) {\tiny$3$};
		\node[anchor=east] at (-1.25,0) {\tiny$2$};
		\node[anchor=west] at (1.25,0) {\tiny$1$};
		\draw[white, ultra thick] (0,0) circle (5cm);
		\draw[dotted] (0,0) circle (5cm);
		\node at (0,-7.5) {$\ell_{\text{loc}}Dd$};
		\node at (0,-10) {$y$};
	\end{scope}
	\draw[->, thick] (8,0) -- (25,0) node[midway, above]{\Large $\substack{e_1'=e_1 \\ e_2' = e_3 \\ e_3'=e_2}$};
}
\]
\end{example}

\subsubsection{Graphs dual to spanning trees in $\RP^2$}
\label{sss:quasitrees}

In $S^3$, if $G$ is a Tait graph of a diagram $D$ with spanning tree $T$, the dual spanning tree $\hat{T}$---obtained by taking the edges dual to each $e\not\in T$---has the same activity as $T$. This is not the case in $\RP^2$. In particular, if $T$ is a spanning tree of $G$ in $\RP^2$, then the dual subgraph (still denoted by $\hat{T}$) is not a tree. It is close though:

\begin{lemma}
\label{lem:projectivecompletions}
Let $G$ be a cellular graph in $\RP^2$. Let $\{T_1,\ldots, T_k\}$ denote the spanning trees of $G$ and $\{S_1, \ldots, S_\ell\}$ denote the spanning trees of $G^*$. Then
\[
\{\hat{T}_1,\ldots, \hat{T}_k\} = \{S_i \cup e : [S_i \cup e] = 1 \in H_1(\RP^2), i = 1,\ldots, \ell\}.
\]
\end{lemma}

If $T$ is a spanning tree of $G$ and $[T \cup e] = 1$, then we call the subgraph $T \cup e$ a \emph{projective completion} of $T$. The previous lemma says that the spanning trees of $G\subset \RP^2$ are in bijective correspondence with the projective completions of the spanning trees of $G^*$.

\begin{remark}
The set of spanning trees together with all projective completions coincides with the set of quasi-trees of $G$. Recall that a \emph{ribbon graph} is a pair $(G, S)$ in which $G$ is a graph, $S$ is a surface with boundary, and the inclusion $G \hookrightarrow S$ is a homotopy equivalence. Then, \emph{quasi-trees}, introduced in \cite{MR2665767}, are ribbon graphs with exactly one boundary component. See also \cite{MR2366169, MR2854567}. The quasi-trees associated to a graph in $\RP^2$ are exactly those obtained as thickenings of trees or spanning subgraphs which are projective completions of spanning trees.
\end{remark}

\begin{proof}
Notice that if $G$ has $n$ vertices and $m$ edges, then it has $m-n+1$ many faces, since $G$ is assumed to be a non-local graph in $\RP^2$, and $\chi(\RP^2) = 1$. In particular, since the spanning tree $T_i$ has $n-1$ edges, its dual $\hat{T}_i$ has $m-n+1$ edges. Since $G^*$ has $m-n+1$ vertices, $\hat{T}_i$ has one edge too many to be a spanning tree. However, $\hat{T}_i$ is a spanning subgraph with $m-n+1$ edges, so it has exactly one cycle. It follows that $[\hat{T}_i]=1$ since $G$ is non-local and $T_i$ is a local subgraph without loops.

In the other direction, we show that the dual of the projective completion $S_i \cup e$ is a spanning tree. Notice that $S_i \cup e$ has $m-n+1$ edges, so its dual has $n-1$ edges. To verify that $\widehat{S_i \cup e}$ has no loops, we use the first part of the argument. Namely, $\hat{S}_i$ is a projective closure $T_j \cup e'$ of some spanning tree $T_j$. Thus $\widehat{S_i \cup e}$ is $T_j \cup e' - e''$ for some $e'' \in \cyc(T_j, e')$. Since $T_j \cup e'$ has exactly one (homologically essential) cycle, $T_j \cup e' - e''$ is a spanning tree.
\end{proof}

\subsection{Relation to the generalized Krushkal polynomial for graphs in $\RP^2$}
\label{ss:krushkal}

In \cite{MR2769192}, Krushkal described a polynomial invariant $K_G \in \mathbb{Z}[x, y, A, B]$ of graphs embedded in orientable surfaces. It specializes to each of the Tutte, relative Tutte \cite{MR2607372}, Las Vergnas \cite{MR597150}, and Bollob\'as-Riordan \cite{MR1851080, MR1906909} polynomials. Subsequently, Buttler \cite{MR3739494} extended the Krushkal polynomial to nonorientable surfaces. Again, we refer the interested reader to two excellent surveys: \cite{MR4972578} for a full description of the relationships between these polynomials, and \cite{MR4952626} for descriptive derivations of polynomials for embedded graphs. In this subsection, we observe that the myopic Tutte polynomial for graphs $G \subset \RP^2$ recovers the generalized Krushkal polynomial. In particular, we prove Theorem \ref{thm:equalkrushkal}, which follows from Propositions \ref{prop:fromkrushkal} and \ref{prop:tokrushkal}.

Assume that $G$ is a graph embedded in a surface $\Sigma$. Following \cite{MR3739494}, for any spanning subgraph $H$ of $G$, let 
\[
s(H) = 2b_0(H) - \chi (\mathcal{S}(H)) \qquad \text{and} \qquad s^\perp(H) = 2b_0(\Sigma - H) - \chi(\mathcal{S}^\perp(H))
\]
where $\mathcal{S}(H)$ and $\mathcal{S}^\perp(H)$ are defined as follows. Take a small regular neighborhood $\nu(H)$ of $H \subset \Sigma$ (so that $\partial \nu(H)$ is a collection of circles). Then $\mathcal{S}(H)$ is obtained by attaching a disk to each boundary component of $\nu(H)$, and $\mathcal{S}^\perp(H)$ is obtained by attaching a disk to each boundary component of $\Sigma - \nu(H)$. Recall that $i_*^H: H_1(H; \F) \to H_1(\Sigma; \F)$ is the map induced by the embedding $i: H \hookrightarrow \Sigma$.

\begin{definition}
\label{def:krushkal}
The \emph{generalized Krushkal polynomial} of embedded graphs $G \subset \Sigma$ is given by
\begin{equation}
\label{eq:krushkal}
K_{G}(x, y, A, B) = \sum_{H \subseteq G} (x-1)^{b_0(H) - b_0(G)} (y-1)^{\dim \ker(i_*^H)} A^{\frac{s(H)}{2}} B^{\frac{s^\perp(H)}{2}}.
\end{equation}
\end{definition}

We list pertinent properties of the generalized Krushkal polynomial below.

\begin{theorem}[Theorems 2.4 and 2.5 of \cite{MR3739494}]
\label{thm:krushkalprops}
Let $G$ be a graph embedded in a compact surface $\Sigma$.
\begin{enumerate}
\item $K_G$ recovers the Tutte polynomial:
\[
\chi_G(x, y) = (y-1)^{\frac{s(\Sigma)}{2}} K_G\left(x, y, y-1, \frac{1}{y-1}\right)
\]
where $s(\Sigma) = 2b_0(\Sigma) - \chi(\Sigma)$.
\item If $G$ is cellular, and $G^* \subset \Sigma$ is the embedded graph which is Poincar\'e dual to $G$, then
\[
K_{G^*}(x, y, A, B) = K_G(y, x, B, A).
\]
\end{enumerate}
\end{theorem}

\begin{lemma}
\label{lem:rp2krushkal}
If $\Sigma = \RP^2$, then
\[
K_G \in \sqrt{A} \, \mathbb{Z}[x,y] \oplus \sqrt{B} \, \mathbb{Z}[x,y].
\]
\end{lemma}

\begin{proof}
From Equation (\ref{eq:krushkal}), this follows if, for any spanning subgraph $H$ of a graph $G$ embedded in $\RP^2$, $s(H), s^\perp(H) \in \{0,1\}$ and $s(H) \not= s^\perp(H)$. First of all, notice that $\nu(H)$ is either a collection of $b_0(H)$-many disks, in the case that $H$ is a local graph, or a collection of $(b_0(H) - 1)$-many disks together with one copy of $\RP^2$ minus $\alpha(H)$-many disks. In the former case, we have that $s(H) = 0$ and $s^\perp(H) = 1$, and in the latter case we have that $s(H) = 1$ and $s^\perp(H) = 0$.
\end{proof}

Continuing, it is not hard to see that the myopic Tutte polynomial for graphs in $\RP^2$ can be recovered by the Krushkal polynomial:

\begin{proposition}
\label{prop:fromkrushkal}
Assume that $G$ is a graph embedded into $\RP^2$. Then 
\[
\psi_G(x,y) = K_G(x, y, 0, 1).
\]
\end{proposition}

\begin{proof}
From the proof of Lemma \ref{lem:rp2krushkal}, if $H$ is not local, then it has one component for which $\mathcal{S}(H)$ is $\RP^2$. Thus, terms in (\ref{eq:krushkal}) are annihilated for each non-local spanning subgraph $H$. For the remaining, we have $\dim \ker(i_*^H) = b_1(H)$, and the desired result follows.
\end{proof}

By (2) of Theorem \ref{thm:krushkalprops}, Proposition \ref{prop:fromkrushkal} says that both $\psi_G$ and $\psi_{G^*}$ can be recovered from $K_G$. However, it also says that the converse is true:

\begin{proposition}
\label{prop:tokrushkal}
For any cellular graph $G \subset \RP^2$,
\[
K_G(x, y, A, B) = \sqrt{B} \, \psi_G(x, y) + \sqrt{A} \, \psi_{G^*}(y, x).
\]
\end{proposition}

\begin{proof}
By Lemma \ref{lem:rp2krushkal}, we can write $K_G(x,y, A, B) = \sqrt{B}\, p(x,y) + \sqrt{A}\, q(x,y)$ for some $p, q \in \mathbb{Z}[x,y]$. Proposition \ref{prop:fromkrushkal} says that $p(x,y) = \psi_G(x,y)$. Moreover, (2) of Theorem \ref{thm:krushkalprops} says that $K_{G^*}(x, y, 0, 1) = q(y,x)$, so $\psi_{G^*}(y,x) = q(x,y)$, so we are done.
\end{proof}

Theorem \ref{thm:equalkrushkal} follows. In addition, it is a consequence of (1) from Theorem \ref{thm:krushkalprops} that the Tutte polynomial is a linear combination of myopic Tutte polynomials.

\begin{corollary}
\label{cor:tuttefrommyopic}
The Tutte polynomial of a cellular graph $G\subset\RP^2$ can be recovered from the myopic Tutte polynomials of itself and its dual:
\[
\chi_G(x,y) = \psi_G(x,y) + (y-1) \, \psi_{G^*}(y, x).
\]
\end{corollary}

\section{A spanning tree expansion for the Jones polynomial in $\RP^3$}
\label{s:spanningtreejones}

Let $G$ be a connected graph with ordered edges. We make one further decoration: assign a sign $\{+, -\}$ to each edge as well. Thus, for each spanning tree $T$ of a signed graph $G$, each edge $e$ falls into one of ten states with respect to $T$ depending on whether $e$ is signed $+$ or $-$, internal or external, active or inactive and, if externally active, local or projective. The decision tree of Figure \ref{fig:decisiontree1} is doubled into that of Figure \ref{fig:decisiontree}.

\begin{figure}[h]
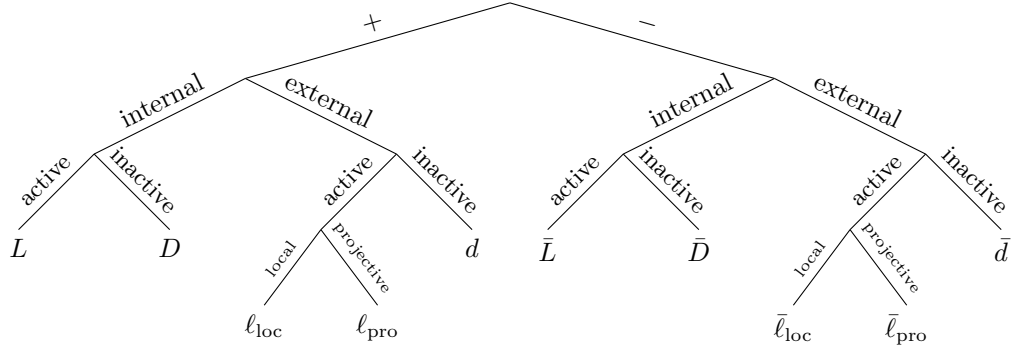

\tikz[]{
	\draw (3.5,1) -- (0,0) node [midway, above, sloped] {$+$};
	\draw (3.5,1) -- (7,0) node [midway, above, sloped] {$-$};
	\draw (0,0) -- (2,-1) node [midway, above, sloped] {external};
	\draw (0,0) -- (-2,-1) node [midway, above, sloped] {internal};
	\draw (2,-1) -- (1,-2) node [midway, above, sloped] {\small active};
	\draw (2,-1) -- (3,-2) node [midway, above, sloped] {\small inactive};
	\draw (-2,-1) -- (-1,-2) node [midway, above, sloped] {\small inactive};
	\draw (-2,-1) -- (-3,-2) node [midway, above, sloped] {\small active};
	\draw (1,-2) -- (0.25, -3) node [midway, above, sloped] {\tiny local};
	\draw (1,-2) -- (1.75, -3) node [midway, above, sloped] {\tiny projective};
	\node[anchor=north] at (-3,-2) {$L$};
	\node[anchor=north] at (-1,-2) {$D$};
	\node[anchor=north] at (0.25,-3) {$\ell_{\text{loc}}$};
	\node[anchor=north] at (1.75,-3) {$\ell_{\text{pro}}$};
	\node[anchor=north] at (3,-2) {$d$};
	\begin{scope}[xshift=7cm]
	\draw (0,0) -- (2,-1) node [midway, above, sloped] {external};
	\draw (0,0) -- (-2,-1) node [midway, above, sloped] {internal};
	\draw (2,-1) -- (1,-2) node [midway, above, sloped] {\small active};
	\draw (2,-1) -- (3,-2) node [midway, above, sloped] {\small inactive};
	\draw (-2,-1) -- (-1,-2) node [midway, above, sloped] {\small inactive};
	\draw (-2,-1) -- (-3,-2) node [midway, above, sloped] {\small active};
	\draw (1,-2) -- (0.25, -3) node [midway, above, sloped] {\tiny local};
	\draw (1,-2) -- (1.75, -3) node [midway, above, sloped] {\tiny projective};
	\node[anchor=north] at (-3,-2) {$\bar{L}$};
	\node[anchor=north] at (-1,-2) {$\bar{D}$};
	\node[anchor=north] at (0.25,-3) {$\bar{\ell}_{\text{loc}}$};
	\node[anchor=north] at (1.75,-3) {$\bar{\ell}_{\text{pro}}$};
	\node[anchor=north] at (3,-2) {$\bar{d}$};
	\end{scope}
}
\caption{Determining the state of any edge in a signed graph relative to a spanning tree.}
\label{fig:decisiontree}
\end{figure}

Let $D$ be a diagram of a nullhomologous link in $\RP^3$. Order the crossings of $D$ and fix a checkerboard surface for it. Assign a sign to each crossing by the following convention.
\begin{equation}
\label{eq:signconvention}
\begin{tikzpicture}[baseline={([yshift=-.5ex]current bounding box.center)}, thick]
\fill[black!20!white] (0,0) -- (.5,.5) -- (1,0);
\fill[black!20!white] (1,1) -- (.5,.5) -- (0,1);
\draw (0,0) -- (1,1);
\draw (0,1) -- (.3,.7);
\draw (1,0) -- (.7,.3);
\node at (0.5,-.35) {$+$};
\begin{scope}[xshift=4cm]
\fill[black!20!white] (0,0) -- (.5,.5) -- (0,1);
\fill[black!20!white] (1,1) -- (.5,.5) -- (1,0);
\draw (0,0) -- (1,1);
\draw (0,1) -- (.3,.7);
\draw (1,0) -- (.7,.3);
\node at (0.5,-.35) {$-$};
\end{scope}
\end{tikzpicture}
\end{equation}
Then, to both of the checkerboard graphs of $D$, we associate a Tait graph with signed (and ordered) edges. For instance:
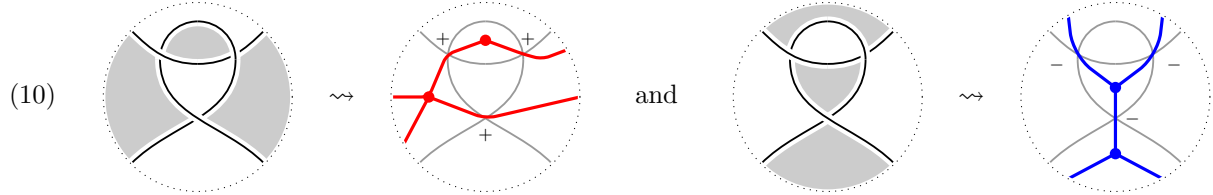
\begin{equation}
\label{eq:PT_taits}
\tikz[baseline={([yshift=-.5ex]current bounding box.center)}, scale=.25]{
	\fill[black!20!white] (-5*1.41/2, 5*1.41/2) to[out=-45, in=135] (-2, 2.25) to[out=-100, in=135] (0,-1.17) to[out=-135, in=45] (-5*1.41/2, -5*1.41/2) to[out=135, in=-90, looseness=0.8] (-5,0) to[out=90,in=-135] (-5*1.41/2, 5*1.41/2);
	\fill[black!20!white] (5*1.41/2, 5*1.41/2) to[out=-135, in=45] (2, 2.25) to[out=-80, in=45] (0,-1.17) to[out=-45, in=135] (5*1.41/2, -5*1.41/2) to[out=45, in=-90, looseness=0.8] (5,0) to[out=90,in=-45] (5*1.41/2, 5*1.41/2);
	\fill[black!20!white] (-2, 2.25) to[out=100, in=180] (0,4) to[out=0, in=80] (2, 2.25) to[out=210, in=0] (0,1.75) to[out=180, in=-30] (-2, 2.25);
	\draw[knot, overcross] (0,1.75) to[out=0, in=-135] (5*1.41/2, 5*1.41/2);
	\draw[knot, overcross] (2,2) to[out=-90, in=45] (-5*1.41/2, -5*1.41/2);
	\draw[knot, overcross] (5*1.41/2, -5*1.41/2) to[out=135, in=-90] (-2,2);
	\draw[knot, overcross] (-2,2) to[out=90, in=180] (0,4);
	\draw[knot, overcross] (-5*1.41/2, 5*1.41/2) to[out=-45, in=180] (0, 1.75);
	\draw[knot, overcross] (0,4) to[out=0, in=90] (2,2);
	\draw[white, ultra thick] (0,0) circle (5cm);
	\draw[dotted] (0,0) circle (5cm);
}
\quad \leadsto \quad
\tikz[baseline={([yshift=-.5ex]current bounding box.center)}, scale=.25]{
	\begin{scope}[white!60!black]
	\draw[knot] (0,1.75) to[out=0, in=-135] (5*1.41/2, 5*1.41/2);
	\draw[knot] (2,2) to[out=-90, in=45] (-5*1.41/2, -5*1.41/2);
	\draw[knot] (5*1.41/2, -5*1.41/2) to[out=135, in=-90] (-2,2);
	\draw[knot] (-2,2) to[out=90, in=180] (0,4);
	\draw[knot] (-5*1.41/2, 5*1.41/2) to[out=-45, in=180] (0, 1.75);
	\draw[knot] (0,4) to[out=0, in=90] (2,2);
	\end{scope}
	\begin{scope}[red, very thick]
		\fill (0, 3) circle (.3cm);
		\fill (-3, 0) circle (.3cm);
		\draw[rounded corners = 1mm] (0,3) -- (-2.05,2.25) -- (-3, 0) -- (-5,0);
		\draw (-3, 0) -- (-5*0.866, -5*0.5);
		\draw[rounded corners = 1mm] (0,3) -- (2.05,2.25) -- (3, 2) -- (5*0.866, 5*0.5);
		\draw[rounded corners] (-3, 0) -- (0,-1.15) -- (5,0);
	\end{scope}
	\node[anchor=north] at (0,-1.18) {\tiny$+$};
	\node[anchor=south] at (2.25, 2.25) {\tiny$+$};
	\node[anchor=south] at (-2.25, 2.25) {\tiny$+$};
	\draw[white, ultra thick] (0,0) circle (5cm);
	\draw[dotted] (0,0) circle (5cm);
}
\qquad \text{and} \qquad
\tikz[baseline={([yshift=-.5ex]current bounding box.center)}, scale=.25]{
	\fill[black!20!white] (-2, 2.25) to[out=-45, in=180] (0,1.75) to[out=0,in=-135] (2,2.25) to[out=-80, in=45] (0,-1.17) to[out=135, in=-100] (-2, 2.25);
	\fill[black!20!white] (-5*1.41/2, 5*1.41/2) to[out=45, in=180] (0,5) to[out=0, in=135] (5*1.41/2, 5*1.41/2) to[out=225, in=30] (2, 2.25) to[out=100, in=0] (0,4) to[out=180, in=80] (-2, 2.25) to[out=150, in=-45] (-5*1.41/2, 5*1.41/2);
	\fill[black!20!white] (-5*1.41/2, -5*1.41/2) to[out=45, in=-135] (0,-1.17) to[out=-45, in=135] (5*1.41/2, -5*1.41/2) to[out=-135, in=0] (0,-5) to[out=180, in=-45] (-5*1.41/2, -5*1.41/2);
	\draw[knot, overcross] (0,1.75) to[out=0, in=-135] (5*1.41/2, 5*1.41/2);
	\draw[knot, overcross] (2,2) to[out=-90, in=45] (-5*1.41/2, -5*1.41/2);
	\draw[knot, overcross] (5*1.41/2, -5*1.41/2) to[out=135, in=-90] (-2,2);
	\draw[knot, overcross] (-2,2) to[out=90, in=180] (0,4);
	\draw[knot, overcross] (-5*1.41/2, 5*1.41/2) to[out=-45, in=180] (0, 1.75);
	\draw[knot, overcross] (0,4) to[out=0, in=90] (2,2);
	\draw[white, ultra thick] (0,0) circle (5cm);
	\draw[dotted] (0,0) circle (5cm);
}
\quad \leadsto \quad
\tikz[baseline={([yshift=-.5ex]current bounding box.center)}, scale=.25]{
	\begin{scope}[white!60!black]
	\draw[knot] (0,1.75) to[out=0, in=-135] (5*1.41/2, 5*1.41/2);
	\draw[knot] (2,2) to[out=-90, in=45] (-5*1.41/2, -5*1.41/2);
	\draw[knot] (5*1.41/2, -5*1.41/2) to[out=135, in=-90] (-2,2);
	\draw[knot] (-2,2) to[out=90, in=180] (0,4);
	\draw[knot] (-5*1.41/2, 5*1.41/2) to[out=-45, in=180] (0, 1.75);
	\draw[knot] (0,4) to[out=0, in=90] (2,2);
	\end{scope}
	\begin{scope}[blue, very thick]
		\fill (0, 0.5) circle (.3cm);
		\fill (0, -3) circle (.3cm);
		\draw (0, 0.5) -- (0, -3);
		\draw[rounded corners = 3mm] (0, 0.5) -- (-2.25,2.2) -- (-5*0.5,5*0.866);
		\draw[rounded corners = 3mm] (0, 0.5) -- (2.25,2.2) -- (5*0.5,5*0.866);
		\draw (-5*0.5,-5*0.866) -- (0,-3);
		\draw (5*0.5,-5*0.866) -- (0,-3);
	\end{scope}
	\node[anchor=west] at (0,-1.18) {\tiny$-$};
	\node[anchor=west] at (2.2, 1.7) {\tiny$-$};
	\node[anchor=east] at (-2.2, 1.7) {\tiny$-$};
	\draw[white, ultra thick] (0,0) circle (5cm);
	\draw[dotted] (0,0) circle (5cm);
}
\end{equation}
Note that each of the edges of the Tait graphs corresponding to $D$ take the same value in $\{+, -\}$ if and only if $D$ is an alternating diagram.

\begin{definition}
Assume that $D \subset \RP^2$ is a diagram of a nullhomologous link in $\RP^3$. Denote the two Tait graphs associated to $D$ by $G_1$ and $G_2$. We define
\begin{equation}
\label{eq:bigpsi}
\Psi_D = \sum_{T_i\subset G_1} \left(\prod_{e_j \in G_1} \mu_{ij}\right) + \sum_{T_i\subset G_2} \left(\prod_{e_j \in G_2} \mu_{ij}\right)
\end{equation}
where $\mu_{ij} \in \mathbb{Z}[A, A^{-1}]$ is a monomial determined by the state of $e_j$ relative $T_i \subset G_\epsilon$, $\epsilon \in \{1,2\}$, according to Table \ref{tab:state_monomialsA}.
\begin{table}[h]
\renewcommand{\arraystretch}{1.5}
\centering
\begin{tabular}{|>{\centering\hspace{0pt}}m{0.092\linewidth}||>{\centering\hspace{0pt}}m{0.065\linewidth}>{\centering\hspace{0pt}}m{0.065\linewidth}>{\centering\hspace{0pt}}m{0.065\linewidth}>{\centering\hspace{0pt}}m{0.065\linewidth}>{\centering\hspace{0pt}}m{0.065\linewidth}|>{\centering\hspace{0pt}}m{0.065\linewidth}>{\centering\hspace{0pt}}m{0.065\linewidth}>{\centering\hspace{0pt}}m{0.065\linewidth}>{\centering\hspace{0pt}}m{0.065\linewidth}>{\centering\arraybackslash\hspace{0pt}}m{0.081\linewidth}|} 
\hline
\multicolumn{1}{|>{\hspace{0pt}}m{0.092\linewidth}||}{state of~$e_j$} & $L$       & $D$ & $\ell_{\text{loc}}$ & $\ell_{\text{pro}}$ & $d$      & $\bar{L}$ & $\bar{D}$ & $\bar{\ell}_{\text{loc}}$ & $\bar{\ell}_{\text{pro}}$ & $\bar{d}$  \\
$\mu_{ij}$                                                            & $-A^{-3}$ & $A$ & $-A^3$              & $A^{-1}$             & $A^{-1}$ & $-A^3$         & $A^{-1}$       & $-A^{-3}$                      & $A$                            & $A$             \\
\hline
\end{tabular}
\caption{The monomial $\mu_{ij}$ associated to an edge $e_j$ of a spanning tree $T_i$ is determined by the state of $e_j$; see Figure \ref{fig:decisiontree}.}
\label{tab:state_monomialsA}
\end{table}
The product $\mu(T_i) := \prod_{e_j\in G_\epsilon} \mu_{ij}$ is called the \emph{weight} of $T_i \subset G_\epsilon$. Recall that the \emph{state} of a tree is the word given by the states of each of $e_j\in G_\epsilon$ (usually presented according to the ordering of the edges, though this is inconsequential). For each Tait graph $G$ associated to $D$, we denote the corresponding summation appearing in Equation (\ref{eq:bigpsi}) by $\Psi_{G}$.
\end{definition}

\begin{example}
\label{ex:mainST}
Returning to $3_1$ of Drobotukhina's table, order the crossings as follows.
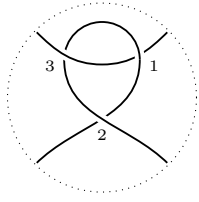
\[
\tikz[baseline={([yshift=-.5ex]current bounding box.center)}, scale=.25]{
	\draw[knot, overcross] (0,1.75) to[out=0, in=-135] (5*1.41/2, 5*1.41/2);
	\draw[knot, overcross] (2,2) to[out=-90, in=45] (-5*1.41/2, -5*1.41/2);
	\draw[knot, overcross] (5*1.41/2, -5*1.41/2) to[out=135, in=-90] (-2,2);
	\draw[knot, overcross] (-2,2) to[out=90, in=180] (0,4);
	\draw[knot, overcross] (-5*1.41/2, 5*1.41/2) to[out=-45, in=180] (0, 1.75);
	\draw[knot, overcross] (0,4) to[out=0, in=90] (2,2);
	\draw[white, ultra thick] (0,0) circle (5cm);
	\draw[dotted] (0,0) circle (5cm);
	\node[anchor=north] at (0,-1.18) {\tiny$2$};
	\node[anchor=west] at (2, 1.7) {\tiny$1$};
	\node[anchor=east] at (-2, 1.7) {\tiny$3$};
}
\]
According to (\ref{eq:PT_taits}) and Table \ref{tab:state_monomialsA}, we can compute $\Psi_D$ as follows.
\[
\tikz[baseline={([yshift=-.5ex]current bounding box.center)}, scale=.225]{
\node[anchor=west] at (-25,0) {Tait graphs:};
\node[anchor=west]  at (-25,-12.5) {Spanning trees:};
\node[anchor=west]  at (-25,-20) {States:};
\node[anchor=west]  at (-25,-22.5) {Weights:};
	\begin{scope}[red, very thick]
		\fill (0, 3) circle (.4cm);
		\fill (0,0) circle (.4cm);
		\draw (0,-5) -- (0,5);
		\draw (-5,0) -- (5,0);
	\end{scope}
	\node[anchor=north] at (2.5, 0) {\tiny$2$};
	\node[anchor=east] at (0,1.5) {\tiny$3$};
	\node[anchor=east] at (0,-2.5) {\tiny$1$};
	\draw[white, ultra thick] (0,0) circle (5cm);
	\draw[dotted] (0,0) circle (5cm);
	\begin{scope}[yshift=-12.5cm, xshift=-6.25cm]
		\draw[black!20!white, thick] (0,-5) -- (0,5);
		\draw[black!20!white, thick] (-5,0) -- (5,0);
		\begin{scope}[red, very thick]
			\fill (0, 3) circle (.4cm);
			\fill (0,0) circle (.4cm);
			\draw (0,-5) -- (0,0);
			\draw (0,3) -- (0,5);
		\end{scope}
		\node[anchor=north] at (2.5, 0) {\tiny$2$};
		\node[anchor=east] at (0,1.5) {\tiny$3$};
		\node[anchor=east] at (0,-2.5) {\tiny$1$};
		\draw[white, ultra thick] (0,0) circle (5cm);
		\draw[dotted] (0,0) circle (5cm);
		\node at (0,-7.5) {$L\ell_{\text{pro}}d$};
		\node at (0,-10) {$-A^{-5}$};
	\end{scope}
	\begin{scope}[yshift=-12.5cm, xshift=6.25cm]
		\draw[black!20!white, thick] (0,-5) -- (0,5);
		\draw[black!20!white, thick] (-5,0) -- (5,0);
		\begin{scope}[red, very thick]
			\fill (0, 3) circle (.4cm);
			\fill (0,0) circle (.4cm);
			\draw (0,0) -- (0,3);
		\end{scope}
		\node[anchor=north] at (2.5, 0) {\tiny$2$};
		\node[anchor=east] at (0,1.5) {\tiny$3$};
		\node[anchor=east] at (0,-2.5) {\tiny$1$};
		\draw[white, ultra thick] (0,0) circle (5cm);
		\draw[dotted] (0,0) circle (5cm);
		\node at (0,-7.5) {$\ell_{\text{pro}} \ell_{\text{pro}} D$};
		\node at (0,-10) {$A^{-1}$};
	\end{scope}
	\begin{scope}[xshift=31.25cm]
		\begin{scope}[blue, very thick]
			\fill (0,2.5) circle (.4cm);
			\fill (0,-2.5) circle (.4cm);
			\draw (0,-2.5) -- (0,-5);
			\draw (0,2.5) -- (0,5);
			\draw (0,-2.5) to[out=30, in=-30] (0,2.5);
			\draw (0,-2.5) to[out=150, in=-150] (0,2.5);
		\end{scope}
		\node[anchor=west] at (0,3.75) {\tiny$2$};
		\node[anchor=east] at (-1.25,0) {\tiny$3$};
		\node[anchor=west] at (1.25,0) {\tiny$1$};
		\draw[white, ultra thick] (0,0) circle (5cm);
		\draw[dotted] (0,0) circle (5cm);
	\end{scope}
	\begin{scope}[xshift=18.75cm, yshift=-12.5cm]
	\draw[black!20!white, thick] (0,-2.5) to[out=150, in=-150] (0,2.5);
	\draw[black!20!white, thick] (0,-2.5) -- (0,-5);
	\draw[black!20!white, thick] (0,2.5) -- (0,5);
		\begin{scope}[blue, very thick]
			\fill (0,2.5) circle (.4cm);
			\fill (0,-2.5) circle (.4cm);
			\draw (0,-2.5) to[out=30, in=-30] (0,2.5);
		\end{scope}
		\node[anchor=west] at (0,3.75) {\tiny$2$};
		\node[anchor=east] at (-1.25,0) {\tiny$3$};
		\node[anchor=west] at (1.25,0) {\tiny$1$};
		\draw[white, ultra thick] (0,0) circle (5cm);
		\draw[dotted] (0,0) circle (5cm);
		\node at (0,-7.5) {$\bar{L}\bar{d}\bar{d}$};
		\node at (0,-10) {$-A^5$};
	\end{scope}
	\begin{scope}[xshift=31.25cm, yshift=-12.5cm]
	\draw[black!20!white, thick] (0,-2.5) to[out=150, in=-150] (0,2.5);
	\draw[black!20!white, thick]  (0,-2.5) to[out=30, in=-30] (0,2.5);
		\begin{scope}[blue, very thick]
			\fill (0,2.5) circle (.4cm);
			\fill (0,-2.5) circle (.4cm);
			\draw (0,-2.5) -- (0,-5);
			\draw (0,2.5) -- (0,5);
		\end{scope}
		\node[anchor=west] at (0,3.75) {\tiny$2$};
		\node[anchor=east] at (-1.25,0) {\tiny$3$};
		\node[anchor=west] at (1.25,0) {\tiny$1$};
		\draw[white, ultra thick] (0,0) circle (5cm);
		\draw[dotted] (0,0) circle (5cm);
		\node at (0,-7.5) {$\bar{\ell}_{\text{pro}}\bar{D}\bar{d}$};
		\node at (0,-10) {$A$};
	\end{scope}
	\begin{scope}[xshift=43.75cm, yshift=-12.5cm]
	\draw[black!20!white, thick] (0,-2.5) to[out=30, in=-30] (0,2.5);
	\draw[black!20!white, thick] (0,-2.5) -- (0,-5);
	\draw[black!20!white, thick] (0,2.5) -- (0,5);
		\begin{scope}[blue, very thick]
			\fill (0,2.5) circle (.4cm);
			\fill (0,-2.5) circle (.4cm);
			\draw (0,-2.5) to[out=150, in=-150] (0,2.5);
		\end{scope}
		\node[anchor=west] at (0,3.75) {\tiny$2$};
		\node[anchor=east] at (-1.25,0) {\tiny$3$};
		\node[anchor=west] at (1.25,0) {\tiny$1$};
		\draw[white, ultra thick] (0,0) circle (5cm);
		\draw[dotted] (0,0) circle (5cm);
		\node at (0,-7.5) {$\bar{\ell}_{\text{loc}}\bar{\ell}_{\text{pro}}\bar{D}$};
		\node at (0,-10) {$-A^{-3}$};
	\end{scope}
}
\]
Therefore, we have that
\[
\Psi_D = (-A^{-5} + A^{-1}) +  (- A^5 + A - A^{-3})
\]
which matches $\langle D \rangle$; see Example \ref{ex:PT_KB}.
\end{example}

\begin{lemma}
\label{lem:orderings}
The polynomial $\Psi_D$ does not depend on the order of crossings of $D$.
\end{lemma}

\begin{proof}
To see this, we return to the situation described in the proof of Proposition \ref{prop:tutte_spanningtrees} and Table \ref{tab:thisthlethwaite}. If $e_i$ and $e_{i+1}$ have the same signs, invariance is apparent. If $e_i$ and $e_{i+1}$ have different signs, there are cases to consider. One verifies that
\begin{itemize}
\item in Case 1, Case 2 when $\ell_{\bullet} = \ell_{\text{loc}}$, and Case 3 when $\ell_{\bullet} = \ell_{\text{pro}}$, the weights of $T$ and $\sigma(T)$ sum to zero in both orderings, and
\item in Case 2 when $\ell_{\bullet} = \ell_{\text{pro}}$ and Case 3 when $\ell_{\bullet} = \ell_{\text{loc}}$, the respective weights of $T$ and $\sigma(T)$ are unaltered by the change of ordering.
\end{itemize}
\end{proof}

\begin{proposition}
\label{prop:KBgraphpolynomial}
Assume $L$ is a nullhomologous link in $\RP^3$ with diagram $D$.
\begin{enumerate}[label=\roman*.]
\item If $D$ is local, then $\Psi_D = 2 \langle D \rangle$; more precisely, $\Psi_{G_1} = \Psi_{G_2} = \langle D \rangle$.
\item If $D$ is non-local, then $\Psi_D = \langle D \rangle$.
\end{enumerate}
\end{proposition}

\begin{proof}
The proof of the first statement is implicit to \cite[\S 4]{MR899051}. We will prove the second statement; our argument follows a structure similar to Thistlethwaite's. 

Notice that bridges and local loops are mutually dual to each other. The story for projective loops is more complicated. Start by observing that there is at most one vertex in a non-local graph $G$ embedded in $\RP^2$ which supports projective loops. If there is exactly one projective loop $p$, and $G - p$ is a local graph, then the edge dual to $p$ is also a projective loop. Otherwise, the dual to each projective loop is a regular edge.

This observation prompts the following definition. We call the diagram $D$ \emph{terminal} if, in both of its Tait graphs, all edges are either bridges, loops, or regular edges whose duals are projective loops. Notice that the dual of a regular edge $e$ is a projective loop if and only if $e$ has the same face on either side of it. Therefore, the Tait graphs of a terminal diagram $D$ fall into precisely one of three types:
\begin{enumerate}[leftmargin=2cm, label=Type \Roman*.]
\item All edges are bridges or local loops;
\item All edges are bridges, local loops, or projective loops;
\item All edges are bridges, local loops, or regular edges whose duals are projective loops.
\end{enumerate}
If $G_1$ and $G_2$ are the Tait graphs of a terminal diagram, then $G_1$ is Type I if and only if $G_2$ is Type I, and $G_1$ is Type II if and only if $G_2$ is Type III. If $G_1$ and $G_2$ are both Type I, then $D$ is evidently a local diagram for the class-0 unknot, a case belonging to Part i of this proposition. Without loss of generality, if $G_1$ is Type II and $G_2$ is Type III, then $D$ is a diagram for $\Upsilon_n$ for some $n$.

Thus, assume that $D$ is terminal, and that $G_1$ is Type II with $p$ positive bridges, $q$ negative bridges, $r$ positive local loops, $s$ negative local loops, $t$ positive projective loops, and $u$ negative projective loops. This means that $G_2$ is Type III with $p$ negative local loops, $q$ positive local loops, $r$ negative bridges, $s$ positive bridges, $t$ negative regular edges, and $u$ positive regular edges. In this case,
\begin{align}
\Psi_D &= (-A^3)^{-p+q+r-s} \sum_{T_i\subset G_1'} \left(\prod_{e_j\in G_1'} \mu_{ij}\right) + (-A^3)^{-p+q+r-s} \sum_{T_i\subset G_2'} \left(\prod_{e_j\in G_2'} \mu_{ij}\right) \nonumber
\\ &= (-A^3)^{-p+q+r-s} \left( A^{-t+u} + \sum_{T_i\subset G_2'} \left(\prod_{e_j\in G_2'} \mu_{ij}\right) \right) \label{eq:M2graphpoly}
\end{align}
for $G_1'$ the graph consisting of a single vertex with $t$ positive projective loops and $u$ negative projective loops, and 
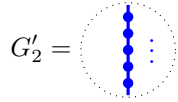
\[
G_2' = 
\tikz[baseline={([yshift=-.5ex]current bounding box.center)}, scale=.125]{
	\begin{scope}[blue, very thick]
		\fill (0,1.75) circle (.55cm);
		\fill (0,3.5) circle (.55cm);
		\fill (0,-1.75) circle (.55cm);
		\fill (0,-3.5) circle (.55cm);
		\fill (0,0) circle (.55cm);
		\node[anchor=west] at (1,0.75) {$\vdots$};
		\draw (0,0) -- (0,-5);
		\draw (0,0) -- (0,5);
	\end{scope}
	\draw[white, ultra thick] (0,0) circle (5cm);
	\draw[dotted] (0,0) circle (5cm);
}
\]
with $t$ negative edges and $u$ positive edges. The corresponding diagram $D'$ with Tait graphs $G_1'$ and $G_2'$ has been obtained from $D$ by removing all Reidemeister I crossings. It can be further reduced by performing $\min\{t, u\}$-many Reidemeister II moves, resulting in a minimal diagram for $\Upsilon_{-t+u}$. The Kauffman bracket is invariant under Reidemeister II moves, so we can conclude the proof for terminal graphs by showing that $\Psi_{D'}$ agrees with Equation (\ref{eq:M2KB}) from Example \ref{ex: twists_KB}, replacing $n$ with $-t+u$. Following (\ref{eq:M2graphpoly}), we compute $\sum_{T_i\subset G_2'} \left(\prod_{e_j\in G_2'} \mu_{ij}\right)$.

Order all the edges of $G_2'$ cyclically, and let $T_i = G_2' - e_i$ for each $i=1,\ldots, t+u$. First, assume that $t=0$, \textit{i.e.}, all edges are signed positively. We list the states and weights associated to each spanning tree of $G_2'$ in this case in Table \ref{tab:statesandweightsG2}. Comparing with (\ref{eq:M2KB}) of Example \ref{ex: twists_KB}, we have the desired result.

\begin{table}[ht]
\centering
\begin{tblr}{
  cells = {c},
  row{5} = {t},
  cell{2}{2} = {r},
  cell{3}{2} = {r},
  cell{4}{2} = {r},
  cell{6}{2} = {r},
  vline{2-3} = {-}{},
  hline{2} = {-}{},
}
Tree     & State                          & Weight      \\
$T_1$    & $\ell_{\text{pro}}DD \cdots D$ & $A^{u-2}$   \\
$T_2$    & $LdD\cdots D$                  & $-A^{u-6}$  \\
$T_3$    & $LLd\cdots D$                  & $A^{u-10}$  \\
$\vdots$ & $\vdots$                       & $\vdots$    \\
$T_u$    & $LLL \cdots d$                 & $(-1)^{u-1} A^{-3n+2}$ 
\end{tblr}
\caption{States and weights associated to spanning trees of $G_2'$.}
\label{tab:statesandweightsG2}
\end{table}

Now, the computation is entirely analogous if $u=0$, \textit{i.e.}, all edges are signed negatively. Thus, we need to show that the computation reduces to one of these cases for $n = -t+u$; that is, when both positive and negative crossings are present. 

We will consider how $\sum_{T_i\subset G_2'} \left(\prod_{e_j\in G_2'} \mu_{ij}\right)$ reduces assuming two neighboring edges $e_i$ and $e_{i+1}$ for $i=1,\ldots, t+u-1$ have differing sign. We will assume that $e_i$ is positive and $e_{i+1}$ is negative (though the proof of the converse will be immediately apparent) unless $i+1 = t+u$, in which case we assume that $e_i+1$ is positive and $e_i$ is negative.

If $1<i<t+u$, then for each tree $T_j$ with $j \not=i$, the state of $T_j$ has either $D \bar{D}$ or $L \bar{L}$, which contributes nothing to the weight of $T_j$. The weight of $T_i$ itself cancels with the weight of $T_{i+1}$. To see this, notice that the states of $T_i$ and $T_{i+1}$ are identical outside of the $i$th and $(i+1)$st letters, where $T_i$ has a $d \bar{D}$ and $T_{i+1}$ has a $ L \bar{d}$. Thus the weights of $T_i$ and $T_{i+1}$ differ only in that $T_{i+1}$ has a factor of $-A^{-2}$ where $T_i$ has a factor of $A^{-2}$, thereby cancelling. This means that the edges $e_i$ and $e_{i+1}$ of differing sign can be completely ignored in the computation of $\sum_{T_i\subset G_2'} \left(\prod_{e_j\in G_2'} \mu_{ij}\right)$.

The above argument works just as well when $i = t+u$, only differing in that the weight of $T_{t+u}$ cancels with the weight of $T_{t+u-1}$. When $i=1$, notice that the state of $T_1$ starts with $\ell_\text{pro} \bar{D}$ (contributing a factor of $A^{-2}$) and differs from the state of $T_2$ only at the beginning, where $T_2$ has $L \bar{d}$ (contributing a factor of $-A^{-2}$). Thus, the weights of $T_1$ and $T_2$ cancel. This may be concerning, since no other state begins with an $\ell_\text{pro}$. However, the states of all other trees begin with $L \bar{L}$, which can again be ignored, so $T_3$ effectively begins with either a $d$ or a $\bar{d}$, which has the same associated monomial as $\ell_\text{pro}$ or $\bar{\ell}_\text{pro}$. So, again, the edges $e_1$ and $e_2$ can be discarded in the computation. This concludes the proof that $\Psi_D = \langle D \rangle$ if $D$ is a terminal graph.

Finally, assume that $D$ is non-terminal. Assume that $e$ is the highest ordered regular edge whose dual is also regular. Denote its dual by $\bar{e}$. Then $e$ is always internally inactive in spanning trees containing it and externally inactive in spanning trees which omit it. The same can be said for $\bar{e}$. Let $\varepsilon\in\{-1,1\}$ denote the sign of $e$, so that $-\varepsilon$ is the sign of $\bar{e}$. Then
\begin{align*}
\Psi_D &= \sum_{T_i\subset G_1} \left(\prod_{e_j \in G_1} \mu_{ij}\right) + \sum_{T_i\subset G_2} \left(\prod_{e_j \in G_2} \mu_{ij}\right) 
\\ &=  
\left(
A^{\varepsilon} \sum_{T_i\subset G_1 - e} \left(\prod_{e_j \in G_1 - e} \mu_{ij}\right) + A^{-\varepsilon} \sum_{T_i\subset G_1 / e} \left(\prod_{e_j \in G_1/e} \mu_{ij}\right)
\right)
\\&
\qquad\qquad\qquad + 
\left(
A^{-\varepsilon}\sum_{T_i\subset G_2 - \bar{e}} \left(\prod_{e_j \in G_2 - \bar{e}} \mu_{ij}\right) + A^{\varepsilon}\sum_{T_i\subset G_2/\bar{e}} \left(\prod_{e_j \in G_2/\bar{e}} \mu_{ij}\right) 
\right)
\\ &=
A^\varepsilon \left( \sum_{T_i\subset G_1 - e} \left(\prod_{e_j \in G_1 - e} \mu_{ij}\right)  + \sum_{T_i\subset G_2/\bar{e}} \left(\prod_{e_j \in G_2/\bar{e}} \mu_{ij}\right) \right)
\\ &
\qquad\qquad\qquad + A^{-\varepsilon} \left( \sum_{T_i\subset G_1 / e} \left(\prod_{e_j \in G_1/e} \mu_{ij}\right) + \sum_{T_i\subset G_2 - \bar{e}} \left(\prod_{e_j \in G_2 - \bar{e}} \mu_{ij}\right)  \right)
\\ &= A^\varepsilon \Psi_{D'} + A^{-\varepsilon} \Psi_{D''}
\end{align*}
for $D'$ and $D''$ the two resolutions of $D$ along the crossing corresponding to the edges $e$ and $\bar{e}$. This agrees with (1) in Definition \ref{def:kauffmanbracket}, so the proof is complete.
\end{proof}

\begin{remark}
We highlight the feature that $\Psi_D$ can be used to compute the bracket of a local link $L$ precisely as long as $D$ is a non-local diagram for $L$. We only risk double-counting the bracket when $D$ is a local diagram.
\end{remark}

We are now ready to prove Theorem \ref{thm:jonespolynomialalternatingtutte}. We assume that $L$ is a non-split, non-local alternating link in $\RP^3$ and that $D$ is a connected, irreducible, alternating diagram for $L$. We denote by $G_+$ and $G_-$ the two checkerboard graphs associated to $D$ (recall that since $D$ is a reduced alternating diagram, we can assume $G_+$ has only positively signed edges and $G_-$ has only negatively signed edges). We will show that there are $k_+, k_- \in \mathbb{Z}[\pm\frac{1}{2}]$ such that
\[
J_L(t) = (-1)^{w(D)} \left( t^{k_+} \psi_{G_+}(-t, -t^{-1}) + t^{k_-} \psi_{G_-}(-t^{-1}, -t) \right).
\]
Moreover, $k_+-k_- = \frac{1}{2}$. Since $J_L(t) \in \Z[t^{\frac{1}{2}}, t^{-\frac{1}{2}}]$, we can conclude that exactly one of $k_+$ and $k_-$ is a half-integer, and the other is an integer.

\begin{proof}[Proof of Theorem \ref{thm:jonespolynomialalternatingtutte}]
Assume $D$ has $m$ crossings, so that $G_+$ and $G_-$ both have $m$ edges. Since $D$ is assumed irreducible, neither $G_+$ nor $G_-$ contain any bridges or loops.

The proof proceeds as in \cite[\S 5]{MR899051}. The state of a spanning tree $T$ of $G_+$ takes the form $L^pD^q \ell_{\text{loc}}^r \ell_{\text{pro}}^s d^t$ for some values satisfying $p + q + r + s + t = m$. According to Table \ref{tab:state_monomialsA}, this tree has weight $(-1)^{p+r} A^{-3p+q +3r - s - t}$. 

Notice that $p+q$ is the number of edges in any spanning tree of $G_+$. If $n_+$ is the number of vertices in $G_+$, this means that $p + q = n_+-1$ and $r + s + t = m - n_+ + 1$. Set $u(T) = p-r$ and $k_+' = 2(n_+-1) - m$; in particular, $k_+'$ is constant among all spanning trees of $G_+$. Thus, we can rewrite the weight of the spanning tree as $(-1)^{u(T)} A^{k_+'-4u(T)}$.

The process for spanning trees $T$ of $G_-$ is similar, only differing in that states take the form $\bar{L}^p \bar{D}^q \bar{\ell}_{\text{loc}}^r \bar{\ell}_{\text{pro}}^s \bar{d}^t$, so that the weight of spanning trees can be written as $(-1)^{u(T)} A^{4u(T)-k_-'}$ for $k_-' = 2(n_--1) - m$ where $n_-$ the number of vertices of $G_-$. By Proposition \ref{prop:KBgraphpolynomial}, we conclude that
\[
\langle D \rangle = \Psi_D =
A^{k_+'} \sum_{T_i \subset G_+} (-1)^{u(T_i)} A^{-4u(T_i)} + A^{-k_-'} \sum_{T_j \subset G_-} (-1)^{u(T_j)} A^{4u(T_j)}
\]
which means that the Jones polynomial takes the following form:
\begin{equation}\label{eq:jlsplitting}
J_L(t) = (-1)^{w(D)} \left( t^{\frac{3w(D) - k_+'}{4}} \left( \sum_{T_i \subset G_+} (-1)^{u(T_i)} t^{u(T_i)}
\right) + t^{\frac{3w(D) + k_-'}{4}} \left( \sum_{T_j \subset G_-} (-1)^{u(T_j)} t^{-u(T_j)}\right)\right).
\end{equation}
Finally, note that $p(T)$, $r(T)$, and $s(T)$ are the internal, local external, and projective external activities of $T$. Using Proposition \ref{prop:tutte_spanningtrees}, we compute
\newline
\begin{minipage}{.5\linewidth}
\begin{align*}
\psi_{G_+}(-t, -t^{-1}) &= \left. \sum_{T \subset G_+} x^{p(T)} y^{r(T)} \right|_{\raisemath{4pt}{\subalign{&x = -t \\ &y = -t^{-1}}}}
\\ &= \sum_{T \subset G_+} (-1)^{u(T)} t^{u(T)},
\end{align*}
\end{minipage}
\begin{minipage}{.5\linewidth}
\begin{align*}
\psi_{G_-}(-t^{-1}, -t) &= \left. \sum_{T \subset G_-} x^{p(T)} y^{r(T)} \right|_{\raisemath{4pt}{\subalign{&x = -t^{-1} \\ &y = -t }}}
\\ &= \sum_{T \subset G_-} (-1)^{u(T)} t^{-u(T)}.
\end{align*}
\end{minipage}
Thus, the result follows setting $k_+ = \frac{3w(D) - k_+'}{4}$ and $k_- = \frac{3w(D) + k_-'}{4}$.

Finally, we verify that $k_+-k_- = \frac{1}{2}$. First, in general, notice that if $G$ is a cellular graph in $\RP^2$ (\textit{e.g.}, when $G$ is a checkerboard graph of a non-local, non-split nullhomologous link), then
\[
v(G) - e(G) + f(G) = 1
\]
since $\chi(\RP^2) = 1$. Let $f_+$ denote the number of faces of $G_+$. Then, we can write
\[
k_+' = 2(n_+-1) - m = 2(m - f_+) - m = m - 2f_+
\]
and, using the fact that $G_-$ is the graph Poincar\'e dual to $G_+$ so $n_- = f_+$,
\[
k_-' = 2(n_- - 1) - m = 2(f_+ - 1) - m = 2 (m - n_+) - m = m - 2n_+.
\]
Thus $k_+' + k_-' = -2(n_+ - m + f_+) = -2$. We conclude that
\[
k_+ - k_- = \cfrac{3w(D) - k_+' - 3w(D) - k_-'}{4} = - \cfrac{k_+' + k_-'}{4} = \cfrac{1}{2}\,.
\]
This concludes the proof.
\end{proof}

\begin{corollary}
\label{cor:detsandtrees}
If $L$ is a non-split, non-local alternating link in $\RP^3$, then $\det_\Re(L)$ and $\det_\Im(L)$ are the number of spanning trees in the distinct Tait graphs of any reduced, alternating diagram of $L$.
\end{corollary}

\begin{proof}
This follows immediately from Theorem \ref{thm:jonespolynomialalternatingtutte}, since $\psi_G(1, 1)$ is equal to the number of spanning trees of $G$ by Proposition \ref{prop:tutte_spanningtrees} (or Proposition \ref{prop:tutte2}, for that matter).
\end{proof}

\begin{corollary}
\label{cor:detsandtreesandmatrices}
If $L$ is a non-split, non-local alternating link in $\RP^3$ with Goeritz matrices $\mathcal{G}_\pm$ obtained from a reduced, alternating diagram of $L$, then
\[
\abs{\det(\mathcal{G}_\pm)} = \mathrm{det}_{\pm}(L).
\]
\end{corollary}

\begin{proof}
Again, by Theorem \ref{thm:jonespolynomialalternatingtutte}, we have that 
\[
\abs{J_L^\pm(-1)} = \abs{\text{spanning trees of}~G_\pm}.
\]
The latter is equal to $\det(\mathcal{G}_\pm)$ by Proposition \ref{prop:goeritzdet}.
\end{proof}

\section{Classical theorems, tabulation, and detection results}
\label{s:applications}

We discuss some applications of Theorem \ref{thm:jonespolynomialalternatingtutte} before progressing to link homology. First, we give a fresh take on classical theorems of Kauffman, Murasugi, and Thistlethwaite and prove Theorem \ref{thm:altbreadth}. This theorem lends credence to a different method of tabulation of nullhomologous links in $\RP^3$, which we describe in \S \ref{ss:tabulation}. Via Theorem \ref{thm:altbreadth}, these tabulation results give a simple proof that several infinite families of nullhomologous links are detected among alternating links in $\RP^3$. This claim was stated as Theorem \ref{thm:detection}.

\subsection{Classical theorems}

In this subsection, we prove Theorem \ref{thm:altbreadth}, which determines the breadth of each summand of $J_L(t) = J_L^+(t) + J_L^-(t)$ when $L$ is a non-split alternating link.

\begin{lemma}[C.f. Proposition 2 of \cite{MR899051}]
\label{lem:mindegs}
If $G$ has $n$ vertices and $m$ edges, then $\psi_G(x,y)$ is degree $n-1$ in $x$ and degree $m-n+1 - \alpha$ in $y$. Furthermore, if $G$ has no bridges and no local loops, then:
\begin{enumerate}
\item $\psi_G$ has precisely one term of maximal degree in $x$, namely $x^{n-1}$; and
\item $\psi_G$ has a term $cy^{m-n+1-\alpha}$ for some $c \in \Z_{>0}$.
\end{enumerate}
\end{lemma}

\begin{proof}
The statements regarding the variable $x$ can be proven exactly as in \cite[Proposition 2]{MR899051}. Notice that the external activity of any tree $T$ cannot exceed $m-n+1$. Since $T$ is local, there are at least $\alpha(G)$-many edges $e$ not in $T$ such that $[e \cup T] = 1$. Thus, the local external activity of $T$ cannot exceed $m-n+1-\alpha$. 

Let $A \subset E(G)$ be a collection of $\alpha(G)$-many edges such that $G - A$ is local. Notice that there is a spanning tree of $G$ which contains none of the edges in $A$. Pick one and call it $T$. Then, order the edges of $G$ so that all the edges not in $T$ are $e_1, e_2, \ldots, e_{m-n+1}$, and that the edges not in $A$ nor $T$ are ordered $e_1, e_2, \ldots, e_{m-n+1-\alpha}$. Thus, the first $m-n+1-\alpha$ edges are locally externally active with respect to $T$, while the next $\alpha$ are projectively externally active. The resulting term is of degree $m-n+1-\alpha$ in the variable $y$. If there are no bridges, then none of the edges of $T$ are internally active, so the corresponding term in $\psi_G$ is $y^{m-n+1-\alpha}$. There may be other trees contributing the same value but, in any case, the result follows from Proposition \ref{prop:tutte_spanningtrees}.
\end{proof}

\begin{example}
\label{ex:alphacomp}
Consider the graph
\[
\tikz[baseline={([yshift=-.5ex]current bounding box.center)}, scale=.225]{
	\begin{scope}[very thick]
		\fill (0,2.5) circle (.4cm);
		\fill (0,0) circle (.4cm);
		\fill (0,-2.5) circle (.4cm);
		\draw (0,-2.5) -- (0,2.5);
		\draw (0,-2.5) to[out=30, in=-30] (0,2.5);
		\draw (0,2.5) to[out=45, in=-135] (0.70711*5, 0.70711*5);
		\draw (0,2.5) to[out=135, in=-45] (-0.70711*5, 0.70711*5);
		\draw (0,-2.5) to[out=-45, in=135] (0.70711*5, -0.70711*5);
		\draw (0,-2.5) to[out=-135, in=45] (-0.70711*5, -0.70711*5);
	\end{scope}
	\draw[white, ultra thick] (0,0) circle (5cm);
	\draw[dotted] (0,0) circle (5cm);
}
\]
which has myopic Tutte polynomial
\[
\psi_G = x^2 + 2x + 1 + xy + 2y.
\]
As promised, there is exactly one term of maximal degree in $x$. The terms $xy$ and $2y$ both take maximal degree in $y$ (here, $\alpha(G) = 2$ so $m - n + 1 - \alpha = 1$); we have proven that there is always a term achieving maximal $y$-degree with no exponent in $x$.
\end{example}

\begin{proof}[Proof of Theorem \ref{thm:altbreadth}]
Applying Lemma \ref{lem:mindegs} to Theorem \ref{thm:jonespolynomialalternatingtutte}, the breadth of $J_L^\pm(t)$ is computed as $(n-1) - (-m + n - 1 + \alpha(G_\pm)) = m - \alpha(G_\pm)$. Theorem \ref{thm:dromain} tells us that $m = c(L)$.
\end{proof}

\begin{remark}
The statement 
\[
(d_{\max} J_L(t) - d_{\min} J_L(t)) - (d_{\max} J_L^\pm(t) - d_{\min} J_L^\pm(t)) + \frac{1}{2} = \alpha(G_\pm)
\]
is equivalent to the one provided in Theorem \ref{thm:jonespolynomialalternatingtutte}, and is obtained by substituting $d_{\max} J_L(t) - d_{\min} J_L(t) = c(L) - \frac{1}{2}$ from Theorem \ref{thm:dromain}.
\end{remark}

\subsection{Tabulation of links in $\RP^3$ by local crossing number}
\label{ss:tabulation}
In \cite{MR1296890}, Drobotukhina initiated the tabulation of links in $\RP^3$ through $6$ crossings. In light of Definition \ref{def:localcrossingno}, we may try to tabulate links in $\RP^3$ by local crossing number, or by $\min\{\beta_\Re(L), \beta_\Im(L)\}$. We begin this tabulation process with Theorem \ref{thm:betatabulation}. We call a link in $\RP^3$ \emph{prime} if it cannot be written as the connected sum of a non-trivial link $L \subset \RP^3$ with a non-trivial link $L' \subset S^3$; see, for example, Figure \ref{fig:hopffam}.

\begin{theorem}
\label{thm:betatabulation}
Assume $L$ is a non-split nullhomologous prime link in $\RP^3$.
\begin{enumerate}
\item If $\beta_\varepsilon(L) = 0$ for some $\varepsilon \in \{\Re, \Im\}$, then $L = \Upsilon_n$ for some $n\in \mathbb{Z}$ (see Figure \ref{fig:upsilon}, reproduced in Figure \ref{subfig:inf0}).
\item If $\beta_\varepsilon(L) = 1$ for some $\varepsilon \in \{\Re, \Im\}$, then $L$ belongs to the infinite family depicted in Figure \ref{subfig:inf1}.
\item If $\beta_\varepsilon(L) = 2$ for some $\varepsilon \in \{\Re, \Im\}$, then $L$ belongs to one of the infinite families depicted in Figure \ref{subfig:inf2}, or it is the Hopf link (see Figure \ref{fig:hopffam}).
\end{enumerate}
\end{theorem}

\begin{figure}[ht]
\begin{subfigure}[t]{.475\linewidth}
\[
\tikz[baseline={([yshift=-.5ex]current bounding box.center)}, scale=.31]{
    \draw[knot] (-5*0.70710678, -5*0.70710678) to[out=45,in=-90] (-0.5,-1.5) -- (-0.5,1.5) to[out=90, in=-45] (-5*0.70710678, 5*0.70710678);
    \draw[knot] (5*0.70710678, -5*0.70710678) to[out=135,in=-90] (0.5,-1.5) -- (0.5,1.5) to[out=90, in=-135] (5*0.70710678, 5*0.70710678);
    \node[draw, fill=white, thin, scale=1.5] at (0,0) {$n$};
	\draw[white, ultra thick] (0,0) circle (5cm);
	\draw[dotted] (0,0) circle (5cm);
}
\]
\subcaption{The family of links with $\beta_\varepsilon(L) = 0$ for some $\varepsilon$. We denote this family by $\Upsilon_n$.}
\label{subfig:inf0}
\end{subfigure}
\bigskip
\begin{subfigure}[t]{.475\linewidth}
\[
\tikz[baseline={([yshift=-.5ex]current bounding box.center)}, scale=.31]{
    \draw[knot, overcross] (0,3.5) to[out=0, in=90] (1.5,2) to[out=-90,in=90] (0.5,0) -- (0.5,-1.5) to[out=-90, in=135] (5*0.70710678, -5*0.70710678);
    \draw[knot, overcross] (-5*0.70710678, 5*0.70710678) to[out=-45, in=-135] (5*0.70710678, 5*0.70710678);
    \draw[knot, overcross] (0,3.5) to[out=180, in=90] (-1.5,2) to[out=-90,in=90] (-0.5,0) -- (-0.5,-1.5) to[out=-90, in=45] (-5*0.70710678, -5*0.70710678);
    \node[draw, fill=white, thin, scale=1.5] at (0,-0.75) {$n$};
	\draw[white, ultra thick] (0,0) circle (5cm);
	\draw[dotted] (0,0) circle (5cm);
}
\]
\subcaption{The family of links with $\beta_\varepsilon(L) = 1$ for some $\varepsilon$. We denote this family by $V_n$.}
\label{subfig:inf1}
\end{subfigure}
\begin{subfigure}[t]{\linewidth}
\[
\tikz[baseline={([yshift=-.5ex]current bounding box.center)}, scale=.31]{
    \draw[knot] (-1.5,2.5) to[out=0,in=180] (-0.5,1.5);
    \draw[knot] (-0.5,2.5) to[out=0,in=180] (0.5,1.5);
    \draw[knot] (0.5,2.5) to[out=0,in=90] (1.5,1.5);
    \draw[knot, overcross] (-1.5,1.5) to[out=90,in=180] (-0.5,2.5);
    \draw[knot, overcross] (-0.5,1.5) to[out=0,in=180] (0.5,2.5);
    \draw[knot, overcross] (0.5,1.5) to[out=0,in=180] (1.5,2.5);
    \draw[knot] (-5*0.70710678, 5*0.70710678) to[out=-45, in=180] (-1.5,2.5);
    \draw[knot] (5*0.70710678, 5*0.70710678) to[out=-135, in=0] (1.5,2.5);
    \draw[knot] (1.5,1.5) to[out=-90,in=90] (0.5,0) -- (0.5,-1.5) to[out=-90, in=135] (5*0.70710678, -5*0.70710678);
    \draw[knot] (-1.5,1.5) to[out=-90,in=90] (-0.5,0) -- (-0.5,-1.5) to[out=-90, in=45] (-5*0.70710678, -5*0.70710678);
    \node[draw, fill=white, thin, scale=1.5] at (0,-0.75) {$n$};
	\draw[white, ultra thick] (0,0) circle (5cm);
	\draw[dotted] (0,0) circle (5cm);
}
\qquad
\tikz[baseline={([yshift=-.5ex]current bounding box.center)}, scale=.31]{
\begin{scope}[yshift=0.25cm]
    \draw[knot, overcross] (-2,1.25) arc (180:90:2cm and -0.65cm);
    \draw[knot, overcross] (2,1.25) arc (0:-90:2cm and -0.65cm);
\end{scope}
    \draw[knot,overcross] (-5*0.70710678, -5*0.70710678) to[out=45,in=-90] (-0.5,-1.5) -- (-0.5,1.5) to[out=90, in=-45] (-5*0.70710678, 5*0.70710678);
    \draw[knot,overcross] (5*0.70710678, -5*0.70710678) to[out=135,in=-90] (0.5,-1.5) -- (0.5,1.5) to[out=90, in=-135] (5*0.70710678, 5*0.70710678);
\begin{scope}[yshift=0.25cm]
    \draw[knot, overcross] (-2,1.25) arc (180:90:2cm and 0.65cm);
    \draw[knot, overcross] (2,1.25) arc (0:-90:2cm and 0.65cm);
\end{scope}
    \node[draw, fill=white, thin, scale=1.5] at (0,-0.75) {$n$};
	\draw[white, ultra thick] (0,0) circle (5cm);
	\draw[dotted] (0,0) circle (5cm);
}
\qquad
\tikz[baseline={([yshift=-.5ex]current bounding box.center)}, scale=.31]{
\begin{scope}[yshift=0.25cm]
    \draw[knot] (-2,1.25) arc (-180:0:2cm and -0.65cm);
\end{scope}
    \draw[knot,overcross] (-5*0.70710678, -5*0.70710678) to[out=45,in=-90] (-0.5,-1.5) -- (-0.5,1.5) to[out=90, in=-45] (-5*0.70710678, 5*0.70710678);
    \draw[knot,overcross] (5*0.70710678, -5*0.70710678) to[out=135,in=-90] (0.5,-1.5) -- (0.5,1.5) to[out=90, in=-135] (5*0.70710678, 5*0.70710678);
\begin{scope}[yshift=0.25cm]
    \draw[knot, overcross] (-2,1.25) arc (-180:0:2cm and 0.65cm);
\end{scope}
    \node[draw, fill=white, thin, scale=1.5] at (0,-0.75) {$n$};
	\draw[white, ultra thick] (0,0) circle (5cm);
	\draw[dotted] (0,0) circle (5cm);
}
\qquad
\tikz[baseline={([yshift=-.5ex]current bounding box.center)}, scale=.31]{
    \draw[knot, overcross] (0,3.5) to[out=0, in=90] (1.5,2) to[out=-90,in=90] (0.5,0.25);
    \draw[knot, overcross] (0,-3.5) to[out=180, in=-90] (-1.5,-2) to[out=90,in=-90] (-0.5,-0.25);
    \draw[knot, overcross] (-5*0.70710678, 5*0.70710678) to[out=-45, in=-135] (5*0.70710678, 5*0.70710678);
    \draw[knot, overcross] (-5*0.70710678, -5*0.70710678) to[out=45, in=135] (5*0.70710678, -5*0.70710678);
    \draw[knot, overcross] (0,-3.5) to[out=0, in=-90] (1.5,-2) to[out=90,in=-90] (0.5,-0.25);
    \draw[knot, overcross] (0,3.5) to[out=180, in=90] (-1.5,2) to[out=-90,in=90] (-0.5,0.25);
    \node[draw, fill=white, thin, scale=1.5] at (0,0) {$n$};
	\draw[white, ultra thick] (0,0) circle (5cm);
	\draw[dotted] (0,0) circle (5cm);
}
\]
\subcaption{The families of prime links with $\beta_\varepsilon(L) = 2$ for some $\varepsilon$. Denote them by $W_n,\, X_n,\, Y_n$, and $Z_n$ respectively.}
\label{subfig:inf2}
\end{subfigure}
\caption{Some families of links with small $\beta_{\varepsilon}$. We have omitted mirrors, split links, and composite links, but see Figure \ref{fig:hopffam} for another family of mostly composite links with $\min\{\beta_\Re, \beta_\Im\}=2$; there are no other non-split links with $\beta_\varepsilon(L)\ \leq 2$.}
\label{fig:betafams}
\end{figure}

In Figure \ref{fig:betafams}, we use the convention that
\[
\ldots \qquad
\tikz[baseline={([yshift=-.5ex]current bounding box.center)}, scale=.25]{    \node[draw, fill=white, thin, scale=1] at (0,0) {$-1$};} = \tikz[baseline={([yshift=-.5ex]current bounding box.center)}, scale=.5]{ \draw[knot, overcross] (0,0) to[out=90,in=-90] (1,1); \draw[knot, overcross](1,0) to[out=90,in=-90] (0,1); }
~ ,\qquad \qquad
\tikz[baseline={([yshift=-.5ex]current bounding box.center)}, scale=.25]{    \node[draw, fill=white, thin, scale=1] at (0,0) {$0$};} = \tikz[baseline={([yshift=-.5ex]current bounding box.center)}, scale=.5]{ \draw[knot, overcross](0,0) to[out=45,in=-45] (0,1); \draw[knot, overcross] (1,0) to[out=135,in=-135] (1,1); }
~ ,\qquad \qquad
\tikz[baseline={([yshift=-.5ex]current bounding box.center)}, scale=.25]{    \node[draw, fill=white, thin, scale=1] at (0,0) {$1$};} = \tikz[baseline={([yshift=-.5ex]current bounding box.center)}, scale=.5]{ \draw[knot, overcross](1,0) to[out=90,in=-90] (0,1); \draw[knot, overcross] (0,0) to[out=90,in=-90] (1,1); }
~ ,\qquad \ldots
\]
and so on. We give names to each of the links referenced in Theorem \ref{thm:betatabulation} in Figure \ref{fig:betafams}. Comparing with Drobotukhina's tabulation \cite{MR1296890}, we note that:
\begin{itemize}
\item the family $\Upsilon_n$ contains the unknot and $1_1^2$, $2_1$, $3_1^2$, $4_1$, $5_1^2$, and $6_1$;
\item the family $V_n$ contains $2_1$, $3_1$, $4_2$, $5_1$, and $6_2$;
\item the family $W_n$ contains $3_1$, $3_1^2$, $4_3$, $4_1^2$, $5_2$, $5_2^2$, $6_3$, and $6_1^2$;
\item the family $X_n$ contains $4_2^2$, $5_1^3$, and $6_{12}^2$;
\item the family $Y_n$ contains $4_3^2$, $5_2^3$, and $6_{35}^2$;
\item the family $Z_n$ contains $4_2^2$, $4_3^2$, and $6_9^2$.
\end{itemize}
We note that $Z_n$ should also contain a 5-crossing link from Drobotukhina's table. It was probably meant to be in the place of the knot (erroneously called a link) labeled $5_4^2$; notice that the pictured knot is a duplicate of $5_8$.

There are overlaps between these families. For example, $V_0 = m \Upsilon_2$, reflecting the fact that $\beta_\Re(\Upsilon_2) = 0$ and $\beta_\Im(\Upsilon_2) = 1$. Moreover:
\begin{enumerate}
\item $V_0 = m\Upsilon_2$ and $V_k = mV_{-k-1}$ for each $k$;
\item $W_0 = m\Upsilon_3 = \Upsilon_{-3}$ and $W_{-1} = V_1$;
\item $Y_{-k} = mY_k$ for each $k$; and
\item $Z_0 = mX_0$, $Z_{-1} = Y_0$, $Z_{-2} = X_0$, and $Z_k = mZ_{-k-2}$ for each $k$.
\end{enumerate}
Otherwise, each link is uniquely represented. Thus, we could tabulate the (non-split, prime) links with small $\beta_\varepsilon$ as in Table \ref{tab:tabulation}.

\begin{table}[ht]
\centering
\begin{tabular}{|c||l|l|l|} 
\hline
$\min\{\beta_\Re(L), \beta_\Im(L)\}$ & \multicolumn{1}{c|}{$0$}                        & \multicolumn{1}{c|}{$1$}                       & \multicolumn{1}{c|}{$2$}       \\ 
\hline
\multirow{5}{*}{Links}             & \multirow{5}{*}{$\Upsilon_n~\text{for each}~n$} & \multirow{5}{*}{$V_n~\text{for each}~n\neq 0$} & $W_n~\text{for each}~n \neq 0, -1$        \\
                                   &                                                 &                                                & $X_n~\text{for each}~n$        \\
                                   &                                                 &                                                & $Y_n~\text{for each}~n$  \\
                                   &                                                 &                                                & $Z_n~\text{for each}~n$  \\
                                   &                                                 &                                                & The Hopf link                  \\
\hline
\end{tabular}
\caption{Tabulating links in $\RP^3$ by $\min\{\beta_\Re(L), \beta_\Im(L)\}$.}
\label{tab:tabulation}
\end{table}

\begin{proof}[Proof of Theorem \ref{thm:betatabulation}]
To prove this, we will consider the checkerboard graphs of the relevant diagrams. Indeed, if we have a complete list of checkerboard graphs $G$ with $\beta(G) = n$, then by considering all possible over/under assignments to the crossings we obtain a complete list of link diagrams $D$ with $\beta_{\varepsilon}(D) = n$. Up to $\beta(G) = 2$, it turns out that a calculation of the Jones polynomials is enough to distinguish all of the resulting links which are not straightforwardly equivalent; see Figure \ref{fig:betafams}. 

To enumerate (connected) graphs $G \subset \RP^2$ with fixed $\beta(G)$, we proceed by induction on $\beta(G)$. The base case is $\beta(G) = 0$. In this case we claim that there can only be one vertex in $G$. Indeed, if there were two distinct vertices connected by an edge, then removing all other edges from $G$ would produce a local graph so that $\beta(G) \geq 1$. Graphs with a single vertex are classified by the number of local loops, and the number of projective loops. Since local loops produce crossings which can be removed by an R1 move, the only resulting links with $\beta_{\varepsilon}(L) = 0$ are those with a checkerboard graph consisting of a single vertex and $n$ projective loops. These are precisely the links $\Upsilon_n$. 

Note that while graphs with $\beta(G) = 0$ may have arbitrarily many edges, they have a finite number of vertices: in this case only one. 

Now suppose we have a (connected) graph $G\subset \RP^2$ with $\beta(G) = n$. By definition of $\beta(G)$, there is an edge $e$ in $G$ so that $\beta(G - e) = n-1$. However, by inductive assumption, we can list all graphs with $\beta = n-1$, and they have at most $n$ vertices. To obtain a complete list of graphs with $\beta = n$, take each graph $G'\subset \RP^2$ with $\beta(G') = n-1$ and consider the collection of connected graphs which can be obtained by adding a single edge to $G'$, either with both vertices in $G'$ or with one vertex in $G'$ and adding one new vertex. Since $G'$ has at most $n$ vertices, this is a finite collection. We can then throw out any of these graphs which we have already tabulated to have $\beta < n$, and the remaining graphs are precisely those with $\beta = n$. 

Following this strategy, we tabulated by hand the complete list of graphs $G \subset \RP^2$ with $\beta(G) \leq 2$. 
\begin{figure}[ht]
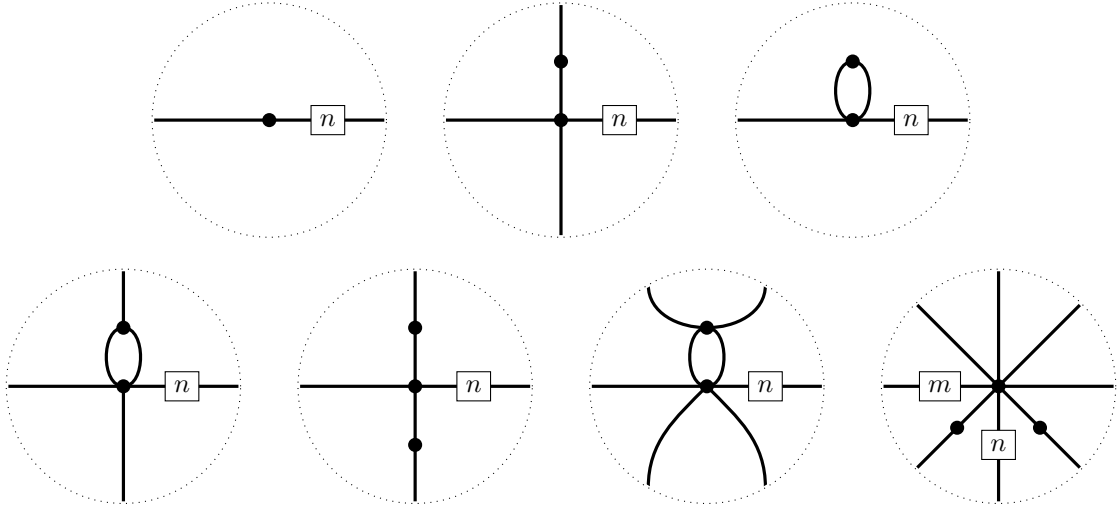

\[
\tikz[baseline={([yshift=-.5ex]current bounding box.center)}, scale=.31]{
    \fill[] (0,0) circle(.3cm);
    \draw[very thick] (-5,0) -- (5,0);
    \node[draw, fill=white, thin] at (2.5,0) {$n$};
	\draw[white, ultra thick] (0,0) circle (5cm);
	\draw[dotted] (0,0) circle (5cm);
}
\qquad
\tikz[baseline={([yshift=-.5ex]current bounding box.center)}, scale=.31]{
    \fill[] (0,0) circle(.3cm);
    \fill[] (0,2.5) circle(.3cm);
    \draw[very thick] (-5,0) -- (5,0);
    \draw[very thick] (0,-5) -- (0,5);
    \node[draw, fill=white, thin] at (2.5,0) {$n$};
	\draw[white, ultra thick] (0,0) circle (5cm);
	\draw[dotted] (0,0) circle (5cm);
}
\qquad
\tikz[baseline={([yshift=-.5ex]current bounding box.center)}, scale=.31]{
    \fill[] (0,0) circle(.3cm);
    \fill[] (0,2.5) circle(.3cm);
    \draw[very thick] (-5,0) -- (5,0);
    \draw[very thick] (0,0) to[out=180,in=180] (0,2.5);
    \draw[very thick] (0,0) to[out=0,in=0] (0,2.5);
    \node[draw, fill=white, thin] at (2.5,0) {$n$};
	\draw[white, ultra thick] (0,0) circle (5cm);
	\draw[dotted] (0,0) circle (5cm);
}
\]

\[
\tikz[baseline={([yshift=-.5ex]current bounding box.center)}, scale=.31]{
    \fill[] (0,0) circle(.3cm);
    \fill[] (0,2.5) circle(.3cm);
    \draw[very thick] (-5,0) -- (5,0);
    \draw[very thick] (0,-5) -- (0,0);
    \draw[very thick] (0,0) to[out=180,in=180] (0,2.5);
    \draw[very thick] (0,0) to[out=0,in=0] (0,2.5);
    \draw[very thick] (0,2.5) -- (0,5);
    \node[draw, fill=white, thin] at (2.5,0) {$n$};
	\draw[white, ultra thick] (0,0) circle (5cm);
	\draw[dotted] (0,0) circle (5cm);
}
\qquad
\tikz[baseline={([yshift=-.5ex]current bounding box.center)}, scale=.31]{
    \fill[] (0,0) circle(.3cm);
    \fill[] (0,2.5) circle(.3cm);
    \fill[] (0,-2.5) circle(.3cm);
    \draw[very thick] (-5,0) -- (5,0);
    \draw[very thick] (0,-5) -- (0,5);
    \node[draw, fill=white, thin] at (2.5,0) {$n$};
	\draw[white, ultra thick] (0,0) circle (5cm);
	\draw[dotted] (0,0) circle (5cm);
}
\qquad
\tikz[baseline={([yshift=-.5ex]current bounding box.center)}, scale=.31]{
    \fill[] (0,0) circle(.3cm);
    \fill[] (0,2.5) circle(.3cm);
    \draw[very thick] (-5,0) -- (5,0);
    \draw[very thick] (0,0) to[out=180,in=180] (0,2.5);
    \draw[very thick] (0,0) to[out=0,in=0] (0,2.5);
    \draw[very thick] (0,2.5) to[out=0,in=-90] (0.5*5, 0.866025*5);
    \draw[very thick] (0,2.5) to[out=180,in=-90] (-0.5*5, 0.866025*5);
    \draw[very thick] (0,0) to[out=-135,in=90] (-0.5*5, -0.866025*5);
    \draw[very thick] (0,0) to[out=-45,in=90] (0.5*5, -0.866025*5);
    \node[draw, fill=white, thin] at (2.5,0) {$n$};
	\draw[white, ultra thick] (0,0) circle (5cm);
	\draw[dotted] (0,0) circle (5cm);
}
\qquad
\tikz[baseline={([yshift=-.5ex]current bounding box.center)}, scale=.31]{
    \fill[] (0,0) circle(.3cm);
    \fill[] (-2.5*0.70710678,-2.5*0.70710678) circle(.3cm);
    \fill[] (2.5*0.70710678,-2.5*0.70710678) circle(.3cm);
    \draw[very thick] (0,-5) -- (0,5);
    \draw[very thick] (-5,0) -- (5,0);
    \draw[very thick] (-5*0.70710678, -5*0.70710678) -- (5*0.70710678, 5*0.70710678);
    \draw[very thick] (5*0.70710678, -5*0.70710678) -- (-5*0.70710678, 5*0.70710678);
    \node[draw, fill=white, thin] at (0,-2.5) {$n$};
    \node[draw, fill=white, thin] at (-2.5,0) {$m$};
	\draw[white, ultra thick] (0,0) circle (5cm);
	\draw[dotted] (0,0) circle (5cm);
}
\]
\caption{All graphs with $\beta(G) = e(G) - \alpha(G) \le 2$, except for those with a pendant edge corresponding to a Reidemeister I move. The indices denote any non-negative number of parallel edges.}
\label{fig:tabulationgraphs}
\end{figure}
Taking the corresponding possible link diagrams, and removing obvious duplicates and mirror images, we obtain the links shown in Figure \ref{fig:betafams} and Figure \ref{fig:hopffam}. Noting that the family in Figure \ref{fig:hopffam} consists of connected sums of $\Upsilon_n$ with the local Hopf link, we are reduced to only the links in Figure \ref{fig:betafams} or the Hopf link itself, as claimed. It is also notable that the first two families of graphs in the second line of Figure \ref{fig:tabulationgraphs} yield the same family $W_n$, observing the isotopy in Figure \ref{fig:doublecountedfamilies}.
\begin{figure}[ht]
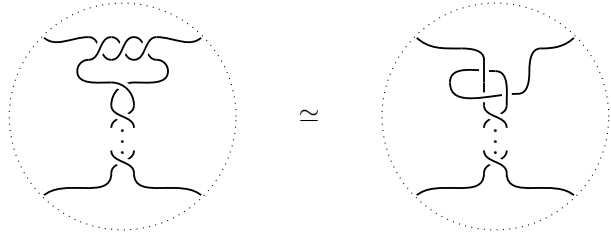

\[
\tikz[baseline={([yshift=-.5ex]current bounding box.center)}, scale=0.3]
{
\draw[knot, overcross] (0,0) to[out=90, in=-90] (1,1);
\begin{scope}[yshift=1cm]
\draw[knot, overcross] (0,2) to[out=90, in=180] (1,3);
\draw[knot, overcross] (0,1) to[out=90,in=-90] (1,2);
\end{scope}
\draw[knot, overcross] (1,0) to[out=90, in=-90] (0,1);
\begin{scope}[yshift=1cm]
\draw[knot, overcross] (1,2) to[out=90,in=0] (0,3);
\draw[knot, overcross] (1,1) to[out=90,in=-90] (0,2);
\end{scope}
\draw[knot] (0.5+5*0.707, 2.5+5*0.707) to[out=-135, in=0] (2,6);
\draw[knot, overcross] (0.5-5*0.707, 2.5+5*0.707) to[out=-45, in=180] (-1,6);
\draw[knot] (0.5+5*0.707, 2.5-5*0.707) to[out=135, in=-90] (1,0);
\draw[knot] (0.5-5*0.707, 2.5-5*0.707) to[out=45, in=-90] (0,0);
\begin{scope}[yshift=1cm]
\draw[knot] (-1,5) to[out=0,in=180] (0,4);
\draw[knot] (0,5) to[out=0,in=180] (1,4);
\draw[knot] (1,5) to[out=0,in=180] (2,4);
\draw[knot, overcross] (-1,4) to[out=0,in=180] (0,5);
\draw[knot, overcross] (0,4) to[out=0,in=180] (1,5);
\draw[knot, overcross] (1,4) to[out=0,in=180] (2,5);
\end{scope}
\draw[knot] (-1,5) to[out=180, in=90] (-1.5,4.5) to[out=-90,in=180] (0,4);
\draw[knot] (2,5) to[out=0,in=90] (2.5,4.5) to[out=-90,in=0] (1,4);
\node at (0.5,1.75) {$\vdots$};
\draw[ultra thick, white] (0.5,2.5) circle(5cm);
\draw[dotted] (0.5,2.5) circle(5cm);
} 
\qquad \simeq \qquad
\tikz[baseline={([yshift=-.5ex]current bounding box.center)}, scale=0.3]
{
\draw[knot, overcross] (0,0) to[out=90, in=-90] (1,1);
\draw[knot, overcross] (0,2) to[out=90, in=-90] (1,3);
\draw[knot, overcross] (1,0) to[out=90, in=-90] (0,1);
\draw[knot, overcross] (1,2) to[out=90,in=-90] (0,3);
\node at (0.5,1.75) {$\vdots$};
\draw[knot] (0.5+5*0.707, 2.5+5*0.707) to[out=-135, in=0] (2.5,5.5);
\draw[knot, overcross] (0.5-5*0.707, 2.5+5*0.707) to[out=-45, in=180] (-1,5.5);
\draw[knot] (0.5+5*0.707, 2.5-5*0.707) to[out=135, in=-90] (1,0);
\draw[knot] (0.5-5*0.707, 2.5-5*0.707) to[out=45, in=-90] (0,0);
\draw[knot] (1.5,3.5) to[out=0,in=-90] (2,4.5) to[out=90,in=180] (2.5,5.5);
\draw[knot, overcross] (0.5,4.5) to[out=180,in=90] (-1.5,4);
\draw[knot, overcross] (0,3) -- (0,4);
\draw[knot, overcross] (-1.5,4) to[out=-90,in=180] (1.5,3.5);
\draw[knot, overcross] (0,4) -- (0,5) to[out=90,in=0] (-1,5.5);
\draw[knot, overcross] (1,3) to[out=90,in=0] (0.5,4.5);
\draw[ultra thick, white] (0.5,2.5) circle(5cm);
\draw[dotted] (0.5,2.5) circle(5cm);
}
\]
\caption{Duplicate families of links provided by distinct graphs with $\beta(G) = 2$.}
\label{fig:doublecountedfamilies}
\end{figure}
\end{proof}

\begin{figure}[ht]
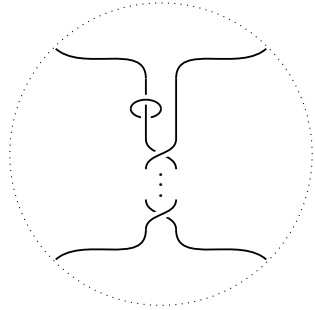

\tikz[baseline={([yshift=-.5ex]current bounding box.center)}, scale=0.4]
{
\draw[knot] (1,0) to[out=90, in=-90] (0,1);
\draw[knot] (1,2) to[out=90,in=-90] (0,3);
\node at (0.5,1.75) {$\vdots$};
\draw[knot, overcross] (0,0) to[out=90, in=-90] (1,1);
\draw[knot, overcross] (0,2) to[out=90, in=-90] (1,3);
\draw[knot] (-0.5,4) arc (-180:0:0.5cm and .3cm);
\draw[knot, overcross] (1,3) -- (1,5);
\draw[knot, overcross] (0,3) -- (0,5);
\draw[knot] (0.5+5*0.707, 2.5+5*0.707) to[out=-135, in=90] (1,5);
\draw[knot, overcross] (0.5-5*0.707, 2.5+5*0.707) to[out=-45, in=90] (0,5);
\draw[knot] (0.5+5*0.707, 2.5-5*0.707) to[out=135, in=-90] (1,0);
\draw[knot] (0.5-5*0.707, 2.5-5*0.707) to[out=45, in=-90] (0,0);
\draw[knot, overcross] (-0.5,4) arc (-180:0:0.5cm and -.3cm);
\draw[ultra thick, white] (0.5,2.5) circle(5cm);
\draw[dotted] (0.5,2.5) circle(5cm);
}
\caption{Another family of links with $\min\{\beta_\Re, \beta_\Im\}=2$. Every member of this family is the connected sum of $\Upsilon_n$ with the Hopf link. Thus every member is composite except for the Hopf link, obtained when $n=0$.}
\label{fig:hopffam}
\end{figure}

We conclude with some questions about local and projective crossing number.

\begin{question} Let $L_1$ be a nullhomologous link in $\RP^3$, let $L_2$ be a link in $S^3$, and let $L_1 \# L_2 \subset \RP^3$ be an arbitrary connected sum of $L_1$ and $L_2$. 
\begin{enumerate}
	\item Is $\beta_{\varepsilon}(L_1 \# L_2) = \beta_{\varepsilon}(L_1) + c(L_2)$?
	\item Is $\alpha_{\varepsilon}(L_1 \# L_2) = \alpha_{\varepsilon}(L_1)$?
	\item Is $\alpha_{\varepsilon}(L_1) + \beta_{\varepsilon}(L_1) = c(L_1)$?
	\item Is there an $n \in \N$ so that for all nullhomologous links $L_1 \in \RP^3$, $\beta_{\varepsilon}(L_1) \leq n$? so that $\alpha_{\varepsilon}(L_1) \leq n$?
\end{enumerate}
\end{question}

\subsection{Detection results}
\label{ss:detection}

It is known that if $L$ is a prime, non-split, alternating link and $J_L(t) = J_{T(2,k)}(t)$, then $L = T(2,k)$; \emph{i.e.}, the Jones polynomial detects the $(2,k)$-torus links among the class of alternating links. This follows from \cite[Theorem 6.5]{MR99667}: for such a link, 
\[
\det(L) \ge 2c(L) - 3
\]
unless $L = T(2,k)$ for some $k$. If $J_L(t) = J_{T(2,k)}(t)$, then $\det(L) = k$, so $L = T(2,k)$ for each $k > 3$. Otherwise, we can use tabulation to conclude the detection. In what follows, we observe that we can use the tabulation of \S \ref{ss:tabulation} in conjunction with Theorem \ref{thm:altbreadth} to detect several infinite families among the class of alternating links. Again, the trick is to replace crossing number with local crossing number. We will write $\mathscr{L}$ to stand for the set
\begin{equation}
\label{eq:scriptell}
\mathscr{L} =\{\Upsilon_n, V_n, W_n, mW_n\}_{n \in \Z} \cup \{Z_n\}_{n \in \Z_{\neq -1 }} 
\end{equation}

Recall that, in Example \ref{ex: twists_KB}, we computed
\begin{equation}
\label{eq:JPtwist}
(-1)^n J_{\Upsilon_n^{\star}} (t) =
\begin{dcases}
t^{\frac{n}{2}} + \sum_{i=0}^{n-1}(-1)^i t^{\frac{n+1}{2} + i} & \text{if $n$ is even or $\star = \, \uparrow \,$,}
\\
t^{-n} + \sum_{i=0}^{n-1} (-1)^i t^{-n+\frac{1}{2}+i} & \text{if $n$ is odd and $\star = \, \downarrow \,$.}
\end{dcases}
\end{equation}
Namely, $\Upsilon_n$ is a knot if $n$ is even, and a 2-component link if $n$ is odd. As in Example \ref{ex: twists_KB}, when $n$ is odd, we let $\Upsilon_n^\uparrow$ denote the oriented link $\Upsilon_n$ which has both components in Figure \ref{fig:upsilon} co-oriented, and $\Upsilon_n^\downarrow$ otherwise. Notice that none of $J_{\Upsilon_n^\uparrow}$, $J_{\Upsilon_{-n}^\uparrow}$, $J_{\Upsilon_n^\downarrow}$, or  $J_{\Upsilon_{-n}^\downarrow}$ are equal for $\abs{n} > 1$. When $n=1$, $J_{\Upsilon_1^\uparrow} = J_{\Upsilon_{-1}^\downarrow}$ and $J_{\Upsilon_{-1}^\uparrow} = J_{\Upsilon_1^\downarrow}$, but this is because $\Upsilon_1^\uparrow = \Upsilon_{-1}^\downarrow$ and $\Upsilon_{-1}^\uparrow = \Upsilon_1^\downarrow$ as oriented links.

For the rest of the members of $\mathscr{L}$, we can hope that each is distinguished by $J^\mathsf{Minor}_L(t)$, where $J^\mathsf{Minor}_L(t)$ is a part of $J_L(t) = J_L^+(t) + J_L^-(t)$ with minimal breadth. We use Proposition \ref{prop:KBgraphpolynomial} (or Theorem \ref{thm:jonespolynomialalternatingtutte}) to isolate these values and streamline their calculation. Assuming that $n\ge 0$, we compute that
\begin{equation}
\label{eq:jpforV}
J_{V_n}^{\mathsf{minor}} (t) =(-1)^n  \left( -t^{-n-\frac{5}{2}} + t^{-n-\frac{3}{2}}\right).
\end{equation}
There is no ambiguity since $V_n$ is a knot for each $n$. 

\begin{proof}[Proof of Theorem \ref{thm:detection}]
Assume that $L'$ is an alternating link and $L\in \mathscr{L}$. If $J_L(t) = J_{L'}(t)$, then $J_{L'}^\varepsilon(t) = J_{\Upsilon_n^\star}^\varepsilon(t)$ for both $\varepsilon \in \{+, -\}$.

\medskip

\noindent \textit{Case I}. If $L$ is $\Upsilon_n$ or $m\Upsilon_n$ for some $n$, then Equation (\ref{eq:JPtwist}) implies that $d_{\max} J_{L'}^\varepsilon - d_{\min} J_{L'}^\varepsilon = 0$ for at least one $\varepsilon$. Since $L'$ is alternating, Theorem \ref{thm:altbreadth} says that $\alpha(G_\varepsilon) = c(L')$, \textit{i.e.}, $\alpha(G_\varepsilon) = e(G_\varepsilon)$. This is true if and only if $G_\varepsilon$ is the graph with a single vertex and $n$ edges, all of which are projective loops. Thus $L'$ must be $\Upsilon_n$ as a non-oriented link. The result then follows from the remarks immediately succeeding Equation (\ref{eq:JPtwist}).

\medskip

\noindent \textit{Case II}. Assume that $L$ is $V_n$ or $mV_n$ for some $n$. Again, we have that $J_{L'}^\varepsilon - d_{\min} J_{L'}^\varepsilon = 1$ for at least one $\varepsilon$, so it follows that $L'$ is $V_k$ or $mV_k$ for some $k$. The result follows from Equation (\ref{eq:jpforV}).

\medskip

\noindent \textit{Case III}. Assume that $L \in \{W_n, mW_n\}_{n \in \Z} \cup \{Z_n\}_{n \in \Z_{\neq -1}}$. Now, we have that $J_{L'}^\varepsilon - d_{\min} J_{L'}^\varepsilon = 2$ for at least one $\varepsilon$. Thus $L'$ must be one of $W_n$, $mW_n$, $X_n$, $mX_n$, $Y_n$, or $Z_n$ for some $n$. We can place $L'$ into one of these families by computing the breadth-minimal parts of the Kauffman brackets. Table \ref{tab:comparingbracket} tells us that no family has $\langle D \rangle_\mathsf{Minor}$ taking the same form, where $\langle D \rangle_\mathsf{Minor}$ is computed via Proposition \ref{prop:KBgraphpolynomial} (using a reduced alternating diagram if one is easily determined).

\begin{table}[ht]
\renewcommand{\arraystretch}{1.5}
\centering
\begin{tabular}{|c|c||c|c|} 
\hline
Family                    & Qualification & $\langle D \rangle_\mathsf{Minor}$ & $\mathrm{det}_\mathsf{Major}$  \\ 
\hhline{|==::==|}
\multirow{2}{*}{$W_n$}    & $n>0$         & $A^{n+7} - A^{n+3} + A^{n-1}$      & $3n+1$                         \\
                          & $n<-1$        & $-A^{-n+4}+A^{-n}-A^{-n-4}$        & $-3n-1$                         \\ 
\hline
\multirow{2}{*}{$X_n$}    & $n\ge 0$      & $-A^{n+6}+A^{n+2}-2A^{n-2}$        & $4n+4$                         \\
                          & $n<0$         & $-A^{-n+6}+A^{-n+2}-2A^{-n-2}$     & \cellcolor{gray!35}                  \\ 
\hline
$Y_n$                     & $n\ge 0$      & $-A^{n+6}-A^{n-6}$                 & \cellcolor{gray!35}                    \\ 
\hline
$Z_n$                     & $n\ge 0$      & $A^{n+8}-2A^{n+4}+A^n$             & $4n+4$                         \\ 
\hline
$\Upsilon \# \mathcal{H}$ & $n\ge 0$      & $-A^{n+4}-A^{n-4}$                 & $2n$                           \\
\hline
\end{tabular}
\caption{We compare the the shapes of $\langle - \rangle_\mathsf{Minor}$ for families of links with $\beta_\varepsilon(L) = 2$ for some $\varepsilon$. We omit the determinants of $X_{n<0}$ and $Y_n$ since the diagrams in Figure \ref{fig:betafams} are non-alternating.}
\label{tab:comparingbracket}
\end{table}

Notice that each of $W_n$, $Z_n$ for $n\neq -1$, $X_n$ for $n \ge 0$ and $\Upsilon \# \mathcal{H}$ are alternating, so we can compute $\mathrm{det}_\mathsf{Major}(L)$ (which is determined by the Jones polynomial) by counting spanning trees via Corollary \ref{cor:detsandtrees}. Notice that each link in the family $\{W_n\}$ is distinguished from one another by $\mathrm{det}_\mathsf{Major}$. We compute:
\begin{equation}
\label{eq:JPW}
(-1)^{n+1} J_{W_n^{\star}}^{\mathsf{Minor}} (t) =
\begin{dcases}
t^{\frac{n+1}{2}} - t^{\frac{n+3}{2}} + t^{\frac{n+5}{2}} & \text{if $n$ is odd or $\star = \, \uparrow \,$,}
\\
t^{-n-4} - t^{-n-3} + t^{-n-2} & \text{if $n$ is even and $\star = \, \downarrow \,$.}
\end{dcases}
\end{equation}
To finish, we note that $J_{W_n}^{\mathsf{Minor}} (t) = J_{mW_n}^{\mathsf{Minor}}(t)$ only when $n = -3$; see Equation (\ref{eq:JPW}). The knot $W_{-3}$ is $5_2$ in Drobotukhina's table, so we can just look up its Jones polynomial:
\[
J_{W_{-3}}(t) = (t^{-1} - 1 + t) + (t^{-\frac{1}{2}} - 2t^{\frac{1}{2}} + 2t^{\frac{3}{2}} - 2t^{\frac{5}{2}} + t^{\frac{7}{2}}).
\]
We conclude that $J_{W_{-3}}(t) \neq J_{mW_{-3}}(t)$, so the result follows for this family. To complete the proof, we compute $J_{Z_n}^\mathsf{Minor}(t)$:
\begin{equation*}
\begin{aligned}
&n~\text{even}:\quad J_{Z_n}^\mathsf{Minor}(t) = t^{-n-2}-2t^{-n-1} +t^{-n}; \\
&n~\text{odd}:\quad J_{Z_n}^\mathsf{Minor}(t) =
{\begin{dcases}
-t^{3-n} + 2t^{2-n} - t^{1-n} & \text{or}
\\
-t^{-n-3} + 2t^{-n-4} - t^{-n-5}. &
\end{dcases}}
\end{aligned}
\end{equation*}
\end{proof}

\section{Spanning tree complexes and Khovanov homology}
\label{s:spanningtreekh}

For any link $L$ in $S^3$, it was observed by Champanerkar-Kofman \cite{MR2480298} (see also \cite{MR2468373}) and Wehrli \cite{MR2477595}, that there is a cochain complex generated by the spanning trees of any Tait graph of $L$ whose homology is the reduced Khovanov homology of $L$. In this section, we define an analogue of the spanning tree complex for links in $\RP^3$, following \cite{MR2480298}. Namely, the Champanerkar-Kofman partial order on spanning trees provides a filtration on the reduced Khovanov complex. The differentials for the resulting spectral sequence are trivial for alternating links. We conclude that the reduced Khovanov homology of an alternating link in $\RP^3$ is determined by its Jones polynomial and its two signatures.

\subsection{Twisted unknots and spanning trees}
\label{ss:twistedunknots}

In this subsection we provide some relations between twisted unknots, defined momentarily, and the spanning trees, following Champanerkar-Kofman \cite[\S 2.1]{MR2480298}.

First, notice that in the computation of $\langle D \rangle$, it is computationally benefitial to stop when the leaves of the skein-resolution tree are \emph{twisted unknots}---that is, unknots obtained from a crossingless diagram of the unknot by performing only Reidemeister I, IV, and V moves. Let $\mathscr{U}(D)$ denote the leaves of this incomplete skein-resolution tree, which we call the \emph{twisted unknots of $D$}. Then, by Definition \ref{def:kauffmanbracket},
\[
\langle D \rangle = \sum_{U\in \mathscr{U}(D)} A^{\sigma(U)} (-A)^{3w(U)}
\]
where $\sigma = \#\{\text{A-smoothings}\} - \#\{\text{B-smoothings}\}$. We call the elements of $\mathscr{U}(D)$ the twisted unknots of $D$. If $D$ is a diagram of a nullhomologous link, then Lemma \ref{lem:ManWillis} says that each of the elements of $\mathscr{U}(D)$ are class-0 unknots.

\begin{example}
To obtain a skein-resolution tree, we start by enumerating the crossings of $D$. For reasons which will become clear momentarily, we smooth crossings from greatest to least. See Figure \ref{fig:mainskeinrestree} for the incomplete skein-resolution tree corresponding to the link diagram of Example \ref{ex:mainST}.
\begin{figure}[ht]
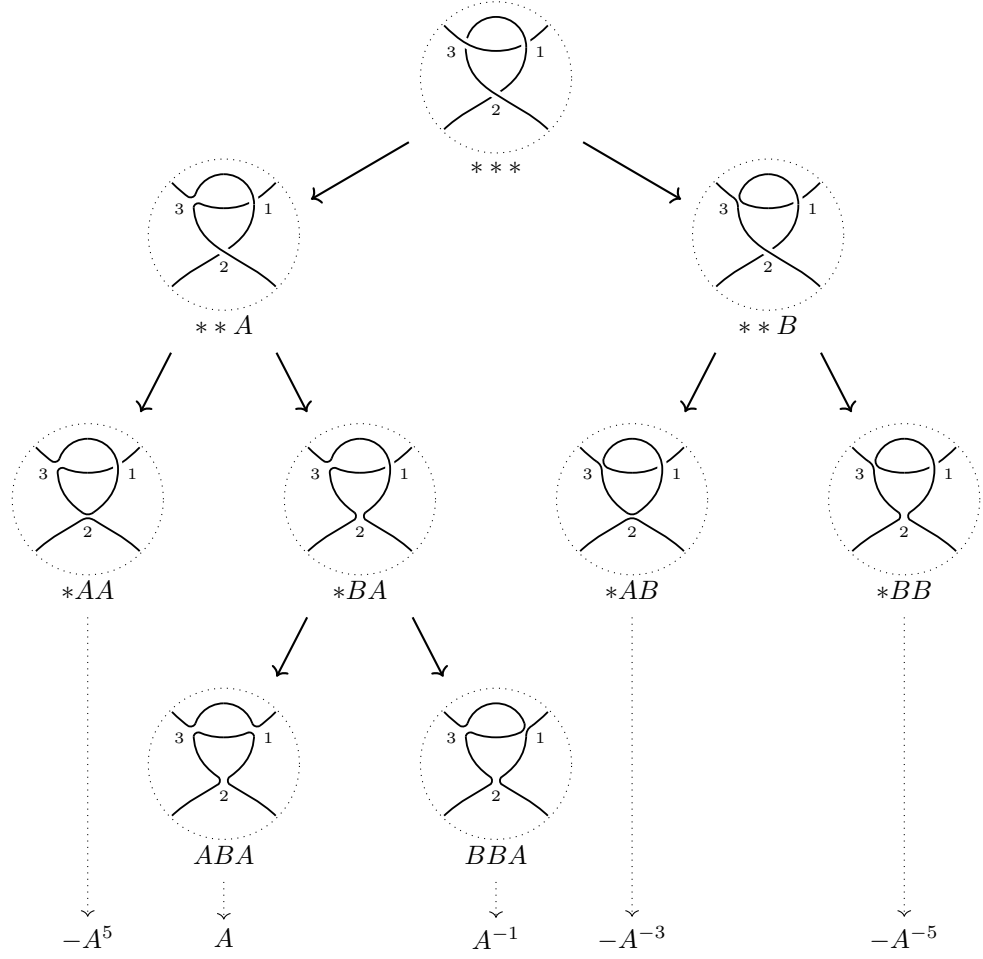

\[
\tikz[yscale=1.75, xscale=1.5, scale=0.8]{
	\node(***) at (0,0) {$\tikz[baseline={([yshift=-.5ex]current bounding box.center)}, scale=.2]{
	\node at (0,-6) {$***$};
	\draw[knot, overcross] (0,1.75) to[out=0, in=-135] (5*1.41/2, 5*1.41/2);
	\draw[knot, overcross] (2,2) to[out=-90, in=45] (-5*1.41/2, -5*1.41/2);
	\draw[knot, overcross] (5*1.41/2, -5*1.41/2) to[out=135, in=-90] (-2,2);
	\draw[knot, overcross] (-2,2) to[out=90, in=180] (0,4);
	\draw[knot, overcross] (-5*1.41/2, 5*1.41/2) to[out=-45, in=180] (0, 1.75);
	\draw[knot, overcross] (0,4) to[out=0, in=90] (2,2);
	\draw[white, ultra thick] (0,0) circle (5cm);
	\draw[dotted] (0,0) circle (5cm);
	\node[anchor=north] at (0,-1.18) {\tiny$2$};
	\node[anchor=west] at (2, 1.7) {\tiny$1$};
	\node[anchor=east] at (-2, 1.7) {\tiny$3$};
}$};
		\node(**A) at (-3,-1.5) {$\tikz[baseline={([yshift=-.5ex]current bounding box.center)}, scale=.2]{
	\node at (0,-6) {$**A$};
	\draw[knot, overcross] (0,1.75) to[out=0, in=-135] (5*1.41/2, 5*1.41/2);
	\draw[knot, overcross] (2,2) to[out=-90, in=45] (-5*1.41/2, -5*1.41/2);
	\draw[knot, overcross] (5*1.41/2, -5*1.41/2) to[out=135, in=-90] (-2,2);
	\draw[knot, overcross] (-2,2) to[out=90, in=180] (0,4);
	\draw[knot, overcross] (-5*1.41/2, 5*1.41/2) to[out=-45, in=180] (0, 1.75);
	\draw[knot, overcross] (0,4) to[out=0, in=90] (2,2);
	\draw[white, ultra thick] (0,0) circle (5cm);
	\draw[dotted] (0,0) circle (5cm);
	\node[anchor=north] at (0,-1.18) {\tiny$2$};
	\node[anchor=west] at (2, 1.7) {\tiny$1$};
	\node[anchor=east] at (-2, 1.7) {\tiny$3$};
	\fill[white, very thick] (-2, 2.25) circle (.45cm);
        \draw[knot] (-2.375, 2.51) to[out=-30,in=-105] (-1.88, 2.685);
        \draw[knot] (-1.995, 1.8) to[out=90, in=155] (-1.58, 2.055);
}$};
			\node(*AA) at (-4.5, -4) {$\tikz[baseline={([yshift=-.5ex]current bounding box.center)}, scale=.2]{
	\node at (0,-6) {$*AA$};
	\draw[knot, overcross] (0,1.75) to[out=0, in=-135] (5*1.41/2, 5*1.41/2);
	\draw[knot, overcross] (2,2) to[out=-90, in=45] (-5*1.41/2, -5*1.41/2);
	\draw[knot, overcross] (5*1.41/2, -5*1.41/2) to[out=135, in=-90] (-2,2);
	\draw[knot, overcross] (-2,2) to[out=90, in=180] (0,4);
	\draw[knot, overcross] (-5*1.41/2, 5*1.41/2) to[out=-45, in=180] (0, 1.75);
	\draw[knot, overcross] (0,4) to[out=0, in=90] (2,2);
	\draw[white, ultra thick] (0,0) circle (5cm);
	\draw[dotted] (0,0) circle (5cm);
	\node[anchor=north] at (0,-1.18) {\tiny$2$};
	\node[anchor=west] at (2, 1.7) {\tiny$1$};
	\node[anchor=east] at (-2, 1.7) {\tiny$3$};
	\fill[white, very thick] (0,-1.15) circle (.45cm);
        \draw[knot] (-0.4,-1.38) to[out=32,in=148] (0.4,-1.38);
        \draw[knot] (-0.375,-0.875) to[out=-36,in=-144] (0.375,-0.875);
	\fill[white, very thick] (-2, 2.25) circle (.45cm);
        \draw[knot] (-2.375, 2.51) to[out=-30,in=-105] (-1.88, 2.685);
        \draw[knot] (-1.995, 1.8) to[out=90, in=155] (-1.58, 2.055);
}$};
			\node(*BA) at (-1.5,-4) {$\tikz[baseline={([yshift=-.5ex]current bounding box.center)}, scale=.2]{
	\node at (0,-6) {$*BA$};
	\draw[knot, overcross] (0,1.75) to[out=0, in=-135] (5*1.41/2, 5*1.41/2);
	\draw[knot, overcross] (2,2) to[out=-90, in=45] (-5*1.41/2, -5*1.41/2);
	\draw[knot, overcross] (5*1.41/2, -5*1.41/2) to[out=135, in=-90] (-2,2);
	\draw[knot, overcross] (-2,2) to[out=90, in=180] (0,4);
	\draw[knot, overcross] (-5*1.41/2, 5*1.41/2) to[out=-45, in=180] (0, 1.75);
	\draw[knot, overcross] (0,4) to[out=0, in=90] (2,2);
	\draw[white, ultra thick] (0,0) circle (5cm);
	\draw[dotted] (0,0) circle (5cm);
	\node[anchor=north] at (0,-1.18) {\tiny$2$};
	\node[anchor=west] at (2, 1.7) {\tiny$1$};
	\node[anchor=east] at (-2, 1.7) {\tiny$3$};
	\fill[white, very thick] (0,-1.15) circle (.45cm);
        \draw[knot] (-0.4,-1.38) to[out=30,in=-33] (-0.375,-0.875);
        \draw[knot] (0.4,-1.38) to[out=150,in=213] (0.375,-0.875);
	\fill[white, very thick] (-2, 2.25) circle (.45cm);
        \draw[knot] (-2.375, 2.51) to[out=-30,in=-105] (-1.88, 2.685);
        \draw[knot] (-1.995, 1.8) to[out=90, in=155] (-1.58, 2.055);
}$};
				\node(ABA) at (-3,-6.5) {$\tikz[baseline={([yshift=-.5ex]current bounding box.center)}, scale=.2]{
	\node at (0,-6) {$ABA$};
	\draw[knot, overcross] (0,1.75) to[out=0, in=-135] (5*1.41/2, 5*1.41/2);
	\draw[knot, overcross] (2,2) to[out=-90, in=45] (-5*1.41/2, -5*1.41/2);
	\draw[knot, overcross] (5*1.41/2, -5*1.41/2) to[out=135, in=-90] (-2,2);
	\draw[knot, overcross] (-2,2) to[out=90, in=180] (0,4);
	\draw[knot, overcross] (-5*1.41/2, 5*1.41/2) to[out=-45, in=180] (0, 1.75);
	\draw[knot, overcross] (0,4) to[out=0, in=90] (2,2);
	\draw[white, ultra thick] (0,0) circle (5cm);
	\draw[dotted] (0,0) circle (5cm);
	\node[anchor=north] at (0,-1.18) {\tiny$2$};
	\node[anchor=west] at (2, 1.7) {\tiny$1$};
	\node[anchor=east] at (-2, 1.7) {\tiny$3$};
	\fill[white, very thick] (2, 2.25) circle (.45cm);
        \draw[knot] (2.375, 2.51) to[out=210,in=-75] (1.88, 2.685);
        \draw[knot] (1.995, 1.8) to[out=90, in=25] (1.58, 2.055);
	\fill[white, very thick] (0,-1.15) circle (.45cm);
        \draw[knot] (-0.4,-1.38) to[out=30,in=-33] (-0.375,-0.875);
        \draw[knot] (0.4,-1.38) to[out=150,in=213] (0.375,-0.875);
	\fill[white, very thick] (-2, 2.25) circle (.45cm);
        \draw[knot] (-2.375, 2.51) to[out=-30,in=-105] (-1.88, 2.685);
        \draw[knot] (-1.995, 1.8) to[out=90, in=155] (-1.58, 2.055);
}$};
				\node(BBA) at (0,-6.5) {$\tikz[baseline={([yshift=-.5ex]current bounding box.center)}, scale=.2]{
	\node at (0,-6) {$BBA$};
	\draw[knot, overcross] (0,1.75) to[out=0, in=-135] (5*1.41/2, 5*1.41/2);
	\draw[knot, overcross] (2,2) to[out=-90, in=45] (-5*1.41/2, -5*1.41/2);
	\draw[knot, overcross] (5*1.41/2, -5*1.41/2) to[out=135, in=-90] (-2,2);
	\draw[knot, overcross] (-2,2) to[out=90, in=180] (0,4);
	\draw[knot, overcross] (-5*1.41/2, 5*1.41/2) to[out=-45, in=180] (0, 1.75);
	\draw[knot, overcross] (0,4) to[out=0, in=90] (2,2);
	\draw[white, ultra thick] (0,0) circle (5cm);
	\draw[dotted] (0,0) circle (5cm);
	\node[anchor=north] at (0,-1.18) {\tiny$2$};
	\node[anchor=west] at (2, 1.7) {\tiny$1$};
	\node[anchor=east] at (-2, 1.7) {\tiny$3$};
	\fill[white, very thick] (2, 2.25) circle (.45cm);
       \draw[knot] (2.375, 2.51) to[out=210,in=90] (1.995, 1.8);
        \draw[knot] (1.88, 2.685) to[out=-75,in=25] (1.58, 2.055);
	\fill[white, very thick] (0,-1.15) circle (.45cm);
        \draw[knot] (-0.4,-1.38) to[out=30,in=-33] (-0.375,-0.875);
        \draw[knot] (0.4,-1.38) to[out=150,in=213] (0.375,-0.875);
	\fill[white, very thick] (-2, 2.25) circle (.45cm);
        \draw[knot] (-2.375, 2.51) to[out=-30,in=-105] (-1.88, 2.685);
        \draw[knot] (-1.995, 1.8) to[out=90, in=155] (-1.58, 2.055);
}$};
		\node(**B) at (3,-1.5) {$\tikz[baseline={([yshift=-.5ex]current bounding box.center)}, scale=.2]{
	\node at (0,-6) {$**B$};
	\draw[knot, overcross] (0,1.75) to[out=0, in=-135] (5*1.41/2, 5*1.41/2);
	\draw[knot, overcross] (2,2) to[out=-90, in=45] (-5*1.41/2, -5*1.41/2);
	\draw[knot, overcross] (5*1.41/2, -5*1.41/2) to[out=135, in=-90] (-2,2);
	\draw[knot, overcross] (-2,2) to[out=90, in=180] (0,4);
	\draw[knot, overcross] (-5*1.41/2, 5*1.41/2) to[out=-45, in=180] (0, 1.75);
	\draw[knot, overcross] (0,4) to[out=0, in=90] (2,2);
	\draw[white, ultra thick] (0,0) circle (5cm);
	\draw[dotted] (0,0) circle (5cm);
	\node[anchor=north] at (0,-1.18) {\tiny$2$};
	\node[anchor=west] at (2, 1.7) {\tiny$1$};
	\node[anchor=east] at (-2, 1.7) {\tiny$3$};
	\fill[white, very thick] (-2, 2.25) circle (.45cm);
        \draw[knot] (-2.375, 2.51) to[out=-30,in=90] (-1.995, 1.8);
        \draw[knot] (-1.88, 2.685) to[out=-105,in=155] (-1.58, 2.055);
}$};
			\node(*AB) at (1.5,-4) {$\tikz[baseline={([yshift=-.5ex]current bounding box.center)}, scale=.2]{
	\node at (0,-6) {$*AB$};
	\draw[knot, overcross] (0,1.75) to[out=0, in=-135] (5*1.41/2, 5*1.41/2);
	\draw[knot, overcross] (2,2) to[out=-90, in=45] (-5*1.41/2, -5*1.41/2);
	\draw[knot, overcross] (5*1.41/2, -5*1.41/2) to[out=135, in=-90] (-2,2);
	\draw[knot, overcross] (-2,2) to[out=90, in=180] (0,4);
	\draw[knot, overcross] (-5*1.41/2, 5*1.41/2) to[out=-45, in=180] (0, 1.75);
	\draw[knot, overcross] (0,4) to[out=0, in=90] (2,2);
	\draw[white, ultra thick] (0,0) circle (5cm);
	\draw[dotted] (0,0) circle (5cm);
	\node[anchor=north] at (0,-1.18) {\tiny$2$};
	\node[anchor=west] at (2, 1.7) {\tiny$1$};
	\node[anchor=east] at (-2, 1.7) {\tiny$3$};
	\fill[white, very thick] (0,-1.15) circle (.45cm);
        \draw[knot] (-0.4,-1.38) to[out=32,in=148] (0.4,-1.38);
        \draw[knot] (-0.375,-0.875) to[out=-36,in=-144] (0.375,-0.875);
	\fill[white, very thick] (-2, 2.25) circle (.45cm);
        \draw[knot] (-2.375, 2.51) to[out=-30,in=90] (-1.995, 1.8);
        \draw[knot] (-1.88, 2.685) to[out=-105,in=155] (-1.58, 2.055);
}$};
			\node(*BB) at (4.5, -4) {$\tikz[baseline={([yshift=-.5ex]current bounding box.center)}, scale=.2]{
	\node at (0,-6) {$*BB$};
	\draw[knot, overcross] (0,1.75) to[out=0, in=-135] (5*1.41/2, 5*1.41/2);
	\draw[knot, overcross] (2,2) to[out=-90, in=45] (-5*1.41/2, -5*1.41/2);
	\draw[knot, overcross] (5*1.41/2, -5*1.41/2) to[out=135, in=-90] (-2,2);
	\draw[knot, overcross] (-2,2) to[out=90, in=180] (0,4);
	\draw[knot, overcross] (-5*1.41/2, 5*1.41/2) to[out=-45, in=180] (0, 1.75);
	\draw[knot, overcross] (0,4) to[out=0, in=90] (2,2);
	\draw[white, ultra thick] (0,0) circle (5cm);
	\draw[dotted] (0,0) circle (5cm);
	\node[anchor=north] at (0,-1.18) {\tiny$2$};
	\node[anchor=west] at (2, 1.7) {\tiny$1$};
	\node[anchor=east] at (-2, 1.7) {\tiny$3$};
	\fill[white, very thick] (0,-1.15) circle (.45cm);
        \draw[knot] (-0.4,-1.38) to[out=30,in=-33] (-0.375,-0.875);
        \draw[knot] (0.4,-1.38) to[out=150,in=213] (0.375,-0.875);
	\fill[white, very thick] (-2, 2.25) circle (.45cm);
        \draw[knot] (-2.375, 2.51) to[out=-30,in=90] (-1.995, 1.8);
        \draw[knot] (-1.88, 2.685) to[out=-105,in=155] (-1.58, 2.055);
}$};
	\draw[->, thick] (***) -- (**A);
		\draw[->, thick] (**A) -- (*AA);
		\draw[->, thick] (**A) -- (*BA);
			\draw[->, thick] (*BA) -- (ABA);
			\draw[->, thick] (*BA) -- (BBA);
	\draw[->, thick] (***) -- (**B);
		\draw[->, thick] (**B) -- (*AB);
		\draw[->, thick] (**B) -- (*BB);
	\node(M1) at (-4.5, -8) {$-A^5$};
	\node(M2) at (-3, -8) {$A$};
	\node(M3) at (0, -8) {$A^{-1}$};
	\node(M4) at (1.5, -8) {$-A^{-3}$};
	\node(M5) at (4.5, -8) {$-A^{-5}$};
	\draw[dotted, ->] (*AA) -- (M1);
	\draw[dotted, ->] (ABA) -- (M2);
	\draw[dotted, ->] (BBA) -- (M3);
	\draw[dotted, ->] (*AB) -- (M4);
	\draw[dotted, ->] (*BB) -- (M5);
}
\]
\caption{An incomplete skein-resolution tree whose leaves are twisted unknots.}
\label{fig:mainskeinrestree}
\end{figure}
First, notice that when $D$ is non-local, elements of $\mathscr{U}(D)$ may be crossingless, contrary to the situation in $S^3$---this is because there is a nontrivial single-crossing nullhomologous link in $\RP^3$ (namely, $\Upsilon_1$), unlike in $S^3$. Secondly, compare this computation to our computation in Example \ref{ex:mainST}. Not only is there a bijection between the spanning trees there and the twisted unknots here, but the monomials they contribute are identical. This is known to be the case in $S^3$ \cite[Theorem 2]{MR2480298}; we will prove this result holds without modification for nullhomologous links in $\RP^3$ as well.
\end{example}

\begin{lemma}
\label{lem:champkof1}
Let $D\subset \RP^2$ be a diagram for a nullhomologous link in $\RP^3$. There is a bijective correspondence between the spanning trees of the two Tait graphs of $D$ and the twisted unknots of $D$. If a twisted unknot $U$ corresponds to a spanning tree $T$, then $\mu(T) = A^{\sigma(U)}(-A)^{3w(U)}$.
\end{lemma}

We will set some notation before proving the lemma. Assume that $D \subset \RP^2$ is a diagram for a non-local nullhomologous link in $\RP^3$. In addition, fix an ordering on the crossings of $D$. Let $\mathscr{T}(D)$ denote the set of spanning trees for the two Tait graphs of $D$. Now, we describe inverse maps
\[
\mathcal{U}: \mathscr{U}(D) \leftrightarrows \mathscr{T}(D): \mathcal{T}.
\]
These maps are described in \cite[\S 2]{MR2480298}, with the caveat that $\ell_\text{pro}$ and $d$ are treated equivalently. To a spanning tree $T$ of $D$, the twisted unknot $\mathcal{U}(T)$ is defined by leaving crossings corresponding to internally active and locally externally active edges alone, and smoothing internally inactive, projectively externally active, and externally inactive edges according to Table \ref{tab:Usmooth}. 

\begin{table}[ht]
\centering
\begin{tblr}{
  column{even} = {c},
  column{3} = {c},
  column{5} = {c},
  column{7} = {c},
  vline{2} = {-}{},
  hline{2} = {-}{},
}
State     & $D$ & $\ell_{\text{pro}}$ & $d$ & $\bar{D}$ & $\bar{\ell}_{\text{pro}}$ & $\bar{d}$ \\
Smoothing & $A$ & $B$                 & $B$ & $B$       & $A$                       & $A$       
\end{tblr}
\caption{Edge smoothings in the definition of $\mathcal{U}$.}
\label{tab:Usmooth}
\end{table}

\begin{example}
\label{ex:Umain}
Each spanning tree has a state which we use to obtain a twisted unknot of $D$. The process is depicted below for our favorite example.
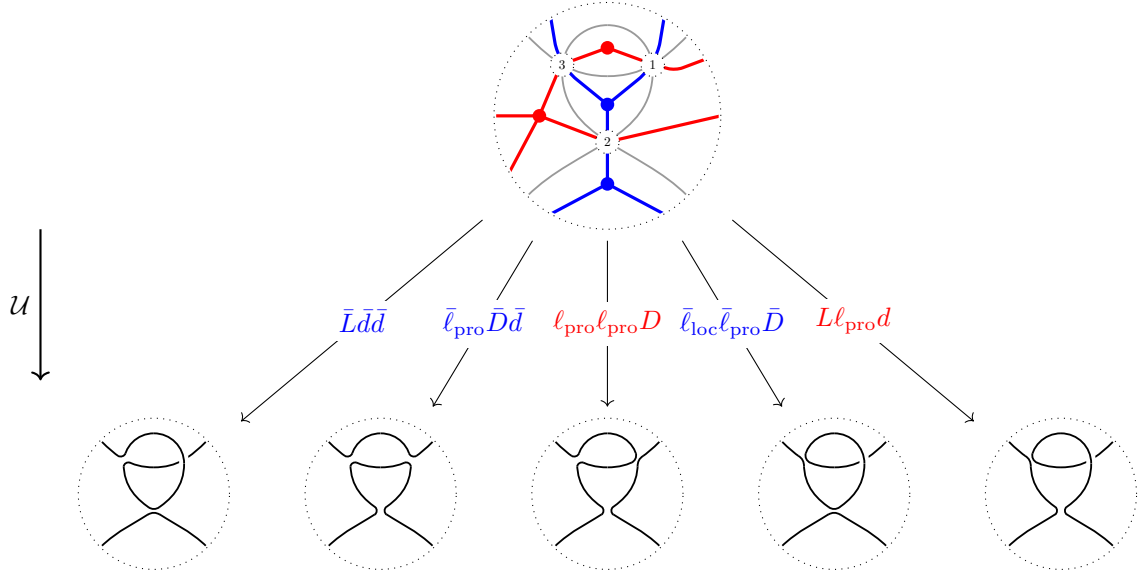
\[
\tikz[]{
\node(Start) at (0,0) {$
\tikz[baseline={([yshift=-.5ex]current bounding box.center)}, scale=.3]{
	\begin{scope}[white!60!black]
	\draw[knot] (0,1.75) to[out=0, in=-135] (5*1.41/2, 5*1.41/2);
	\draw[knot] (2,2) to[out=-90, in=45] (-5*1.41/2, -5*1.41/2);
	\draw[knot] (5*1.41/2, -5*1.41/2) to[out=135, in=-90] (-2,2);
	\draw[knot] (-2,2) to[out=90, in=180] (0,4);
	\draw[knot] (-5*1.41/2, 5*1.41/2) to[out=-45, in=180] (0, 1.75);
	\draw[knot] (0,4) to[out=0, in=90] (2,2);
	\end{scope}
	\begin{scope}[blue, very thick]
		\fill (0, 0.5) circle (.3cm);
		\fill (0, -3) circle (.3cm);
		\draw (0, 0.5) -- (0, -3);
		\draw[rounded corners = 3mm] (0, 0.5) -- (-2.15,2.25) -- (-5*0.5,5*0.866);
		\draw[rounded corners = 3mm] (0, 0.5) -- (2.15,2.25) -- (5*0.5,5*0.866);
		\draw (-5*0.5,-5*0.866) -- (0,-3);
		\draw (5*0.5,-5*0.866) -- (0,-3);
	\end{scope}
	\begin{scope}[red, very thick]
		\fill (0, 3) circle (.3cm);
		\fill (-3, 0) circle (.3cm);
		\draw[rounded corners = 1mm] (0,3) -- (-2.05,2.25) -- (-3, 0) -- (-5,0);
		\draw (-3, 0) -- (-5*0.866, -5*0.5);
		\draw[rounded corners = 1mm] (0,3) -- (2.05,2.25) -- (3, 2) -- (5*0.866, 5*0.5);
		\draw[rounded corners] (-3, 0) -- (0,-1.15) -- (5,0);
	\end{scope}
	\node[circle, draw, dotted, fill=white!20, scale=.5] at (2, 2.25) {$1$};
	\node[circle, draw, dotted, fill=white!20, scale=.5] at (0,-1.15) {$2$};
	\node[circle, draw, dotted, fill=white!20, scale=.5] at (-2, 2.25) {$3$};
	\draw[white, ultra thick] (0,0) circle (5cm);
	\draw[dotted] (0,0) circle (5cm);
}$};
\node(End1) at (-6, -5) {$\tikz[baseline={([yshift=-.5ex]current bounding box.center)}, scale=.2]{
	\draw[knot, overcross] (0,1.75) to[out=0, in=-135] (5*1.41/2, 5*1.41/2);
	\draw[knot, overcross] (2,2) to[out=-90, in=45] (-5*1.41/2, -5*1.41/2);
	\draw[knot, overcross] (5*1.41/2, -5*1.41/2) to[out=135, in=-90] (-2,2);
	\draw[knot, overcross] (-2,2) to[out=90, in=180] (0,4);
	\draw[knot, overcross] (-5*1.41/2, 5*1.41/2) to[out=-45, in=180] (0, 1.75);
	\draw[knot, overcross] (0,4) to[out=0, in=90] (2,2);
	\draw[white, ultra thick] (0,0) circle (5cm);
	\draw[dotted] (0,0) circle (5cm);
	\fill[white, very thick] (0,-1.15) circle (.45cm);
        \draw[knot] (-0.4,-1.38) to[out=32,in=148] (0.4,-1.38);
        \draw[knot] (-0.375,-0.875) to[out=-36,in=-144] (0.375,-0.875);
	\fill[white, very thick] (-2, 2.25) circle (.45cm);
        \draw[knot] (-2.375, 2.51) to[out=-30,in=-105] (-1.88, 2.685);
        \draw[knot] (-1.995, 1.8) to[out=90, in=155] (-1.58, 2.055);
}$};
\node(End2) at (-3, -5) {$\tikz[baseline={([yshift=-.5ex]current bounding box.center)}, scale=.2]{
	\draw[knot, overcross] (0,1.75) to[out=0, in=-135] (5*1.41/2, 5*1.41/2);
	\draw[knot, overcross] (2,2) to[out=-90, in=45] (-5*1.41/2, -5*1.41/2);
	\draw[knot, overcross] (5*1.41/2, -5*1.41/2) to[out=135, in=-90] (-2,2);
	\draw[knot, overcross] (-2,2) to[out=90, in=180] (0,4);
	\draw[knot, overcross] (-5*1.41/2, 5*1.41/2) to[out=-45, in=180] (0, 1.75);
	\draw[knot, overcross] (0,4) to[out=0, in=90] (2,2);
	\draw[white, ultra thick] (0,0) circle (5cm);
	\draw[dotted] (0,0) circle (5cm);
	\fill[white, very thick] (2, 2.25) circle (.45cm);
        \draw[knot] (2.375, 2.51) to[out=210,in=-75] (1.88, 2.685);
        \draw[knot] (1.995, 1.8) to[out=90, in=25] (1.58, 2.055);
	\fill[white, very thick] (0,-1.15) circle (.45cm);
        \draw[knot] (-0.4,-1.38) to[out=30,in=-33] (-0.375,-0.875);
        \draw[knot] (0.4,-1.38) to[out=150,in=213] (0.375,-0.875);
	\fill[white, very thick] (-2, 2.25) circle (.45cm);
        \draw[knot] (-2.375, 2.51) to[out=-30,in=-105] (-1.88, 2.685);
        \draw[knot] (-1.995, 1.8) to[out=90, in=155] (-1.58, 2.055);
}$};
\node(End3) at (0, -5) {$\tikz[baseline={([yshift=-.5ex]current bounding box.center)}, scale=.2]{
	\draw[knot, overcross] (0,1.75) to[out=0, in=-135] (5*1.41/2, 5*1.41/2);
	\draw[knot, overcross] (2,2) to[out=-90, in=45] (-5*1.41/2, -5*1.41/2);
	\draw[knot, overcross] (5*1.41/2, -5*1.41/2) to[out=135, in=-90] (-2,2);
	\draw[knot, overcross] (-2,2) to[out=90, in=180] (0,4);
	\draw[knot, overcross] (-5*1.41/2, 5*1.41/2) to[out=-45, in=180] (0, 1.75);
	\draw[knot, overcross] (0,4) to[out=0, in=90] (2,2);
	\draw[white, ultra thick] (0,0) circle (5cm);
	\draw[dotted] (0,0) circle (5cm);
	\fill[white, very thick] (2, 2.25) circle (.45cm);
       \draw[knot] (2.375, 2.51) to[out=210,in=90] (1.995, 1.8);
        \draw[knot] (1.88, 2.685) to[out=-75,in=25] (1.58, 2.055);
	\fill[white, very thick] (0,-1.15) circle (.45cm);
        \draw[knot] (-0.4,-1.38) to[out=30,in=-33] (-0.375,-0.875);
        \draw[knot] (0.4,-1.38) to[out=150,in=213] (0.375,-0.875);
	\fill[white, very thick] (-2, 2.25) circle (.45cm);
        \draw[knot] (-2.375, 2.51) to[out=-30,in=-105] (-1.88, 2.685);
        \draw[knot] (-1.995, 1.8) to[out=90, in=155] (-1.58, 2.055);
}$};
\node(End4) at (3, -5) {$\tikz[baseline={([yshift=-.5ex]current bounding box.center)}, scale=.2]{
	\draw[knot, overcross] (0,1.75) to[out=0, in=-135] (5*1.41/2, 5*1.41/2);
	\draw[knot, overcross] (2,2) to[out=-90, in=45] (-5*1.41/2, -5*1.41/2);
	\draw[knot, overcross] (5*1.41/2, -5*1.41/2) to[out=135, in=-90] (-2,2);
	\draw[knot, overcross] (-2,2) to[out=90, in=180] (0,4);
	\draw[knot, overcross] (-5*1.41/2, 5*1.41/2) to[out=-45, in=180] (0, 1.75);
	\draw[knot, overcross] (0,4) to[out=0, in=90] (2,2);
	\draw[white, ultra thick] (0,0) circle (5cm);
	\draw[dotted] (0,0) circle (5cm);
	\fill[white, very thick] (0,-1.15) circle (.45cm);
        \draw[knot] (-0.4,-1.38) to[out=32,in=148] (0.4,-1.38);
        \draw[knot] (-0.375,-0.875) to[out=-36,in=-144] (0.375,-0.875);
	\fill[white, very thick] (-2, 2.25) circle (.45cm);
        \draw[knot] (-2.375, 2.51) to[out=-30,in=90] (-1.995, 1.8);
        \draw[knot] (-1.88, 2.685) to[out=-105,in=155] (-1.58, 2.055);
}$};
\node(End5) at (6, -5) {$\tikz[baseline={([yshift=-.5ex]current bounding box.center)}, scale=.2]{
	\draw[knot, overcross] (0,1.75) to[out=0, in=-135] (5*1.41/2, 5*1.41/2);
	\draw[knot, overcross] (2,2) to[out=-90, in=45] (-5*1.41/2, -5*1.41/2);
	\draw[knot, overcross] (5*1.41/2, -5*1.41/2) to[out=135, in=-90] (-2,2);
	\draw[knot, overcross] (-2,2) to[out=90, in=180] (0,4);
	\draw[knot, overcross] (-5*1.41/2, 5*1.41/2) to[out=-45, in=180] (0, 1.75);
	\draw[knot, overcross] (0,4) to[out=0, in=90] (2,2);
	\draw[white, ultra thick] (0,0) circle (5cm);
	\draw[dotted] (0,0) circle (5cm);
	\fill[white, very thick] (0,-1.15) circle (.45cm);
        \draw[knot] (-0.4,-1.38) to[out=30,in=-33] (-0.375,-0.875);
        \draw[knot] (0.4,-1.38) to[out=150,in=213] (0.375,-0.875);
	\fill[white, very thick] (-2, 2.25) circle (.45cm);
        \draw[knot] (-2.375, 2.51) to[out=-30,in=90] (-1.995, 1.8);
        \draw[knot] (-1.88, 2.685) to[out=-105,in=155] (-1.58, 2.055);
}$};
\draw[->] (Start) to node[pos=.5, fill=white!20]{$\blue{\bar{L}\bar{d}\bar{d}}$} (End1);
\draw[->] (Start) to node[pos=.5, fill=white!20]{$\blue{\bar{\ell}_{\text{pro}}\bar{D}\bar{d}}$}  (End2);
\draw[->] (Start) to node[pos=.5, fill=white!20]{$\red{\ell_{\text{pro}} \ell_{\text{pro}}  D}$}  (End3);
\draw[->] (Start) to node[pos=.5, fill=white!20]{$\blue{\bar{\ell}_{\text{loc}} \bar{\ell}_{\text{pro}} \bar{D}}$}  (End4);
\draw[->] (Start) to node[pos=.5, fill=white!20]{$\red{L \ell_{\text{pro}} d}$} (End5);
\draw[->, thick] (-7.5,-1.5) -- node[pos=0.5, left]{$\mathcal{U}$} (-7.5,-3.5);
}
\]
\end{example}

Next, to a twisted unknot $U$ of $D$, we define $\mathcal{T}(U)$ as follows. Consider both of the checkerboard colorings of $U$. The checkerboard colorings correspond to the Tait graphs of $D$. For both colorings, consider the subgraph obtained by including an edge $e \in G$ if 
\begin{enumerate}
\item the vertices of $e$ are contained in the same shaded component, or
\item resolving the crossing corresponding to $e$ according to Figure \ref{fig:Tsmooth}  separates two shaded regions (here, the sign describes whether the crossing corresponds to a positive or a negative twist).
\end{enumerate}
By Lemma \ref{lem:projectivecompletions}, exactly one of these subgraphs is a tree (and the other is a non-tree corresponding to a quasi-tree); denote it by $\mathcal{T}(U)$.

\begin{figure}[ht]
\begin{tikzpicture}[thick, scale=0.8]
\node at (1.5,1.5) {$+$};
\draw (0,0) -- (1,1);
\draw (0,1) -- (.3,.7);
\draw (1,0) -- (.7,.3);
\draw[->] (.5,-.2) -- (.5, -.8);
\begin{scope}[yshift=-2cm]
\draw (0,0) to[out=30, in=-30] (0,1);
\draw (1,0) to[out=150, in=-150] (1,1);
\end{scope}
\begin{scope}[xshift=2cm]
\draw (0,1) -- (1,0);
\draw (0,0) -- (.3,.3);
\draw (.7, .7) -- (1,1);
\draw[->] (.5,-.2) -- (.5, -.8);
\begin{scope}[yshift=-2cm]
\draw (0,0) to[out=60, in=120] (1,0);
\draw (0,1) to[out=-60, in=-120] (1,1);
\end{scope}
\end{scope}
\begin{scope}[xshift=5cm]
\node at (1.5,1.5) {$-$};
\draw (0,0) -- (1,1);
\draw (0,1) -- (.3,.7);
\draw (1,0) -- (.7,.3);
\draw[->] (.5,-.2) -- (.5, -.8);
\begin{scope}[yshift=-2cm]
\draw (0,0) to[out=60, in=120] (1,0);
\draw (0,1) to[out=-60, in=-120] (1,1);
\end{scope}
\begin{scope}[xshift=2cm]
\draw (0,1) -- (1,0);
\draw (0,0) -- (.3,.3);
\draw (.7, .7) -- (1,1);
\draw[->] (.5,-.2) -- (.5, -.8);
\begin{scope}[yshift=-2cm]
\draw (0,0) to[out=30, in=-30] (0,1);
\draw (1,0) to[out=150, in=-150] (1,1);
\end{scope}
\end{scope}
\end{scope}
\end{tikzpicture}
\caption{Crossing smoothings in the definition of $\mathcal{T}$.}
\label{fig:Tsmooth}
\end{figure}
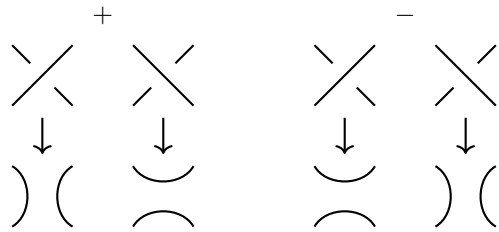

\begin{example}
\label{ex:Tmain}
Each twisted unknot $U$ produces two spanning subgraphs. One of them is a tree, and we denote it by $\mathcal{T}(U)$. We provide the spanning subgraphs below for our main example.
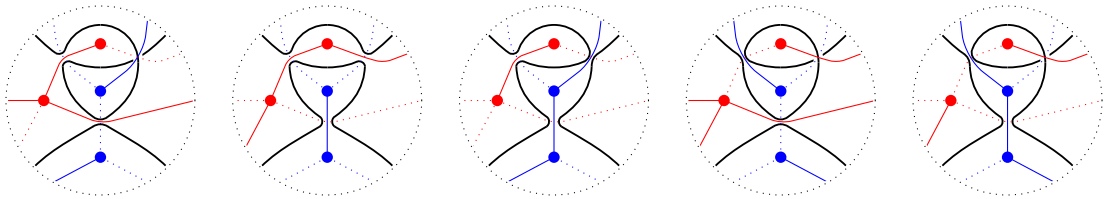
\[
\tikz[]{
\node(End1) at (-6, -5) {$\tikz[baseline={([yshift=-.5ex]current bounding box.center)}, scale=.25]{
	\draw[knot, overcross] (0,1.75) to[out=0, in=-135] (5*1.41/2, 5*1.41/2);
	\draw[knot, overcross] (2,2) to[out=-90, in=45] (-5*1.41/2, -5*1.41/2);
	\draw[knot, overcross] (5*1.41/2, -5*1.41/2) to[out=135, in=-90] (-2,2);
	\draw[knot, overcross] (-2,2) to[out=90, in=180] (0,4);
	\draw[knot, overcross] (-5*1.41/2, 5*1.41/2) to[out=-45, in=180] (0, 1.75);
	\draw[knot, overcross] (0,4) to[out=0, in=90] (2,2);
	\draw[white, ultra thick] (0,0) circle (5cm);
	\draw[dotted] (0,0) circle (5cm);
	\fill[white, very thick] (0,-1.15) circle (.45cm);
        \draw[knot] (-0.4,-1.38) to[out=32,in=148] (0.4,-1.38);
        \draw[knot] (-0.375,-0.875) to[out=-36,in=-144] (0.375,-0.875);
	\fill[white, very thick] (-2, 2.25) circle (.45cm);
        \draw[knot] (-2.375, 2.51) to[out=-30,in=-105] (-1.88, 2.685);
        \draw[knot] (-1.995, 1.8) to[out=90, in=155] (-1.58, 2.055);
	\begin{scope}[blue]
		\fill (0, 0.5) circle (.3cm);
		\fill (0, -3) circle (.3cm);
		\draw[dotted] (0, 0.5) -- (0, -3);
		\draw[dotted, rounded corners = 3mm] (0, 0.5) -- (-2.25,2.25) -- (-5*0.5,5*0.866);
		\draw[rounded corners = 3mm] (0, 0.5) -- (2.25,2.25) -- (5*0.5,5*0.866);
		\draw (-5*0.5,-5*0.866) -- (0,-3);
		\draw[dotted] (5*0.5,-5*0.866) -- (0,-3);
	\end{scope}
	\begin{scope}[red]
		\fill (0, 3) circle (.3cm);
		\fill (-3, 0) circle (.3cm);
		\draw[rounded corners = 1mm] (0,3) -- (-2.05,2.25) -- (-3, 0) -- (-5,0);
		\draw[dotted] (-3, 0) -- (-5*0.866, -5*0.5);
		\draw[dotted, rounded corners = 1mm] (0,3) -- (2.05,2.25) -- (3, 2) -- (5*0.866, 5*0.5);
		\draw[rounded corners] (-3, 0) -- (0,-1.225) -- (5,0);
	\end{scope}
	\draw[white, ultra thick] (0,0) circle (5cm);
	\draw[dotted] (0,0) circle (5cm);
}$};
\node(End2) at (-3, -5) {$\tikz[baseline={([yshift=-.5ex]current bounding box.center)}, scale=.25]{
	\draw[knot, overcross] (0,1.75) to[out=0, in=-135] (5*1.41/2, 5*1.41/2);
	\draw[knot, overcross] (2,2) to[out=-90, in=45] (-5*1.41/2, -5*1.41/2);
	\draw[knot, overcross] (5*1.41/2, -5*1.41/2) to[out=135, in=-90] (-2,2);
	\draw[knot, overcross] (-2,2) to[out=90, in=180] (0,4);
	\draw[knot, overcross] (-5*1.41/2, 5*1.41/2) to[out=-45, in=180] (0, 1.75);
	\draw[knot, overcross] (0,4) to[out=0, in=90] (2,2);
	\draw[white, ultra thick] (0,0) circle (5cm);
	\draw[dotted] (0,0) circle (5cm);
	\fill[white, very thick] (2, 2.25) circle (.45cm);
        \draw[knot] (2.375, 2.51) to[out=210,in=-75] (1.88, 2.685);
        \draw[knot] (1.995, 1.8) to[out=90, in=25] (1.58, 2.055);
	\fill[white, very thick] (0,-1.15) circle (.45cm);
        \draw[knot] (-0.4,-1.38) to[out=30,in=-33] (-0.375,-0.875);
        \draw[knot] (0.4,-1.38) to[out=150,in=213] (0.375,-0.875);
	\fill[white, very thick] (-2, 2.25) circle (.45cm);
        \draw[knot] (-2.375, 2.51) to[out=-30,in=-105] (-1.88, 2.685);
        \draw[knot] (-1.995, 1.8) to[out=90, in=155] (-1.58, 2.055);
	\begin{scope}[blue]
		\fill (0, 0.5) circle (.3cm);
		\fill (0, -3) circle (.3cm);
		\draw (0, 0.5) -- (0, -3);
		\draw[dotted, rounded corners = 3mm] (0, 0.5) -- (-2.25,2.25) -- (-5*0.5,5*0.866);
		\draw[dotted, rounded corners = 3mm] (0, 0.5) -- (2.25,2.25) -- (5*0.5,5*0.866);
		\draw[dotted] (-5*0.5,-5*0.866) -- (0,-3);
		\draw[dotted] (5*0.5,-5*0.866) -- (0,-3);
	\end{scope}
	\begin{scope}[red]
		\fill (0, 3) circle (.3cm);
		\fill (-3, 0) circle (.3cm);
		\draw[dotted, rounded corners = 1mm] (-3, 0) -- (-5,0);
		\draw[rounded corners = 1mm] (0,3) -- (-2.05,2.25) -- (-3, 0);
		\draw (-3, 0) -- (-5*0.866, -5*0.5);
		\draw[rounded corners = 1mm] (0,3) -- (2.05,2.25) -- (3, 2) -- (5*0.866, 5*0.5);
		\draw[dotted, rounded corners] (-3, 0) -- (0,-1.225) -- (5,0);
	\end{scope}
	\draw[white, ultra thick] (0,0) circle (5cm);
	\draw[dotted] (0,0) circle (5cm);
}$};
\node(End3) at (0, -5) {$\tikz[baseline={([yshift=-.5ex]current bounding box.center)}, scale=.25]{
	\draw[knot, overcross] (0,1.75) to[out=0, in=-135] (5*1.41/2, 5*1.41/2);
	\draw[knot, overcross] (2,2) to[out=-90, in=45] (-5*1.41/2, -5*1.41/2);
	\draw[knot, overcross] (5*1.41/2, -5*1.41/2) to[out=135, in=-90] (-2,2);
	\draw[knot, overcross] (-2,2) to[out=90, in=180] (0,4);
	\draw[knot, overcross] (-5*1.41/2, 5*1.41/2) to[out=-45, in=180] (0, 1.75);
	\draw[knot, overcross] (0,4) to[out=0, in=90] (2,2);
	\draw[white, ultra thick] (0,0) circle (5cm);
	\draw[dotted] (0,0) circle (5cm);
	\fill[white, very thick] (2, 2.25) circle (.45cm);
       \draw[knot] (2.375, 2.51) to[out=210,in=90] (1.995, 1.8);
        \draw[knot] (1.88, 2.685) to[out=-75,in=25] (1.58, 2.055);
	\fill[white, very thick] (0,-1.15) circle (.45cm);
        \draw[knot] (-0.4,-1.38) to[out=30,in=-33] (-0.375,-0.875);
        \draw[knot] (0.4,-1.38) to[out=150,in=213] (0.375,-0.875);
	\fill[white, very thick] (-2, 2.25) circle (.45cm);
        \draw[knot] (-2.375, 2.51) to[out=-30,in=-105] (-1.88, 2.685);
        \draw[knot] (-1.995, 1.8) to[out=90, in=155] (-1.58, 2.055);
	\begin{scope}[blue]
		\fill (0, 0.5) circle (.3cm);
		\fill (0, -3) circle (.3cm);
		\draw (0, 0.5) -- (0, -3);
		\draw[dotted, rounded corners = 3mm] (0, 0.5) -- (-2.25,2.25) -- (-5*0.5,5*0.866);
		\draw[rounded corners = 3mm] (0, 0.5) -- (2.25,2.25) -- (5*0.5,5*0.866);
		\draw (-5*0.5,-5*0.866) -- (0,-3);
		\draw[dotted] (5*0.5,-5*0.866) -- (0,-3);
	\end{scope}
	\begin{scope}[red]
		\fill (0, 3) circle (.3cm);
		\fill (-3, 0) circle (.3cm);
		\draw[dotted, rounded corners = 1mm] (-3, 0) -- (-5,0);
		\draw[rounded corners = 1mm] (0,3) -- (-2.05,2.25) -- (-3, 0);
		\draw[dotted] (-3, 0) -- (-5*0.866, -5*0.5);
		\draw[dotted, rounded corners = 1mm] (0,3) -- (2.05,2.25) -- (3, 2) -- (5*0.866, 5*0.5);
		\draw[dotted, rounded corners] (-3, 0) -- (0,-1.225) -- (5,0);
	\end{scope}
	\draw[white, ultra thick] (0,0) circle (5cm);
	\draw[dotted] (0,0) circle (5cm);
}$};
\node(End4) at (3, -5) {$\tikz[baseline={([yshift=-.5ex]current bounding box.center)}, scale=.25]{
	\draw[knot, overcross] (0,1.75) to[out=0, in=-135] (5*1.41/2, 5*1.41/2);
	\draw[knot, overcross] (2,2) to[out=-90, in=45] (-5*1.41/2, -5*1.41/2);
	\draw[knot, overcross] (5*1.41/2, -5*1.41/2) to[out=135, in=-90] (-2,2);
	\draw[knot, overcross] (-2,2) to[out=90, in=180] (0,4);
	\draw[knot, overcross] (-5*1.41/2, 5*1.41/2) to[out=-45, in=180] (0, 1.75);
	\draw[knot, overcross] (0,4) to[out=0, in=90] (2,2);
	\draw[white, ultra thick] (0,0) circle (5cm);
	\draw[dotted] (0,0) circle (5cm);
	\fill[white, very thick] (0,-1.15) circle (.45cm);
        \draw[knot] (-0.4,-1.38) to[out=32,in=148] (0.4,-1.38);
        \draw[knot] (-0.375,-0.875) to[out=-36,in=-144] (0.375,-0.875);
	\fill[white, very thick] (-2, 2.25) circle (.45cm);
        \draw[knot] (-2.375, 2.51) to[out=-30,in=90] (-1.995, 1.8);
        \draw[knot] (-1.88, 2.685) to[out=-105,in=155] (-1.58, 2.055);
	\begin{scope}[blue]
		\fill (0, 0.5) circle (.3cm);
		\fill (0, -3) circle (.3cm);
		\draw[dotted] (0, 0.5) -- (0, -3);
		\draw[rounded corners = 3mm] (0, 0.5) -- (-2.25,2.25) -- (-5*0.5,5*0.866);
		\draw[dotted, rounded corners = 3mm] (0, 0.5) -- (2.25,2.25) -- (5*0.5,5*0.866);
		\draw[dotted] (-5*0.5,-5*0.866) -- (0,-3);
		\draw (5*0.5,-5*0.866) -- (0,-3);
	\end{scope}
	\begin{scope}[red]
		\fill (0, 3) circle (.3cm);
		\fill (-3, 0) circle (.3cm);
		\draw[rounded corners = 1mm] (-3, 0) -- (-5,0);
		\draw[dotted, rounded corners = 1mm] (0,3) -- (-2.05,2.25) -- (-3, 0);
		\draw (-3, 0) -- (-5*0.866, -5*0.5);
		\draw[rounded corners = 1mm] (0,3) -- (2.05,2.25) -- (3, 2) -- (5*0.866, 5*0.5);
		\draw[rounded corners] (-3, 0) -- (0,-1.225) -- (5,0);
	\end{scope}
	\draw[white, ultra thick] (0,0) circle (5cm);
	\draw[dotted] (0,0) circle (5cm);
}$};
\node(End5) at (6, -5) {$\tikz[baseline={([yshift=-.5ex]current bounding box.center)}, scale=.25]{
	\draw[knot, overcross] (0,1.75) to[out=0, in=-135] (5*1.41/2, 5*1.41/2);
	\draw[knot, overcross] (2,2) to[out=-90, in=45] (-5*1.41/2, -5*1.41/2);
	\draw[knot, overcross] (5*1.41/2, -5*1.41/2) to[out=135, in=-90] (-2,2);
	\draw[knot, overcross] (-2,2) to[out=90, in=180] (0,4);
	\draw[knot, overcross] (-5*1.41/2, 5*1.41/2) to[out=-45, in=180] (0, 1.75);
	\draw[knot, overcross] (0,4) to[out=0, in=90] (2,2);
	\draw[white, ultra thick] (0,0) circle (5cm);
	\draw[dotted] (0,0) circle (5cm);
	\fill[white, very thick] (0,-1.15) circle (.45cm);
        \draw[knot] (-0.4,-1.38) to[out=30,in=-33] (-0.375,-0.875);
        \draw[knot] (0.4,-1.38) to[out=150,in=213] (0.375,-0.875);
	\fill[white, very thick] (-2, 2.25) circle (.45cm);
        \draw[knot] (-2.375, 2.51) to[out=-30,in=90] (-1.995, 1.8);
        \draw[knot] (-1.88, 2.685) to[out=-105,in=155] (-1.58, 2.055);
	\begin{scope}[blue]
		\fill (0, 0.5) circle (.3cm);
		\fill (0, -3) circle (.3cm);
		\draw (0, 0.5) -- (0, -3);
		\draw[rounded corners = 3mm] (0, 0.5) -- (-2.25,2.25) -- (-5*0.5,5*0.866);
		\draw[dotted, rounded corners = 3mm] (0, 0.5) -- (2.25,2.25) -- (5*0.5,5*0.866);
		\draw[dotted] (-5*0.5,-5*0.866) -- (0,-3);
		\draw (5*0.5,-5*0.866) -- (0,-3);
	\end{scope}
	\begin{scope}[red]
		\fill (0, 3) circle (.3cm);
		\fill (-3, 0) circle (.3cm);
		\draw[dotted, rounded corners = 1mm] (-3, 0) -- (-5,0);
		\draw[dotted, rounded corners = 1mm] (0,3) -- (-2.05,2.25) -- (-3, 0);
		\draw (-3, 0) -- (-5*0.866, -5*0.5);
		\draw[rounded corners = 1mm] (0,3) -- (2.05,2.25) -- (3, 2) -- (5*0.866, 5*0.5);
		\draw[dotted, rounded corners] (-3, 0) -- (0,-1.225) -- (5,0);
	\end{scope}
	\draw[white, ultra thick] (0,0) circle (5cm);
	\draw[dotted] (0,0) circle (5cm);
}$};
}
\]
\end{example}

\begin{proof}[Proof of Lemma \ref{lem:champkof1}]
It is straightforward to verify that $\mathcal{T}(\mathcal{U}(T)) = T$ and $\mathcal{U}(\mathcal{T}(U)) = U$. We will argue that both maps are well-defined.

To see that $\mathcal{U}(T) \in \mathscr{U}(D)$ for any spanning tree $T$ associated to $D$, start by taking a small regular neighborhood of $T$. It is a crossingless unknot by definition. For each $e \not\in T$, add the corresponding crossing if $e$ is locally externally active. Since these edges are active, the added crossing is the only crossing in $\cyc(T, e)$. Moreover, since $[\cyc(T,e)] = 0$, each added crossing can be undone by a Reidemeister I move, so the result is a twisted unknot. Next, add in crossings for each live edge $e \in T$. For each, since $e\in T$ is live, it is ordered first in $\cut(T,e)$. This means that any edge $f \in \cut(T,e)$ must be externally inactive, since $e \in \cyc(T,f)$. Thus, the crossing corresponding to $e \in T$ amounts to a nugatory crossing between two twisted unknots, so the result is still a twisted unknot.

Next, we verify that $\mathcal{T}(U) \in \mathscr{T}(D)$ for each twisted unknot $U$ associated to $D$. Picking a checkerboard coloring of $U$, each shaded region corresponds to at least one vertex of a Tait graph for $D$. Start by adding edges between vertices lying within the same shaded component. Notice that local cycles cannot be introduced at this stage, otherwise $U$ has more than one component, which is a contradiction. If a homologically essential cycle is introduced, then the process is restarted with the complementary checkerboard coloring. Now, we add edges corresponding to crossings between shaded regions. Since $U$ is a twisted unknot, there is a sequence of edges corresponding to a sequence of Reidemeister I moves taking $U$ to the crossingless unknot. Since $\RP^2 - D$ is colored in a checkerboard fashion, the rules of Figure \ref{fig:Tsmooth} guarantee that the result is a spanning tree.

We can rephrase the last statement as $\mu(T) = A^{\sigma(\mathcal{U}(T))} (-A)^{3 w(\mathcal{U}(T))}$. This follows from our definitions, comparing the preceding discussion with Table \ref{tab:state_monomialsA}. If an edge $e$ is in state $L$ or $\ell_{\text{loc}}$ relative $T$, the corresponding crossing is unresolved in $\mathcal{U}(T)$, and the edge determines the sign of its crossing in $\mathcal{U}(T)$. Otherwise, the edge is resolved according to Table \ref{tab:Usmooth}.
\end{proof}

In addition, Champanerkar and Kofman provide a partial ordering on $\mathscr{T}(D)$. Given a nullhomologous diagram $D\subset \RP^2$ with $n$ ordered crossings, every twisted unknot in $\mathscr{U}(D)$ is described by a partial smoothing, \emph{i.e}, a word in $\{*, A, B\}^n$ with letters in order of the crossings. For example, the leaves of the tree in Figure \ref{fig:mainskeinrestree} are twisted unknots pictured with their partial smoothings.

\begin{definition}
For any spanning trees $T, T' \in \mathscr{T}(D)$, let $(x_1, \ldots, x_n)$ and $(y_1, \ldots, y_n)$ be the corresponding partial smoothings of $D$. We declare that $T > T'$ if 
\begin{enumerate}
\item for each $i$, $y_i = A$ implies that $x_i = A$ or $x_i = *$, and
\item there exists an $i$ such that $x_i = A$ and $y_i = B$.
\end{enumerate}
The transitive closure of this relation is a partial order. Let $\mathscr{P}(D)$ be the poset $\mathscr{T}(D)$ with this partial order. We confuse spanning trees for their corresponding partial smoothing and write $(x_1,\ldots, x_n) > (y_1, \ldots, y_n)$.
\end{definition}

\begin{example}
Returning to Figure \ref{fig:mainskeinrestree}, we have two sequences in $\mathscr{P}(D)$:
\[
*AA > ABA > BBA > *BB
\qquad \text{and} \qquad
*AA > * AB > *BB.
\]
\end{example}

\subsection{Spanning tree complex}

For a signed graph $G$, let $e_\pm(G)$ denote the number of $\pm$-signed edges of $G$. Assume that $D \subset \RP^2$ is a non-local diagram for a non-split nullhomologous link $L\subset \RP^3$. We will adapt the notation of alternating diagrams to arbitrary diagrams and write $G_\pm$ to denote the Tait graph of $D$ which has $e_\pm(G) \ge e_\mp(G)$. If both Tait graphs have an equal number of $(+)$-signed and $(-)$-signed edges, then one can assign the notation arbitrarily.

For any spanning tree $T \subset G_\pm$, define
\[
\begin{aligned}
	u(T) & = \pm(\# L - \# \ell_{\mathrm{loc}} - \# \bar{L} + \# \bar{\ell}_{\mathrm{loc}}) \\
		&= \mp w(\mathcal{U}(T))\\
	v(T) &= e_\pm(T)\\
		&= \begin{cases} \# L + \# D ~ \text{if} ~ + \\ \# \bar{L} + \# \bar{D} ~ \text{if} ~ - \end{cases}
\end{aligned}
\]
and let
\begin{align*}
\mathcal{C}^{u, v}_\pm(D) = \F \langle T \subset G_\pm : u(T) = u, v(T) = v \rangle, \quad \text{setting} \quad
\mathcal{C}_\pm(D) = \bigoplus_{u,v} \mathcal{C}^{u,v}_\pm(D).
\end{align*}
The \emph{spanning tree complex} of $D$ is defined as $\mathcal{C}(D) = \mathcal{C}_+ (D)\oplus \mathcal{C}_-(D)$.

\begin{example}
Assume that $D$ is the diagram with crossings ordered as in Example \ref{ex:mainST}. Then $\mathcal{C}(D) = \mathcal{C}_+(D) \oplus \mathcal{C}_-(D)$ where
\[
\mathcal{C}_+ = \mathcal{C}_+^{0,1}\langle \ell_{\mathrm{pro}}\ell_{\mathrm{pro}}D \rangle \oplus \mathcal{C}_+^{1,1} \langle L \ell_{\mathrm{pro}} d \rangle
\qquad \text{and} \qquad
\mathcal{C}_-(D) = \mathcal{C}_-^{1,1} \langle \bar{L} \bar{d} \bar{d} \rangle \oplus \mathcal{C}_-^{0,1} \langle \bar{\ell}_{\mathrm{pro}} \bar{D} \bar{d} \rangle \oplus \mathcal{C}_-^{-1,1} \langle \bar{\ell}_{\mathrm{loc}} \bar{\ell}_{\mathrm{pro}} \bar{D} \rangle.
\]
Notice that the $v$-grading is constant throughout. In general, when $D$ is a non-split, alternating diagram, then $\mathcal{C}_+$ and $\mathcal{C}_-$ will each have constant, though potentially distinct, $v$-gradings. For example, for the reduced alternating diagram of $\Upsilon_2$ in Figure \ref{fig:upsilon}, 
\[
\mathcal{C}_+(D) = \mathcal{C}_+^{0,1} \langle \ell_{\mathrm{pro}} D \rangle \oplus \mathcal{C}_+^{1,1} \langle L d \rangle
\qquad \text{and} \qquad
\mathcal{C}_-(D) = \mathcal{C}_-^{0,0} \langle \bar{\ell}_{\mathrm{pro}} \bar{\ell}_{\mathrm{pro}} \rangle.
\]
\end{example}

\begin{proposition}[C.f. Proposition 2 of \cite{MR2480298}]
\label{prop:gradedeuler}
Assume that $D\subset \RP^2$ is a non-local diagram for a nullhomologous link $L$ in $\RP^3$. For any differential $\partial = \partial_+ \oplus \partial_-$ with $\partial_\pm: \mathcal{C}_\pm^{u,v} \to \mathcal{C}_\pm^{u\mp1,v\mp1}$, the Jones polynomial can be expressed as the graded Euler characteristic of $\{\mathcal{C}(D), \partial\}$. Precisely,
\[
 (-1)^{w(D)} J_L(t) = t^{\frac{3w(D) + k_+}{4}} \sum_{u, v} (-1)^u t^{u-v} \dim \mathcal{C}_+^{u,v} + t^{\frac{3w(D) - k_-}{4}} \sum_{u, v} (-1)^u t^{v - u} \dim \mathcal{C}_-^{u,v}
\]
where $k_\pm = \pm(e_+(G_\pm) - e_-(G_\pm)) + 2 (v(G_\pm) - 1)$. Moreover, if $D$ is a reduced, alternating diagram of $L$, then
\[
J_L^\pm(t) = t^{\frac{3w(D) \pm k_{\pm}}{4}} \sum_{u,v} (-1)^u t^{\pm(u-v)} \dim \mathcal{C}_\pm^{u,v}.
\]
\end{proposition}

\begin{proof}
An arbitrary spanning tree $T \subset G_\pm$ has activity word $L^a D^b \ell_{\mathrm{loc}}^c \ell_{\mathrm{pro}}^d d^e \bar{L}^f \bar{D}^g \bar{\ell}_{\mathrm{loc}}^h \bar{\ell}_{\mathrm{pro}}^i \bar{d}^j$. Using Table \ref{tab:state_monomialsA}, the corresponding weight is
\[
\mu(T) = (-1)^{a + c + f + h} A^{-3a + b + 3c - d - e + 3f - g - 3h + i + j}.
\]
Note that
\begin{itemize}
\item $a + b + f + g = v(G_\pm) - 1$,
\item $c + d + e + h + i + j = e(G_\pm) - v(G_\pm) + 1$,
\item $a + b + c + d + e = e_+(G_\pm)$, and
\item $f + g  + h + i + j = e_-(G_\pm)$.
\end{itemize}
Then:
\begin{itemize}
\item If $T \subset G_+$, then $u = a - c - f + h$ and $v = a + b$, so 
\[
\mu(T) = (-1)^u A^{-4(u-v) - k_+}.
\]
\item If $T \subset G_-$, then $u = -a + c + f - h$ and $v = f + g$, so 
\[
\mu(T) = (-1)^u A^{4(u - v) + k_-}.
\]
\end{itemize}
By Proposition \ref{prop:KBgraphpolynomial}, this means that
\begin{align*}
\langle D \rangle & = \sum_{T \subset G_+} \mu (T) + \sum_{T \subset G_-} \mu(T) \\
	&= A^{-k_+} \sum_{u, v} (-1)^u A^{-4(u - v)} \dim C_+^{u,v} + A^{k_-} \sum_{u, v} (-1)^u A^{4(u - v)} \dim C_-^{u,v}.
\end{align*}
The first statement follows from the normalization in Equation (\ref{eq:jpdefinition}). For the second statement, note that $k_\pm = e(G_\pm) + 2(v(G_\pm) - 1)$ if $D$ is an alternating diagram. Since $G_+$ and $G_-$ are dual to each other, we have that $e(G_+) = e(G_-)$ and $v(G_-) = f(G_+)$. Since $v-e+f = 1$ implies that $v + e + f = 2e+1$, we have
\[
\frac{3w(D) + k_+}{4} - \frac{3w(D) - k_-}{4} = \frac{2(v+e+f) - 4}{4} = \frac{1}{2}(2e + 1) - 1.
\]
The result follows since the difference is a half-integer.
\end{proof}

Next, we must verify that there exists a differential $\partial = \partial_+ \oplus \partial_-$ on $\mathcal{C}(D)$ where $\partial_\pm$ has bidegree $(\mp 1, \mp 1)$. We show that $\widetilde{C}(D)$, the reduced Khovanov complex described in \S \ref{ss:khbackground}, is a deformation retract of $\mathcal{C}(D)$.

We can import the argument of Champanerkar-Kofman \cite[\S 3]{MR2480298} with practically no changes. To summarize, let $U \in \mathscr{U}(D)$ be a twisted unknot. Then $\mathcal{C}(U)$ is contractible with homology $\F_{(0,-1)}$ (again, we choose this normalization for consistency). Let $Z_U \in \widetilde{C}(U)$ denote the single generator of this homology, called the \emph{fundamental cycle}. Let $\bigcirc$ denote the round local unknot. The map $f_U: \widetilde{C}(\bigcirc) \to \widetilde{C}(U)$ defined by $Z_\bigcirc \mapsto Z_U$ is a grading-preserving chain homotopy \cite[Lemma 4]{MR2480298}.

The fundamental cycle $Z_U$ is a linear combination of pure tensors, called \emph{enhanced states}, with bigrading
\begin{equation}
\label{eq:enhancedstatesgrading}
i = \frac{w-\sigma}{2} \qquad \text{and} \qquad j = i + w - \tau
\end{equation}
where $w$ is the writhe of $U$, $\sigma$ is the value defined in \S \ref{ss:twistedunknots}, and $\tau$ is the difference $\#(\text{circles labeled with X}) - \#(\text{circles labeled with 1})$. Let $\iota: \widetilde{C}(U) \to \widetilde{C}(D)$ be the inclusion of enhanced states of $U$ into the enhanced states of $D$ with grading shift
\begin{equation}\label{eq:gradingshift}
i + \frac{w(D) - w(U) - \sigma(U)}{2} \qquad \text{and} \qquad j + \frac{3(w(D) - w(U)) - \sigma(U)}{2}.
\end{equation}

\begin{example}
\label{ex:mainplusinclusion}
Let's illustrate Champanerkar-Kofman's process for our running example. In Figure \ref{fig:cubeandinclusion}, we have provided each enhanced state appearing in $\widetilde{C}(D)$ together with its bigrading. The fundamental cycle of each twisted unknot for $D$ with this particular crossing order is naturally identified with a linear combination of enhanced states in $\widetilde{C}(D)$.
\begin{figure}[ht]
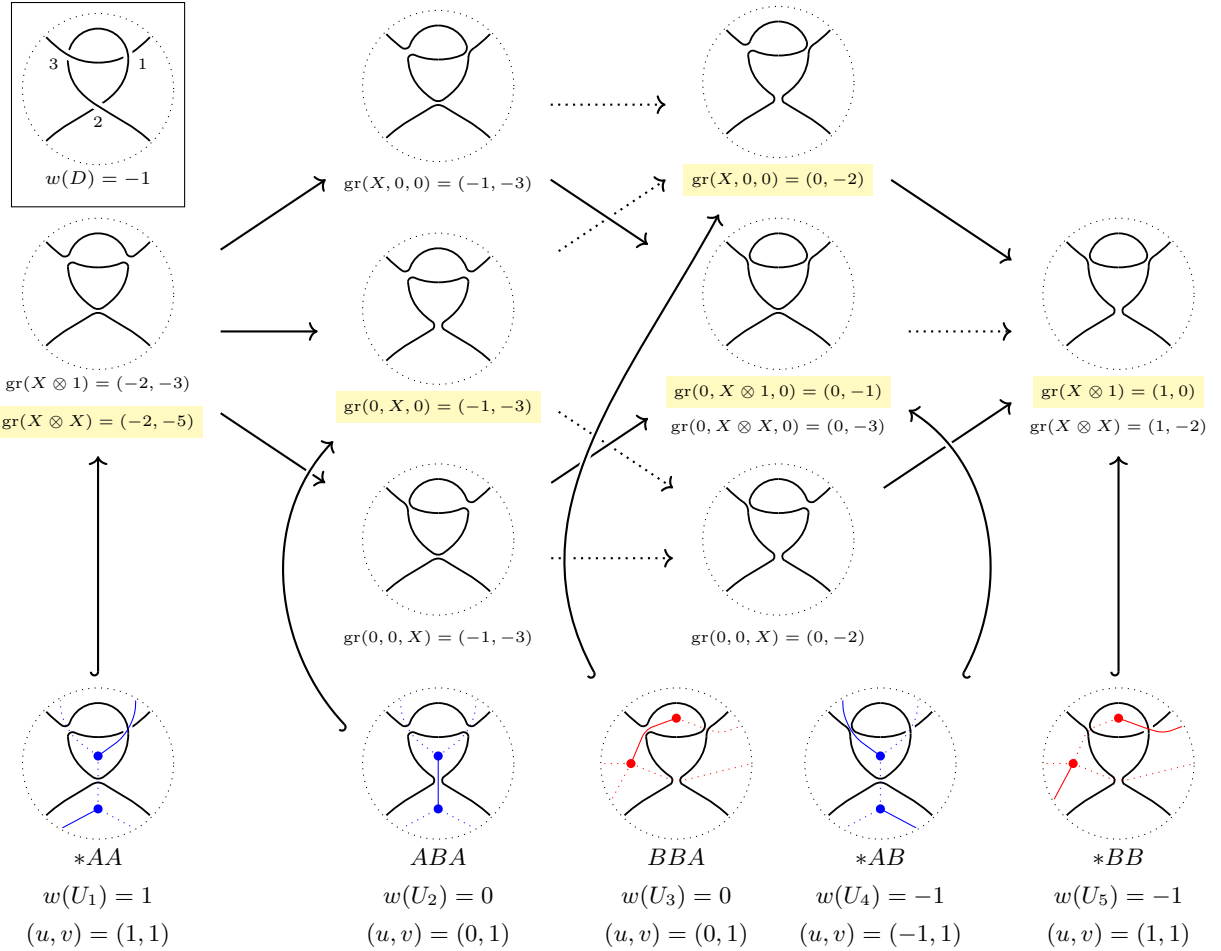

\tikz[yscale=1, xscale=0.9]{
\node[draw](***) at (0,3) {$\tikz[baseline={([yshift=-.5ex]current bounding box.center)}, scale=.2]{
	\node at (0,-6) {\scriptsize$w(D)=-1$};
	\draw[knot, overcross] (0,1.75) to[out=0, in=-135] (5*1.41/2, 5*1.41/2);
	\draw[knot, overcross] (2,2) to[out=-90, in=45] (-5*1.41/2, -5*1.41/2);
	\draw[knot, overcross] (5*1.41/2, -5*1.41/2) to[out=135, in=-90] (-2,2);
	\draw[knot, overcross] (-2,2) to[out=90, in=180] (0,4);
	\draw[knot, overcross] (-5*1.41/2, 5*1.41/2) to[out=-45, in=180] (0, 1.75);
	\draw[knot, overcross] (0,4) to[out=0, in=90] (2,2);
	\node[anchor=north] at (0,-1.18) {\tiny$2$};
	\node[anchor=west] at (2, 1.7) {\tiny$1$};
	\node[anchor=east] at (-2, 1.7) {\tiny$3$};
	\draw[white, ultra thick] (0,0) circle (5cm);
	\draw[dotted] (0,0) circle (5cm);
}$};
	\node(AAA) at (0,0) {$\tikz[baseline={([yshift=-.5ex]current bounding box.center)}, scale=.2]{
	\node at (0,-7.5) {\tiny$\begin{gathered} \mathrm{gr}(X \otimes 1) = (-2,-3) \\ \highlight{\mathrm{gr}(X \otimes X) = (-2,-5)} \end{gathered}$};
	\draw[knot, overcross] (0,1.75) to[out=0, in=-135] (5*1.41/2, 5*1.41/2);
	\draw[knot, overcross] (2,2) to[out=-90, in=45] (-5*1.41/2, -5*1.41/2);
	\draw[knot, overcross] (5*1.41/2, -5*1.41/2) to[out=135, in=-90] (-2,2);
	\draw[knot, overcross] (-2,2) to[out=90, in=180] (0,4);
	\draw[knot, overcross] (-5*1.41/2, 5*1.41/2) to[out=-45, in=180] (0, 1.75);
	\draw[knot, overcross] (0,4) to[out=0, in=90] (2,2);
	\draw[white, ultra thick] (0,0) circle (5cm);
	\draw[dotted] (0,0) circle (5cm);
	\fill[white, very thick] (2, 2.25) circle (.45cm);
        \draw[knot] (2.375, 2.51) to[out=210,in=-75] (1.88, 2.685);
        \draw[knot] (1.995, 1.8) to[out=90, in=25] (1.58, 2.055);
	\fill[white, very thick] (0,-1.15) circle (.45cm);
        \draw[knot] (-0.4,-1.38) to[out=32,in=148] (0.4,-1.38);
        \draw[knot] (-0.375,-0.875) to[out=-36,in=-144] (0.375,-0.875);
	\fill[white, very thick] (-2, 2.25) circle (.45cm);
        \draw[knot] (-2.375, 2.51) to[out=-30,in=-105] (-1.88, 2.685);
        \draw[knot] (-1.995, 1.8) to[out=90, in=155] (-1.58, 2.055);
}$};
\node(*AA) at (0, -6.5) {$\tikz[baseline={([yshift=-.5ex]current bounding box.center)}, scale=.2]{
	\node at (0,-9) {\small$\begin{gathered} *AA \\ w(U_1) = 1 \\ (u,v) = (1,1) \end{gathered}$};
	\draw[knot, overcross] (0,1.75) to[out=0, in=-135] (5*1.41/2, 5*1.41/2);
	\draw[knot, overcross] (2,2) to[out=-90, in=45] (-5*1.41/2, -5*1.41/2);
	\draw[knot, overcross] (5*1.41/2, -5*1.41/2) to[out=135, in=-90] (-2,2);
	\draw[knot, overcross] (-2,2) to[out=90, in=180] (0,4);
	\draw[knot, overcross] (-5*1.41/2, 5*1.41/2) to[out=-45, in=180] (0, 1.75);
	\draw[knot, overcross] (0,4) to[out=0, in=90] (2,2);
	\draw[white, ultra thick] (0,0) circle (5cm);
	\draw[dotted] (0,0) circle (5cm);
	\fill[white, very thick] (0,-1.15) circle (.45cm);
        \draw[knot] (-0.4,-1.38) to[out=32,in=148] (0.4,-1.38);
        \draw[knot] (-0.375,-0.875) to[out=-36,in=-144] (0.375,-0.875);
	\fill[white, very thick] (-2, 2.25) circle (.45cm);
        \draw[knot] (-2.375, 2.51) to[out=-30,in=-105] (-1.88, 2.685);
        \draw[knot] (-1.995, 1.8) to[out=90, in=155] (-1.58, 2.055);
	\begin{scope}[blue]
		\fill (0, 0.5) circle (.3cm);
		\fill (0, -3) circle (.3cm);
		\draw[dotted] (0, 0.5) -- (0, -3);
		\draw[dotted, rounded corners = 3mm] (0, 0.5) -- (-2.4,2.25) -- (-5*0.5,5*0.866);
		\draw[rounded corners = 3mm] (0, 0.5) -- (2.4,2.25) -- (5*0.5,5*0.866);
		\draw (-5*0.5,-5*0.866) -- (0,-3);
		\draw[dotted] (5*0.5,-5*0.866) -- (0,-3);
	\end{scope}
	\draw[white, ultra thick] (0,0) circle (5cm);
	\draw[dotted] (0,0) circle (5cm);
}$};
	\node(BAA) at (5,3) {$\tikz[baseline={([yshift=-.5ex]current bounding box.center)}, scale=.2]{
	\node at (0,-6.5) {\tiny$\mathrm{gr}(X,0,0) = (-1,-3)$};
	\draw[knot, overcross] (0,1.75) to[out=0, in=-135] (5*1.41/2, 5*1.41/2);
	\draw[knot, overcross] (2,2) to[out=-90, in=45] (-5*1.41/2, -5*1.41/2);
	\draw[knot, overcross] (5*1.41/2, -5*1.41/2) to[out=135, in=-90] (-2,2);
	\draw[knot, overcross] (-2,2) to[out=90, in=180] (0,4);
	\draw[knot, overcross] (-5*1.41/2, 5*1.41/2) to[out=-45, in=180] (0, 1.75);
	\draw[knot, overcross] (0,4) to[out=0, in=90] (2,2);
	\draw[white, ultra thick] (0,0) circle (5cm);
	\draw[dotted] (0,0) circle (5cm);
	\fill[white, very thick] (2, 2.25) circle (.45cm);
       \draw[knot] (2.375, 2.51) to[out=210,in=90] (1.995, 1.8);
        \draw[knot] (1.88, 2.685) to[out=-75,in=25] (1.58, 2.055);
	\fill[white, very thick] (0,-1.15) circle (.45cm);
        \draw[knot] (-0.4,-1.38) to[out=32,in=148] (0.4,-1.38);
        \draw[knot] (-0.375,-0.875) to[out=-36,in=-144] (0.375,-0.875);
	\fill[white, very thick] (-2, 2.25) circle (.45cm);
        \draw[knot] (-2.375, 2.51) to[out=-30,in=-105] (-1.88, 2.685);
        \draw[knot] (-1.995, 1.8) to[out=90, in=155] (-1.58, 2.055);
}$};
	\node(ABA) at (5,0) {$\tikz[baseline={([yshift=-.5ex]current bounding box.center)}, scale=.2]{
	\node at (0,-6.5) {\tiny$\highlight{\mathrm{gr}(0,X,0) = (-1,-3)}$};
	\draw[knot, overcross] (0,1.75) to[out=0, in=-135] (5*1.41/2, 5*1.41/2);
	\draw[knot, overcross] (2,2) to[out=-90, in=45] (-5*1.41/2, -5*1.41/2);
	\draw[knot, overcross] (5*1.41/2, -5*1.41/2) to[out=135, in=-90] (-2,2);
	\draw[knot, overcross] (-2,2) to[out=90, in=180] (0,4);
	\draw[knot, overcross] (-5*1.41/2, 5*1.41/2) to[out=-45, in=180] (0, 1.75);
	\draw[knot, overcross] (0,4) to[out=0, in=90] (2,2);
	\draw[white, ultra thick] (0,0) circle (5cm);
	\draw[dotted] (0,0) circle (5cm);
	\fill[white, very thick] (2, 2.25) circle (.45cm);
        \draw[knot] (2.375, 2.51) to[out=210,in=-75] (1.88, 2.685);
        \draw[knot] (1.995, 1.8) to[out=90, in=25] (1.58, 2.055);
	\fill[white, very thick] (0,-1.15) circle (.45cm);
        \draw[knot] (-0.4,-1.38) to[out=30,in=-33] (-0.375,-0.875);
        \draw[knot] (0.4,-1.38) to[out=150,in=213] (0.375,-0.875);
	\fill[white, very thick] (-2, 2.25) circle (.45cm);
        \draw[knot] (-2.375, 2.51) to[out=-30,in=-105] (-1.88, 2.685);
        \draw[knot] (-1.995, 1.8) to[out=90, in=155] (-1.58, 2.055);
}$};
\node(*ABA) at (5,-6.5) {$\tikz[baseline={([yshift=-.5ex]current bounding box.center)}, scale=.2]{
	\node at (0,-9) {\small$\begin{gathered} ABA \\ w(U_2) = 0 \\ (u,v) = (0,1) \end{gathered}$};
	\draw[knot, overcross] (0,1.75) to[out=0, in=-135] (5*1.41/2, 5*1.41/2);
	\draw[knot, overcross] (2,2) to[out=-90, in=45] (-5*1.41/2, -5*1.41/2);
	\draw[knot, overcross] (5*1.41/2, -5*1.41/2) to[out=135, in=-90] (-2,2);
	\draw[knot, overcross] (-2,2) to[out=90, in=180] (0,4);
	\draw[knot, overcross] (-5*1.41/2, 5*1.41/2) to[out=-45, in=180] (0, 1.75);
	\draw[knot, overcross] (0,4) to[out=0, in=90] (2,2);
	\draw[white, ultra thick] (0,0) circle (5cm);
	\draw[dotted] (0,0) circle (5cm);
	\fill[white, very thick] (2, 2.25) circle (.45cm);
        \draw[knot] (2.375, 2.51) to[out=210,in=-75] (1.88, 2.685);
        \draw[knot] (1.995, 1.8) to[out=90, in=25] (1.58, 2.055);
	\fill[white, very thick] (0,-1.15) circle (.45cm);
        \draw[knot] (-0.4,-1.38) to[out=30,in=-33] (-0.375,-0.875);
        \draw[knot] (0.4,-1.38) to[out=150,in=213] (0.375,-0.875);
	\fill[white, very thick] (-2, 2.25) circle (.45cm);
        \draw[knot] (-2.375, 2.51) to[out=-30,in=-105] (-1.88, 2.685);
        \draw[knot] (-1.995, 1.8) to[out=90, in=155] (-1.58, 2.055);
	\begin{scope}[blue]
		\fill (0, 0.5) circle (.3cm);
		\fill (0, -3) circle (.3cm);
		\draw (0, 0.5) -- (0, -3);
		\draw[dotted, rounded corners = 3mm] (0, 0.5) -- (-2.4,2.25) -- (-5*0.5,5*0.866);
		\draw[dotted, rounded corners = 3mm] (0, 0.5) -- (2.4,2.25) -- (5*0.5,5*0.866);
		\draw[dotted] (-5*0.5,-5*0.866) -- (0,-3);
		\draw[dotted] (5*0.5,-5*0.866) -- (0,-3);
	\end{scope}
	\draw[white, ultra thick] (0,0) circle (5cm);
	\draw[dotted] (0,0) circle (5cm);
}$};
	\node(AAB) at (5,-3) {$\tikz[baseline={([yshift=-.5ex]current bounding box.center)}, scale=.2]{
	\node at (0,-6.5) {\tiny$\mathrm{gr}(0,0,X) = (-1,-3)$};
	\draw[knot, overcross] (0,1.75) to[out=0, in=-135] (5*1.41/2, 5*1.41/2);
	\draw[knot, overcross] (2,2) to[out=-90, in=45] (-5*1.41/2, -5*1.41/2);
	\draw[knot, overcross] (5*1.41/2, -5*1.41/2) to[out=135, in=-90] (-2,2);
	\draw[knot, overcross] (-2,2) to[out=90, in=180] (0,4);
	\draw[knot, overcross] (-5*1.41/2, 5*1.41/2) to[out=-45, in=180] (0, 1.75);
	\draw[knot, overcross] (0,4) to[out=0, in=90] (2,2);
	\draw[white, ultra thick] (0,0) circle (5cm);
	\draw[dotted] (0,0) circle (5cm);
	\fill[white, very thick] (2, 2.25) circle (.45cm);
        \draw[knot] (2.375, 2.51) to[out=210,in=-75] (1.88, 2.685);
        \draw[knot] (1.995, 1.8) to[out=90, in=25] (1.58, 2.055);
	\fill[white, very thick] (0,-1.15) circle (.45cm);
        \draw[knot] (-0.4,-1.38) to[out=32,in=148] (0.4,-1.38);
        \draw[knot] (-0.375,-0.875) to[out=-36,in=-144] (0.375,-0.875);
	\fill[white, very thick] (-2, 2.25) circle (.45cm);
        \draw[knot] (-2.375, 2.51) to[out=-30,in=90] (-1.995, 1.8);
        \draw[knot] (-1.88, 2.685) to[out=-105,in=155] (-1.58, 2.055);
}$};
	\node(BBA) at (10,3) {$\tikz[baseline={([yshift=-.5ex]current bounding box.center)}, scale=.2]{
	\node at (0,-6.5) {\tiny$\highlight{\mathrm{gr}(X,0,0) = (0,-2)}$};
	\draw[knot, overcross] (0,1.75) to[out=0, in=-135] (5*1.41/2, 5*1.41/2);
	\draw[knot, overcross] (2,2) to[out=-90, in=45] (-5*1.41/2, -5*1.41/2);
	\draw[knot, overcross] (5*1.41/2, -5*1.41/2) to[out=135, in=-90] (-2,2);
	\draw[knot, overcross] (-2,2) to[out=90, in=180] (0,4);
	\draw[knot, overcross] (-5*1.41/2, 5*1.41/2) to[out=-45, in=180] (0, 1.75);
	\draw[knot, overcross] (0,4) to[out=0, in=90] (2,2);
	\draw[white, ultra thick] (0,0) circle (5cm);
	\draw[dotted] (0,0) circle (5cm);
	\fill[white, very thick] (2, 2.25) circle (.45cm);
       \draw[knot] (2.375, 2.51) to[out=210,in=90] (1.995, 1.8);
        \draw[knot] (1.88, 2.685) to[out=-75,in=25] (1.58, 2.055);
	\fill[white, very thick] (0,-1.15) circle (.45cm);
        \draw[knot] (-0.4,-1.38) to[out=30,in=-33] (-0.375,-0.875);
        \draw[knot] (0.4,-1.38) to[out=150,in=213] (0.375,-0.875);
	\fill[white, very thick] (-2, 2.25) circle (.45cm);
        \draw[knot] (-2.375, 2.51) to[out=-30,in=-105] (-1.88, 2.685);
        \draw[knot] (-1.995, 1.8) to[out=90, in=155] (-1.58, 2.055);
}$};
	\node(BAB) at (10,0) {$\tikz[baseline={([yshift=-.5ex]current bounding box.center)}, scale=.2]{
	\node at (0,-7.5) {\tiny$\begin{gathered} \highlight{\mathrm{gr}(0,X \otimes 1,0) = (0,-1)} \\ \mathrm{gr}(0,X \otimes X,0) = (0,-3) \end{gathered}$};
	\draw[knot, overcross] (0,1.75) to[out=0, in=-135] (5*1.41/2, 5*1.41/2);
	\draw[knot, overcross] (2,2) to[out=-90, in=45] (-5*1.41/2, -5*1.41/2);
	\draw[knot, overcross] (5*1.41/2, -5*1.41/2) to[out=135, in=-90] (-2,2);
	\draw[knot, overcross] (-2,2) to[out=90, in=180] (0,4);
	\draw[knot, overcross] (-5*1.41/2, 5*1.41/2) to[out=-45, in=180] (0, 1.75);
	\draw[knot, overcross] (0,4) to[out=0, in=90] (2,2);
	\draw[white, ultra thick] (0,0) circle (5cm);
	\draw[dotted] (0,0) circle (5cm);
	\fill[white, very thick] (2, 2.25) circle (.45cm);
       \draw[knot] (2.375, 2.51) to[out=210,in=90] (1.995, 1.8);
        \draw[knot] (1.88, 2.685) to[out=-75,in=25] (1.58, 2.055);
	\fill[white, very thick] (0,-1.15) circle (.45cm);
        \draw[knot] (-0.4,-1.38) to[out=32,in=148] (0.4,-1.38);
        \draw[knot] (-0.375,-0.875) to[out=-36,in=-144] (0.375,-0.875);
	\fill[white, very thick] (-2, 2.25) circle (.45cm);
        \draw[knot] (-2.375, 2.51) to[out=-30,in=90] (-1.995, 1.8);
        \draw[knot] (-1.88, 2.685) to[out=-105,in=155] (-1.58, 2.055);
}$};
\node(*BBA) at (8.5,-6.5) {$\tikz[baseline={([yshift=-.5ex]current bounding box.center)}, scale=.2]{
	\node at (0,-9) {\small$\begin{gathered} BBA \\ w(U_3) = 0 \\ (u,v) = (0,1)  \end{gathered}$};
	\draw[knot, overcross] (0,1.75) to[out=0, in=-135] (5*1.41/2, 5*1.41/2);
	\draw[knot, overcross] (2,2) to[out=-90, in=45] (-5*1.41/2, -5*1.41/2);
	\draw[knot, overcross] (5*1.41/2, -5*1.41/2) to[out=135, in=-90] (-2,2);
	\draw[knot, overcross] (-2,2) to[out=90, in=180] (0,4);
	\draw[knot, overcross] (-5*1.41/2, 5*1.41/2) to[out=-45, in=180] (0, 1.75);
	\draw[knot, overcross] (0,4) to[out=0, in=90] (2,2);
	\draw[white, ultra thick] (0,0) circle (5cm);
	\draw[dotted] (0,0) circle (5cm);
	\fill[white, very thick] (2, 2.25) circle (.45cm);
       \draw[knot] (2.375, 2.51) to[out=210,in=90] (1.995, 1.8);
        \draw[knot] (1.88, 2.685) to[out=-75,in=25] (1.58, 2.055);
	\fill[white, very thick] (0,-1.15) circle (.45cm);
        \draw[knot] (-0.4,-1.38) to[out=30,in=-33] (-0.375,-0.875);
        \draw[knot] (0.4,-1.38) to[out=150,in=213] (0.375,-0.875);
	\fill[white, very thick] (-2, 2.25) circle (.45cm);
        \draw[knot] (-2.375, 2.51) to[out=-30,in=-105] (-1.88, 2.685);
        \draw[knot] (-1.995, 1.8) to[out=90, in=155] (-1.58, 2.055);
	\begin{scope}[red]
		\fill (0, 3) circle (.3cm);
		\fill (-3, 0) circle (.3cm);
		\draw[dotted, rounded corners = 1mm] (-3, 0) -- (-5,0);
		\draw[rounded corners = 1mm] (0,3) -- (-2.05,2.25) -- (-3, 0);
		\draw[dotted] (-3, 0) -- (-5*0.866, -5*0.5);
		\draw[dotted, rounded corners = 1mm] (0,3) -- (2.05,2.25) -- (3, 2) -- (5*0.866, 5*0.5);
		\draw[dotted, rounded corners] (-3, 0) -- (0,-1.225) -- (5,0);
	\end{scope}
	\draw[white, ultra thick] (0,0) circle (5cm);
	\draw[dotted] (0,0) circle (5cm);
}$};
\node(*AB) at (11.5,-6.5) {$\tikz[baseline={([yshift=-.5ex]current bounding box.center)}, scale=.2]{
	\node at (0,-9) {\small$\begin{gathered} *AB \\ w(U_4) = -1 \\ (u,v) = (-1,1) \end{gathered}$};
	\draw[knot, overcross] (0,1.75) to[out=0, in=-135] (5*1.41/2, 5*1.41/2);
	\draw[knot, overcross] (2,2) to[out=-90, in=45] (-5*1.41/2, -5*1.41/2);
	\draw[knot, overcross] (5*1.41/2, -5*1.41/2) to[out=135, in=-90] (-2,2);
	\draw[knot, overcross] (-2,2) to[out=90, in=180] (0,4);
	\draw[knot, overcross] (-5*1.41/2, 5*1.41/2) to[out=-45, in=180] (0, 1.75);
	\draw[knot, overcross] (0,4) to[out=0, in=90] (2,2);
	\draw[white, ultra thick] (0,0) circle (5cm);
	\draw[dotted] (0,0) circle (5cm);
	\fill[white, very thick] (0,-1.15) circle (.45cm);
        \draw[knot] (-0.4,-1.38) to[out=32,in=148] (0.4,-1.38);
        \draw[knot] (-0.375,-0.875) to[out=-36,in=-144] (0.375,-0.875);
	\fill[white, very thick] (-2, 2.25) circle (.45cm);
        \draw[knot] (-2.375, 2.51) to[out=-30,in=90] (-1.995, 1.8);
        \draw[knot] (-1.88, 2.685) to[out=-105,in=155] (-1.58, 2.055);
	\begin{scope}[blue]
		\fill (0, 0.5) circle (.3cm);
		\fill (0, -3) circle (.3cm);
		\draw[dotted] (0, 0.5) -- (0, -3);
		\draw[rounded corners = 3mm] (0, 0.5) -- (-2.4,2.25) -- (-5*0.5,5*0.866);
		\draw[dotted, rounded corners = 3mm] (0, 0.5) -- (2.4,2.25) -- (5*0.5,5*0.866);
		\draw[dotted] (-5*0.5,-5*0.866) -- (0,-3);
		\draw (5*0.5,-5*0.866) -- (0,-3);
	\end{scope}
	\draw[white, ultra thick] (0,0) circle (5cm);
	\draw[dotted] (0,0) circle (5cm);
}$};
	\node(ABB) at (10,-3) {$\tikz[baseline={([yshift=-.5ex]current bounding box.center)}, scale=.2]{
	\node at (0,-6.5) {\tiny$\mathrm{gr}(0,0,X) = (0,-2)$};
	\draw[knot, overcross] (0,1.75) to[out=0, in=-135] (5*1.41/2, 5*1.41/2);
	\draw[knot, overcross] (2,2) to[out=-90, in=45] (-5*1.41/2, -5*1.41/2);
	\draw[knot, overcross] (5*1.41/2, -5*1.41/2) to[out=135, in=-90] (-2,2);
	\draw[knot, overcross] (-2,2) to[out=90, in=180] (0,4);
	\draw[knot, overcross] (-5*1.41/2, 5*1.41/2) to[out=-45, in=180] (0, 1.75);
	\draw[knot, overcross] (0,4) to[out=0, in=90] (2,2);
	\draw[white, ultra thick] (0,0) circle (5cm);
	\draw[dotted] (0,0) circle (5cm);
	\fill[white, very thick] (2, 2.25) circle (.45cm);
        \draw[knot] (2.375, 2.51) to[out=210,in=-75] (1.88, 2.685);
        \draw[knot] (1.995, 1.8) to[out=90, in=25] (1.58, 2.055);
	\fill[white, very thick] (0,-1.15) circle (.45cm);
        \draw[knot] (-0.4,-1.38) to[out=30,in=-33] (-0.375,-0.875);
        \draw[knot] (0.4,-1.38) to[out=150,in=213] (0.375,-0.875);
	\fill[white, very thick] (-2, 2.25) circle (.45cm);
        \draw[knot] (-2.375, 2.51) to[out=-30,in=90] (-1.995, 1.8);
        \draw[knot] (-1.88, 2.685) to[out=-105,in=155] (-1.58, 2.055);
}$};
	\node(BBB) at (15,0) {$\tikz[baseline={([yshift=-.5ex]current bounding box.center)}, scale=.2]{
	\node at (0,-7.5) {\tiny$\begin{gathered} \highlight{\mathrm{gr}(X \otimes 1) = (1,0)} \\ \mathrm{gr}(X \otimes X) = (1,-2) \end{gathered}$};
	\draw[knot, overcross] (0,1.75) to[out=0, in=-135] (5*1.41/2, 5*1.41/2);
	\draw[knot, overcross] (2,2) to[out=-90, in=45] (-5*1.41/2, -5*1.41/2);
	\draw[knot, overcross] (5*1.41/2, -5*1.41/2) to[out=135, in=-90] (-2,2);
	\draw[knot, overcross] (-2,2) to[out=90, in=180] (0,4);
	\draw[knot, overcross] (-5*1.41/2, 5*1.41/2) to[out=-45, in=180] (0, 1.75);
	\draw[knot, overcross] (0,4) to[out=0, in=90] (2,2);
	\draw[white, ultra thick] (0,0) circle (5cm);
	\draw[dotted] (0,0) circle (5cm);
	\fill[white, very thick] (2, 2.25) circle (.45cm);
       \draw[knot] (2.375, 2.51) to[out=210,in=90] (1.995, 1.8);
        \draw[knot] (1.88, 2.685) to[out=-75,in=25] (1.58, 2.055);
	\fill[white, very thick] (0,-1.15) circle (.45cm);
        \draw[knot] (-0.4,-1.38) to[out=30,in=-33] (-0.375,-0.875);
        \draw[knot] (0.4,-1.38) to[out=150,in=213] (0.375,-0.875);
	\fill[white, very thick] (-2, 2.25) circle (.45cm);
        \draw[knot] (-2.375, 2.51) to[out=-30,in=90] (-1.995, 1.8);
        \draw[knot] (-1.88, 2.685) to[out=-105,in=155] (-1.58, 2.055);
}$};
\node(*BB) at (15, -6.5) {$\tikz[baseline={([yshift=-.5ex]current bounding box.center)}, scale=.2]{
	\node at (0,-9) {\small$\begin{gathered} *BB \\ w(U_5) = -1 \\ (u,v) = (1,1)\end{gathered}$};
	\draw[knot, overcross] (0,1.75) to[out=0, in=-135] (5*1.41/2, 5*1.41/2);
	\draw[knot, overcross] (2,2) to[out=-90, in=45] (-5*1.41/2, -5*1.41/2);
	\draw[knot, overcross] (5*1.41/2, -5*1.41/2) to[out=135, in=-90] (-2,2);
	\draw[knot, overcross] (-2,2) to[out=90, in=180] (0,4);
	\draw[knot, overcross] (-5*1.41/2, 5*1.41/2) to[out=-45, in=180] (0, 1.75);
	\draw[knot, overcross] (0,4) to[out=0, in=90] (2,2);
	\draw[white, ultra thick] (0,0) circle (5cm);
	\draw[dotted] (0,0) circle (5cm);
	\fill[white, very thick] (0,-1.15) circle (.45cm);
        \draw[knot] (-0.4,-1.38) to[out=30,in=-33] (-0.375,-0.875);
        \draw[knot] (0.4,-1.38) to[out=150,in=213] (0.375,-0.875);
	\fill[white, very thick] (-2, 2.25) circle (.45cm);
        \draw[knot] (-2.375, 2.51) to[out=-30,in=90] (-1.995, 1.8);
        \draw[knot] (-1.88, 2.685) to[out=-105,in=155] (-1.58, 2.055);
	\begin{scope}[red]
		\fill (0, 3) circle (.3cm);
		\fill (-3, 0) circle (.3cm);
		\draw[dotted, rounded corners = 1mm] (-3, 0) -- (-5,0);
		\draw[dotted, rounded corners = 1mm] (0,3) -- (-2.05,2.25) -- (-3, 0);
		\draw (-3, 0) -- (-5*0.866, -5*0.5);
		\draw[rounded corners = 1mm] (0,3) -- (2.05,2.25) -- (3, 2) -- (5*0.866, 5*0.5);
		\draw[dotted, rounded corners] (-3, 0) -- (0,-1.225) -- (5,0);
	\end{scope}
	\draw[white, ultra thick] (0,0) circle (5cm);
	\draw[dotted] (0,0) circle (5cm);
}$};
\draw[->, thick] (AAA) -- node[above, midway, sloped]{} (BAA);
\draw[->, thick] (AAA) -- node[above, midway, sloped]{} (ABA);
\draw[->, thick] (AAA) -- node[above, midway, sloped]{} (AAB);
\draw[->, thick, dotted] (ABA) -- (BBA);
\draw[->, thick, dotted] (ABA) -- (ABB);
\draw[->, thick, dotted] (BAA) -- (BBA);
\draw[->, line width=3pt, white] (BAA) -- (BAB);
\draw[->, thick] (BAA) -- node[above, pos=0.3, sloped]{} (BAB);
\draw[->, thick, dotted] (AAB) -- (ABB);
\draw[->, line width=3pt, white] (AAB) -- (BAB);
\draw[->, thick] (AAB) -- node[above, pos=0.3, sloped]{}(BAB);
\draw[->, thick] (BBA) -- node[above, midway, sloped]{} (BBB);
\draw[->, thick, dotted] (BAB) -- (BBB);
\draw[->, thick] (ABB) -- node[above, midway, sloped]{} (BBB);
\draw[arrows = {Hooks[right]->}, thick] (*AA) -- (AAA);
\draw[white, line width=3pt] (*ABA) to[out=135, in=-135] (ABA);
\draw[arrows = {Hooks[right]->}, thick] (*ABA) to[out=135, in=-135] (ABA);
\draw[white, line width=3pt] (*BBA) to[out=120, in=-120] (BBA);
\draw[arrows = {Hooks[right]->}, thick] (*BBA) to[out=120, in=-120] (BBA);
\draw[white, line width=3pt] (*AB) to[out=60, in=-30] (BAB);
\draw[arrows = {Hooks[left]->}, thick] (*AB) to[out=60, in=-30] (BAB);
\draw[arrows = {Hooks[right]->}, thick] (*BB) -- (BBB);
}
\caption{An illustration of $\iota: \widetilde{C}(U_i) \to \widetilde{C}(D)$ for each twisted unknot in $\mathcal{U}(D)$ for $D$ our typical diagram of $3_1$. The image of $Z_{U_i}$ is the enhanced state indicated in yellow.}
\label{fig:cubeandinclusion}
\end{figure}
Remembering that the bidegree of each fundamental cycle is in bidegree $(0,-1)$, the reader is invited to verify that the grading-shift is as described in (\ref{eq:gradingshift}). 
\end{example}

Now, $\mathcal{C}(D)$ is generated by $\mathscr{T}(D)$ which, by Lemma \ref{lem:champkof1}, is in 1-1 correspondence with $\mathscr{U}(D)$. As hinted in Example \ref{ex:mainplusinclusion}, we can embed $\phi: \mathcal{C}(D) \hookrightarrow \widetilde{C}(D)$ by $T \mapsto \iota(Z_{\mathcal{U}(T)})$. From (\ref{eq:enhancedstatesgrading}), we can use the same argument as \cite[Proof of Theorem 3]{MR2480298} to verify that if $T\subset G_\pm$ has bidegree $(u,v)$, then the $(i,j)$-degree of $\phi(T)$ is given by
\begin{equation}
\label{eq:treegradings}
i = \pm(u - 2v) + \frac{w(D) \pm k_{\pm}}{2} \qquad \text{and} \qquad j = \pm(2u - 2v) + \frac{3w(D) \pm (k_\pm +2)}{2}.
\end{equation}

Finally, let $\widetilde{U}_k$ denote $\iota(\widetilde{C}(\mathcal{U}(T_k))) \subset \widetilde{C}(D)$. The idea, described in detail in \cite{MR2480298}, is to perform Gaussian elimination \cite{MR2320156} (or \emph{elementary collapses} in the language of \cite{MR2480298}) on $\widetilde{U}_k$, taking it to its fundamental cycle. By \cite[Lemma 4]{MR2480298}, there is a sequence of Gaussian eliminations which are deformation retractions $r_k: C_0 \to C_k$ with respect to $f_{\widetilde{U}_k}$. The result is a complex $C_s$ whose generators are in bijective correspondence with the spanning trees of $G$, and $\widetilde{H}^{i,j}(C_s) \cong \widetilde{\Kh}^{i,j}(D)$. Then the map
\begin{equation}\label{eq:isom}
r_s \circ \phi : \mathcal{C}(D) \to C_s
\end{equation}
is an isomorphism of bigraded abelian groups. Recall that $\{\widetilde{C}(D), d\}$ splits as a complex if $D$ is a non-local diagram for a nullhomologous link. Since $d$ has bidegree $(1,0)$, Equation (\ref{eq:treegradings}) informs us that the induced differential on $\mathcal{C}(D)$ has bidegree $(-1,-1)$ on generators in $\mathcal{C}_+(D)$ and $(1,1)$ on generators in $\mathcal{C}_-(D)$. Then, the actual deformation retraction $\widetilde{C}(D) \to \mathcal{C}(D)$ is given by $(r_s \circ \phi)^{-1} \circ r_s$. Summarizing, 

\begin{theorem}
Assume that $D$ is a non-local diagram of a nullhomologous link in $\RP^3$. Then there is a spanning tree complex $\mathcal{C}(D) = \{\mathcal{C}_+(D) \oplus \mathcal{C}_-(D), \partial_+ \oplus \partial_-\}$ with $\partial_+$ of bidegree $(-1,-1)$ and $\partial_-$ of bidegree $(1,1)$ which is a deformation retract of the Khovanov complex $\widetilde{C}(D)$.
\end{theorem}

\begin{example}
The isomorphism (\ref{eq:isom}) is apparent in Figure \ref{fig:cubeandinclusion} of Example \ref{ex:mainplusinclusion}. Indeed, there is a sequence of Gaussian eliminations which eliminate all non-highlighted generators.
\end{example}

\subsection{The Champanerkar-Kofman spectral sequence}

Define $\rho:\mathscr{P}(D) \to \widetilde{C}(D)$ by $\rho(T) = \sum_{T \ge T_i} \widetilde{U}_i$. Champanerkar-Kofman use this map to define a decreasing linearly ordered filtration on $\widetilde{C}(D)$. Start by enumerating the maximal descending ordered sequences of spanning trees in $\mathscr{P}(D)$, and let $T_k^j$ denote the $k$th element of the $j$th such sequence. Define $F^p\widetilde{C}(D) := \sum_j \rho(T_p^j)$. By \cite[Lemma 2]{MR2480298}, $\partial F^p \subseteq F^p$, so the filtration determines a spectral sequence, and
\[
E_0^{p, *} = F^p \widetilde{C}(D) / F^{p+1}\widetilde{C}(D) = \bigoplus_k \widetilde{U}_k.
\]

\begin{theorem}[C.f. Theorem 5 of \cite{MR2480298}]
\label{thm:ckspectral}
For any non-local diagram $D \subset \RP^2$ of a nullhomologous link $L \subset \RP^3$, there is a spectral sequence $E_r^{*, *}$ which converges to $\widetilde{\Kh}(L; \F)$ such that
\begin{enumerate}
\item $E_1^{*,*}$ is isomorphic to the spanning tree complex $\mathcal{C}^{*,*}(D)$, and
\item there is some $r \le c(D)$ for which the spectral sequence collapses.
\end{enumerate}
\end{theorem}

Since the proof of Theorem \ref{thm:ckspectral} is completely derivative of the proof of \cite[Theorem 5]{MR2480298}, we content ourselves with an example. Let $D$ be the 4-crossing diagram for the knot $3_1$ pictured in Figure \ref{fig:finalexample}. We also list each Tait graph of $D$ and their spanning trees. Notice that $w(D) = 0$, so we pick $G_+$ and $G_-$ arbitrarily.

\begin{figure}[ht]
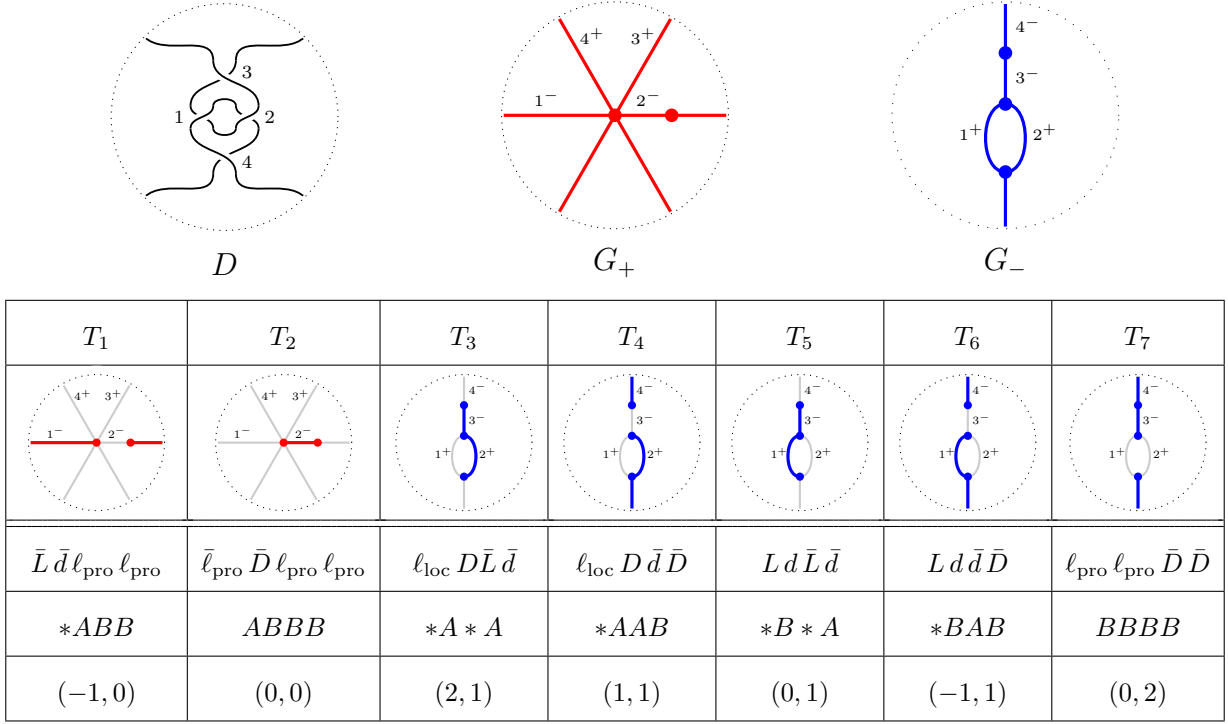

\[
\tikz[baseline={([yshift=-.5ex]current bounding box.center)}, scale=0.3]
{
\draw[knot, looseness=1] (0,0) to[out=90, in=-90] (2,2);
\draw[knot, looseness=1] (0,2) to[out=90, in=-90] (-1,3);
\draw[knot, looseness=1] (2,2) to[out=90, in=-90] (1,3);
\draw[knot, looseness=1] (-1,3) to[out=90, in=-90] (1,5);
\draw[knot, overcross, looseness=1] (1,0) to[out=90, in=-90] (-1,2);
\draw[knot, overcross, looseness=1] (-1,2) to[out=90, in=-90] (0,3);
\draw[knot, overcross, looseness=1] (1,2) to[out=90, in=-90] (2,3);
\draw[knot, overcross, looseness=1] (2,3) to[out=90, in=-90] (0,5);
\draw[knot, looseness=1] (0,2) to[out=-90,in=-90] (1,2);
\draw[knot, looseness=1] (0,3) to[out=90,in=90] (1,3);
\draw[knot] (0.5+5*0.707, 2.5+5*0.707) to[out=-135, in=90] (1,5);
\draw[knot] (0.5-5*0.707, 2.5+5*0.707) to[out=-45, in=90] (0,5);
\draw[knot] (0.5+5*0.707, 2.5-5*0.707) to[out=135, in=-90] (1,0);
\draw[knot] (0.5-5*0.707, 2.5-5*0.707) to[out=45, in=-90] (0,0);
\node at (-1.5, 2.5) {\scriptsize$1$};
\node at (2.5, 2.5) {\scriptsize$2$};
\node at (1.5,4.5) {\scriptsize$3$};
\node at (1.5,0.5) {\scriptsize$4$};
\draw[white, ultra thick] (0.5,2.5) circle (5cm);
\draw[dotted] (0.5,2.5) circle(5cm);
\node at (0.5,-4) {\Large$D$};
}
\qquad \qquad \qquad 
\tikz[baseline={([yshift=-.5ex]current bounding box.center)}, scale=0.3]
{
\node at (-2.5,3.25) {\tiny$1^-$};
\node at (2,3.25) {\tiny$2^-$};
\node at (1.75,6) {\tiny$3^+$};
\node at (-0.5,6) {\tiny$4^+$};
\draw[red, very thick] (0.5-5*0.5, 2.5-5*0.86602540378) -- (0.5+5*0.5, 2.5+5*0.86602540378);
\draw[red, very thick] (0.5-5*0.5, 2.5+5*0.86602540378) -- (0.5+5*0.5, 2.5-5*0.86602540378);
\draw[red, very thick] (-4.5, 2.5) -- (5.5, 2.5);
\fill[red] (0.5,2.5) circle (.3cm);
\fill[red] (3,2.5) circle (.3cm);
\draw[white, ultra thick] (0.5,2.5) circle (5cm);
\draw[dotted] (0.5,2.5) circle(5cm);
\node at (0.5,-4) {\Large$G_+$};
}
\qquad \qquad \qquad 
\tikz[baseline={([yshift=-.5ex]current bounding box.center)}, scale=0.3]
{
\node at (-1,1.75) {\tiny$1^+$};
\node at (2.25,1.75) {\tiny$2^+$};
\node at (1.5,4.25) {\tiny$3^-$};
\node at (1.5,6.5) {\tiny$4^-$};
\draw[blue, very thick] (0.5,3) -- (0.5,7.5);
\draw[blue, very thick] (0.5,-2.5) -- (0.5,0);
\draw[blue, very thick] (0.5,0) to[out=180,in=180] (0.5,3);
\draw[blue, very thick] (0.5,0) to[out=0,in=0] (0.5,3);
\fill[blue] (0.5,0) circle(.3cm);
\fill[blue] (0.5,3) circle(.3cm);
\fill[blue] (0.5,5.25) circle(.3cm);
\draw[white, ultra thick] (0.5,2.5) circle (5cm);
\draw[line cap=round, dash pattern=on 0pt off 3.5pt] (0.5,2.5) circle(5cm);
\node at (0.5,-4) {\Large$G_-$};
}
\]
\[
\renewcommand{\arraystretch}{2}
\begin{tabular}{|c|c|c|c|c|c|c|} 
\hline
$T_1$                                                              & $T_2$                                                                                & $T_3$                                         & $T_4$                                            & $T_5$                          & $T_6$                          & $T_7$                                                               \\ 
\hline
$\tikz[baseline={([yshift=-.5ex]current bounding box.center)}, scale=0.18]
{
\node[scale=0.7] at (-2.5,3.25) {\tiny$1^-$};
\node[scale=0.7] at (2,3.25) {\tiny$2^-$};
\node[scale=0.7] at (1.75,6) {\tiny$3^+$};
\node[scale=0.7] at (-0.5,6) {\tiny$4^+$};
\draw[black!20!white, thick] (0.5-5*0.5, 2.5-5*0.86602540378) -- (0.5+5*0.5, 2.5+5*0.86602540378);
\draw[black!20!white, thick] (0.5-5*0.5, 2.5+5*0.86602540378) -- (0.5+5*0.5, 2.5-5*0.86602540378);
\draw[red, very thick] (-4.5, 2.5) -- (0.5, 2.5);
\draw[black!20!white, thick] (0.5, 2.5) -- (3, 2.5);
\draw[red, very thick] (3, 2.5) -- (5.5, 2.5);
\fill[red] (0.5,2.5) circle (.3cm);
\fill[red] (3,2.5) circle (.3cm);
\draw[white, ultra thick] (0.5,2.5) circle (5.5cm);
\draw[white, ultra thick] (0.5,2.5) circle (5cm);
\draw[dotted] (0.5,2.5) circle(5cm);
}$
& $\tikz[baseline={([yshift=-.5ex]current bounding box.center)}, scale=0.18]
{
\node[scale=0.7] at (-2.5,3.25) {\tiny$1^-$};
\node[scale=0.7] at (2,3.25) {\tiny$2^-$};
\node[scale=0.7] at (1.75,6) {\tiny$3^+$};
\node[scale=0.7] at (-0.5,6) {\tiny$4^+$};
\draw[black!20!white, thick] (0.5-5*0.5, 2.5-5*0.86602540378) -- (0.5+5*0.5, 2.5+5*0.86602540378);
\draw[black!20!white, thick] (0.5-5*0.5, 2.5+5*0.86602540378) -- (0.5+5*0.5, 2.5-5*0.86602540378);
\draw[black!20!white, thick] (-4.5, 2.5) -- (0.5, 2.5);
\draw[red, very thick] (0.5, 2.5) -- (3, 2.5);
\draw[black!20!white, thick] (3, 2.5) -- (5.5, 2.5);
\fill[red] (0.5,2.5) circle (.3cm);
\fill[red] (3,2.5) circle (.3cm);
\draw[white, ultra thick] (0.5,2.5) circle (5cm);
\draw[dotted] (0.5,2.5) circle(5cm);
}$
& $\tikz[baseline={([yshift=-.5ex]current bounding box.center)}, scale=0.18]
{
\node[scale=0.7] at (-1,1.75) {\tiny$1^+$};
\node[scale=0.7] at (2.25,1.75) {\tiny$2^+$};
\node[scale=0.7] at (1.5,4.25) {\tiny$3^-$};
\node[scale=0.7] at (1.5,6.5) {\tiny$4^-$};
\draw[blue, very thick] (0.5,3) -- (0.5,5.25); 
\draw[black!20!white, thick] (0.5,5.25) -- (0.5,7.5); 
\draw[black!20!white, thick] (0.5,-2.5) -- (0.5,0); 
\draw[black!20!white, thick] (0.5,0) to[out=180,in=180] (0.5,3); 
\draw[blue, very thick] (0.5,0) to[out=0,in=0] (0.5,3); 
\fill[blue] (0.5,0) circle(.3cm);
\fill[blue] (0.5,3) circle(.3cm);
\fill[blue] (0.5,5.25) circle(.3cm);
\draw[white, ultra thick] (0.5,2.5) circle (5cm);
\draw[dotted] (0.5,2.5) circle(5cm);
}$
& $\tikz[baseline={([yshift=-.5ex]current bounding box.center)}, scale=0.18]
{
\node[scale=0.7] at (-1,1.75) {\tiny$1^+$};
\node[scale=0.7] at (2.25,1.75) {\tiny$2^+$};
\node[scale=0.7] at (1.5,4.25) {\tiny$3^-$};
\node[scale=0.7] at (1.5,6.5) {\tiny$4^-$};
\draw[black!20!white, thick] (0.5,3) -- (0.5,5.25); 
\draw[blue, very thick] (0.5,5.25) -- (0.5,7.5); 
\draw[blue, very thick] (0.5,-2.5) -- (0.5,0); 
\draw[black!20!white, thick] (0.5,0) to[out=180,in=180] (0.5,3); 
\draw[blue, very thick] (0.5,0) to[out=0,in=0] (0.5,3); 
\fill[blue] (0.5,0) circle(.3cm);
\fill[blue] (0.5,3) circle(.3cm);
\fill[blue] (0.5,5.25) circle(.3cm);
\draw[white, ultra thick] (0.5,2.5) circle (5cm);
\draw[dotted] (0.5,2.5) circle(5cm);
}$
& $\tikz[baseline={([yshift=-.5ex]current bounding box.center)}, scale=0.18]
{
\node[scale=0.7] at (-1,1.75) {\tiny$1^+$};
\node[scale=0.7] at (2.25,1.75) {\tiny$2^+$};
\node[scale=0.7] at (1.5,4.25) {\tiny$3^-$};
\node[scale=0.7] at (1.5,6.5) {\tiny$4^-$};
\draw[blue, very thick] (0.5,3) -- (0.5,5.25); 
\draw[black!20!white, thick] (0.5,5.25) -- (0.5,7.5); 
\draw[black!20!white, thick] (0.5,-2.5) -- (0.5,0); 
\draw[blue, very thick] (0.5,0) to[out=180,in=180] (0.5,3); 
\draw[black!20!white, thick] (0.5,0) to[out=0,in=0] (0.5,3); 
\fill[blue] (0.5,0) circle(.3cm);
\fill[blue] (0.5,3) circle(.3cm);
\fill[blue] (0.5,5.25) circle(.3cm);
\draw[white, ultra thick] (0.5,2.5) circle (5cm);
\draw[dotted] (0.5,2.5) circle(5cm);
}$
& $\tikz[baseline={([yshift=-.5ex]current bounding box.center)}, scale=0.18]
{
\node[scale=0.7] at (-1,1.75) {\tiny$1^+$};
\node[scale=0.7] at (2.25,1.75) {\tiny$2^+$};
\node[scale=0.7] at (1.5,4.25) {\tiny$3^-$};
\node[scale=0.7] at (1.5,6.5) {\tiny$4^-$};
\draw[black!20!white, thick] (0.5,3) -- (0.5,5.25); 
\draw[blue, very thick] (0.5,5.25) -- (0.5,7.5); 
\draw[blue, very thick] (0.5,-2.5) -- (0.5,0); 
\draw[blue, very thick] (0.5,0) to[out=180,in=180] (0.5,3); 
\draw[black!20!white, thick] (0.5,0) to[out=0,in=0] (0.5,3); 
\fill[blue] (0.5,0) circle(.3cm);
\fill[blue] (0.5,3) circle(.3cm);
\fill[blue] (0.5,5.25) circle(.3cm);
\draw[white, ultra thick] (0.5,2.5) circle (5cm);
\draw[dotted] (0.5,2.5) circle(5cm);
}$
& $\tikz[baseline={([yshift=-.5ex]current bounding box.center)}, scale=0.18]
{
\node[scale=0.7] at (-1,1.75) {\tiny$1^+$};
\node[scale=0.7] at (2.25,1.75) {\tiny$2^+$};
\node[scale=0.7] at (1.5,4.25) {\tiny$3^-$};
\node[scale=0.7] at (1.5,6.5) {\tiny$4^-$};
\draw[blue, very thick] (0.5,3) -- (0.5,5.25); 
\draw[blue, very thick] (0.5,5.25) -- (0.5,7.5); 
\draw[blue, very thick] (0.5,-2.5) -- (0.5,0); 
\draw[black!20!white, thick] (0.5,0) to[out=180,in=180] (0.5,3); 
\draw[black!20!white, thick] (0.5,0) to[out=0,in=0] (0.5,3); 
\fill[blue] (0.5,0) circle(.3cm);
\fill[blue] (0.5,3) circle(.3cm);
\fill[blue] (0.5,5.25) circle(.3cm);
\draw[white, ultra thick] (0.5,2.5) circle (5cm);
\draw[dotted] (0.5,2.5) circle(5cm);
}$
\\ 
\hhline{|=======|}
$\bar{L} \, \bar{d} \, \ell_{\mathrm{pro}} \, \ell_{\mathrm{pro}}$ & $\bar{\ell}_{\mathrm{pro}} \, \bar{D} \, \ell_{\mathrm{pro}} \, \ell_{\mathrm{pro}}$ & $\ell_{\mathrm{loc}} \, D \bar{L} \, \bar{d}$ & $\ell_{\mathrm{loc}} \, D \, \bar{d} \, \bar{D}$ & $L \, d \, \bar{L} \, \bar{d}$ & $L \, d \, \bar{d} \, \bar{D}$ & $\ell_{\mathrm{pro}} \, \ell_{\mathrm{pro}} \, \bar{D} \, \bar{D}$  \\ 
\hline
$*ABB$                                                             & $ABBB$                                                                               & $*A*A$                                        & $*AAB$                                           & $*B*A$                         & $*BAB$                         & $BBBB$                                                              \\ 
\hline
$(-1,0)$                                                           & $(0,0)$                                                                              & $(2,1)$                                       & $(1,1)$                                          & $(0,1)$                        & $(-1,1)$                       & $(0,2)$                                                             \\
\hline
\end{tabular}
\]
\caption{A four-crossing diagram for $3_1$ (from Drobotukhina's table \cite{MR1296890}) with its Tait graphs and their spanning trees. We list the activity word of each tree, followed by the partial smoothing which determines the corresponding twisted unknot, and the $(u,v)$-grading in the spanning tree complex.}
\label{fig:finalexample}
\end{figure}

Table \ref{tab:Usmooth} tells us how to determine the partial smoothing for the corresponding twisted unknot. From the partial smoothings, we find that there are three maximal sequences in $\mathscr{P}(D)$:
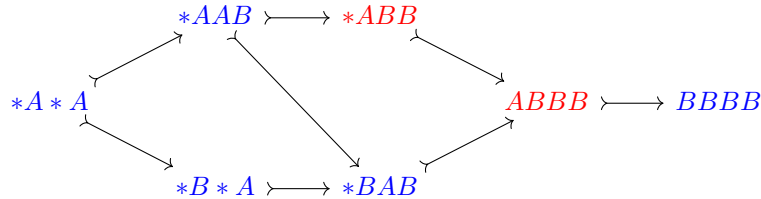
\[
\begin{tikzcd}
& \blue{*AAB} \arrow[ddr, rightarrowtail] \arrow[r, rightarrowtail] & \red{*ABB} \arrow[dr, rightarrowtail] & & 
\\
\blue{*A*A} \arrow[ur, rightarrowtail]  \arrow[dr, rightarrowtail] & & & \red{ABBB} \arrow[r, rightarrowtail] & \blue{BBBB}
\\
& \blue{*B*A} \arrow[r, rightarrowtail] & \blue{*BAB} \arrow[ur, rightarrowtail] & &
\end{tikzcd}
\]
where $X \rightarrowtail Y$ means $X > Y$. Therefore:
\[
E_0^{1,q} = \blue{\widetilde{U}_3}
\qquad
E_0^{2,q} = \blue{\widetilde{U}_4} \oplus \blue{\widetilde{U}_5}
\qquad
E_0^{3,q} = \red{\widetilde{U}_1} \oplus \blue{\widetilde{U}_6}
\qquad
E_0^{4,q} = \red{\widetilde{U}_2}
\qquad
E_0^{5,q} = \blue{\widetilde{U}_7}
\]
Notice that $k_+ = 2$ and $k_- = 4$, so we can compute the $q$-grading via the equalities $p+q = i = u - 2v + 1$ for $T \subset G_+$ and $p+q = i = -u + 2v -2$ for $T \subset G_-$. The $E_1$, $E_2$, and $E_3$ pages of the corresponding spectral sequence are pictured in Figure \ref{fig:finalss}. The spectral sequence collapses on the $E_3$ page. Notice that there is no differential on $E_2$ between $T_5$ and $T_2$ since there are no incidences between $\mathcal{C}_+(D)$ and $\mathcal{C}_-(D)$.

\begin{figure}[h!]
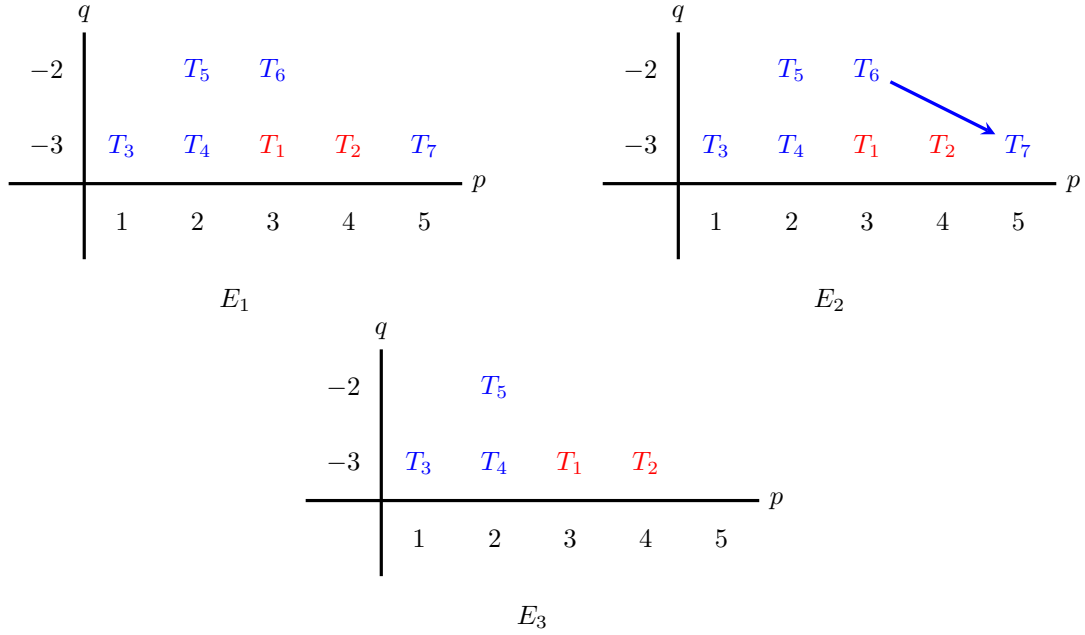

\[
\tikz[baseline={([yshift=-.5ex]current bounding box.center)}, scale=1]
{
\draw[very thick] (-0.5,-3.5) -- node[midway, below=1.25cm]{$E_1$}  (5.5,-3.5) node[anchor=west]{$p$};
\draw[very thick] (0.5,-4.5) -- (0.5, -1.5) node[anchor=south]{$q$};
\node at (1,-3) {$\blue{T_3}$};
\node at (2,-3) {$\blue{T_4}$};
\node at (2,-2) {$\blue{T_5}$};
\node at (3,-3) {$\red{T_1}$};
\node at (3,-2) {$\blue{T_6}$};
\node at (4,-3) {$\red{T_2}$};
\node at (5,-3) {$\blue{T_7}$};
\node at (0,-3) {$-3$};
\node at (0,-2) {$-2$};
\node at (1,-4) {$1$};
\node at (2,-4) {$2$};
\node at (3,-4) {$3$};
\node at (4,-4) {$4$};
\node at (5,-4) {$5$};
}
\qquad\qquad
\tikz[baseline={([yshift=-.5ex]current bounding box.center)}, scale=1]
{
\draw[very thick] (-0.5,-3.5) -- node[midway, below=1.25cm]{$E_2$}  (5.5,-3.5) node[anchor=west]{$p$};
\draw[very thick] (0.5,-4.5) -- (0.5, -1.5) node[anchor=south]{$q$};
\node at (1,-3) {$\blue{T_3}$};
\node at (2,-3) {$\blue{T_4}$};
\node at (2,-2) {$\blue{T_5}$};
\node at (3,-3) {$\red{T_1}$};
\node (X) at (3,-2) {$\blue{T_6}$};
\node at (4,-3) {$\red{T_2}$};
\node (Y) at (5,-3) {$\blue{T_7}$};
\node at (0,-3) {$-3$};
\node at (0,-2) {$-2$};
\node at (1,-4) {$1$};
\node at (2,-4) {$2$};
\node at (3,-4) {$3$};
\node at (4,-4) {$4$};
\node at (5,-4) {$5$};
\draw[-stealth, blue, line width=1.25pt] (X) -- (Y);
}
\]
\vspace{-0.85em}
\[
\tikz[baseline={([yshift=-.5ex]current bounding box.center)}, scale=1]
{
\draw[very thick] (-0.5,-3.5) -- node[midway, below=1.25cm]{$E_3$}  (5.5,-3.5) node[anchor=west]{$p$};
\draw[very thick] (0.5,-4.5) -- (0.5, -1.5) node[anchor=south]{$q$};
\node at (1,-3) {$\blue{T_3}$};
\node at (2,-3) {$\blue{T_4}$};
\node at (2,-2) {$\blue{T_5}$};
\node at (3,-3) {$\red{T_1}$};
\node at (4,-3) {$\red{T_2}$};
\node at (0,-3) {$-3$};
\node at (0,-2) {$-2$};
\node at (1,-4) {$1$};
\node at (2,-4) {$2$};
\node at (3,-4) {$3$};
\node at (4,-4) {$4$};
\node at (5,-4) {$5$};
}
\]
\caption{Pages $E_1$, $E_2$, and $E_3$ of the Champanerkar-Kofman spectral sequence associated to the diagram of $3_1$ in Figure \ref{fig:finalexample}.}
\label{fig:finalss}
\end{figure}

\begin{proof}[Proof of Theorem \ref{thm:main}]
Let $D\subset \RP^2$ be a non-local, reduced, alternating diagram. The Tait graphs determined by $D$ are $G_+$, which has all positive edges, and $G_-$, which has all negative edges. For each spanning tree $T$ of $G_\pm$, we have $v(T) = e_\pm(T) = e(T) = v(G_\pm) - 1$. In other words, the $v$-grading is constant for all the generators of $\mathcal{C}_\pm(D)$. The differential $\partial_{\pm}$ on $\mathcal{C}_\pm(D)$ has degree $(\mp1, \mp1)$, hence it is trivial. From Theorem \ref{thm:ckspectral}, we conclude that $\widetilde{Kh}^{i,j}(D; \F) \cong H^{u,v}(\mathcal{C}(D; \F)) \cong \mathcal{C}^{u,v}(D; \F)$. In addition, writing $J^\pm_L(t) = \sum a_n t^n$, Proposition \ref{prop:gradedeuler} implies that
\[
\dim \mathcal{C}_\pm^{u,v} = \abs{a_{\pm(u - v) + \frac{3w(D) \pm k_{\pm}}{4}}} = \abs{a_{\pm(u-v) + \frac{3w(D) \pm(c(D) + 2v)}{4}}}
\]
using the fact that $e(G_\pm) = c(D)$ and $v(T) = e_\pm(T) = v(G_\pm) -1$. Thus the Betti numbers are determined by the Jones polynomial. Finally, by Proposition \ref{prop:goeritzsig}, $\sigma_\pm(L) = 1 + n_-(D)  - s_B(D^\pm)$. Since $D$ is alternating, $v(G(D^\pm)) = v(G_{\pm}) = v+1$, where $G(D^\pm)$ is the Tait graph of the diagram $D^\pm$ defined in \S \ref{ss:Goeritz}. Thus
\[
v(T) = \frac{c(D) - w(D)}{2} - \sigma_\pm(L)
\]
for any $T \subset G_\pm$.
\end{proof}

\bibliographystyle{alpha}
\bibliography{MyopicTutte}

\end{document}